\documentclass[11pt,reqno]{amsart}
\usepackage[a4paper,margin=30mm,headheight=14pt]{geometry}
\usepackage{amsmath,amssymb,amsthm,mathtools,mathrsfs,booktabs,longtable,array,needspace}
\usepackage{xcolor,enumitem,etoolbox,microtype}
\usepackage[T1]{fontenc}
\usepackage{lmodern}
\usepackage{hyperref}
\usepackage{aliascnt}
\usepackage[nameinlink,capitalize,noabbrev]{cleveref}
\hypersetup{
  colorlinks=true,
  linkcolor=blue!45!black,
  citecolor=blue!45!black,
  urlcolor=blue!45!black,
  pdftitle={On the Allen-Cahn equation},
  pdfauthor={Zhouyu Long and Wenming Zou},
  pdfsubject={Exterior-stable Allen-Cahn interfaces, natural energy growth and finite Morse index},
  pdfkeywords={Allen-Cahn equation, exterior stability, finite Morse index, natural energy growth, De Giorgi conjecture}
}
\numberwithin{equation}{section}
\newtheorem{theorem}{Theorem}[section]
\theoremstyle{plain}
\newaliascnt{lemma}{theorem}
\newtheorem{lemma}[lemma]{Lemma}
\aliascntresetthe{lemma}
\theoremstyle{plain}
\newaliascnt{proposition}{theorem}
\newtheorem{proposition}[proposition]{Proposition}
\aliascntresetthe{proposition}
\theoremstyle{plain}
\newaliascnt{corollary}{theorem}
\newtheorem{corollary}[corollary]{Corollary}
\aliascntresetthe{corollary}
\theoremstyle{definition}
\newaliascnt{definition}{theorem}

\aliascntresetthe{definition}
\theoremstyle{remark}
\newaliascnt{remark}{theorem}
\newtheorem{remark}[remark]{Remark}
\aliascntresetthe{remark}
\theoremstyle{remark}
\newaliascnt{example}{theorem}

\aliascntresetthe{example}
\theoremstyle{plain}
\crefname{theorem}{Theorem}{Theorems}
\Crefname{theorem}{Theorem}{Theorems}
\crefname{lemma}{Lemma}{Lemmas}
\Crefname{lemma}{Lemma}{Lemmas}
\crefname{proposition}{Proposition}{Propositions}
\Crefname{proposition}{Proposition}{Propositions}
\crefname{corollary}{Corollary}{Corollaries}
\Crefname{corollary}{Corollary}{Corollaries}
\crefname{definition}{Definition}{Definitions}
\Crefname{definition}{Definition}{Definitions}
\crefname{remark}{Remark}{Remarks}
\Crefname{remark}{Remark}{Remarks}
\crefname{example}{Example}{Examples}
\Crefname{example}{Example}{Examples}
\theoremstyle{plain}
\newtheorem*{mainresult}{Theorem A}
\newtheorem*{mainrigidity}{Corollary B}
\newtheorem*{maininteraction}{Theorem C}
\setlist{itemsep=3pt,topsep=5pt}
\newcommand{\R}{\mathbb R}
\newcommand{\Sph}{\mathbb S}

\newcommand{\fint}{\mathop{\mathchoice{\displaystyle\int\mkern-18mu-}{\textstyle\int\mkern-15mu-}{\scriptstyle\int\mkern-12mu-}{\scriptscriptstyle\int\mkern-10mu-}}\nolimits}
\allowdisplaybreaks[1]
\title[Energy growth and finite-index rigidity for Allen--Cahn]{$\text{\bf\Large  On the Allen--Cahn equation}$}
\author[Zhouyu Long]{\bfseries Zhouyu Long$^{*}$}
\thanks{$^{*}$Department of Mathematical Sciences, Tsinghua University,
Beijing 100084, China.\newline
E-mail: \href{mailto:longzy25@mails.tsinghua.edu.cn}{\nolinkurl{longzy25@mails.tsinghua.edu.cn}}}
\author[Wenming Zou]{\bfseries Wenming Zou$^{\dagger}$}
\thanks{$^{\dagger}$Department of Mathematical Sciences, Tsinghua University,
Beijing 100084, China.\newline
E-mail: \href{mailto:zou-wm@mail.tsinghua.edu.cn}{\nolinkurl{zou-wm@mail.tsinghua.edu.cn}}}

\makeatletter
\patchcmd{\@setauthors}{\centering\footnotesize}{\centering\large}{}%
  {\ClassError{amsart}{Unable to set author font}{Check the amsart author layout.}}
\patchcmd{\@setauthors}{\MakeUppercase{\authors}}{\authors}{}%
  {\ClassError{amsart}{Unable to preserve author case}{Check the amsart author layout.}}
\def\author@andify{\nxandlist{\unskip,\quad}{\unskip,\quad}{\unskip,\quad}}
\let\addresses\@empty
\makeatother
\date{}

\begin{document}
\raggedbottom
\maketitle
\vskip0.6in
\begin{center}
\begin{minipage}{130mm}
\textbf{Abstract:} Using the framework developed by
OpenAI~\cite{OpenAI2026}, we prove that finite Morse index implies
$O(R^{n-1})$ energy growth for bounded entire solutions of the quartic
Allen--Cahn equation in dimensions $4\leq n\leq9$, uniformly in the
center of the ball. Combined with the cited finite-index
classification results, this gives planar rigidity for $4\leq n\leq7$
and, for nonconstant solutions, multiplicity-one stable minimal cone
blowdowns through dimension nine. Within the framework, we derive an
additive adjacent-layer interaction estimate, a weighted curvature
and defect estimate with exponent $p=9/4$, and the localization needed
for exterior stability. The limiting interaction coefficient $48$ is
sharp in the unrestricted low-defect class with sufficiently large
logarithmic selection depth.

\vskip0.23in
\textbf{Keywords:} Allen--Cahn equation, exterior stability, finite Morse
index, natural energy growth, De Giorgi conjecture, interaction estimates.
\vskip0.231in
\textbf{MSC Classification (2020):} Primary 35J61; Secondary 35B35, 35B40, 53A10.
\vskip0.23in
\end{minipage}
\end{center}
\newpage
\tableofcontents
\newpage

\section{\bf Introduction}
\label{sec:s12-introduction}

We consider bounded entire solutions of the Allen--Cahn equation
\[
 \Delta u=u^3-u\qquad\hbox{in }\mathbb R^n.
\]
For the quartic potential $W(s)=(1-s^2)^2/4$, the associated energy is
\[
 E_u(A)=\int_A\left(\tfrac12|\nabla u|^2+W(u)\right)\,dx.
\]
The energy of a one-dimensional transition grows like the area of a
hypersurface. This surface-order scaling also underlies the connection
between diffuse-interface energies and perimeter~\cite{Modica1987}.
For general critical points, the varifold compactness theory starts from
a bound on the rescaled energies~\cite{HutchinsonTonegawa2000,TonegawaWickramasekera2012}.
The question addressed here is whether finite Morse index implies
\[
 E_u(B_R(x_c))\leq C_uR^{n-1}\qquad(R\geq1),
\]
with a constant independent of the observation center $x_c$.

\vskip0.12in
The relation with minimal hypersurfaces has long guided the study of
entire solutions. De Giorgi's conjecture asks whether bounded solutions
that are strictly monotone in one direction have planar level sets
when $n\leq8$~\cite{DeGiorgi1979}. Ghoussoub--Gui~\cite{GhoussoubGui1998} and
Ambrosio--Cabr\'e~\cite{AmbrosioCabre2000} proved the conjecture in
dimensions two and three, respectively. Savin proved the conclusion
through dimension eight under the additional limits
$u(x',x_n)\to\pm1$ as $x_n\to\pm\infty$~\cite{Savin2009}.
Wang subsequently gave a proof of the flat-level-set regularity theorem
based on decay of the tilt excess~\cite{WangSavin2017}.
Nonplanar monotone solutions exist in dimensions $n\geq9$ by
del Pino--Kowalczyk--Wei~\cite{DelPinoKowalczykWei2011}, while
nonplanar stable solutions already occur in dimension
eight~\cite{PacardWei2013,LiuWangWei2017}.

\vskip0.12in
Finite Morse index allows instability in a compact region and leads to
a broader geometric problem. In the plane, Wang--Wei obtained linear
energy growth and finitely many ends by proving curvature decay for
the transition layers~\cite{WangWeiFiniteEnds2019}. In dimension
three, the constructions of del Pino--Kowalczyk--Wei give nonplanar
finite-index solutions near dilations of certain complete embedded
minimal surfaces of finite total curvature, including the
catenoid~\cite{DelPinoKowalczykWei2013}.
In dimension four, Florit-Simon--Serra~\cite{FloritSimonSerra2025}
proved the stable classification under surface-order energy growth.
Under the same energy condition, Florit-Simon~\cite{FloritSimon2026}
extended this classification to finite Morse index and reduced the
problem in dimensions four through seven to the stable classification.
Liu--Luo--Wang--Wei--Wei--Wu~\cite{LiuSheeting2026} subsequently
obtained finite-index classification and blowdown results assuming
surface-order growth. Conjecture E in~\cite{FloritSimon2026} asks whether finite
Morse index alone implies one-dimensionality in dimensions
$4\leq n\leq7$.

\vskip0.12in
Our main result establishes this uniform energy bound in dimensions
$4\leq n\leq9$ for every bounded entire solution with finite Morse
index, or equivalently, stable outside a compact set. Combining it with the finite-index
classification in~\cite{LiuSheeting2026} gives planar rigidity for
$4\leq n\leq7$. For $4\leq n\leq9$, every nonconstant such solution
has multiplicity-one stationary stable minimal cone blowdowns.
Thus the growth estimate supplies the hypothesis needed to pass from
the spectral assumption to these geometric conclusions.

\vskip0.12in
To prove the growth estimate, one has to control the total contribution
of the transition layers on large balls. Individual layers can be
close to one-dimensional profiles even when several layers lie near
one another. Wang--Wei's analysis of clustering interfaces identifies
their interaction through a Toda system; their second-order estimates
start from prior geometric control~\cite{WangWei2019}.
Here the local estimates must be combined with exterior stability
before an area bound is available. In particular, the errors from
exceptional regions and from cutting off near the unstable core must
be controlled at the same scale as the interaction terms.

\vskip0.12in
We first estimate the defect from the one-dimensional profile and
derive an additive interaction bound involving both adjacent layers.
A weighted curvature and defect estimate with exponent $p=9/4$ then
controls the exceptional regions and the cutoff errors. This exponent
leaves a strict margin in the radial comparison at the largest
dimension considered. We use the same amplitude in the penalized-reach
inequality and in the ambient stability test. If the normalized energy
were unbounded, this comparison would give an upper bound for a
weighted interface area incompatible with the lower bound obtained
from the energy and a normal-tube argument.

\vskip0.12in
The main proof framework used in this paper is due to
OpenAI~\cite{OpenAI2026}. Its positive normal equation, penalized
reach, root-sum-squares (RSS) gluing, and radial comparison provide
the architecture of our argument. The present treatment adapts this
framework to exterior stability and develops the adjacent-layer
interaction estimates and the $p=9/4$ localization used below.
The precise main results and proof outline are given in
Section~\ref{sec:s12-main-results}. The sources of the individual
constructions and the external results used in the proof are specified
in Section~\ref{sec:s1-external-inputs}.

\vskip0.62in
\section{\bf Main results}
\label{sec:s12-main-results}

We consider the quartic potential $W(s)=(1-s^2)^2/4$. For a bounded
entire solution, set
\[
 e_u=\tfrac12|\nabla u|^2+W(u),\qquad E_u(A)=\int_A e_u\,dx,
 \qquad Q_u(\xi)=\int_{\mathbb R^n}\bigl(|\nabla\xi|^2+(3u^2-1)\xi^2\bigr)\,dx.
\]
We also write $\mu_u=e_u\,dx$ and $E(u;A)=E_u(A)$. Throughout
the paper, $n=m+1$ denotes the ambient dimension and $m$ the
dimension of the interface.
Positive constants denoted by $C$ and $c$ may change from line to line.
Their allowed dependencies are specified in the corresponding section or
statement. Constants assigned a value or used as fixed parameters keep
that value.

\vskip0.32in
\subsection{Finite Morse index solutions}

The Morse index of $u$ is the maximal dimension of a subspace of
$C_c^\infty(\mathbb R^n)$ on which $Q_u$ is negative definite.
A solution is stable outside a compact set if, for some fixed ball
$B_{R_0}(0)$, $Q_u(\xi)\geq0$ for every
$\xi\in C_c^\infty(\mathbb R^n\setminus\overline B_{R_0}(0))$.
We call this exterior stability.
For the smooth potential $3u^2-1$, bounded below by $-1$, it is
equivalent to finite Morse index; see
\cite[Lemma~1.1 and Theorem~1.2]{Devyver2012} and
Proposition~\ref{thm:r13-rigidity-reduction}. The center of an
observation ball is independent of the center of the fixed ball.

\vskip0.12in
Our first main result is the following energy estimate.

\vskip0.12in
\begin{mainresult}
Let $4\leq n\leq9$. Every bounded entire solution of
$\Delta u=u^3-u$ that is stable outside a fixed ball satisfies
\begin{equation}\label{eq:r11-intro-growth}
 \sup_{x_c\in\R^n}E_u(B_R(x_c))\leq C_uR^{n-1}\qquad(R>0).
\end{equation}
Here $C_u<\infty$ may depend on $n$ and $u$, but is independent of
$R$ and $x_c$.
\end{mainresult}

\vskip0.12in
The proof is given in Section~\ref{sec:r13-actual-growth}. It uses
the local estimates and the weighted stability estimates proved in
Parts~\ref{part:s12-local} and~\ref{part:s12-stability}, respectively.
As a consequence, we obtain the following classification and
blowdown statements.

\vskip0.12in
\begin{mainrigidity}
For $4\leq n\leq7$, every exterior-stable bounded entire solution is
\[
 u\equiv\pm1\quad\hbox{or}\quad
 u(x)=\tanh((e\cdot x-s_0)/\sqrt2),\qquad e\in\Sph^{n-1}.
\]
Its Morse index is zero, and
\[
 E_u(B_R(x_c))\leq\frac{2\sqrt2}{3}\omega_{n-1}R^{n-1}.
\]
For $4\leq n\leq9$, every nonconstant exterior-stable bounded entire
solution has multiplicity-one stationary stable minimal cone blowdowns.
\end{mainrigidity}

\vskip0.12in
Here $\omega_{n-1}$ is the volume of the unit ball in $\mathbb R^{n-1}$.
Corollary B follows from Theorem A and
\cite[Corollary~6.8 and Theorem~1.3]{LiuSheeting2026}. We give the
transfer in Proposition~\ref{thm:r13-rigidity-reduction}. The
conclusions in dimensions eight and nine concern energy growth and
cone compactness; planar classification is excluded by the examples
cited in the introduction.

\vskip0.32in
\subsection{Interaction between adjacent interfaces}

We next state the local estimate used in the proof of Theorem A.
Let $z$ be a regular zero in the buffered low-defect region of
Theorem~\ref{thm:r13-local-master}, where the zero set consists of
ordered graphs. Select any subset of the two first adjacent graphs
whose shortest gaps from $z$ satisfy $\ell_\pm\leq C_0\log N$,
where $N$ is the clean scale and $C_0$ is fixed. The signs refer to the
two sides of the gradient-oriented normal. We use the notation
\[
 \begin{gathered}
 q=1-u^2,\qquad U=\sqrt2\operatorname{arctanh}u,\qquad
 D=1-|\nabla U|^2,\qquad
 \nu(z)=\frac{\nabla U(z)}{|\nabla U(z)|},\\
 A=d\nu|_{T_z\{u=0\}},\qquad H=\operatorname{tr}A,\qquad
 \sigma=\frac{4\sqrt2}{3},\\
 I_P(z)=\frac1\sigma\int_{-P}^P
 q(z+s\nu(z))^3D(z+s\nu(z))\,ds.
 \end{gathered}
\]
These quantities are defined for nonconstant solutions with $|u|<1$;
the constant wells are treated separately. Thus $D$ measures the
defect from the one-dimensional profile, and $I_P$ is its weighted
integral on a fixed normal interval. Our convention for mean curvature
is $H=\operatorname{tr}A$ with $A=d\nu$. The normalization $\sigma$ is
twice the one-dimensional transition energy $\sigma_1=2\sqrt2/3$.

\vskip0.32in

\needspace{10\baselineskip}
\begin{maininteraction}
Fix $m\geq2$. Under the clean adjacent-window hypotheses of
Theorem~\ref{thm:r13-local-master}, let
$\mu_\pm=e^{-\sqrt2\ell_\pm}$ on selected first-neighbor sides,
and set $\mu_\pm=0$ on omitted sides. For every $\varepsilon>0$,
there is $P_\varepsilon<\infty$ such that, for each fixed
$P\geq P_\varepsilon$, a sufficiently large clean scale gives
\begin{equation}\label{eq:r12-intro-diamond}
 |H-12\sqrt2(\mu_+-\mu_-)|+12\sqrt2(\mu_++\mu_-)
 \leq\bigl(1/(2\sqrt2)+\varepsilon\bigr)I_P.
\end{equation}
\end{maininteraction}

\vskip0.12in
In particular, after decreasing the error tolerance, this gives
\[
 I_P\geq(48-\varepsilon)(\mu_++\mu_-),\qquad
 |H|+24\sqrt2\min(\mu_+,\mu_-)
 \leq\bigl(1/(2\sqrt2)+\varepsilon\bigr)I_P.
\]
For a fixed $P>0$, the coefficient on the right of
\eqref{eq:r12-intro-diamond} can be replaced by $C_P+\varepsilon$,
where $C_P=[2\sqrt2\tanh(P/\sqrt2)]^{-1}$. The corresponding local
mass coefficient is $48\tanh(P/\sqrt2)$, up to an arbitrarily small
error at large scale. Both incoming modes are subtracted at the same
base point and on the same normal interval. This is what gives the
sum of the two interactions. The periodic examples in Proposition~\ref{prop:periodic-sharpness48}
give sharpness in its stated parameter range.

\vskip0.32in

\subsection{Outline of the proof}
\label{sec:s12-proof-outline}

For $R\geq1$, write
\[
 K(R)=1+\sup_{x_c\in\mathbb R^n}R^{-m}E_u(B_R(x_c)).
\]
If $K$ is unbounded, choose $R_j\to\infty$ with
$K(R_j)\to\infty$. The elementary bound $K(R)\leq1+CR$ then shows
that $L_j=R_j/\sqrt{K(R_j)}\to\infty$, while $L_j/R_j\to0$.
We derive a contradiction along this sequence. The local analysis controls normal derivatives, contact
points, and adjacent-layer interactions. The weighted $p=9/4$
estimate then controls the errors in the ambient stability test.
Using the same amplitude in this test and in the reach inequality
gives an upper bound for a weighted area. Energy comparison and a
normal-tube argument give a positive lower bound for the same quantity.

\vskip0.12in
The exponent $9/4$ is chosen for a quantitative reason. In the radial
family of Theorem~\ref{thm:s2-exact-feasibility}, the moment condition
is expressed by
\[
 \frac{8-4/p}{m-2}.
\]
For $m=8$, this is $1$ at $p=2$ and $28/27$ at $p=9/4$.
The coefficient required in our comparison satisfies
\[
 \begin{gathered}
 1<g_0\leq\frac{72159}{99800}\sqrt2<1.02253<\frac{28}{27},
 \\g_0=\kappa(1.005)(.718)/.998,\quad
 \kappa=\frac{\sqrt2}{G_1},\quad1\leq G_1\leq1.001.
 \end{gathered}
\]
Thus the larger exponent leaves a strict margin at the largest
dimension considered. This condition belongs to the chosen radial
comparison. Here $G_1$ is the constant value of the angular penalty on
its plateau. The curvature-power argument uses the identities
of~\cite{Simons1968,SchoenSimonYau1975}.

\vskip0.12in
The paper is divided into three parts. In
Part~\ref{part:s12-applications} we prove Theorem A and Corollary B
using the local estimates, the weighted moment estimate, and the
energy-growth criterion stated in the later parts.
Part~\ref{part:s12-local}
contains the preliminary results, the local normal equation, contact
geometry, and the interaction estimate. In
Part~\ref{part:s12-stability} we use exterior stability to obtain
large-scale regularity, prove the $p=9/4$ estimate, and establish the
growth criterion. The isolated-layer mean-curvature calculation and
the scalar parameter calculation are given in the appendices.
Appendix~\ref{sec:r13-final-gap} discusses the dimension restriction
and the additional estimate that would be needed beyond the stated
range.

\vskip0.32in
\part{Finite Morse index solutions}
\label{part:s12-applications}
In this part we prove the main conclusions using the estimates
established in Parts~\ref{part:s12-local} and
\ref{part:s12-stability}. We first derive energy growth from the
quantitative criterion, and then apply the finite-index classification
results.

\vskip0.32in
\section{\bf Energy growth: Proof of Theorem A}
\label{sec:r13-actual-growth}
We prove the energy estimate using the local estimates in
Proposition~\ref{prop:r13-nine-dimensional-local-package}, the weighted
moment estimate in Theorem~\ref{thm:r13-actual-nine-quarter-budget},
and the criterion in Theorem~\ref{thm:r13-moment-radial-master}.
Their proofs are given in Parts~\ref{part:s12-local} and
\ref{part:s12-stability}. The parameter choice required by the
criterion is verified in Corollary~\ref{cor:r13-uniform-growth-certificate}.
Here $K$ is the normalized energy defined in
Section~\ref{sec:s12-proof-outline}.

\vskip0.32in
\begin{corollary}
\label{thm:r13-growth-49}\label{thm:r13-actual-growth}
\label{thm:r11-growth-48}\label{thm:r11-growth-47}
\label{thm:r10-growth-45}\label{thm:r9-actual-growth}\label{thm:r8-main-growth}
Let $4\leq n\leq9$ and let $u:\mathbb R^n\to\mathbb R$ be a
bounded solution of $\Delta u=u^3-u$, stable outside one fixed ball.
Then
\begin{equation}\label{eq:r13-growth-49}
 \sup_{x_c\in\mathbb R^n}E_u(B_R(x_c))\leq C_uR^{n-1}
 \qquad(R>0).
\end{equation}
\end{corollary}
\begin{proof}
The maximum principle gives $|u|\leq1$: translate a sequence
approaching a supremum greater than one and use elliptic compactness
to contradict $\Delta u=u^3-u>0$ at a maximum of the limit.
The infimum is analogous. If either endpoint is attained, the strong maximum principle gives
a constant well. Otherwise $|u|<1$.
The constant zero is not exterior stable, since a sufficiently large
scaled compact test is negative and can be translated outside any
fixed ball. If $u$ is zero-free, apply Lemma~\ref{lem:r4-phase-ball}
to $u$ or $-u$ on arbitrarily large balls to obtain $u\equiv1$ or
$u\equiv-1$. If $D$ vanishes somewhere,
Lemma~\ref{lem:audited-relative-defect} gives a planar heteroclinic,
whose growth follows by slicing.

\vskip0.12in
It remains to verify the radial inputs for a nonplanar solution.
Put $m=n-1$. Proposition~\ref{prop:r13-nine-dimensional-local-package}
supplies the local estimates with $G_1\leq1.001$, $J_*\leq1.01$
and the required common finite normal window. Choose contact
accuracy at most $1/32000$ before the final good threshold.
Theorem~\ref{thm:r13-actual-nine-quarter-budget} and
Corollary~\ref{cor:s8-density-envelope-moment} supply the moment estimate. Their full-power transport and
Theorem~\ref{lem:near-critical-phase-decay} fix the decay
rate. Corollary~\ref{cor:r13-nine-dimensional-fixed-window}
verifies every radial parameter condition.

\vskip0.12in
Theorem~\ref{thm:r13-moment-radial-master} therefore bounds $K$ for $R\geq1$.
For $R<1$, $e_u\leq1/2$ by \eqref{eq:audited-q} gives the same bound after enlarging $C_u$.
\end{proof}

\vskip0.32in

\section{\bf Geometric consequences: Proof of Corollary B}
\label{sec:r10-corollaries}

Theorem A supplies the surface-order energy bound required by the
classification and blowdown results for finite-index solutions.

\begin{proposition}
\label{thm:r13-rigidity-reduction}
\label{cor:r13-blowdowns-through-nine}
Let $4\leq n\leq9$, and let $u$ be a bounded entire solution of
$\Delta u=u^3-u$. Suppose that $u$ is exterior-stable, or has
finite Morse index, and that
\[
 E_u(B_R(0))\leq K R^{n-1}\qquad(R\geq1).
\]
Then $u$ has finite Morse index. If $4\leq n\leq7$, then
\begin{equation}\label{eq:r10-profile-class}
 u\in\mathcal P_n:=\{-1,1\}\cup
 \left\{x\mapsto\tanh\frac{e\cdot x-s_0}{\sqrt2}:
          e\in\mathbb S^{n-1},\ s_0\in\mathbb R\right\}.
\end{equation}
For every nonconstant such $u$, put $m=n-1$ and
$\sigma_1=2\sqrt2/3$. For
$\Phi\in C_c(\mathbb R^n\times G(n,m))$, where $G(n,m)$ is the
Grassmannian of unoriented $m$-planes, define
\[
 V_R(\Phi)=\frac1{\sigma_1R^m}
   \int\Phi\bigl(x/R,(\nabla u(x))^\perp\bigr)|\nabla u(x)|^2\,dx,
 \qquad \mu_R(A)=\frac{E_u(RA)}{\sigma_1R^m}.
\]
Every sequence $R_j\to\infty$ has a subsequence for which
\[
 V_{R_j}\rightharpoonup|C|,\qquad
 \mu_{R_j}\rightharpoonup\mathcal H^m\!\lfloor C,
\]
where $C$ is a nonempty stationary stable minimal hypercone with
multiplicity one. It is a hyperplane if $n\leq7$.
Consequently, Corollary~\ref{thm:r13-actual-growth} gives these
conclusions for every exterior-stable bounded entire solution in
the stated dimensions.
\end{proposition}
\begin{proof}
The maximum-principle reduction in the proof of
Corollary~\ref{thm:r13-actual-growth} leaves the constant wells
or $-1<u<1$; the constant zero violates natural growth.

\vskip0.12in
For $H=-\Delta+3u^2-1$, the potential is smooth and bounded below by
$-1$, so the operator is bounded below on the complete space
$\mathbb R^n$. Exterior stability gives a positive solution of
$H\phi=0$ outside a compact set by
\cite[Lemma~1.1]{Devyver2012}. Extending it as a positive smooth
function across a larger compact set, if necessary, allows
\cite[Theorem~1.2]{Devyver2012} to give finite Morse index. Conversely, finite index implies exterior stability: otherwise, one
constructs arbitrarily many test functions with pairwise disjoint supports
and negative quadratic form.

\vskip0.12in
The quartic potential is smooth and nonnegative, has precisely
the critical points $-1,0,1$, and satisfies
$W''(\pm1)=2$ and $W''(0)=-1$. Thus all potential, dimension,
index and growth assumptions of
\cite[Theorem~1.3]{LiuSheeting2026} hold, with its ambient
dimension $N=n$ and its $\sigma_0=\sigma_1$. That theorem,
together with its Lemma~6.1, gives the simultaneous varifold and
full-energy convergence above. Its Corollary~6.8 gives
\eqref{eq:r10-profile-class} for $4\leq n\leq7$.
For a nonconstant solution, Lemma~\ref{lem:r4-phase-ball} guarantees
at least one zero. The fixed-ball lower energy bound and monotonicity
in Lemmas~\ref{lem:r9-zero-baseline} and~\ref{lem:r9-density-baseline},
followed by ball inclusion, give a positive lower bound for
$R^{-m}E_u(B_R(0))$ at large $R$. Testing the energy-measure convergence
against a nonnegative compactly supported function equal to one on $B_1(0)$ then
shows that $C$ is nonempty.
Assign a fixed $m$-plane on $\{\nabla u=0\}$; its choice does not
affect $V_R$ because the weight vanishes there. The wells have zero
blowdown. The final assertion follows
by inserting Corollary~\ref{thm:r13-actual-growth}.
\end{proof}

\vskip0.12in
\begin{corollary}
\label{cor:s1-routine-consequences}
\label{cor:r9-index-dichotomy}\label{cor:r10-index-dichotomy}
\label{cor:r9-disjoint-instability}\label{cor:r10-disjoint-instability}
\label{cor:r9-sharp-universal-energy}\label{cor:r10-sharp-universal-energy}
\label{cor:r9-finite-index-compactness}\label{cor:r10-finite-index-compactness}
For fixed $4\leq n\leq7$, the following statements hold.
\begin{enumerate}[label=\textup{(\roman*)}]
\item Exterior stability, finite Morse index and global stability
are equivalent for bounded entire solutions. They hold exactly
on $\mathcal P_n$, and every bounded entire solution has index
zero or infinity.
\item If $u\notin\mathcal P_n$, every compact
$K\subset\mathbb R^n$ admits a smooth compactly supported
$\varphi$ outside $K$ with $Q_u(\varphi)<0$. Such tests may be
chosen with disjoint supports escaping to infinity. For every
$N\geq1$, the first $N$ span a space on which $Q_u$ is negative
definite. This applies,
in particular, to a solution with a critical zero, a nonempty
compact zero component, or a nonempty zero set that is not an
affine hyperplane.
\item Every exterior-stable bounded entire solution satisfies
\begin{equation}\label{eq:r10-sharp-universal-energy}
 E_u(B_R(x_c))\leq\sigma_1\omega_{n-1}R^{n-1},\qquad
 \sigma_1=\frac{2\sqrt2}{3},\quad R>0,\quad x_c\in\mathbb R^n.
\end{equation}
The optimal coefficients for $n=4,5,6,7$, respectively, are
\[
 \frac{8\pi\sqrt2}{9},\quad \frac{\pi^2\sqrt2}{3},\quad
 \frac{16\pi^2\sqrt2}{45},\quad \frac{\pi^3\sqrt2}{9}.
\]
At a fixed center the normalized energy tends to zero for a
well and to $\sigma_1\omega_{n-1}$ for a transition.
\item The class of all bounded entire finite-index solutions is compact in
$C^\infty_{\rm loc}(\mathbb R^n)$. Every limit is a well or a single
planar transition.
\end{enumerate}
\end{corollary}
\begin{proof}
Corollary~\ref{thm:r13-actual-growth} and
Proposition~\ref{thm:r13-rigidity-reduction} identify the
exterior-stable class with $\mathcal P_n$.
Finite index implies exterior stability, and the wells are stable.
For a transition, $v=\partial_eu>0$ solves $Hv=0$, so
\[
 Q_u(\varphi)=\int v^2|\nabla(\varphi/v)|^2\geq0.
\]
This proves \textup{(i)}. Its contrapositive gives a negative
test outside every compact set. Successively avoiding the
previous supports proves \textup{(ii)}, since disjoint supports
annihilate all quadratic cross terms.

\vskip0.12in
For \textup{(iii)}, the wells are immediate. For a transition,
write $b=e\cdot x_c-s_0$ and
$e_g(t)=\tfrac12\operatorname{sech}^4(t/\sqrt2)$. Since
$\int_{\mathbb R}e_g=\sigma_1$, slicing gives
\[
 E_u(B_R(x_c))=\omega_{n-1}
   \int_{-R}^R e_g(t+b)(R^2-t^2)^{(n-1)/2}\,dt
 \leq\sigma_1\omega_{n-1}R^{n-1}.
\]
Dominated convergence after $v=t+b$ proves the density limit
and sharpness. Finally, a sequence with infinitely many wells
has a constant subsequence. Otherwise \textup{(iv)} follows
from the transition parameters: pass to $e_j\to e$ and
$s_j\to s\in[-\infty,+\infty]$. A finite limit gives a
transition, while the two infinite limits give wells, locally
smoothly. The locally smooth topology is metrizable.
\end{proof}

\vskip0.32in

\part{Local estimates on interfaces}
\label{part:s12-local}
We begin with the results used from the literature and the equation for
the defect. We then study contact points and adjacent interfaces. Except
for the transfer of exterior stability to compactly supported interface
tests, these estimates follow from the bounded entire equation.

\vskip0.32in
\section{\bf Preliminary results}
\label{sec:s12-preliminaries}
\label{sec:s2-coefficient-theorem}
\label{sec:s1-external-inputs}

The local sheeting result gives the regular graphs used in
Part~\ref{part:s12-stability}. We apply the finite-index classification
in Section~\ref{sec:r10-corollaries}, once the energy estimate is available.

\vskip0.12in
All numbered references to~\cite{LiuSheeting2026} refer to
arXiv:2610.08499v1, dated 6 October 2026. We use two distinct
parts of that paper. Its Theorem~1.1 is a local input to the
growth argument; its finite-index results are used only after
growth has been proved.

\vskip0.12in
For the local input, let $2\leq N\leq10$ and let
$u_j:B_1^N\to[-1,1]$ be stable smooth solutions of
$\varepsilon_j^2\Delta u_j=W'(u_j)$, where
$\varepsilon_j\downarrow0$ and $W$ is the quartic potential.
The hypothesis required by that theorem is
\[
 \varepsilon_j|\nabla u_j|^2\,dx
 \rightharpoonup Q\sigma_1\mathcal H^{N-1}\!\lfloor\{x_N=0\},
 \qquad Q\in\mathbb N,\qquad \sigma_1=\frac{2\sqrt2}{3}.
\]
Here convergence is locally in $B_1^N$.
It gives exactly $Q$ ordered smooth zero graphs, with uniform
interior $C^{2,\theta}$ bounds for every $0<\theta<1$.
Its application requires both a stable observation domain and
this planar measure limit; slab containment alone does not
supply the latter. The near-reach argument below must establish
these hypotheses before invoking sheeting.

\vskip0.12in
The variational link between diffuse-interface energies and perimeter
is classical~\cite{Modica1987}; the convergence framework for critical
phase interfaces is developed in~\cite{HutchinsonTonegawa2000}.
For context, stable phase-interface limit regularity is developed in
\cite{TonegawaWickramasekera2012}, while the second-order transition
layer estimates of \cite{WangWei2019} require prior graph or curvature
control. Our local graph entry is the stated sheeting theorem; the
stronger mean-derivative jets used for reach are proved in the appendix.

\vskip0.12in
The growth-to-geometry implication is proved in Proposition~\ref{thm:r13-rigidity-reduction}:
it combines the exterior positive-solution criterion and finite-index
theorem in~\cite[Lemma~1.1 and Theorem~1.2]{Devyver2012} with
\cite[Theorem~1.3 and Corollary~6.8]{LiuSheeting2026}.

\vskip0.12in
The proof architecture and the principal constructions in the
following sections are taken from~\cite{OpenAI2026}. More precisely,
we adapt its positive response and parity/reach construction
(Sections~5.5--5.6 and~7.3--7.6), its lifted form and exact-rate
comparison (Section~8.2 and Propositions~8.5, 8.10), its RSS
construction (Section~9.2), and its radial comparison
(Sections~10.4--10.8). Within this framework, the arguments below
include the exterior-core cutoff terms and mean-curvature errors,
derive the joint adjacent-layer contribution and its limiting
coefficient $48$, and establish the full-power $p=9/4$ localization.
All references to that work use the repository revision pinned in
the bibliography. Its final classification theorems are not used as
inputs.

\vskip0.12in
\begin{remark}
\label{rem:s1-observation-centers}
Ball inclusion and energy monotonicity
\eqref{eq:audited-monotonicity} give, as extended limits,
\[
 \lim_{r\to\infty}r^{-m}E_u(B_r(x))
 =\lim_{r\to\infty}r^{-m}E_u(B_r(0))=:M_\infty.
\]
If $M_\infty<\infty$, centered monotonicity yields
$E_u(B_R(x))\leq M_\infty R^m$ for all centers and radii.
The proof uses the uniform-center envelope before this finiteness is known.
\end{remark}

\vskip0.32in

\section{\bf Defect estimates and the normal equation}
\label{sec:s12-defect}
Starting from the equation on the whole space, we estimate the defect
and derive its limiting normal equation. Positivity of the limiting
operator then gives estimates for the normal response. We also lift
compactly supported interface functions to transfer exterior stability.
The defect and response estimates apply before any energy-growth or
integrated moment bound has been established.

\vskip0.32in
\subsection{Defect estimates and the density envelope}
\label{sec:r10-analytic-module}\label{sec:r9-baseline-core}
Fix $3\leq m\leq8$, $n=m+1$, and a bounded entire solution of
$\Delta u=u^3-u$. Unless indicated otherwise, constants in this section
depend only on $m$; dependence on windows, tolerances and derivative orders
is recorded explicitly. The analytic estimates below hold in every fixed
finite dimension. Apart from the constant wells,
$|u|<1$ by the strong maximum principle. Put
\[
 g(t)=\tanh(t/\sqrt2),\quad U=g^{-1}(u),\quad q=1-u^2,
 \quad\boldsymbol p=\nabla U,\quad D=1-|\boldsymbol p|^2,
 \quad P=q^3D.
\]
Modica's estimate~\cite{Modica1985} gives $0\leq D\leq1$. Direct calculation
gives
\begin{align}
 \nabla u&=q\boldsymbol p/\sqrt2,&\Delta U&=-\sqrt2uD,&
 e_u&=\tfrac14q^2(2-D),\label{eq:audited-transformed}\\
 2e_u&\leq q^2\leq4e_u,&e_u&\leq\tfrac12,&
 (-\Delta+3u^2-1)q&=-q^2D.\label{eq:audited-q}
\end{align}
Consequently, for compactly supported Lipschitz $\eta$,
\begin{equation}\label{eq:audited-exact-q-test}
 Q_u(q\eta)=\int q^2|\nabla\eta|^2-\int P\eta^2.
\end{equation}
Exterior stability therefore gives $\int P\eta^2\leq\int q^2|\nabla\eta|^2$
whenever the support avoids the unstable core.

\vskip0.12in
\begin{lemma}
\label{lem:audited-relative-defect}\label{lem:r3-defect-jets}
For each fixed $r<\infty$ and integer $k\geq0$, there is a constant
$C_{r,k}$ depending only on $m,r,k$ such that
\begin{equation}\label{eq:r3-fixed-defect-jets}
 \|D\|_{C^k(B_r(z))}\leq C_{r,k}D(z),\qquad
 \|\nabla^2U\|_{C^k(B_r(z))}\leq C_{r,k}\sqrt{D(z)}.
\end{equation}
In particular,
\begin{equation}\label{eq:audited-relative-defect}
 |\nabla^2U|^2\leq CD,\qquad |\nabla D|\leq CD,
 \qquad|\nabla\log q|\leq\sqrt2.
\end{equation}
Either $D\equiv0$ and $u$ is a planar heteroclinic, or $D>0$ and
\[
 e^{-Cr}D(z)\leq D(x)\leq e^{Cr}D(z)\quad(x\in B_r(z)),
 \qquad|\nabla\log P|\leq C.
\]
\end{lemma}
\begin{proof}
We first establish uniform bounds for the positive-order derivatives of $U$.
On $B_2(z)$, the function $U-U(z)$ is bounded by $2$ and solves
$\Delta U=-\sqrt2\tanh(U/\sqrt2)(1-|\nabla U|^2)$.
Interior $W^{2,s}$ estimates with $s>n$ and successive Schauder estimates
give these bounds. Indeed, all derivatives of the nonlinearity are bounded
for $|\nabla U|\leq1$, uniformly in $U(z)$.

\vskip0.12in
Choose a fixed orthonormal basis with $e=\boldsymbol p(z)/|\boldsymbol p(z)|$
if possible, and put $v=\partial_eU$, $w=1-v$,
$t_j=\partial_{e_j}U$, $b=2\sqrt2u\boldsymbol p$. Then
\[
 \Delta w-b\cdot\nabla w=qD(1-w),\quad
 \Delta t_j-b\cdot\nabla t_j=-qDt_j,\quad
 D=2w-w^2-\sum_{j=2}^nt_j^2.
\]
Here $w\geq0$, $D\leq2w$, $w(z)\leq D(z)$ and
$|qD(1-w)|\leq2w$. Harnack~\cite{GilbargTrudinger2001}, applied with
the bounded measurable coefficient $qD(1-w)/w$ (zero at zeros), gives
$\sup_{B_{r+3}(z)}w\leq C_rD(z)$. Since $|t_j|^2\leq2w$,
interior estimates give $\|t_j\|_{C^{k,\gamma}(B_{r+2})}\leq C\sqrt{D(z)}$.
For $w$ use its smooth-coefficient equation
\[
 \Delta w-b\cdot\nabla w-q(1-w)(2-w)w
 =-q(1-w)\sum_{j=2}^nt_j^2.
\]
The right hand side is $O(D(z))$ in every fixed H\"older norm. Therefore
Schauder estimates give $\|w\|_{C^{k,\gamma}(B_{r+1})}\leq CD(z)$.
We now use the displayed identity for $D$ and write the Hessian columns as
$-\nabla w,\nabla t_2,\ldots,\nabla t_n$. This proves
\eqref{eq:r3-fixed-defect-jets}. Also $\nabla q=-\sqrt2uq\boldsymbol p$.
If $D(z)=0$, Harnack forces $w=t_j=0$ on every fixed ball, giving an
affine unit-gradient $U$. Otherwise $|\nabla\log D|\leq C$ integrates
along segments, and $\nabla\log P=3\nabla\log q+\nabla\log D$.
\end{proof}

\begin{lemma}
\label{lem:positive-directional-source}\label{lem:r12-positive-directional-source}
For any fixed unit vector $n_0$, the bounded nonnegative function
$w=q(1-\boldsymbol p\cdot n_0)=q-\sqrt2\,\partial_{n_0}u$ satisfies
\begin{equation}\label{eq:r12-positive-directional-source}
 (-\Delta+2)w=F_{n_0}:=qw+q^2|\boldsymbol p-n_0|^2\geq qw,
 \qquad w=\mathcal K*F_{n_0},
\end{equation}
where $\mathcal K$ is the positive fundamental solution of $-\Delta+2$.
\end{lemma}
\begin{proof}
For $L=-\Delta+2-3q$, $L\partial_{n_0}u=0$, $Lq=-q^2D$ and
$D=2(1-\boldsymbol p\cdot n_0)-|\boldsymbol p-n_0|^2$ give the equation.
Since the solution is bounded, the heat representation gives
$w=\int_0^Te^{-2t}H_t*F_{n_0}\,dt+e^{-2T}H_T*w$.
The last term tends to zero uniformly as $T\to\infty$.
\end{proof}

\begin{lemma}
\label{lem:r11-kernel-horizontal-mass}
Write $a=\sqrt2$. The kernel above satisfies
\begin{equation}\label{eq:r11-exact-horizontal-mass}
 \int_{\mathbb R^m}\mathcal K(\tau,h)\,d\tau=(2a)^{-1}e^{-a|h|}.
\end{equation}
As $\ell\to\infty$, for each fixed $M<\infty$ and uniformly on
$|h-\ell|\leq M$,
\begin{equation}\label{eq:r11-horizontal-truncation}
 \int_{|\tau|\leq\ell^{2/3}}\mathcal K(\tau,h)\,d\tau
 =(1-o(1))(2a)^{-1}e^{-ah}.
\end{equation}
For each fixed $c_0>0$ and all sufficiently large $\ell$,
\begin{equation}\label{eq:massive-coarse-disk}
 \int_{|\tau|\leq c_0\sqrt\ell}\mathcal K(\tau,h)\,d\tau
 \geq c_{m,c_0,M}e^{-\sqrt2\ell}.
\end{equation}
A uniform additive $o(1)$ perturbation of $(\tau,h)$ changes the lower
bound in \eqref{eq:r11-horizontal-truncation} by a factor $1-o(1)$.
\end{lemma}
\begin{proof}
Horizontal integration gives
$\int_0^\infty(4\pi t)^{-1/2}e^{-2t-h^2/(4t)}dt$. This is the
one-dimensional Green function whose derivative has jump $-1$ at zero,
and proves
\eqref{eq:r11-exact-horizontal-mass}. For radial distance $r\geq1$, put $t=rs$ in the full heat formula.
The phase $2s+(4s)^{-1}$ has minimum $a$, quadratic near
$s=(2a)^{-1}$, and grows linearly at infinity and as $1/s$ at zero.
The resulting Gaussian and complementary tail bounds give
\begin{equation}\label{eq:r11-kernel-two-sided}
 cr^{-m/2}e^{-ar}\leq\mathcal K(r)\leq Cr^{-m/2}e^{-ar},\qquad r\geq1.
\end{equation}
For $R=\ell^{2/3}$, the inequalities
$\sqrt{h^2+|\tau|^2}-h\geq c|\tau|^2/h$ when $R<|\tau|\leq h$,
and $\geq c|\tau|$ when $|\tau|>h$, make the omitted mass
$O(e^{-ah}\operatorname{poly}(\ell)e^{-c\ell^{1/3}})$.
On $|\tau|\leq c_0\sqrt\ell$, instead, $|(\tau,h)|\leq\ell+C$;
the lower kernel bound times the disk volume proves
\eqref{eq:massive-coarse-disk}. Finally Harnack and interior derivative
estimates for $(-\Delta+2)\mathcal K=0$ give
$|\nabla\log\mathcal K|\leq C$ for $|x|\geq1$, proving perturbation stability.
\end{proof}

\begin{lemma}
\label{lem:r9-zero-baseline}
Let $\Sigma=\{u=0\}$ and $\Sigma_g=\{u=0,D<\delta\}$, $\delta\leq1/4$.
Every good zero has a unique smooth zero graph in a universal ball
$B_{r_*}$, with uniform fixed-order derivatives and area at most $Cr_*^m$.
On $\Sigma_{\rm reg}$,
\[
 H=\frac{\partial_\nu D}{2(1-D)};
 \qquad a^2\leq CD,\quad |H|+|\nabla_\Sigma H|\leq CD
 \quad\hbox{on }\Sigma_g.
\]
For $r\geq1$,
\[
 \mathcal H^m(\Sigma_g\cap B_r(x))\leq CE(u;B_{r+1}(x))
 \leq C(1+r)^{m+1}.
\]
Every zero, including a singular zero, has energy at least $c$ in its
$r_*$-ball. Thus every unit-separated zero family has at most
$CE(u;B_{r+1}(x))$ points in $B_r(x)$.
\end{lemma}
\begin{proof}
At good zeros, $|\nabla u|\geq\sqrt{3/8}$. The uniform derivative
estimates for entire solutions keep one directional derivative positive
on a fixed ball. Thus the implicit function theorem gives a unique graph
and the stated derivative bounds.
On a regular zero, $\nabla^2U|_{T\Sigma}=|\boldsymbol p|A$ and
$\Delta U=0$, so
$H=-\nabla^2U(\nu,\nu)/|\boldsymbol p|$ and
$\partial_\nu D=-2|\boldsymbol p|\nabla^2U(\nu,\nu)$.
The relative derivative estimates give the assertions on the good set.
On the full regular set, the bound obtained in this way is only
$a^2\leq CD/(1-D)$. At every zero $q=1$;
$|\nabla\log q|\leq\sqrt2$ and $e_u\geq q^2/4$ give the fixed-ball
energy bound. A maximal $r_*/4$-separated good family covers the good
surface by its graph balls, with bounded overlap of smaller energy
balls in $B_{r+1}$. Summation proves the area estimate. Disjoint smaller
balls give the unit-separated counting estimate for arbitrary zeros.
\end{proof}

\begin{lemma}
\label{lem:r9-density-baseline}
Set $M(R,x)=R^{-m}E(u;B_R(x))$ and $K(R)=1+\sup_xM(R,x)$, $R\geq1$.
For each fixed $x$, $M(R,x)$ and $K(R)$ are nondecreasing in $R$, and
\[
 \begin{gathered}
 K(R)\leq1+CR,\quad K(s)\leq C(s/R)K(R)\ (s\geq R\geq1),
 \\
 E(u;B_s(x))\leq K(R)s^m\ (0<s\leq R).
 \end{gathered}
\]
If $K$ is unbounded, every sufficiently large $R$ admits $x_R$ with
$E(u;B_R(x_R))\geq K(R)R^m/4$, and $L_R=R/\sqrt{K(R)}$ obeys
$L_R\geq c\sqrt R\to\infty$, $L_R/R\to0$ and $L_R/\log R\to\infty$.
\end{lemma}
\begin{proof}
The divergence-free stress tensor $e_u\delta_{ij}-u_iu_j$ gives the
classical monotonicity identity~\cite{Modica1989}
\begin{equation}\label{eq:audited-monotonicity}
 RE'(B_R(x))-mE(B_R(x))
 =R\int_{\partial B_R(x)}|\partial_ru|^2
 +\int_{B_R(x)}(W(u)-\tfrac12|\nabla u|^2)\geq0.
\end{equation}
The last inequality follows from Modica's estimate. Integrating $e_u\leq1/2$ over a ball proves the first density bound.
For the other bound on $K$, cover a ball of radius $s$ by
$C(s/R)^{m+1}$ balls of radius $R$ and sum their energies. If $K\to\infty$, choose
$M(R,x_R)\geq(K(R)-1)/2\geq K(R)/4$; all three scale conclusions follow.
\end{proof}

\subsection{The limiting normal equation}
Use $\nu=\nabla U/|\nabla U|$, $A=d\nu$, $a=|A|$ and $H=\operatorname{tr}A$
on the regular zero set, and exclude the planar case. Put
\[
 G(s)=1-g(s)^2,\quad\sigma=\int_{\mathbb R}G^2\,ds=\frac{4\sqrt2}3,
 \quad I_P(z)=\sigma^{-1}\int_{-P}^Pq(z+s\nu)^3D(z+s\nu)\,ds.
\]
The estimates for the normal equation and the lifted test function are based on
\cite[Sections 8.1--8.2]{OpenAI2026}. We include the cutoff near the
physical core in the stability argument below. Fixed-window estimates
may depend on $m$ and on the specified windows and tolerances. We choose
these parameters before the defect threshold.

\begin{lemma}
\label{lem:normalized-defect-limit}
For zeros $z_j$ with $\epsilon_j=D_j(z_j)\downarrow0$, rotate and translate
to $z_j=0$, $\nu_j(0)=e_n$. A subsequence converges smoothly locally:
\begin{gather*}
 U_j\to s,\quad D_j/\epsilon_j\to\mathcal D>0,\quad
 \nabla^2U_j/\sqrt{\epsilon_j}\to\mathcal T(x'),\\
 V_j=(\boldsymbol p_j-e_n)/\sqrt{\epsilon_j}\to V(x'),\quad
 W_j=q_j(1-\boldsymbol p_j\cdot e_n)/\epsilon_j\to W\geq0.
\end{gather*}
Here $\mathcal T$ is tangential, harmonic and trace-free,
$V$ is tangential, $\nabla'V=\mathcal T$, $V(0)=0$,
$\mathcal D(0)=1$, $W(0)=1/2$. With $F=G\mathcal D$,
\begin{equation}\label{eq:normalized-model-sources}
 L_0F=2G|\mathcal T|^2,\quad L_0W=G^2|V|^2,
 \qquad L_0=-\Delta+2-G(s).
\end{equation}
All zeros in each fixed ball form one graph tending to $s=0$, and
\begin{equation}\label{eq:normalized-model-traces}
 \begin{aligned}
 a_j(0)^2/\epsilon_j&\to|\mathcal T(0)|^2,\\
 H_j(0)/\epsilon_j&\to W_s(0)=F_s(0)/2,\\
 \nabla_{\Sigma_j}H_j(0)/\epsilon_j&\to\nabla'F_s(0)/2,\\
 I_{P,j}(0)/\epsilon_j&\to\sigma^{-1}\int_{-P}^PG^2F(0,s)\,ds\\
 &=2\sigma^{-1}\int_{-P}^PG^2W(0,s)\,ds>0
 \end{aligned}
\end{equation}
for each fixed $P>0$.
\end{lemma}
\begin{proof}
The relative derivative estimates give local compactness and $U_j\to s$.
The identities
\[
 (\nabla^2U)\boldsymbol p=-\nabla D/2,\qquad
 (\boldsymbol p\cdot\nabla)\nabla^2U=-(\nabla^2U)^2-\nabla^2D/2
\]
show that $\mathcal T$ is tangential and independent of $s$; differentiating
$\Delta U=-\sqrt2uD$ gives its trace and harmonicity.
Divide the exact defect equation
$(-\Delta+2\sqrt2u\boldsymbol p\cdot\nabla+2q|\boldsymbol p|^2)D
=2|\nabla^2U|^2$ by $\epsilon_j$. Conjugation by $G$, using
$G'=-\sqrt2gG$, $G''=(2-3G)G$, gives the first model equation.
The tangential components of $\boldsymbol p_j$ are $O(\sqrt{\epsilon_j})$;
$p_{j,n}=\sqrt{1-D_j-|\boldsymbol p_j^\top|^2}=1+O(\epsilon_j)$
locally in $C^k$ for each fixed $k$. Symmetry gives
$\partial_s\boldsymbol p_j=O(\epsilon_j)$, hence compactness of $W_j$ and
the assertions on $V$. The equation for $w$ in Lemma~\ref{lem:positive-directional-source} gives $L_0W=G^2|V|^2$.
The graph convergence follows from the implicit function theorem. The
curvature limits then follow from the zero-level identities established
above. Finally,
$qD=2q(1-p_n)-q|\boldsymbol p-e_n|^2$ gives
$F=2W-G|V|^2$; $V(0,s)=0$ yields the equal, positive column masses.
\end{proof}

\subsubsection{Estimates for the positive linearized operator}
\label{sec:r3-mean-absorption}\label{sec:sharp-mean-transfer}
\begin{theorem}
\label{thm:positive-even-source}\label{lem:r3-horizontal-positive}\label{lem:sharp-mean-positive}
Suppose $F\geq0$ is smooth on $\mathbb R^m\times\mathbb R$,
$L_0F=\mathcal S\geq0$, and $\mathcal S$ is even in $s$. The zero-Dirichlet inverses on $B_j^m\times(-j,j)$ increase
to the minimal nonnegative potential $E\leq F$; $E$ is even and
$F-E$ is nonnegative $L_0$-harmonic. Put
\[
 M_P(F)=\sigma^{-1}\int_{-P}^PG^2F(0,s)\,ds,\quad
 T_P=\tanh(P/\sqrt2),\quad h_*=(2\sqrt2)^{-1},\quad C_P=h_*/T_P.
\]
For every $P>0$ the sharp positive functionals are
\begin{equation}\label{eq:sharp-mean-positive}
 \Lambda_{P,\pm}(F)=C_PM_P(F)\mp F_s(0)/2\geq0.
\end{equation}
If also $\mathcal S\geq2G\mathcal B$ with
$\mathcal B=\mathcal B(x')\geq0$ subharmonic on $\mathbb R^m$, then
\begin{equation}\label{eq:positive-subharmonic-source}
 M_P(F)\geq\beta_P\mathcal B(0),\qquad
 \beta_P=[1-(1-T_P)^2(T_P+2)]_+.
\end{equation}
For every $S>0$,
\begin{align}
 \tfrac12|\nabla'F_s(0)|&\leq\frac{3\sqrt{m+2}}{8S}
                  \fint_{B_S^m}F(x',0)\,dx',\label{eq:r3-positive-horizontal}\\
 \tfrac14|\nabla'F_s(0)|^2&\leq\frac{9(m+2)}{64\omega_mS^{m+2}}
                  \int_{B_S^m}F(x',0)^2\,dx'.\label{eq:r3-positive-horizontal-L2}
\end{align}
The modes $\varphi_\pm=e^{\pm\sqrt2s}(2\mp g(s))$, $d_0=3/\sqrt2$, satisfy
\begin{equation}\label{eq:r12-Jost-normal-mass}
 L_0\varphi_\pm=0,\quad\varphi_\pm'(0)=\pm d_0,\quad
 M_P(\varphi_\pm)=3T_P,\quad\int_{\mathbb R}G^2\varphi_\pm=3\sigma.
\end{equation}
If $F\geq c_+\varphi_++c_-\varphi_-$ globally, $c_\pm\geq0$, then
\begin{equation}\label{eq:positive-mode-subtraction}
 \tfrac12|F_s(0)-d_0(c_+-c_-)|+\tfrac{d_0}2(c_++c_-)
 \leq C_PM_P(F).
\end{equation}
For each $\varepsilon>0$, one common finite $P_\varepsilon$ makes
$C_P\leq h_*+\varepsilon$ and $\beta_P\geq1-\varepsilon$ for all
$P\geq P_\varepsilon$. No growth bound on $F$ or $\mathcal S$ is assumed.
\end{theorem}
\begin{proof}
Since $2-G\geq1$, the comparison principle gives even potentials
$0\leq E_j\leq F$ increasing to $E$. Comparing with any nonnegative
supersolution also proves minimality. It follows that $h=F-E\geq0$
is harmonic and $E_s(x',0)=0$.

For \eqref{eq:positive-subharmonic-source}, the finite-cylinder Green
kernel at $(0,s)$ is horizontally radial. Spherical averaging therefore
gives $E_j(0,s)\geq\mathcal B(0)f_j(0,s)$, where $L_0f_j=2G$ with
zero boundary data. The increasing limit $0\leq f\leq2$ is one dimensional
by uniqueness of bounded solutions. Pairing its equation with
$(-\partial_s^2+2-G)G=2G^2$ yields $\int G^2f=\sigma$. Hence
$1-M_P(f)\leq2\sigma^{-1}\int_{|s|>P}G^2=(1-T_P)^2(T_P+2)$.

It remains to prove the estimates for positive harmonic functions.
The normalized convex set
$\mathcal K=\{h\geq0:L_0h=0,\ h(0)=1\}$ is compact in the
locally smooth topology by Harnack and interior estimates. Since the
coefficients are horizontally translation invariant and $1\leq2-G\leq2$,
horizontal Harnack chains give $h(x'+t,s)\leq C_th(x',s)$ with
$C_t$ independent of $(x',s)$ and $h\in\mathcal K$. At an extreme point, subtracting a sufficiently small multiple of
any horizontal translate leaves a nonnegative harmonic function.
Extremality therefore makes the translate proportional to $h$. The
proportionality factors are positive continuous characters. Hence
\[
 h=e^{b\cdot x'}f(s/\sqrt2),\qquad
 -f''-2\operatorname{sech}^2y\,f=-a^2f,\qquad a^2=4-2|b|^2.
\]
The positive-solution ground-state identity, tested on compact
approximations to the eigenfunction $\operatorname{sech}y$, gives
$1\leq a\leq2$. For $a>1$, every positive solution is a nonnegative
combination of $a^{-1}e^{\pm ay}(a\mp\tanh y)$, since the coefficients
have the signs of the dominant asymptotics at the two ends; an extreme
point uses one term. At $a=1$, reduction of order leaves only
$\operatorname{sech}y$: its companion
$\operatorname{sech}y\int_0^y\cosh^2t\,dt$ has opposite end signs.

For $Y=P/\sqrt2$, $T=\tanh Y$, and
$\psi_a^\pm=e^{\pm ay}(a\mp\tanh y)$, the finite-window inequality is
\begin{equation}\label{eq:sharp-mean-ratio}
 M_P(\psi_a^\pm)\geq(a^2-1)T,\qquad1\leq a\leq2.
\end{equation}
Indeed, reflection reduces to $+$ and integration by parts gives
\begin{equation}\label{eq:s8-window-convexity}
 K_Y(a):=\frac{M_P(\psi_a^+)}a
 =\frac98\int_0^Y\operatorname{sech}^4y\cosh(ay)\,dy
 +\frac38\operatorname{sech}^4Y\int_0^Y\cosh(ay)\,dy.
\end{equation}
Thus $K_Y$ is convex, $K_Y(2)=3T/2$, and direct differentiation gives
\begin{equation}\label{eq:s8-window-endpoint-slope}
 K_Y'(2)=\frac98T-\frac3{16}(1-T^2)[(5-T^2)Y+T]\leq\frac98T.
\end{equation}
The tangent bound implies
$K_Y(a)\geq(9a/8-3/4)T\geq(a-a^{-1})T$, the last difference being
$(a-2)(a-4)T/(8a)\geq0$. Since
$|\partial_s\psi_a^\pm(0)|=(a^2-1)/\sqrt2$, this proves the sharp
normal inequality on extreme points.

On these points $|h_s(0)|\leq3h(0)/(2\sqrt2)$ and
$|\nabla'h_s(0)|\leq3|b|h(0)/(2\sqrt2)$, while
\[
 \fint_{B_S^m}e^{b\cdot x'}\geq1+\frac{|b|^2S^2}{2(m+2)},\qquad
 \frac{|b|}{1+|b|^2S^2/[2(m+2)]}\leq\frac{\sqrt{2(m+2)}}{2S}.
\]
These estimates give the first horizontal inequality. For fixed $P,S$,
each signed directional inequality compares continuous linear functionals
on $\mathcal K$. By the Krein--Milman theorem, it holds on all of
$\mathcal K$. Scaling and adding $E$ prove the assertions for $F$.
The second horizontal inequality follows from Cauchy--Schwarz. This
argument extends only functionals on fixed finite windows.

The $a=2$ modes and $\varphi_++\varphi_-=4/G$ give
\eqref{eq:r12-Jost-normal-mass}. Moreover
$\Lambda_{P,+}(\varphi_+)=\Lambda_{P,-}(\varphi_-)=0$ and
$\Lambda_{P,+}(\varphi_-)=\Lambda_{P,-}(\varphi_+)=d_0$,
proving optimality. Apply both functionals to
$F-c_+\varphi_+-c_-\varphi_-$ to obtain
\eqref{eq:positive-mode-subtraction}. Finally
$C_P-h_*=2h_*/(e^{\sqrt2P}-1)$ and
$(1-T_P)^2(T_P+2)\leq12e^{-2\sqrt2P}$; one may take
\[
 P_\varepsilon=\max\{2^{-1/2}\log(1+2h_*/\varepsilon),
                  (2\sqrt2)^{-1}\log(12/\varepsilon)\}.
\]
\end{proof}

\begin{lemma}
\label{lem:r3-normal-mass}\label{prop:sharp-mean-actual}
For every $\epsilon>0$, there is $1\leq P_\epsilon<\infty$ such that,
for each fixed $P\geq P_\epsilon$, there is a threshold
$\delta_{\epsilon,P}=\delta(m,\epsilon,P)>0$ for which every actual zero
with $0<D(z)<\delta_{\epsilon,P}$ satisfies
\[
 cD(z)\leq I_P(z)\leq C_{m,P}D(z),\qquad I_P(z)\geq(1-\epsilon)a(z)^2,
\]
\begin{equation}\label{eq:sharp-mean-actual}
 |H(z)|\leq(h_*+\epsilon)I_P(z).
\end{equation}
\end{lemma}
\begin{proof}
The relative derivative estimates and $q(z)=1$ give $I_P\asymp_PD$.
Choose the window in the preceding theorem with tolerance $\epsilon/2$.
If either $I_P(z)\geq(1-\epsilon)a(z)^2$ or
$|H(z)|\leq(h_*+\epsilon)I_P(z)$ failed, normalize a failure sequence as in
Lemma~\ref{lem:normalized-defect-limit}. Its limiting mass is positive,
and $\mathcal B=|\mathcal T|^2$ is subharmonic. Applying the two model
estimates contradicts the trace limits. Curvature failure forces $|\mathcal T(0)|>0$ since
$I_P\geq cD$. Letting $P\to\infty$ also gives
\begin{equation}\label{eq:r3-source-mass}
 \sigma^{-1}\int_{\mathbb R}G^2F(0,s)\,ds\geq|\mathcal T(0)|^2.
\end{equation}
\end{proof}

\begin{proposition}
\label{prop:r3-small-mean}
Fix $p\geq1$, $S\geq1$ and a slope tolerance $0<\theta<1/10$, and set
$c_{{\rm av},m,p}=1+[3\sqrt{m+2}/8]^p/\omega_m$.
There is $\delta_{S,p}=\delta(m,S,p,\theta)>0$ such that every zero
with $0<D(z)<\delta_{S,p}$ lies on a unique zero graph in $B_{8S}(z)$,
with slope at most $\theta$, and satisfies
\begin{equation}\label{eq:r3-small-mean-actual}
 |\nabla_\Sigma H(z)|^p\leq
 \frac{c_{{\rm av},m,p}}{S^{m+p}}\int_{\Sigma\cap B_{2S}(z)}D^p\,dS.
\end{equation}
Thus choosing $S^p\geq c_{{\rm av},m,p}/\varepsilon$ gives a prescribed
$\varepsilon S^{-m}$ coefficient; the graph and the estimate both follow
from small defect at the base point.
\end{proposition}
\begin{proof}
For a failure sequence the normalized model gives graphs over $B_S^m$
and $\nabla_\Sigma H/\epsilon_j\to\nabla'F_s(0)/2$.
By \eqref{eq:r3-positive-horizontal} and Jensen,
\[
 2^{-p}|\nabla'F_s(0)|^p\leq
 \frac{[3\sqrt{m+2}/8]^p}{\omega_mS^{m+p}}\int_{B_S^m}F(x',0)^p\,dx'.
\]
The lower limit of the normalized surface integral is at least the
positive integral on the right. The additional $1$ in
$c_{{\rm av},m,p}$ makes this incompatible with a failure of the stated
estimate. Convergence on $B_{16S}$ gives the graphical assertion.
\end{proof}

\subsubsection{Stability on a compact interface patch}
\label{sec:r3-exterior-surface}
\begin{proposition}
\label{prop:r3-actual-transfer}
Assume $Q_u\geq0$ outside $\overline B_{R_0}$. For every $\epsilon>0$
there are $P_\epsilon\geq1$ and $C_\epsilon<\infty$, depending only
on $m,\epsilon$, such that for each fixed $P\geq P_\epsilon$ there
exist $Q\geq P$ and $\delta>0$ with the following property.
Every compactly supported Lipschitz $\psi$ on
$\Sigma\cap\{D<\delta,\ |z|-R_0>8Q\}$ satisfies
\begin{equation}\label{eq:r3-actual-transfer}
 (1-\epsilon)\int_\Sigma\frac{a^2+I_P}{2}\psi^2\,dS
 \leq\int_\Sigma|\nabla_\Sigma\psi|^2\,dS
 +C_\epsilon\int_\Sigma e^{-(|z|-R_0)/8}\psi^2\,dS.
\end{equation}
In particular, at smaller tolerances,
\begin{equation}\label{eq:r3-beta99}
 .99\int_\Sigma a^2\psi^2\,dS\leq\int_\Sigma|\nabla_\Sigma\psi|^2\,dS
 +C\int_\Sigma e^{-(|z|-R_0)/8}\psi^2\,dS.
\end{equation}
\end{proposition}
\begin{proof}
For small fixed $\alpha>0$, choose even smooth cutoffs $\beta_0,\gamma_0$
equal to one on $[-1/2,1/2]$ and supported in $(-1,1)$. Put
\[
 \begin{gathered}
 \Phi(z,s)=z+s\nu(z),\quad J=\det(I+sA),\quad
 \beta=\beta_0(s\sqrt{D(z)}/\alpha),\\
 d_c=|z|-R_0,\quad r=d_c/4,\quad\gamma=\gamma_0(s/r).
 \end{gathered}
\]
Since $a\leq C\sqrt D$, on the support of $\beta$ we have
$|s|a\leq C\alpha$. Consequently $J>0$, the map $\Phi$ is locally
invertible, and its inverse tangent metric is $I+O(\alpha)$. The lifted amplitude $G\beta\gamma\psi$ has compact
support outside the core, since $|\Phi|-R_0\geq d_c-|s|\geq3d_c/4$.
We define an ambient test function by taking the root sum of squares of
the amplitudes over all inverse branches. At each ambient point the
relevant fiber is discrete in the compact lifted support, and is therefore
finite. In a sufficiently small ambient neighborhood, these preimages
determine all relevant inverse branches, each counted once;
the branch amplitudes extend by zero across their support boundaries.
Their fiberwise root sum of squares is locally Lipschitz. The area
formula gives exactly the sum of branchwise potential energies, while
$|\nabla\sqrt{\sum_j\xi_j^2}|^2\leq\sum_j|\nabla\xi_j|^2$ almost
everywhere by Cauchy--Schwarz. The lifted quadratic form is therefore nonnegative by ambient stability.
For $f=\beta\gamma\psi$, integration by parts in the normal
variable writes this form exactly as
\begin{equation}\label{eq:r3-lift-exact}
 \int_{\Sigma\times\mathbb R}\!
 \{G^2J|\nabla f|^2+f^2[-J_sGG'+3(u(\Phi)^2-g^2)G^2J]\}\,dS\,ds.
\end{equation}
Here gradients are taken in the pullback of the Euclidean metric. The
calculation is made on the normal bundle and introduces no chart boundary
terms.

\vskip0.16in

The base gradient coefficient is at most $(1+C\alpha+\eta)\sigma$ for a
fixed Young tolerance $\eta>0$. On the derivative support of $\beta$,
$|s|\sqrt D\geq\alpha/2$, and $|\nabla D|\leq CD$ gives
$|\nabla\beta|\leq C_\alpha(1+|s|)\sqrt D$ in the pullback metric.
Its contribution, including the cross terms from Young's inequality,
is $o_\delta(D)\psi^2$, by the exponential decay of $G^2$. Likewise $|\nabla r|\leq1/4$,
$|\nabla\gamma|\leq C/r$, and $r\geq2Q\geq1$ give
$\int G^2J|\nabla\gamma|^2\leq Cr^{-2}e^{-\sqrt2r}\leq Ce^{-d_c/8}$,
which also controls the associated cross terms.

\vskip0.16in

The polynomial $J$ and evenness give
$\int_{\mathbb R}-J_sGG'=\sigma(H^2-a^2)/2+O(a^3)$.
The $\beta$ tail is superpolynomial in $D$, the $\gamma$ tail is
$O(e^{-d_c/8})$, and $H^2+a^3=o_\delta(D)$. Thus the resulting curvature term is
$-\sigma a^2/2$ up to the indicated errors.

\vskip0.16in

We next estimate the profile term. It is nonpositive everywhere, since
$|\nabla U|\leq1$ implies $|U(z+s\nu)|\leq|s|$. On a fixed window
it contributes at least $(1-\eta)\sigma I_P/2$ to the negative part.
To prove this, set $U_\pm(t)=\pm U(z\pm t\nu)$.
For fixed $Q$, relative jets give
\[
 U_\pm(t)=t+O_Q(D(z)),\qquad
 t-U_\pm(t)\geq\tfrac12\int_0^tD(z\pm s\nu)\,ds.
\]
Taylor expansion of $g^2$ therefore bounds
$3(g^2-u(\Phi)^2)G^2$ below by
$-\tfrac12(G^3)'(t)\int_0^tD(z\pm s\nu)ds$, with $o_\delta(D(z))$
integrated error. Restrict the integral of $D$ to $|s|\leq P$ and integrate in $t$ up to
$Q\geq P$. More explicitly, for
$A_\pm(t)=\int_0^{\min\{t,P\}}D(z\pm s\nu)\,ds$,
\[
 -\tfrac12\int_0^Q(G^3)'(t)A_\pm(t)\,dt
 =\tfrac12\int_0^P[G(s)^3-G(Q)^3]D(z\pm s\nu)\,ds.
\]
Thus this contribution is half the $G^3D$ mass on this
interval, less the endpoint term $G(Q)^3\int_{-P}^PD/2$. Relative jets bound the latter integral by
$C_{m,P}D(z)$; choose $Q$ large after $P$. On this fixed interval,
$q^3=G^3+O_Q(D)$, $J=1+o_\delta(1)$, and $\beta=\gamma=1$
after decreasing $\delta$, since $d_c>8Q$.
Finally, $I_P\geq cD$ absorbs all errors that vanish with the defect.
Choose $\eta$, then $\alpha$, and then $\delta$. Divide the nonnegative
form by its gradient coefficient, which is close to $\sigma$. This proves
\eqref{eq:r3-actual-transfer}; Lemma~\ref{lem:r3-normal-mass} gives
\eqref{eq:r3-beta99}.
\end{proof}

\vskip0.16in

For a fixed application choose the common normal window $P$, then $Q$
and the horizontal averaging radius $S$, before decreasing $\delta$.
The additional interaction requirement on $P$ is recorded in
Section~\ref{sec:r13-nine-dimensional-local}.

\vskip0.16in

\subsection{Curvature identities and normal rays}
\begin{lemma}
\label{lem:nonminimal-kato-simons}
For a smooth Euclidean hypersurface of dimension $m\geq2$, with
$a=|A|$, $H=\operatorname{tr}A$, and every $e>0$,
\begin{equation}\label{eq:nonminimal-kato-family}
 |\nabla A|^2\geq\kappa_{m,e}|\nabla a|^2-C_{m,e}|\nabla H|^2,
 \quad\kappa_{m,e}=\frac{m+2}{m(1+e)},\quad
 C_{m,e}=\frac3{m+2}\{\kappa_{m,e}(1+e^{-1})-1\}.
\end{equation}
For $e=1/(m+1)$ this becomes
\begin{equation}\label{eq:r10-kato-general}
 \kappa_m=1+1/m,\qquad C_{K,m}=3(m+1)/m-3/(m+2).
\end{equation}
The exact Simons identity is
\begin{equation}\label{eq:nonminimal-simons-identity}
 \tfrac12\Delta_\Sigma a^2=|\nabla A|^2+A_{ij}H_{ij}
                         +H\operatorname{tr}A^3-a^4.
\end{equation}
These classical identities~\cite{Simons1968,SchoenSimonYau1975} are used
with their actual mean errors in Section~\ref{sec:r12-higher-moment-attack}.
\end{lemma}
\begin{proof}
The symmetric Codazzi tensor decomposes orthogonally as $\nabla A=T_0+S$,
where $S_{ijk}=(H_ig_{jk}+H_jg_{ik}+H_kg_{ij})/(m+2)$ and
$|S|^2=3|\nabla H|^2/(m+2)$. For trace-free $A_0$ and unit $v$,
the trace-free projection of $\operatorname{sym}(A_0\otimes v)$ has norm squared
$|A_0|^2/3+2m|A_0v|^2/[3(m+2)]\leq m|A_0|^2/(m+2)$.
The trace part of $A$ contracts trivially with $T_0$. Thus
$|\nabla a|\leq\sqrt{m/(m+2)}|T_0|+|S|$; Young's inequality
gives \eqref{eq:nonminimal-kato-family}, using regularized norms at $a=0$.
Commuting Codazzi with the Gauss equation gives
$\Delta A_{ij}=H_{ij}+HA_{ik}A_{kj}-a^2A_{ij}$; contraction proves Simons.
\end{proof}

\vskip0.16in

\subsubsection{Derivative estimates along normal rays}
\label{sec:r9-accurate-rays}
We now extend the profile comparison to growing normal intervals and prove
uniqueness of the corresponding preimage. These facts will be used when
combining the normal test functions; compare
\cite[Lemma 5.1 and Section 9.2]{OpenAI2026}.
\begin{lemma}
\label{lem:r4-near-critical-growth}
For every $\gamma>0$ there is $C_\gamma=C(m,\gamma)$ such that every
nonplanar bounded entire solution satisfies
\begin{equation}\label{eq:r4-near-critical-growth}
 D(y)\leq C_\gamma e^{(2\sqrt2+\gamma)|y-z|}D(z).
\end{equation}
Consequently, for each fixed $k$,
\[
 |\nabla^kD(y)|\leq C_{k,\gamma}e^{(2\sqrt2+\gamma)|y-z|}D(z),\quad
 |\nabla^k\nabla^2U(y)|\leq
 C_{k,\gamma}e^{(\sqrt2+\gamma/2)|y-z|}\sqrt{D(z)}.
\]
\end{lemma}
\begin{proof}
For $T=\nabla^2U$,
$\mathscr L=-\Delta+2\sqrt2u\boldsymbol p\cdot\nabla+2q|\boldsymbol p|^2$
satisfies $\mathscr LD=2|T|^2\geq0$ and has a nonnegative zeroth-order coefficient.
For $\lambda=2\sqrt2+\gamma$, $r=|x-y|$,
\[
 e^{\lambda r}\mathscr L(e^{-\lambda r})
 \leq-\lambda^2+m\lambda/r+2\sqrt2\lambda|u|+2(1-u^2)
 \leq-\lambda(\lambda-2\sqrt2)+m\lambda/r<0
\]
outside a fixed radius. Harnack makes $c_\gamma D(y)e^{-\lambda r}\leq D(x)$
on that sphere. Subtract $\varepsilon$, compare on an annulus whose outer
sphere has negative subsolution, and let $\varepsilon\downarrow0$.
Harnack covers the inner ball; interchanging $x,y$ proves the claim.
Unit-ball relative jets give the derivative bounds.
\end{proof}

\vskip0.16in

\begin{lemma}
\label{lem:r4-accurate-ray}
Fix $\xi,\alpha>0$. There are $b_\xi=b(m,\xi,\alpha)>0$ and
$C_\xi=C(m,\xi)<\infty$ such that, if $z\in\Sigma$, $x=z+s\nu(z)$ and
$B_z(s):=D(z)e^{(2\sqrt2+\xi)|s|}\leq b_\xi$, then $z$ is regular,
$|s|a(z)<\alpha/2$, and
\begin{equation}\label{eq:r4-accurate-profile}
 |\boldsymbol p(x)-\nu(z)|+|U(x)-s|
 \leq C_\xi B_z(s)e^{-\xi|s|/2}.
\end{equation}
Each ambient point $x$ has at most one such preimage $(z,s)$. In its inverse chart,
\begin{equation}\label{eq:r4-accurate-ratio}
 \left|q(x)/G(s)-1\right|+
 \left|\nabla_x\log(G(s)/q(x))\right|
 \leq C_\xi B_z(s)e^{-\xi|s|/2}.
\end{equation}
Both signs of $s$ are included.
\end{lemma}
\begin{proof}
Apply the preceding lemma with loss $\xi/2$. Along $z+t\nu(z)$,
$Y=\boldsymbol p-\nu(z)$ satisfies
$Y'=-\nabla D/2-TY$, $|Y(0)|\leq D(z)$, and
\[
 D+|\nabla D|\leq C_\xi D(z)e^{(2\sqrt2+\xi/2)|t|},\qquad
 |T|\leq C_\xi\sqrt{D(z)}e^{(\sqrt2+\xi/4)|t|}.
\]
The integral of $T$ on $[0,s]$ is at most $C_\xi\sqrt{b_\xi}$.
Gronwall's inequality bounds $Y$ by the first exponential. We then
integrate $(U-t)'=Y\cdot\nu(z)$ to obtain \eqref{eq:r4-accurate-profile}.
Also $|s|a(z)\leq C\sqrt{b_\xi}|s|e^{-(\sqrt2+\xi/2)|s|}<\alpha/2$.

\vskip0.16in

For two accurate preimages, the profile estimate gives
$|s-\widetilde s|\leq C_\xi b_\xi$. Comparing both normals with
$\boldsymbol p(x)$ and using $z=x-s\nu(z)$ then gives
$|z-\widetilde z|\leq C_\xi b_\xi(1+|s|)e^{-\xi|s|/2}$.
For small $b_\xi$ both bases lie in one fixed good graph; Harnack and
$a\leq C\sqrt D$ give
$|\nu(z)-\nu(\widetilde z)|\leq C\sqrt{D(z)}|z-\widetilde z|$.
Tangential projection of their equal images yields
$.9|z-\widetilde z|\leq C|\widetilde s|\sqrt{D(z)}|z-\widetilde z|$,
whose right coefficient is less than $.9$ for small $b_\xi$.
Thus the preimages agree. Finally $\nabla_xs=\nu(z)$,
$\nabla\log q=-\sqrt2u\boldsymbol p$, and
$\nabla\log G=-\sqrt2g(s)\nu(z)$ give the ratio estimate.
\end{proof}

\vskip0.16in

\section{\bf Penalized reach and contact points}
\label{sec:s12-contact}
In this section we define a penalized distance to a competing zero and
derive the corresponding contact estimates. We use the relative defect
estimates of Section~\ref{sec:s12-defect} and one fixed angular penalty.
The arguments in this section and the next two sections require neither
stability nor an energy-growth bound.

\vskip0.16in

\label{sec:r10-reach-geometry}\label{sec:r9-reach-geometry}\label{sec:r4-reach-gate}
Write $n=m+1$, $3\leq m\leq8$, $\Sigma=\{u=0\}$,
$\nu=\nabla U/|\nabla U|$, $A=d\nu|_{T\Sigma}$, $a=|A|$ and
$H=\operatorname{tr}A$ on $\Sigma_{\rm reg}$. For another orientation
$N$ use $A_N,H_N$. Balls are ambient $n$-balls and surface measure is
$\mathcal H^m$. Constants may depend on dimension and the fixed penalty.
The contact formulas themselves hold for every fixed $m\geq1$.
Constant wells are separate; $D\equiv0$ is the planar profile by the
baseline rigidity. We therefore assume $D\not\equiv0$ when discussing
finite reaches. We prove the contact estimates of \cite[Sections~7.5--7.6]{OpenAI2026}
directly for the zero set, using the sign convention $A=d\nu$.

\vskip0.16in

Fix a smooth nonincreasing $\Gamma:(0,1]\to(0,\infty)$ with
\begin{equation}\label{eq:r4-penalty-assumptions}
 \begin{gathered}
 s^{-1}\leq\Gamma(s)\leq C_\Gamma s^{-1},\qquad
 \Gamma(s)=s^{-1}\quad(s\leq1-\delta_1),\\
 \Gamma(s)=G_1\quad(s\geq1-\delta_2),\qquad
 0<\delta_2<\delta_1<\tfrac1{10}.
 \end{gathered}
\end{equation}
The actual penalized reach is
\begin{equation}\label{eq:r4-actual-reach}
 d(z)=\inf_{y\in\Sigma\setminus\{z\}}
 |y-z|\,\Gamma\left(\left|\nu(z)\cdot\frac{y-z}{|y-z|}\right|\right),
\end{equation}
where the objective is assigned the value $+\infty$ when
$\nu(z)\cdot(y-z)=0$. An attaining distinct zero is an
active off-diagonal contact, and is a plateau contact if
$|\nu(z)\cdot(y-z)|/|y-z|\geq1-\delta_2$. A limiting tangent direction is a diagonal contact.
Differential comparisons for $d$ below are upper-touching comparisons.

\vskip0.16in

\subsection{One fixed angular penalty}

\begin{lemma}
\label{lem:r7-small-penalty}
For every $\varepsilon_1>0$ there is a penalty of the form used in
\eqref{eq:r4-penalty-assumptions}, with
\begin{equation}\label{eq:r7-penalty-bounds}
 \begin{gathered}
 s^{-1}\leq\Gamma(s)\leq(1+C\delta_1)s^{-1},\quad
 1\leq G_1\leq\Gamma(s),\quad
 |\Gamma'(s)|\leq C_\Gamma s^{-2}\quad(0<s\leq1),\\
 -(1-\delta_1)^{-2}\leq\Gamma'(s)\leq0\quad(s\geq1-\delta_1),\\
 0\leq(1-s)\Gamma''(s)\leq\varepsilon_1
                \quad(s\geq1-\delta_1).
 \end{gathered}
\end{equation}
The number $\delta_1$ can first be prescribed sufficiently small;
the smaller positive plateau width $\delta_2$ is fixed afterwards.
\end{lemma}
\begin{proof}
Choose smooth nonincreasing $\chi:\R\to[0,1]$, equal to one on
$(-\infty,0]$ and zero on $[1,\infty)$.  For $s\geq1-\delta_1$ set
\[
 \Gamma(s)=\frac1{1-\delta_1}
       -\int_{1-\delta_1}^s u^{-2}
       \chi\left(\frac{\log(\delta_1/(1-u))}{L}\right)\,du,
\]
and set $\Gamma(s)=1/s$ below that interval.  The integrand is zero
near one.  With $v=\log(\delta_1/(1-s))/L$,
\[
 \Gamma'=-s^{-2}\chi(v),\qquad
 \Gamma''=2s^{-3}\chi(v)-\frac{s^{-2}\chi'(v)}{L(1-s)}\geq0.
\]
Consequently
\[
 (1-s)\Gamma''\leq
 \frac{2\delta_1}{(1-\delta_1)^3}
       +\frac{\|\chi'\|_\infty}{L(1-\delta_1)^2}.
\]
First choose $\delta_1$ small and then $L$ large.  The penalty is
constant on $s\geq1-\delta_2$, where $\delta_2=\delta_1e^{-L}$.
All derivatives match at both endpoints because $\chi$ is smooth
and constant on the indicated half-lines.  Its derivative is less
negative than that of $1/s$, so $\Gamma\geq1/s$; its upper bound
follows from $\Gamma\leq(1-\delta_1)^{-1}$ on the modified interval.
Monotonicity proves the assertion about $G_1$.
\end{proof}

\vskip0.16in

We now fix all angular tolerances needed below, including the target
$0<\omega<1/4$ in Proposition~\ref{prop:r7-actual-weak-reach}.
Choose $\varepsilon_1,\delta_1$ sufficiently small and then fix
$L,\delta_2,\Gamma$. These penalty parameters remain fixed as the
reach tends to infinity.

\vskip0.16in

\subsection{Phase clearing and tangent contact balls}

\begin{lemma}
\label{lem:r4-phase-ball}
There is a dimensional $C$ such that every positive solution of
$\Delta v=v^3-v$, $0<v\leq1$, on a neighborhood of the closed ball
$\overline B_r(x)$ satisfies
\[
 v(x)\geq\sqrt{1-\lambda_1/r^2}\geq1-C/r^2
 \quad(r^2\geq2\lambda_1),
\]
where $\lambda_1$ is the first Dirichlet eigenvalue of the unit
$n$-dimensional ball.
\end{lemma}
\begin{proof}
Let $\varphi$ be the positive radial first eigenfunction of the unit
ball, with $\varphi(0)=\|\varphi\|_\infty=1$, and set
$w=\sqrt{1-\lambda_1/r^2}\,\varphi((\cdot-x)/r)$. Then $w=0$
on the boundary and $-\Delta w\leq w-w^3$. If
$\Lambda=\max_{\overline B_r}w/v>1$, the maximum is attained in the
interior. At a point of contact $w=\Lambda v$ the function
$w-\Lambda v\leq0$ has maximum zero, whereas
\[
 -\Delta(w-\Lambda v)\leq
 (\Lambda v-(\Lambda v)^3)-\Lambda(v-v^3)
 =-\Lambda(\Lambda^2-1)v^3<0,
\]
a contradiction. Thus $w\leq v$, giving the assertion at the center.
If only an open zero-free phase ball is specified in an application,
apply the assertion on its concentric ball of half the radius. After choosing the sign of the phase, the solution is positive on a neighborhood
of the closure of this smaller ball, and the same $Cr^{-2}$ bound
follows after changing the fixed dimensional constant.
\end{proof}

\vskip0.16in

\begin{theorem}
\label{lem:r5-distance-energy}\label{lem:near-critical-phase-decay}
Let $u$ be a bounded entire Allen--Cahn solution with nonempty zero
set $\Sigma$, and put $a=\sqrt2$. For every $0<\varepsilon<a$ there is
$C_\varepsilon=C(m,\varepsilon)$ such that
\begin{equation}\label{eq:near-critical-phase-decay}
 1-|u(x)|+q(x)\leq C_\varepsilon
 e^{-(a-\varepsilon)\operatorname{dist}(x,\Sigma)},\qquad
 e_u(x)\leq C_\varepsilon
 e^{-2(a-\varepsilon)\operatorname{dist}(x,\Sigma)}.
\end{equation}
For every fixed $k$, the corresponding bound with exponent
$a-\varepsilon$ holds for $|\nabla^k(u-\operatorname{sgn}u)|$
on $B_{1/2}(x)$ when $\operatorname{dist}(x,\Sigma)\geq2$, with a
constant depending only on $m,k,\varepsilon$.
For $x\notin\Sigma$, every nearest zero is regular and its chord
to $x$ is normal.
\end{theorem}
\begin{proof}
Put $r=\operatorname{dist}(x,\Sigma)$, and flip the sign so that
$u>0$ on $B_r(x)$. Choose $a-\varepsilon<k_0<a$ and
$\delta>0$ with $(1-\delta)(2-\delta)>k_0^2$.
Lemma~\ref{lem:r4-phase-ball} supplies a fixed $M=M_\delta$
so that $u\geq1-\delta$ on $B_{r-M}(x)$ for $r\geq2M$.
With $R=r-M$ and $w=1-u$, therefore,
\[
 (-\Delta+k_0^2)w\leq0\quad\text{on }B_R(x),\qquad0\leq w\leq1.
\]
A positive comparison function is
\[
 \Phi_R(y)=C_{k_0}(1+R)^{(n-1)/2}e^{-k_0R}
 \fint_{\mathbb S^{n-1}}e^{k_0\omega\cdot(y-x)}\,d\omega.
\]
It solves $(-\Delta+k_0^2)\Phi_R=0$. On $|y-x|=R$, a cap of
angular radius $(1+R)^{-1/2}$ has normalized measure at least
$c_n(1+R)^{-(n-1)/2}$ and satisfies
$\omega\cdot(y-x)\geq R-1/2$. Thus one fixed $C_{k_0}$ makes
$\Phi_R\geq1$ on the boundary. The maximum principle and
$k_0>a-\varepsilon$ give
\[
 w(x)\leq C_{k_0}(1+R)^{(n-1)/2}e^{-k_0R}
 \leq C_\varepsilon e^{-(a-\varepsilon)r}.
\]
Use $q\leq2w$ and Modica's $e_u\leq q^2/2$; boundedness treats
$r<2M$. Fixed-radius interior elliptic estimates for
$\Delta w=u(1+u)w$, iterated with the global fixed-order bounds,
give all stated derivatives, since distance changes by at most one.
At a nearest zero $z\in\partial B_r(x)$ the Hopf lemma for
$\Delta-1$ applies to the positive function $u$, because
$(\Delta-1)u=u^3-2u<0$. Hence $\nabla u(z)\ne0$.
Stationarity of squared distance along this regular zero sheet
then gives $x-z\perp T_z\Sigma$.
\end{proof}

\vskip0.16in

\begin{lemma}
\label{lem:r5-Poisson}
Put $G(s)=1-g(s)^2$, $g(s)=\tanh(s/\sqrt2)$, and
$\mathscr L=-\Delta+2-G(x_n)$ on $B_P\subset\mathbb R^n$.
Let $h\geq0$ solve $\mathscr Lh=0$ with nonnegative boundary data $f$
on $\partial B_P\cap\{x_n>0\}$ and zero data on the lower hemisphere.
For every fixed $\zeta>0$ there is $P_\zeta=P(m,\zeta)$ such that,
for $P\geq P_\zeta$ and $x_*=(P/2)e_n$, one has
\begin{equation}\label{eq:r5-Poisson}
 h(0)\leq CP^{-m/2} e^{-P}\int_{\partial B_P\cap\{x_n>0\}}f\,dS,
 \qquad
 h(x_*)\geq c_\zeta
 e^{-(\sqrt2+\zeta)\sqrt5P/2}
 \int_{\partial B_P\cap\{x_n>0\}}f\,dS.
\end{equation}
Consequently $h(0)\leq C_\zeta P^C e^{\kappa P}h(x_*)$, where
$\kappa=(\sqrt2+\zeta)\sqrt5/2-1$.
\end{lemma}
\begin{proof}
For $1\leq V\leq2$, maximum-principle comparison sandwiches the
Dirichlet Poisson density between those for $-\Delta+1$ and
$-\Delta+2$. The former has center density
$[|\mathbb S^m|P^m M(P)]^{-1}$, where
$M(t)=\fint_{\mathbb S^m}e^{t\theta_1}\,d\theta$.
A cap of radius $P^{-1/2}$ gives $M(P)\geq cP^{-m/2}e^P$,
proving the upper bound.

\vskip0.16in

For the lower density from $x_*=(P/2)e_n$ to an upper-hemisphere
point $\xi$, choose a fixed transverse disk of radius $B$ with
$\sqrt{2+\lambda_1(B)}<\sqrt2+\zeta$. Propagate the positive Green
function along the chord $[x_*,\xi]$ using its first transverse
Dirichlet mode and the longitudinal hyperbolic-sine solution. The
cylinder lies inside $B_P$ after fixed end portions are removed:
$x_*$ has depth $P/2$ and the incidence cosine at $\xi$ is at least
$1/\sqrt5$. Fixed-ball Harnack comparisons at the two ends and an
interior tangent-ball barrier at $\xi$ therefore give inward Green
derivative at least $c_\zeta e^{-(\sqrt2+\zeta)|x_*-\xi|}$.
Here $B=B(m,\zeta)$ is fixed, and the end comparisons are uniform
in $P$ and in the upper-hemisphere point $\xi$, including the equator. Since $|x_*-\xi|\leq\sqrt5P/2$, integration against $f$
proves both bounds and their ratio. Nonnegative bounded data follow
by monotone approximation.
\end{proof}

\vskip0.16in

\begin{theorem}
\label{thm:r5-positive-trace}\label{lem:r4-tangent-phase-ball}
There are dimensional constants $r_0,C<\infty$ and $\gamma>0$ with
the following property.  Suppose $u:\mathbb R^n\to(-1,1)$ is an
entire Allen--Cahn solution, $B_r(c)$ contains no zero, $r\geq r_0$,
and $x\in\partial B_r(c)\cap\{u=0\}$.  Let $N=(c-x)/r$ be the
inward normal.  Then $x$ is regular and, with the sign of $u$ chosen
positive inside the ball,
\begin{equation}\label{eq:r5-positive-trace}
 D(x)+|A(x)|+|H_N(x)|\leq C/r,
 \qquad H_N(x)\leq Cr^{-1-\gamma}.
\end{equation}
Moreover, along the inward radius,
\begin{equation}\label{eq:r5-inward-U}
 U(x+sN)\geq s-C(1+s)/r\qquad(0\leq s\leq\log r).
\end{equation}
Only the mean-curvature conclusion is one-sided; no assertion
$|H_N|\leq Cr^{-1-\gamma}$ is made.
\end{theorem}
\begin{proof}
Translate $x=0$ and rotate $N=e_n$, choosing $u>0$ in the ball.
The phase-decay theorem gives, with fixed-order derivatives,
\begin{equation}\label{eq:r5-exp-phase}
 1-u(y)\leq Ce^{-c\operatorname{dist}(y,\Sigma)}.
\end{equation}
For inward radial depth $0<s<Q=\sqrt r$, put
\[
 w_0(s)=g'(s)\int_0^s(g'(t))^{-2}
                   \int_t^\infty(g'(v))^2\,dv\,dt.
\]
Then $w_0(0)=0$, $(-\partial_s^2+3g^2-1)w_0=g'$ and
$|w_0|+|w_0'|\leq C(1+s)g'$. For fixed $C_0>2m$, the positive
function $v_*=g-C_0w_0/r$ obeys
\[
 v_*''-\frac m{r-s}v_*'+v_*-v_*^3>0\qquad(0<s<Q).
\]
Indeed the leading term is $(C_0/r-m/(r-s))g'$, and the remainder
is bounded by $Cr^{-2}[(1+s)g'+(1+s)^2(g')^2]$.
Multiplication by $1-C_1e^{-cQ}$ preserves this strict inequality
and, by \eqref{eq:r5-exp-phase}, puts the inner boundary below $u$.
The trace on the outer boundary is zero. We compare by increasing the
amplitude of this subsolution. Positivity and the Hopf lemma permit a
small initial amplitude. The strict differential inequality excludes an
interior first contact, and the Hopf lemma excludes a first obstruction
on the boundary.
Smaller amplitudes remain strict because
$\mathcal F(av_*)=a\mathcal F(v_*)+(a-a^3)v_*^3>0$.
Thus
\[
 u(se_n)\geq g(s)-Cr^{-1}(1+s)g'(s)-Ce^{-c\sqrt r}
 \qquad(0\leq s\leq\log r).
\]
Inversion of $g$ gives \eqref{eq:r5-inward-U}; near zero use the
zero trace and the same radial comparison. At zero it gives
$\partial_Nu\geq1/\sqrt2-C/r$, hence regularity and $D\leq C/r$.
The zero-level identity and relative jets give $|H_N|\leq CD$.
The tangent ball gives $A_N\geq-r^{-1}\operatorname{Id}$; the trace
bound therefore yields $|A|\leq C/r$.

\vskip0.16in

For the improved upper trace take $P=\alpha\log r$, with $\alpha$
fixed small below. The relative defect estimates on unit balls imply, for
each derivative order used below,
\begin{equation}\label{eq:r5-growing-jets}
 \begin{gathered}
 \|D\|_{C^k(B_{2P})}\leq C_kr^{-1+C_k\alpha},\\
 \|T\|_{C^k(B_{2P})}+\|U-x_n\|_{C^k(B_{2P})}
 \leq C_kr^{-1/2+C_k\alpha},\quad T=D^2U.
 \end{gathered}
\end{equation}
Here $\boldsymbol p(0)=\sqrt{1-D(0)}e_n$ also controls the affine
part of $U$. Differentiating $D=1-|\boldsymbol p|^2$ gives
$\nabla_pT=-T^2-\tfrac12D^2D$; hence
$\partial_nT=O(r^{-1+C\alpha})$ on $B_P$.
For $F(y,s)=G(s)D(y,s)$ the exact defect equation implies
\begin{equation}\label{eq:r5-even-source}
 (-\Delta+2-G)F=2G(s)|T(y,0)|^2+\mathcal E,
 \qquad\|\mathcal E\|_\infty\leq Cr^{-3/2+C\alpha}.
\end{equation}
Indeed the drift and potential before conjugation are
$2\sqrt2u\boldsymbol p$ and $2q(1-D)$; replacing them by
$2\sqrt2g e_n$ and $2G$ costs
$O(r^{-1/2+C\alpha})(|D|+|\nabla D|)$.
The identities $G'/G=-\sqrt2g$, $G''=2G-3G^2$ give the displayed
operator, and the bound on $\partial_nT$ controls the replacement of
$|T(y,s)|^2$ by $|T(y,0)|^2$ on the right hand side.

\vskip0.16in

On $B_P$ decompose $F=F_{\rm ev}+F_{\rm err}+h_++h_-$.
The first two terms solve the zero-Dirichlet problems with right hand sides
$2G(s)|T(y,0)|^2$ and $\mathcal E$, respectively. The functions
$h_\pm$ solve the homogeneous equation and take the nonnegative trace
of $F$ on the corresponding hemisphere.
Thus $F_{\rm ev}\geq0$ is even,
$\|F_{\rm err}\|_\infty+|\nabla F_{\rm err}(0)|\leq
Cr^{-3/2+C\alpha}$. On the upper half-ball, reflection comparison
gives $h_-(y,s)\leq h_-(y,-s)$: on its curved boundary the left
side is zero and the right side is nonnegative, and the traces agree
on $s=0$. Hence $\partial_sh_-(0)\leq0$. Local Harnack gives
$|\partial_sh_+(0)|\leq Ch_+(0)$.
At $x_*=(P/2)e_n$, positivity gives
$h_+(x_*)\leq F(x_*)+Cr^{-3/2+C\alpha}$.
Since $D\leq2(1-\partial_sU)$, put $a=P/2+1$ and integrate
\eqref{eq:r5-inward-U} to obtain
\[
 \int_0^aD(0,s)\,ds\leq2[a-U(0,a)]\leq C(1+P)/r.
\]
Applying the relative Harnack inequality on unit balls along the segment
$[P/2,P/2+1]e_n$ gives
\[
 D(x_*)\leq C(1+P)/r,\qquad
 F(x_*)\leq C(1+P)r^{-1}e^{-\sqrt2P/2}.
\]
Lemma~\ref{lem:r5-Poisson} yields
\[
 h_+(0)\leq CP^Cr^{-1}e^{-(\sqrt2/2-\kappa)P}
       +CP^Cr^{-3/2+C\alpha}e^{\kappa P}.
\]
Choose $\zeta>0$ so that $\kappa<\sqrt2/2$, then $\alpha>0$ so
small that $(C+\kappa)\alpha<1/4$ for all fixed losses above.
Absorbing powers of $P$ gives $F_s(0)\leq Cr^{-1-\gamma}$ for
some fixed $\gamma>0$. Finally $G(0)=1$, $G'(0)=0$ and
$H_N=D_s(0)/(2|\boldsymbol p(0)|^2)$ prove the assertion.
\end{proof}

\vskip0.16in

\begin{proposition}
\label{prop:r4-actual-contact-regularity}
At every regular zero $d(z)>0$, and
\begin{equation}\label{eq:r4-chord-bounds}
 \max_i|\lambda_i(z)|\leq2/d(z),\qquad
 |\nu(z)\cdot(y-z)|\leq C_\Gamma |y-z|^2/d(z)
 \quad(y\in\Sigma).
\end{equation}
Here $1/\infty=0$. There is $d_*=d_*(m,\Gamma)$ such that, whenever
$d_*\leq d(z)<\infty$,
\begin{equation}\label{eq:r4-contact-defect}
 D(z)+a(z)\leq C_\Gamma/d(z).
\end{equation}
For the same large-reach threshold, every active off-diagonal contact
$y$ is regular and satisfies
\begin{equation}\label{eq:r4-receiving-defect}
 D(y)+a(y)\leq C_\Gamma/d(z).
\end{equation}
If $D\not\equiv0$, then $d$ is finite on $\Sigma_{\rm reg}$.  Moreover,
for every fixed physical ball $B_T$ there is $L_{u,T}<\infty$ such that
every sufficiently large-reach active pair with $y\in B_T$ satisfies
\begin{equation}\label{eq:r4-core-contact-localization}
 d(z)\leq L_{u,T},\qquad z\in B_{T+L_{u,T}}.
\end{equation}
In particular, active pairs whose receiving foot lies in the fixed unstable core
have uniformly bounded reach at their base points.
\end{proposition}

\vskip0.16in

\begin{proof}
For nearby points, a local graph gives a positive lower bound for the
chord objectives. Outside the graph neighborhood, the chord length itself
gives a positive lower bound.
Along a principal geodesic the argument of $\Gamma$ tends to zero,
while the objective tends to $2/|\lambda_i|$, proving the first inequality in
\eqref{eq:r4-chord-bounds}.  The upper bound on $\Gamma$ proves the second.
Consequently the two balls with centers $z\pm c_\Gamma d(z)\nu(z)$ and
radius $c_\Gamma d(z)$ contain no zeros, for a sufficiently small fixed
$c_\Gamma$.  Theorem~\ref{lem:r4-tangent-phase-ball} proves
\eqref{eq:r4-contact-defect}.

For an active off-diagonal chord choose $N=\pm\nu(z)$ so that
$N\cdot(y-z)>0$.  The open set
\[
 \mathcal O_z=\{z+te:N\cdot e>0,\ 0<t<d(z)/\Gamma(N\cdot e)\}
\]
contains no zero by definition of $d$.  Its boundary contains $y$.
After scaling by $d(z)$ it is a fixed smooth compact embedded
hypersurface.  Near $z$ its equation is the sphere equation
$|x-z|^2=d(z)N\cdot(x-z)$, since $\Gamma(s)=1/s$ near zero.
Elsewhere it is the smooth positive radial graph $1/\Gamma(s)$.
It therefore has an interior rolling ball of a fixed positive radius
depending only on $\Gamma$.  To see that the radius is uniform, use the local $C^2$ graph bounds
near each boundary point. After removing a fixed neighborhood of the
point, the remaining compact boundary has positive separation from it.  Decreasing the local
ball radius then prevents contact with any other part of the boundary.
Scaling back gives a zero-free tangent ball at $y$ of radius
$c_\Gamma d(z)$.  The same theorem proves
\eqref{eq:r4-receiving-defect}.

\vskip0.16in

If $d(z)=\infty$, the base chord inequality supplies tangent phase
balls of every radius.  The lemma then gives $D(z)=0$.  The relative
Harnack estimate forces $D\equiv0$.  Thus $D\not\equiv0$ excludes
infinite reaches.  In this case $D>0$ everywhere, again by relative
Harnack.  Set $b_{u,T}=\min_{\overline B_T}D>0$.
Equation~\eqref{eq:r4-receiving-defect} gives
$d(z)\leq C_\Gamma/b_{u,T}$ whenever $y\in B_T$; the chord length
is at most $d(z)/G_1\leq d(z)$, giving
\eqref{eq:r4-core-contact-localization}.
\end{proof}

\vskip0.16in

\subsection{Ordinary tubes}

\begin{lemma}\label{lem:r5-ordinary-tubes}
There is $c_\Gamma>0$ such that the normal map
\[
 (z,t)\longmapsto z+t\nu(z),\qquad
 z\in\Sigma_{\rm reg},\quad |t|<c_\Gamma d(z),
\]
is jointly injective. Every image has $z$ as its unique nearest zero. Its Jacobian satisfies fixed positive upper and lower bounds. Consequently, for $R\geq1$ and $L>0$,
\begin{equation}\label{eq:r5-reach-area}
 \mathcal H^m(\{z\in\Sigma_{\rm reg}\cap B_R:d(z)\geq L\})
 \leq C_\Gamma R^n/\min\{L,R\}.
\end{equation}
\end{lemma}
\begin{proof}
The actual chord bound gives, for any other zero $y$,
\[
 |z+t\nu(z)-y|^2-t^2
 \geq\left(1-2C_\Gamma |t|/d(z)\right)|y-z|^2>0
\]
when $c_\Gamma$ is sufficiently small. Thus $z$ is the unique closest zero. Two normal images can coincide only if their bases and signed depths agree. Since $|\lambda_i|\leq2/d$, the $m$ factors $1+t\lambda_i$ stay between two fixed positive constants. Apply the area formula to the normal interval $|t|<c_\Gamma\min\{L,R\}$ over the indicated bases. Its image lies in $B_{(1+c_\Gamma)R}$, and its volume is at least a fixed multiple of $\min\{L,R\}$ times the surface area. This proves \eqref{eq:r5-reach-area}. The countable regular-chart area formula gives the same conclusion for measurable subsets.
\end{proof}

\subsection{Semiconcavity of the penalized reach}
\label{sec:coarse-actual-reach-gradient}

\begin{theorem}
\label{lem:r4-reach-semiconcavity}
\label{prop:coarse-actual-reach-gradient}
Suppose $D\not\equiv0$. The actual reach is locally the infimum
of a compact family of smooth functions of the base point, with a common local
upper Hessian bound. Thus it is locally semiconcave and locally
Lipschitz on $\Sigma_{\rm reg}$, and the singular part of its
surface Laplacian is nonpositive. Moreover, there is
$C_\Gamma=C(m,\Gamma)$ such that
\begin{equation}\label{eq:coarse-actual-reach-gradient}
 |\nabla_\Sigma d|\leq C_\Gamma(1+d)^2
 \quad\text{almost everywhere on }\{D\leq1/4\}\cap\Sigma.
\end{equation}
\end{theorem}
\begin{proof}
On a compact regular graph of base points $X(v)=(v,h(v))$, local graph
geometry and a fixed competitor give $0<c\leq d\leq C$.
Relevant long chords $\ell=|y-X(v)|\geq\rho$ have
$s\geq\ell/(2C)\geq\rho/(2C)$ and $\ell\leq2C/G_1$; hence
$\ell\Gamma(s)$ has uniformly bounded derivatives with respect to $v$ through
order two, without any regularity assumption at $y$.
For short chords, choose $\rho$ so the penalty is $1/s$ and write
\[
 \phi_{t,\omega}(v)=\frac{N(v,t,\omega)}{|K(v,t,\omega)|},
 \quad |\omega|=1,\quad |t|\leq\rho,
\]
where the smooth extensions to $t=0$ are
\[
 \begin{aligned}
 N&=\left|\int_0^1DX(v+\theta t\omega)\omega\,d\theta\right|^2\geq1,\\
 K&=\nu(X(v))\cdot\int_0^1(1-\theta)
 D^2X(v+\theta t\omega)[\omega,\omega]\,d\theta.
 \end{aligned}
\]
A branch of value at most $2C$ has $|K|\geq1/(2C)$, so $K$ has a fixed sign
and the branch has uniformly bounded first and second derivatives with respect to $v$.
The compact parameters include all diagonal limits. Enlarging the
value cutoff by a fixed factor and shrinking the graph patch once
preserves these denominator bounds by uniform continuity. The resulting
compact family represents every minimizing sequence, including its
off-diagonal or tangent-direction limit. Subtract a common quadratic
from its branches: their infimum is concave. This proves local
semiconcavity and Lipschitz continuity. The singular Hessian is a
nonpositive matrix measure, so its metric trace is nonpositive; the
lower-order metric term is locally bounded.

When $D\leq1/4$, universal graph radius $\rho$ and fourth-order
graph bounds follow from $|\nabla U|\geq\sqrt3/2$.
At a differentiability point an active branch has derivative $\nabla d$.
For a long chord, $s\geq\ell/d$, $|\nabla\ell|\leq1$,
$|\nabla s|\leq C+\ell^{-1}$ and
$|\Gamma'|\leq C_\Gamma s^{-2}$ give
\[
 |\nabla(\ell\Gamma(s))|\leq
 C_\Gamma(d/\ell+d^2/\ell+d^2/\ell^2)
 \leq C_{\Gamma,\rho}(1+d)^2.
\]
For a short or diagonal branch, $N\geq1$,
$|\nabla N|+|\nabla K|\leq C$ and $|K|=N/d$ give
$|\nabla\phi_{t,\omega}|\leq C(d+d^2)$.
\end{proof}

\begin{corollary}
\label{cor:coarse-gradient-profile-exhaustion}
Let $f:[t_{\min},\infty)\to(0,1]$ be locally Lipschitz,
$|f'|\leq C_f f$, and a positive multiple of $e^{-\kappa t/2}$
for all sufficiently large $t$, with fixed $\kappa>0$. On any good
set $\mathcal G\subset\Sigma_g\cap\{d\geq t_{\min}\}$, the restriction
$f(d)$ is locally Lipschitz and
\[
 |\nabla_\Sigma f(d)|
 \leq C_\Gamma(1+d)^2|f'(d)|\leq C_{\Gamma,f}<\infty
 \quad\hbox{almost everywhere}.
\]
Consequently $\int_{\mathcal G}\chi^2|\nabla_\Sigma f(d)|^2\,dS<\infty$
for every fixed spatial weight satisfying
$0<\chi(z)\leq C_\chi e^{-c_\chi|z|}$, with $c_\chi>0$.
In the planar case put $f(\infty)=0$.
\end{corollary}
\begin{proof}
The coarse gradient bound and the Lipschitz chain rule give the displayed
estimates. At a corner of the profile, the gradient of $d$ vanishes
almost everywhere on the corresponding level set.  Uniform good charts and
an ambient unit-ball cover give
$\mathcal H^m(\Sigma_g\cap B_r)\leq C(1+r)^n$, so the fixed
exponential spatial weight makes the bound integrable. This integrability
argument allows dependence on the fixed profile and spatial weight.
\end{proof}

\subsection{Oriented mean curvature at contact points}

\begin{theorem}
\label{thm:r4-positive-foot-transfer}
Choose a measurable family of actual active plateau pairs $(z,y(z))$,
one pair for each base point $z$ in a measurable subset of
$\Sigma_{\rm reg}$.  Suppose $d(z)$ exceeds the fixed threshold of
Proposition~\ref{prop:r4-actual-contact-regularity}.  Put
\[
 \ell=|z-y|=d(z)/G_1,\qquad N_y=(z-y)/\ell,\qquad
 s=|N_y\cdot\nu(z)|\geq1-\delta_2.
\]
The receiving normal is $N_y$, and $d(y)\leq d(z)$.  Restricted to pairs
with $H_{N_y}(y)>0$, each receiving foot is paired with at most one base point.
For every nonnegative measurable function $F$ of a selected pair,
\begin{equation}\label{eq:r4-positive-foot-area}
 \int F(z,y(z))\mathbf1_{\{H_{N_y}(y(z))>0\}}\,dS_z
 =\int_{Y_+}F(z(y),y)
       \frac{|\det(\operatorname{Id}+\ell A_{N_y}(y))|}{s}\,dS_y.
\end{equation}
Here $Y_+$ consists of the selected receiving feet $y$ with $H_{N_y}(y)>0$;
null critical images may be omitted on the left.  The inverse Jacobian obeys
\begin{equation}\label{eq:r4-positive-foot-J}
 J_{y\to z}\leq C_\Gamma,\qquad
 J_{y\to z}\leq\frac{(1+\ell a(y))^m}{1-\delta_2}.
\end{equation}
Thus $J_{y\to z}=1+O_m(\delta_2+\varepsilon)$ as an upper bound when
$\ell a(y)\leq\varepsilon\leq1$.
\end{theorem}

\begin{proof}
The plateau objective is $G_1|y-z|$.  Its first variation in $y$ along
the smooth receiving sheet vanishes, so the chord is normal at $y$.
The zero-free region $\mathcal O_z$ shows that its inward normal is
$N_y$.  The reverse normal chord at $y$ has objective $G_1\ell$,
which proves $d(y)\leq d(z)$.

At a point with $H_{N_y}>0$ there is only one choice of normal giving
positive trace.  Every base point whose selected receiving trace is positive at this
foot must therefore lie on the same oriented ray $y+tN_y$.
If two occurred at lengths $0<\ell_1<\ell_2$, the nearer base point
$z_1$ would be a competing zero for the farther base point $z_2$.
Its chord direction from $z_2$ is the same as the direction towards
$y$, but its length is $\ell_2-\ell_1<\ell_2$.  Its penalty is unchanged,
contradicting activity of $(z_2,y)$.  This proves multiplicity one.

Near a selected regular pair the implicit function theorem applied
to the sheet through $z$ supplies a smooth inverse branch
$z(y)=y+\ell(y)N_y$, since its transverse derivative has absolute
value proportional to $s>0$.  Tangential projection of its derivative
onto $T_y\Sigma$ is $\operatorname{Id}+\ell A_{N_y}$: the term
$(d\ell)N_y$ projects to zero.  Orthogonal projection from
$T_z\Sigma$ to $T_y\Sigma$ has area factor $s$.  This proves the
displayed Jacobian exactly.

To apply this calculation globally, cover the regular pair relation by
countably many inverse charts. Partition the selected graph into disjoint
measurable pieces subordinate to this cover.  The area formula on each chart gives the equality above;
the image of its Jacobian-zero set has zero surface measure.
When we sum over these pieces, the ray argument and the choice of one
contact at each base point exclude overlaps.
Finally \eqref{eq:r4-receiving-defect} gives $\ell a(y)\leq C_\Gamma$,
and the determinant bound follows from its singular values.
\end{proof}

A measurable selection exists on the plateau-contact set. Indeed $d$
is continuous on regular patches by the preceding theorem, so the active
plateau relation is Borel. Its fibers are bounded by $|y-z|\leq d/G_1$
and closed; they cannot approach the diagonal, since
$|\nu(z)\cdot(y-z)|/|y-z|\to0$ as $y\to z$ along the regular graph,
whereas plateau contacts have this quantity at least $1-\delta_2$.
Thus the fibers are compact. The measurable selection
theorem on a countable regular-patch exhaustion gives a selection for
the completed surface measure.

\begin{corollary}
\label{cor:r4-weighted-foot-transfer}
Let $W$ be nonnegative and nonincreasing.  Let $\chi>0$ satisfy
$|\nabla\log\chi|\leq R^{-1}$ in physical space.  On selected plateau
pairs with $\ell\leq L$, for every nonnegative receiving function $h$,
\begin{equation}\label{eq:r4-weighted-foot-transfer}
 \int\chi(z)^2W(d(z))h(y(z))\mathbf1_{\{H_{N_y}>0\}}\,dS_z
 \leq C_\Gamma e^{2L/R}
        \int\chi(y)^2W(d(y))h(y)\,dS_y.
\end{equation}
If $\ell a(y)\leq\varepsilon$ throughout the selected pairs, replace
$C_\Gamma$ by $(1+\varepsilon)^m/(1-\delta_2)$.
Pairs landing in the fixed physical core have their base points in one fixed
enlargement of that core, by \eqref{eq:r4-core-contact-localization}.
For bounded multipliers their integrals are consequently finite
physical-core errors, independent of the exhaustion radius.  More
precisely, there is a fixed $R_{u,T}<\infty$ such that the contribution
of pairs landing in $B_T$ is at most
\begin{equation}\label{eq:r4-core-foot-located}
 C_{u,T,W,h}\,\chi(0)^2 e^{2R_{u,T}/R},
\end{equation}
provided $W(d(z))h(y(z))$ is bounded on those pairs.  The estimate
locates the error at the physical core for translated exponential cutoffs.
\end{corollary}

\begin{proof}
Use \eqref{eq:r4-positive-foot-area},
$\chi(z)^2/\chi(y)^2\leq e^{2\ell/R}$, and
$W(d(z))\leq W(d(y))$.  For the final assertion, both endpoints of
a large-reach pair have uniformly nonvanishing gradient by
\eqref{eq:r4-contact-defect}--\eqref{eq:r4-receiving-defect}.
The part of their zero sets in the fixed enlarged ball therefore has
finite area, by fixed-radius graph representations and a finite cover.
All base points in question lie in $B_{R_{u,T}}$, so
$\chi(z)^2\leq\chi(0)^2e^{2R_{u,T}/R}$, which proves
\eqref{eq:r4-core-foot-located}.
\end{proof}

\vskip0.36in

\section{\bf Comparison between defect and interaction}
\label{sec:s12-comparison}
\label{sec:r10-comparison}\label{sec:r9-comparison}
We next prove $D\leq C(J_\lambda+E)$ and $E\leq C(J_\lambda+D)$
on the same logarithmic window. For the first inequality, we compare the
two inward slopes with those of auxiliary half-space solutions. For the
second, we use the nonnegative source $F_{n_0}$ in \eqref{eq:r12-positive-directional-source}. The chamber construction
below verifies the local geometric hypotheses needed in
\cite[Proposition~8.5]{OpenAI2026}. We denote a unit plane normal by
$\mathbf N$ and its horizontal component by $\mathbf N'$.

\subsection{Logarithmic graph ordering}
\label{sec:r5-sharp-lower}

Fix the penalty in \eqref{eq:r4-penalty-assumptions}.  In particular
$G_1=\min\Gamma$, and $\Gamma=G_1$ on a fixed neighbourhood of $1$.
For $0<\delta\leq1/4$, put
\[
 X_\delta=\{y\in\Sigma:D(y)\geq\delta\},\qquad
 \bar a_\delta(y)=\begin{cases}1,&y\in X_\delta,\\a(y),&y\notin X_\delta,
 \end{cases}
 \qquad
 J_\lambda(z)=\sup_{y\in\Sigma}e^{-\lambda|y-z|}\bar a_\delta(y)^2.
\]
Every singular zero belongs to $X_\delta$, so this definition never asks
for its second fundamental form.  Recall that $d$ is the actual reach
\eqref{eq:r4-actual-reach}, and set $E=e^{-\sqrt2 d/G_1}$, with $E=0$
when $d=\infty$.

\begin{lemma}
\label{lem:r5-log-ordering}
Let $P_j\to\infty$ and $\varepsilon_jP_j^2\to0$.  Suppose that all zeros
in $B_{10P_j}(0)$ are regular, that $D\leq1/4$ there, and that
$|A|\leq\varepsilon_j$ there.  Assume $0\in\Sigma$ and $\nu(0)=e_n$.
For each sufficiently small fixed $\eta>0$, and all large $j$, the
following assertions hold.

The sheet through $0$ is a graph $t=h_0(x')$ on $|x'|<2P_j$, with
$h_0(0)=0$, $\nabla h_0(0)=0$, and
$|\nabla^2h_0|\leq C\varepsilon_j$.  Every zero in $B_{\eta P_j}$ has
unoriented normal at angle $O(\eta)+o(1)$ from $e_n$.  Every such zero
in a sufficiently smaller fixed box belongs to a graph over a common
horizontal disk of radius $cP_j$.  The graphs meeting a closed smaller
box are finite in number and are ordered.  Here $c>0$ can be fixed
after $\eta$.

If one of these graphs has axis intercept $t>0$ with
$t<c\eta P_j$, its distance $\ell$ from $0$, minimized in this graph
patch, is attained at an interior foot $y$.  With $\mathbf N=y/\ell$,
\begin{equation}\label{eq:r5-chord-angle}
 \ell\asymp t,\qquad
 |\mathbf N-e_n|\leq C t/P_j+C P_j\varepsilon_j.
\end{equation}
The analogous assertions hold below $0$.
\end{lemma}

\begin{proof}
At each zero, $|\nabla u|\geq\sqrt{3/8}$ and universal interior
bounds give one uniform implicit-graph radius. Tangent projection
from any zero in $B_{P_j}$ continues along radial segments to radius
$2P_j$: while the projected path has length at most $2P_j$, its lift
has length at most $3P_j$ and normal variation at most
$3\varepsilon_jP_j=o(1)$. It stays in $B_{10P_j}$ and cannot lose
invertibility. A terminal limit is a regular zero, so continuation
cannot stop. Straight-line homotopies between the radial path and
its short endpoint perturbations have length $O(P_j)$ and the same
property. Uniqueness of local lifts then gives a single smooth graph,
with slope $O(\varepsilon_jP_j)$ and tangent-plane height error
$O(\varepsilon_jP_j^2)=o(1)$.

For a zero $y\in B_{\eta P_j}$ put
$\beta=|\operatorname{proj}_{T_y\Sigma}e_n|$, and choose tangent
unit $\tau$ with $e_n\cdot\tau=\beta$. Its graph points at
$\pm c_0P_j\tau$ have heights relative to the base graph
$y_n\pm c_0P_j\beta+o(1)$. If $\beta>C\eta+o(1)$, the signs
are opposite, so the joining path meets the base graph. At that
intersection the normal differs by $o(1)$ from both its original
normal and $e_n$, a contradiction. This argument also excludes an
initially horizontal unoriented normal without presupposing a
horizontal graph for that sheet.

In a smaller box, every piece therefore extends as a horizontal graph
over one disk of radius $cP_j$. By local uniqueness and continuation,
two such graphs either coincide or are disjoint. Thus they are ordered. Infinitely many meeting a closed
smaller box would accumulate at a regular zero, contradicting its
uniform local graph radius. The number need not be uniform in $j$.

For a graph $h$ with $h(0)=t>0$, apply Taylor's formula to
$h-h_0>0$ a distance $cP_j$ opposite to $\nabla h(0)$:
\[
 |\nabla h(0)|\leq Ct/P_j+CP_j\varepsilon_j.
\]
Since the graph slope is small, its distance from zero is comparable
to $t$ and at most $t$. Its boundary is at distance at least $cP_j\gg t$,
so the minimizing foot is interior and the chord is normal there.
The Hessian bound propagates the same slope estimate to the foot,
giving \eqref{eq:r5-chord-angle}. Reflection treats the lower side.
\end{proof}

\subsection{The linearized equation on two half-spaces}
Put $g(t)=\tanh(t/\sqrt2)$, $V_g=3g^2-1$, and let
\[
 L_g=-\Delta_y-\partial_t^2+V_g(t),\qquad t>0.
\]
Fix $p=\max\{12,n+1\}>n=m+1$. Weighted uniformly local norms will mean
\[
 \|f\|_{k,p,b}
 =\sup_{(y,t)\in\mathbb R^m\times[0,\infty)}
 e^{-b|y|}\|f\|_{W^{k,p}(B_1(y,t)\cap\{t>0\})}.
\]
The $k=0$ norm denotes $L^p$. For $n\leq9$ we use $p=12$.
Inverse and Sobolev constants may depend on $m,p$ and on the displayed
weight exponent $b$.

\begin{lemma}
\label{lem:r4-paired-slopes}
There are $b_*>0$, $\epsilon_*>0$ and $C<\infty$ with the following
property. Let $h\in C^\infty(\mathbb R^m)$ have bounded derivatives,
\[
 h(0)=0,\quad\nabla h(0)=0,\qquad
 \|\nabla h\|_\infty+\|D^2h\|_\infty\leq\epsilon_*.
\]
For $0<b\leq b_*$ set
\[
 M_b=\sup_y e^{-b|y|}(|\nabla h(y)|+|D^2h(y)|).
\]
There are positive bounded solutions $v_h^\pm=g+\phi_\pm$, with zero
trace at $t=0$, such that
\[
 x=(y,h(y)\pm t),\qquad
 v_h^\pm(y,t)\text{ solves }\Delta_xv=v^3-v
\]
in the corresponding upper or lower graph half-space. They satisfy
$0<v_h^\pm\leq1$, are the unique solutions in a fixed small uniformly
local $W^{2,p}$ neighborhood of $g$, and obey
\begin{equation}\label{eq:r4-paired-slope-bound}
 \|\phi_\pm\|_{2,p,b}\leq CM_b,
 \qquad
 |\partial_t\phi_+(0,0)+\partial_t\phi_-(0,0)|
 \leq C M_b^2.
\end{equation}
The estimate constant $C$ depends only on $m,p,b$.
\end{lemma}
\begin{proof}
We first construct the inverse of the linear operator. Write
$\mathcal D=\partial_t+\sqrt2g$ and
$\mathcal D^*=-\partial_t+\sqrt2g$, so that
\begin{equation}\label{eq:s7-darboux-factorization}
 -\partial_t^2+V_g=\mathcal D^*\mathcal D,\qquad
 \mathcal D\mathcal D^*=-\partial_t^2+2-G,
 \qquad\mathcal Dg'=0,
 \quad G=1-g^2.
\end{equation}
Odd-extend a half-space source $f$ to an $L^p_{\rm loc}$ function $F$,
and let $K$ be the positive whole-space fundamental solution of $-\Delta+2$.
Its heat-kernel representation gives $\int K=1/2$, exponential tails,
and $|\nabla K(X)|\leq C|X|^{1-n}$ near zero.
Since $p>n$, H\"older's inequality and a unit-block sum give bounded
maps $K*,\nabla K*:L^p_{\rm uloc,b}\to L^\infty_b$.
In the norm
$\|Z\|_{\infty,b}=\sup_{y,t}e^{-b|y|}|Z(y,t)|$, solve
\begin{equation}\label{eq:s7-darboux-resolvent}
 Z=(\partial_tK)*F+\sqrt2K*(gF)+K*(GZ).
\end{equation}
The last term has norm at most
$\int K(X)e^{b|X'|}\,dX=1/2+o_{b\downarrow0}(1)<3/4$ for small
fixed $b$.  Contraction therefore gives an even function $Z$ with
$\|Z\|_{\infty,b}\leq C\|f\|_{0,p,b}$ and, in distributions,
$(-\Delta_y-\partial_t^2+2-G)Z=\mathcal DF$.
The nonsmooth extension $F$ enters this estimate only through its norm.

Define
\begin{equation}\label{eq:s7-pinned-recovery}
 v(y,t)=g'(t)\int_0^t\frac{Z(y,s)}{g'(s)}\,ds.
\end{equation}
For $s$ between $0$ and $t$,
$g'(t)/g'(s)\leq4e^{-\sqrt2(|t|-|s|)}$; hence
$\|v\|_{\infty,b}\leq2\sqrt2\|Z\|_{\infty,b}$.
The function $v$ is odd, vanishes at $t=0$, and satisfies
$\mathcal Dv=Z$.
Thus $\mathcal D(L_gv-F)=0$ in distributions.  Its residual is
$C(y)g'(t)$; it is also odd, whereas $g'$ is even, so the residual
vanishes.  Unit-scale elliptic estimates on the odd full-space
equation $L_gv=F$ now give
\begin{equation}\label{eq:s7-halfspace-inverse-bound}
 \|L_g^{-1}f\|_{2,p,b}\leq C\|f\|_{0,p,b}.
\end{equation}
Regularity across $t=0$ follows after the distributional equation.

For uniqueness, the odd extension of a homogeneous Dirichlet solution
in the weighted uniformly local $W^{2,p}$ class introduces no boundary
distribution.  Its transform
$Z=\mathcal Dv$ is weighted bounded and solves the homogeneous partner
equation.  The residual $Z-K*(GZ)$ solves the homogeneous equation
for $-\Delta+2$.  The barrier $e^{c\sqrt{1+|y|^2+t^2}}$, with $b<c$
sufficiently small, forces this weighted bounded residual to vanish.
The contraction estimate gives $Z=0$. The zero boundary value in the
recovery formula then gives $v=0$.

Flattening the graph gives the fixed point equations
\begin{align}\label{eq:r4-artificial-equation}
 L_g\phi_\pm={}&|\nabla h|^2(g''+\partial_{tt}\phi_\pm)
 \mp2\nabla h\cdot\nabla_y\partial_t\phi_\pm
 \mp\Delta h(g'+\partial_t\phi_\pm)\\
 &-3g\phi_\pm^2-\phi_\pm^3.\nonumber
\end{align}
The inverse estimate and the ambient-dimensional embedding
$W^{2,p}_{\rm uloc}(\mathbb R^n_+)\subset C^{1,\alpha}$,
$\alpha=1-n/p>0$, make the right side followed by $L_g^{-1}$ a contraction in a
small unweighted ball when $\epsilon_*$ is small. The same estimates
with the small unweighted coefficient terms absorbed prove
$\|\phi_\pm\|_{2,p,b}\leq CM_b$. The solutions are positive by
$C^1$ smallness near zero and $C^0$ smallness for $t$ bounded away from
zero. For the bound by one put $w=(v_h^\pm-1)_+$ in physical
coordinates. Kato's inequality gives $(-\Delta+2)w\leq0$, since
$v^3-v\geq2(v-1)$ for $v\geq1$. The bounded maximum principle
for $-\Delta+2$, with zero bottom trace, gives $w=0$.

The sharper estimate for the sum of the two slopes follows from the
cancellation of the linear graph error. To see this, add the two equations. Put
$S=\phi_++\phi_-$ and $T=\phi_+-\phi_-$. The linear sources
$\mp\Delta h\,g'$ cancel exactly, leaving
\[
 L_gS=|\nabla h|^2(2g''+S_{tt})
       -2\nabla h\cdot\nabla_yT_t-\Delta h\,T_t
       -3g(\phi_+^2+\phi_-^2)-(\phi_+^3+\phi_-^3).
\]
Every term has weighted $L^p$ norm with exponent $2b$ at most $CM_b^2$:
the first derivatives of $\phi_\pm$ are bounded by the $W^{2,p}$
embedding; a cubic has one unweighted small factor. Choose $b_*$ so that
the inverse is bounded also at $2b$. Consequently
$\|S\|_{2,p,2b}\leq CM_b^2$, and boundary Sobolev embedding gives
\eqref{eq:r4-paired-slope-bound}.
\end{proof}

\subsection{Comparison across a chamber}
\label{sec:r5-sharp-upper}
We incorporate the decaying potential error into a positive
one-dimensional comparison function to obtain the exponent $\sqrt2$
throughout a long chamber. The local hypotheses for the chamber argument
in \cite[Proposition~8.5, Step~4]{OpenAI2026} are given below.

Write $a=\sqrt2$, $g(t)=\tanh(t/\sqrt2)$, and $V_g=3g^2-1$.
Let $Q_R=\{y\in\mathbb R^m:|y|_\infty<R\}$ and fix
$\vartheta_m=[100(1+\sqrt m)]^{-1}$.
The chamber estimates use the norm conversion
$|y|\leq\sqrt m|y|_\infty$, with dimensional factors included in the constants.

\begin{lemma}
\label{lem:r5-sharp-barrier}
Fix $0<\gamma<a$. There are $\theta_0,\eta_0,\kappa_0,c,C>0$,
depending only on $m,\gamma$, and smooth functions $\varphi,v_0$ on
$[-\theta_0,\infty)$ such that
\[
 c\leq\varphi,v_0\leq C,\quad
 |\varphi'|\leq Ce^{-\gamma t},\quad
 |v_0'|+|v_0''|\leq C.
\]
Suppose $t_1=x_n$, $t_2=\ell-\mathbf N\cdot x$, $|\mathbf N|=1$,
$\mathbf N_n\geq1-\kappa_0$, and
\[
 V(x)\geq V_g(t_1)-\eta_0 e^{-\gamma t_1}.
\]
On $t_1\geq-\theta_0$, both
$e^{-at_2}\varphi(t_1)$ and $v_0(t_1)$ are supersolutions of
$-\Delta+V$, the latter strictly so. If $l$ is any inward affine unit
face distance, $v_0(t_1)e^{-c_0l}$ is also a supersolution for a fixed
sufficiently small $c_0>0$.
\end{lemma}
\begin{proof}
Put $A(t)=g'(t)$ and $t_*=-\theta_0$. Since $A>0$,
\[
 Y(t)=A(t)\left(1+\int_{t_*}^t A(s)^{-2}\,ds\right)>0,
 \qquad (-\partial_t^2+V_g)Y=0.
\]
Moreover $Y(t)\asymp e^{at}$ as $t\to\infty$. Define
\[
 k'(t)=-Y(t)^{-2}\int_{t_*}^tY(s)^2e^{-\gamma s}\,ds.
\]
Then $|k'|\leq Ce^{-\gamma t}$ and $k'$ is integrable. Choose the
additive constant so that $1\leq k\leq C$, and set
$\varphi=e^{-at}Yk$. Direct differentiation gives
\begin{equation}\label{eq:r5-phi-reserve}
 -\varphi''-2a\varphi'+(V_g-2)\varphi
   =e^{-at}Y e^{-\gamma t}\geq ce^{-\gamma t}.
\end{equation}
The explicit formula for $g'$ also gives
$(e^{-at}Y)'=O((1+t)e^{-at})$, so the asserted bounds for
$\varphi,\varphi'$ follow. Consequently
\[
 (-\Delta+V)(e^{-at_2}\varphi(t_1))
 =e^{-at_2}\{-\varphi''-2a\mathbf N_n\varphi'+(V-2)\varphi\}\geq0
\]
when $\kappa_0,\eta_0$ are sufficiently small, by
\eqref{eq:r5-phi-reserve}.

We also need a strictly positive supersolution for the remaining boundary
terms. We construct $v_0$ explicitly.
Set $Y_0(t)=A(t)\int_{t_*}^t A(s)^{-2}\,ds$ and
\[
 v_D(t)=A(t)\int_{t_*}^tY_0(s)\,ds
             +Y_0(t)\int_t^\infty A(s)\,ds,
 \qquad v_0=v_D+A.
\]
Then $v_D(t_*)=0$, $v_D>0$ for $t>t_*$,
$(-\partial_t^2+V_g)v_D=1$, and $v_D\to1/2$ at infinity.
These facts and the displayed formulas give all the stated bounds
for $v_0$, including its strictly positive lower bound. Thus
$(-\Delta+V)v_0\geq1-C\eta_0\geq1/2$.
Applying the operator to $v_0e^{-c_0l}$ adds terms bounded by
$C(c_0+c_0^2)e^{-c_0l}$, proving the last assertion.
\end{proof}

\begin{lemma}
\label{lem:r5-polytope-clearing}
Let $l_1,\ldots,l_k$ be inward affine unit distances and let their
bounded positive intersection be contained in a positive phase of an
Allen--Cahn solution $0<v\leq1$. There are constants $M,C$, depending
only on $k$ and dimension, such that, whenever $\min_jl_j(x)\geq M$,
\[
 1-v(x)\leq C\sum_{j=1}^k e^{-a l_j(x)}.
\]
The same assertion holds for a bounded positive artificial graph
half-space solution on such an intersection.
\end{lemma}
\begin{proof}
At every point where all $l_j\geq M$, the phase contains its ball of
radius $M$ (or every slightly smaller closed ball). The phase-ball
clearing lemma therefore gives $w=1-v\leq C/M^2$ there.
Choose $\zeta>0$ small with $k\zeta<1/8$, and then choose $M$ so
that $w\leq\zeta/2$. On the inner polytope $\{l_j>M\}$ put
\[
 r_j=\zeta e^{-a(l_j-M)},\qquad
 W=\sum_{j=1}^k(r_j-kr_j^2).
\]
One has $0<W\leq k\zeta$, and on each boundary face
$W\geq\zeta-k\zeta^2\geq\zeta/2$. Since $|\nabla l_j|=1$,
\[
 (-\Delta+2)W=6k\sum_jr_j^2
      \geq6\left(\sum_jr_j\right)^2\geq3W^2.
\]
The deficit $w$ satisfies
$(-\Delta+2)w=3w^2-w^3$. The zeroth-order coefficient in the equation
for $w-W$ is positive while both functions are at most $1/8$.
The maximum principle yields $w\leq W$, which proves the claim.
\end{proof}

\begin{theorem}
\label{thm:r5-local-sharp-intercept}
Fix $b>0$ sufficiently small for
Lemma~\ref{lem:r4-paired-slopes}. There exist fixed
$\kappa_{\rm fin},\epsilon,\delta>0$ and $\ell_0,C,c>0$ with the following
property. Let $L\geq1$ and let $h$ be a smooth function on
$Q_{4L}$ whose graph lies in $\{u=0\}$ and which satisfies
\begin{equation}\label{eq:r5-flat-trap}
 \begin{gathered}
 h(0)=0,\quad \nabla h(0)=0,\quad
 D(y,h(y))\leq\delta\quad(y\in Q_{2L}),\\
 \|h\|_\infty+\|\nabla h\|_\infty+\|D^2h\|_\infty\leq\epsilon.
 \end{gathered}
\end{equation}
Orient the solution so that its upper chamber is positive and its
lower chamber is negative. Suppose that on each side, after using the inward coordinates
$x=(y,s)$ with $s=x_n$ on the upper side and $s=-x_n$ on the lower
side, one of the following alternatives holds.

\begin{enumerate}
\item There is a smooth adjacent actual zero graph $s=k_\pm(y)$ on $Q_{4L}$
with no zero between it and the base graph, and the actual solution has
the expected sign throughout that chamber. There are a unit vector
$\mathbf N_\pm$ with $(\mathbf N_\pm)_n>0$ and an intercept $\ell_\pm$ such that
\begin{equation}\label{eq:r5-neighbour-trap}
 \ell_0\leq\ell_\pm\leq\kappa_{\rm fin} L,\qquad
 |(\mathbf N_\pm)'|L\leq\vartheta_m\ell_\pm,\qquad
 |\mathbf N_\pm\cdot(y,k_\pm(y))-\ell_\pm|\leq\epsilon.
\end{equation}
Here the base graph is $s=h(y)$ or $s=-h(y)$, respectively.
\item The actual solution has the expected constant sign on the
empty buffered region between the base graph and $s=2\kappa_{\rm emp} L$,
for all $y\in Q_{4L}$. In this case put $\ell_\pm=+\infty$.
\end{enumerate}
Here $\kappa_{\rm emp}=\kappa_{\rm fin}/100$ is fixed independently
of $L$. The separation of these constants permits a near/far neighbor
split without a threshold gap.
Define the actual curvature-square envelope
\[
 \mathcal J_b=\sup_{y\in Q_{4L}}
                  e^{-2b|y|}|A(y,h(y))|^2.
\]
Then
\begin{equation}\label{eq:r5-sharp-defect-upper}
 D(0)\leq C\mathcal J_b
       +C\bigl(e^{-\sqrt2\ell_+}+e^{-\sqrt2\ell_-}\bigr)
       +Ce^{-cL}.
\end{equation}
All constants are independent of $L$ and of the adjacent intercepts.
The exponential constant $c$ can be chosen independently of sufficiently
small $b$; only $C$ in front of $\mathcal J_b$ may depend on $b$.
The $L^\infty$ height bound in \eqref{eq:r5-flat-trap} is part of the
hypothesis; it has not been inferred from a small slope on a long box.
\end{theorem}
\begin{proof}
Extend $h$ by a cutoff equal to one on $Q_{2L}$ and supported in
$Q_{3L}$, with derivatives $O(L^{-1}),O(L^{-2})$. The global graph
$\widetilde h$ has unweighted $C^2$ size $O(\epsilon)$ and supplies
the two artificial solutions $v_*^\pm$ of
Lemma~\ref{lem:r4-paired-slopes}. Radial integration from
$h(0)=\nabla h(0)=0$, $|D^2h|\leq C|A|$, and
$\int_0^r(r-s)e^{bs}\,ds\leq C_be^{br}$ give
\begin{equation}\label{eq:r5-envelope-extension}
 M_b^2\leq C_b\mathcal J_b.
\end{equation}
Indeed $e^{-b|y|}(|h|+|\nabla h|+|D^2h|)\leq
C_b\mathcal J_b^{1/2}$ on $Q_{3L}$, and the cutoff preserves this
bound since $L\geq1$. Using the whole graph window in this extension makes the rate of the
boundary error independent of $b$.

We prove the comparison in the upper chamber. Reflection gives the
lower-chamber estimate with the same extension.
Set $t_1=x_n$. In the finite case set $t_2=\ell-\mathbf N\cdot x$
and take the outer polytope
$\mathcal P=\{t_1>\epsilon,t_2>\epsilon,y\in Q_{4L}\}$.
In the empty case take
$\mathcal P=\{\epsilon<t_1<2\kappa_{\rm emp}L,y\in Q_{4L}\}$.
The graph traps put both polytopes in the actual positive phase and
above the artificial graph. The inner domain $\Omega$ lies over
$Q_L$, above the base, and below $t_1=t_2$ in the finite case or
$t_1=\kappa_{\rm emp}L$ in the empty case.
In its finite top, $t_1,t_2\asymp\ell$ by the tilt bound. Taking
$\ell_0$ larger than the fixed clearing threshold gives
\begin{equation}\label{eq:r5-midpoint-data}
 1-u,\ 1-v_*^+\leq Ce^{-at_2}+Ce^{-cL}
 \qquad(t_1=t_2,\ y\in Q_L)
\end{equation}
by Lemma~\ref{lem:r5-polytope-clearing}.

The difference $w=u-v_*^+$ has zero bottom trace and satisfies
$(-\Delta+V_*)w=0$, with
$V_*=u^2+uv_*^++(v_*^+)^2-1$. Uniformly in chamber length,
\begin{equation}\label{eq:r5-decaying-potential}
 V_*\geq V_g(t_1)-\eta_0e^{-\gamma t_1}\quad\hbox{on }\Omega.
\end{equation}
To check this, fix $0<\gamma<a$ first. On every fixed depth interval
$[-\epsilon,T]$, the base relative jets give
$|u-g(t_1)|\leq C_T(\sqrt\delta+\epsilon)$: integrate
$\boldsymbol p=\nabla U$, using
$\boldsymbol p=\sqrt{1-D}\,\nu$ on the zero graph and
$|D^2U|\leq C_T\sqrt\delta$. The artificial solution has the same
fixed-depth convergence by its construction.
For increasing depth, the bottom and top face distances of
$\mathcal P$ in $\Omega$ are at least $t_1-\epsilon$, while lateral
distances are at least $3L$. Decreasing $\kappa_{\rm fin}$ ensures
$t_1\leq3L$. Polytope clearing therefore yields
$1-u+1-v_*^+\leq Ce^{-at_1}$ once $t_1$ exceeds a fixed threshold.
First choose $T$ so $Ce^{-(a-\gamma)T}\leq\eta_0/2$, then choose
$\epsilon,\delta$ so the fixed-depth error is at most
$(\eta_0/2)e^{-\gamma T}$. This proves
\eqref{eq:r5-decaying-potential} in both chambers.

The tilt gives $1-\mathbf N_n\leq C\kappa_{\rm fin}^2$.
Lemma~\ref{lem:r5-sharp-barrier} thus supplies the finite-top barrier
$e^{-at_2}\varphi(t_1)$ and, for both chambers, the side barriers
$v_0(t_1)e^{-c_0(L\pm y_j)}$. Add $e^{-cL}v_0$ for the remainder
in \eqref{eq:r5-midpoint-data}. In the empty case replace the
finite-top barrier by
$v_0(t_1)e^{-c_0(\kappa_{\rm emp}L-t_1)}$.
A fixed multiple of the sum dominates $|w|$ on every face.
Division by the strictly positive supersolution $v_0$ gives a positive
zeroth-order coefficient, justifying comparison despite the sign of $V_*$.
Consequently $|w|\leq C(e^{-a\ell}+e^{-cL})$ near zero, omitting
the interaction term in the empty case. Uniform boundary gradient
estimates give
\begin{equation}\label{eq:r5-slope-transmission}
 |\partial_nu(0)-\partial_tv_*^+(0,0)|
 \leq C(e^{-a\ell}+e^{-cL}).
\end{equation}
The required local coefficient and bottom-graph norms are uniformly
bounded by the actual solution estimates and relative jets; enlarging
$C$ covers bounded $L$ if necessary.

Apply \eqref{eq:r5-slope-transmission} on both sides. The actual
inward slopes coincide, whereas the artificial perturbation slopes
sum to $O(M_b^2)$ by Lemma~\ref{lem:r4-paired-slopes}.
Since $D(0)=1-2\partial_nu(0)^2$, \eqref{eq:r5-envelope-extension}
proves \eqref{eq:r5-sharp-defect-upper}.
\end{proof}

\subsection{Envelope localization and pointwise comparison}
\label{sec:r5-global-comparison}
Both comparison inequalities use the following logarithmic window.
After excluding bad zeros and continuing the local graphs, we distinguish
nearby adjacent graphs from empty buffered chambers.

\begin{lemma}
\label{lem:s2-envelope-localization}\label{lem:r5-adjacent-realization}
Fix an angular tolerance $0<\vartheta<1/10$. There are
$A_0,c_1,\lambda>0$, depending only on $m,\Gamma,\vartheta$,
with the following property for every $\delta\in(0,1/4]$. Put
\[
 B=J_\lambda(z)+E(z),\qquad P=A_0\log(1/B),\qquad L=c_1P.
\]
For $0<B$ sufficiently small, after translating $z$ to zero and rotating
its normal to $e_n$, the base graph on $Q_{4L}$ satisfies the height,
slope, Hessian and low-defect assumptions of
Theorem~\ref{thm:r5-local-sharp-intercept}.
On each side, the chamber of constant sign adjacent to the base graph
satisfies either the finite adjacent-graph or the empty-buffer alternative
of that theorem.
Each finite side has
$e^{-\sqrt2\ell_\pm}\leq E(z)$, and its full-window curvature envelope
satisfies $\mathcal J_b\leq J_\lambda(z)$.
Moreover $e^{-cL}\leq B$ for the lateral rate $c$ of that theorem.

If $E(z)\geq J_\lambda(z)$, one of these adjacent sides realizes
\begin{equation}\label{eq:r5-sharp-reach-realization}
 d(z)=G_1\ell,\qquad E(z)=e^{-\sqrt2\ell}.
\end{equation}
The receiving gradient-oriented normal on that side makes angle at most
$\vartheta$ with $-\nu(z)$.
\end{lemma}
\begin{proof}
Fix first the inverse exponent $b$ and chamber constants
$\kappa_{\rm fin},\kappa_{\rm emp},c$. Choose the ordering fraction
$\eta$ sufficiently small in terms of $\vartheta$, and then $c_1$ so small that the box
$Q_{4c_1P}\times[-4c_1P,4c_1P]$ lies in the common graph window,
and every factor $Cc_1$ below is at most
$\min\{1/1000,\vartheta_m/100\}$. The intercepts used below must
also lie in the small-intercept range of
Lemma~\ref{lem:r5-log-ordering}. Choose $A_0$ large so
$cc_1A_0\geq1$ and, in the regime $E\geq J_\lambda$, every
competitor of objective at most $2d$ lies in that window with
axis intercept below $\kappa_{\rm fin}L/2$; here
$d=O(\log(1/B))$. Finally choose
$0<\lambda\leq b$ with $10\lambda A_0<1/2$.
These choices precede $\delta$.

For $B\ll1$ set $\varepsilon=B^{1/4}$. There is no bad zero in
$B_{10P}$, since one would imply
$J_\lambda(0)\geq e^{-10\lambda P}>B$.
All zeros there are regular with $D<\delta\leq1/4$, and
$a^2\leq J_\lambda(0)e^{10\lambda P}\leq B^{1/2}$.
The ordering lemma applies since $\varepsilon P^2\to0$.
It gives the base graph on $Q_{4L}$ with
\begin{equation}\label{eq:r5-global-base-flatness}
 h(0)=\nabla h(0)=0,\qquad
 \|h\|_\infty+\|\nabla h\|_\infty+\|D^2h\|_\infty
 \leq C(1+P)^2\varepsilon=o(1),
\end{equation}
and finitely many ordered graphs in the smaller box. Their coarse
slopes are $O(\eta)+o(1)$; at a small positive intercept $t$ the
stronger bound is
\begin{equation}\label{eq:r5-global-intercept-slope}
 |\nabla k(0)|\leq Ct/P+CP\varepsilon,\qquad
 \|D^2k\|\leq C\varepsilon.
\end{equation}

On the upper side suppose there is an intercept
$0<t_*\leq\kappa_{\rm fin}L/2$, and choose the smallest one.
Its distance $\ell\asymp t_*$ from zero is attained at an interior
foot, and its upward chord normal obeys
$|\mathbf N'|\leq Ct_*/P+CP\varepsilon$.
Thus the chord lies in the fixed plateau and
\begin{equation}\label{eq:r5-global-ell-reach}
 d(0)\leq G_1\ell,\qquad e^{-\sqrt2\ell}\leq E(0),\qquad
 \ell\geq(\log(1/B))/\sqrt2\longrightarrow\infty.
\end{equation}
Also $|\mathbf N'|L\leq Cc_1t_*+CP^2\varepsilon\leq
\vartheta_m\ell$, and Taylor expansion at the foot gives
$|\mathbf N\cdot(y,k(y))-\ell|\leq C\varepsilon P^2=o(1)$
on $Q_{4L}$. No intervening graph can occur, by minimality of
$t_*$; connectivity gives the positive chamber. This is the finite
alternative of the sharp-intercept theorem.

Otherwise the region up to $2\kappa_{\rm emp}L$ is zero-free.
Indeed a zero there would extend to an ordered graph with positive
intercept $t\leq2\kappa_{\rm emp}L+C\eta L+o(1)$.
This is within the ordering lemma's small-intercept range, so
\eqref{eq:r5-global-intercept-slope} improves it to
$t\leq2\kappa_{\rm emp}L+Cc_1t+CP^2\varepsilon
\leq3\kappa_{\rm emp}L<\kappa_{\rm fin}L/2$, a contradiction.
Thus the two alternatives cover all possibilities and give an empty
positive chamber in this case. Reflection gives the corresponding
negative chamber below the base.

For $w=(y,h(y))$, $y\in Q_{2L}$, every non-base zero in the common
cylinder is at distance at least $c\log(1/B)$, by the finite affine
traps or the empty buffers; zeros outside it are at distance at least
$cL$. Base-graph chords satisfy
$|\nu(w)\cdot(v-w)|\leq C\varepsilon|v-w|^2$ and hence have
objective at least $c/\varepsilon$. Therefore
\begin{equation}\label{eq:s7-comparison-window-reach}
 d(w)\geq c\min\{\varepsilon^{-1},\log(1/B)\}
 \geq c\log(1/B).
\end{equation}
Proposition~\ref{prop:r4-actual-contact-regularity} gives
$D(w)\leq C/\log(1/B)\to0$, which supplies the theorem's fixed
low-defect threshold independently of the original $\delta$.
Since $|(y,h(y))|\leq\sqrt2|y|$ and $\lambda\leq b$,
$e^{-2b|y|}a(y,h(y))^2\leq J_\lambda(0)$ on all of $Q_{4L}$.
Also $e^{-cL}\leq B$ by the choice of $A_0$.

If $E\geq J_\lambda$, then $E\leq B\leq2E$ and
$d=O(\log(1/B))$. All competitors of objective at most $2d$ are
in the common window, but base-sheet chords and diagonal limits have
objective at least $c/\varepsilon\gg P$. Compactness therefore
gives an off-diagonal minimizer on one of its regular graphs.
That graph's shortest chord has an interior foot in the plateau;
its objective is no larger, so $d=G_1\ell$. An intervening graph
would meet this segment earlier and have a smaller plateau objective,
contradicting minimality. It is therefore a selected adjacent side.
The sign change at each regular crossing makes its gradient-oriented
normal oppose the base normal. Choosing the fixed ordering angle
sufficiently small in terms of $\vartheta$ gives the stated angle bound.
\end{proof}

\paragraph*{The positive massive kernel}

\begin{lemma}
\label{lem:r5-massive-kernel}
Under the interaction-dominated conclusion of
Lemma~\ref{lem:s2-envelope-localization}, namely $E\geq J_\lambda$ and
$d=G_1\ell$ on its selected adjacent side,
\begin{equation}\label{eq:r5-disk-lower}
 D(0)\geq c e^{-\sqrt2\ell}.
\end{equation}
The constant is independent of $\ell$ and $\delta$.
\end{lemma}

\begin{proof}
Take $n_0=\nu(0)=e_n$. Lemma~\ref{lem:positive-directional-source}
gives $w=q(1-\boldsymbol p\cdot n_0)\geq0$ and
$w=\mathcal K*F_{n_0}\geq\mathcal K*(qw)$.
We apply the lower bound in
Lemma~\ref{lem:r11-kernel-horizontal-mass}. The power is
$r^{-m/2}$, corresponding to a tangent disk of dimension $m$.

Let $y$ be the adjacent minimizing foot.  The opposing orientation
and $D\leq1/4$ imply $qw\geq c>0$ at the zeros of the adjacent
graph in a tangent disk of radius $c_0\sqrt\ell$ centered at $y$.
Their normals change by at most
$C B^{1/4}\sqrt\ell=o(1)$.  Universal local derivative bounds then
give $qw\geq c/2$ in a fixed thin neighbourhood of this disk.  This
neighbourhood can be parametrized as a graph over $T_y\Sigma$ with
a fixed interval in the $\mathbf N=y/\ell$ direction; it is injective and
has $n$-dimensional volume at least $c_m\ell^{m/2}$ because the
tangent disk has $m$ dimensions. Thus integration over this neighborhood cancels the polynomial factor in
the massive kernel in every fixed ambient dimension.

Every point $x$ in it satisfies $|x|\leq\ell+C$.  Indeed its tangent
displacement is $\tau\perp y$, with $|\tau|\leq c_0\sqrt\ell$;
thus $|y+\tau|\leq\ell+C$.  Its graph error is bounded by
$C B^{1/4}|\tau|^2=o(1)$, and the normal thickness is fixed.
For small $B$, $\ell\to\infty$ and $\ell\asymp\log(1/B)$.
The shared convolution and kernel lower bound therefore give
\[
 w(0)\geq c\ell^{m/2}\ell^{-m/2}e^{-\sqrt2\ell}
       =c e^{-\sqrt2\ell}.
\]
At the base zero, $q=1$ and $\boldsymbol p=|\boldsymbol p|n_0$, so
$D(0)=(1-|\boldsymbol p|)(1+|\boldsymbol p|)\geq w(0)$.  This proves
\eqref{eq:r5-disk-lower}.
\end{proof}

\begin{theorem}
\label{thm:s2-pointwise-master}\label{thm:r5-actual-two-sided}
\label{thm:r5-sharp-lower}
There are $\lambda>0$ and $C<\infty$, depending only on dimension and
the fixed penalty, such that every bounded entire Allen--Cahn solution
and every $0<\delta\leq1/4$ satisfy
\begin{equation}\label{eq:r5-actual-two-sided}
 D(z)\leq C\{J_\lambda(z)+E(z)\},\qquad
 E(z)\leq C\{J_\lambda(z)+D(z)\}
 \quad(z\in\Sigma\setminus X_\delta).
\end{equation}
\end{theorem}
\begin{proof}
First suppose $0<B=J_\lambda(0)+E(0)$ is sufficiently small.
Apply Lemma~\ref{lem:s2-envelope-localization} with the fixed angular
tolerance $\vartheta=1/100$. It provides both actual chambers,
the low-defect base and the full-window envelope bound. Applying the
sharp-intercept theorem gives
\[
 D(0)\leq C\mathcal J_b+
       C(e^{-\sqrt2\ell_+}+e^{-\sqrt2\ell_-})+Ce^{-cL}
 \leq C\{J_\lambda(0)+E(0)\}.
\]
The auxiliary half-space solutions use the same extension on the full
window. The cancellation of their slopes and the transmission estimate
with the exact rate were proved above. For the reverse comparison,
$E\leq J_\lambda$ is immediate. Otherwise the interaction-dominated
conclusion of the localization lemma and the positive massive-kernel
lemma give $D(0)\geq cE(0)$.

For $B\geq b_0>0$, the first inequality follows from $D\leq1$.
For the second, if $J_\lambda\geq b_0/2$ use $E\leq1$.
Otherwise $E\geq b_0/2$, so
$d\leq M=G_1\log(2/b_0)/\sqrt2$.
There is $c_M>0$, depending only on $M$, the dimension and the fixed
penalty, such that $d\leq M$ forces $D\geq c_M$. Indeed, if $D(z_j)\to0$, translate and rotate at $z_j$.
The relative defect estimates and elliptic compactness give
$U_j\to x_n$, equivalently $u_j\to g(x_n)$, in local $C^2$.
In a ball larger than $2M/G_1$, all zeros are consequently a single
horizontal graph whose Hessian tends to zero.
Every chord of objective at most $2M$ lies in that graph and has
length at most $2M/G_1$. Taylor's formula and $\Gamma(s)\geq1/s$
make these objectives tend to infinity, a contradiction.
Thus $E\leq c_M^{-1}D$ also in this regime.

If $B=0$, then $d(0)=\infty$ and
Proposition~\ref{prop:r4-actual-contact-regularity} gives $D(0)=0$.
This completes both comparisons.
\end{proof}

\begin{corollary}
\label{cor:s2-pointwise-moments}
For every fixed $r\geq1$ there is $C_r=C(m,\Gamma,r)$ such that
the same actual quantities satisfy
\begin{equation}\label{eq:s2-pointwise-moments}
 D^r+E^r\leq C_r\{J_\lambda^r+DE^{r-1}\}.
\end{equation}
At $r=1$ take $E^0=1$.
\end{corollary}
\begin{proof}
Take the constant $C\geq1$ of the theorem.
If $D>2CE$, the upper comparison gives $D\leq2CJ_\lambda$;
otherwise $D^r\leq(2C)^{r-1}DE^{r-1}$.
If $E>2CD$, the reverse comparison gives $E\leq2CJ_\lambda$;
otherwise $E^r\leq2CDE^{r-1}$.
Add the two alternatives.
\end{proof}

\vskip0.32in

\section{\bf Interaction between adjacent interfaces}
\label{sec:s12-interaction}
We now choose common local coordinates for adjacent interfaces and
estimate their exponential interaction. The main result is
Theorem~\ref{thm:r13-local-master}, which proves the local estimate
stated in Theorem C. We first establish the required projections and
normal bands, then estimate the incoming normal modes.

\subsection{Local graphs and normal bands}
\label{sec:r11-interpolation-geometry}

The relative defect estimates and the fixed angular penalty describe
the local zero set of a bounded entire solution in $\mathbb R^{m+1}$
whose defect is small on a polynomial window.

\begin{lemma}
\label{lem:s3-buffered-projection}
Let $f$ be a graph on $B_r^m$ with $|Df|\leq\theta\leq1/10$,
$|D^2f|\leq\kappa$. Suppose $x'\in B_{r/2}$,
$v=|x_n-f(x')|\leq L$, $4\theta L<r$, and
$\kappa(L+2\theta r)\leq1/4$. Its closest point $p(x)$ is unique,
interior and smooth, and its distance $s_f$ satisfies
\begin{equation}\label{eq:s3-projection-calculus}
 |p'(x)-x'|\leq2\theta v,\qquad
 s_f\leq v\leq\sqrt{1+\theta^2}\,s_f.
\end{equation}
If $0<l=f(0)\leq L$, $x=(0,s)$, $0\leq s<l$, and
$\ell=s_f(0)$, then
\begin{equation}\label{eq:s3-normal-ray-calculus}
 s_f'(s)=-1+O(\theta^2),\qquad
 s+s_f(s)=\ell+O(\theta^2l),\qquad
 |p(0,s)|=\ell+O(\theta^2l).
\end{equation}
For a graph through zero tangent to the horizontal plane, zero is
the foot of its normal segment $0\leq s\leq L$.
\end{lemma}
\begin{proof}
Half squared distance has Hessian
$I+Df\otimes Df-(x_n-f)D^2f\geq3I/4$ and gradient at $x'$
of norm at most $\theta L$. Its inward derivative is negative on
$\partial B_r$, so its unique minimum is interior. The implicit
function theorem and the stationary equation
$p'-x'=(x_n-f(p'))Df(p')$ give smoothness and
$|p'-x'|\leq\theta(v+\theta|p'-x'|)$. Cauchy--Schwarz then gives
$v\leq|x_n-f(p')|+\theta|p'-x'|\leq\sqrt{1+\theta^2}s_f$.
Below the upper graph,
$s_f'=-1/\sqrt{1+|Df(p')|^2}$; integrate and use
$|p'|\leq2\theta l$, $f(p')=l+O(\theta^2l)$ to obtain
\eqref{eq:s3-normal-ray-calculus}. The final assertion follows
from stationarity at zero and uniqueness.
\end{proof}

\begin{lemma}
\label{lem:r11-clean-window}\label{lem:r11-clean-window-geometry}
Fix $\omega>0$, $C_0,k<\infty$ and a fractional interior margin.
The constants and the threshold for $N$ below depend only on
$m,\Gamma,\omega,C_0,k$ and this margin. Suppose
\begin{equation}\label{eq:r11-clean-window-hypothesis}
 z_0\in\Sigma,\qquad
 D\leq N^{-\omega}\quad\hbox{on }\Sigma\cap B_{N^\omega}(z_0),
 \qquad N\longrightarrow\infty.
\end{equation}
At each base point $z$ in the interior region determined by the prescribed
margin, the zero set in the cylinder of lateral radius $(1+\log N)^k$
and height $C_0\log N$ consists of a finite ordered family of graphs
over $T_z\Sigma$. They extend over disks of radius
$cN^{\omega/8}$ with a comparable buffer and curvature
$\kappa_N=CN^{-\omega/2}$. On the logarithmic cylinders their slopes
and differences of unoriented normals are at most
\begin{equation}\label{eq:r11-clean-angle}
 \theta_N=C(N^{-\omega/8}\log N+N^{-3\omega/8}).
\end{equation}
Distinct graphs have Euclidean separation at least $c_0\log N$.
For each graph, the nearest-point projection is unique and smooth
on the smaller cylinder. On fixed logarithmic-power disks, the errors in replacing Euclidean
distances by vertical distances and in the affine approximations are
\begin{equation}\label{eq:r11-additive-error}
 O\bigl(N^{-c}(1+\log N)^{C_k}\bigr)=o(1),\qquad c=c(\omega)>0.
\end{equation}
Normal metrics and Jacobians at logarithmic depths are $1+o(1)$.
The families and projections agree on smaller overlaps. If another graph lies within $C_0\log N$ of $z$, let $\ell_{\min}(z)$
be the least Euclidean distance from $z$ to an adjacent graph. Then
\begin{equation}\label{eq:r11-actual-gap-realization}
 d(z)=G_1\ell_{\min}(z).
\end{equation}
This identity uses actual, uncapped reach.
\end{lemma}
\begin{proof}
Relative jets give $|A|\leq C\sqrt D\leq\kappa_N$ and
$|\nabla u|\geq\sqrt{(1-N^{-\omega})/2}$ at zeros. Interior
estimates therefore provide uniform local graphs and fixed-order
chart derivatives. Put $r=N^{\omega/8}$. Lemma~\ref{lem:r5-log-ordering},
with $P=cr$ and $\varepsilon=CN^{-\omega/2}$, continues these
graphs to buffered tangent disks because
$\varepsilon P^2=O(N^{-\omega/4})$. Their normal variation is
$O(\kappa_Nr)$ and affine error is $O(\kappa_Nr^2)$.

\vskip0.12in

Two such disks with bases a distance $b=o(r)$ apart have
unoriented normal angle at most
\begin{equation}\label{eq:r11-two-disk-angle}
 C(b/r+\kappa_Nr).
\end{equation}
Indeed, a larger tilt makes the second graph cross the first along
the diameter in the direction of maximal tilt. The base displacement
contributes $b$ and the affine error contributes $C\kappa_Nr^2$. Take the diameter of length $cr$ with
$c>0$ sufficiently small; its projected points lie inside the first
disk because the base displacement is $o(r)$. At the resulting regular zero their
tangents agree, although each has varied only $C\kappa_Nr$,
a contradiction. This also applies to pieces of one connected component.
For bases $O(\log N)$ apart it gives \eqref{eq:r11-clean-angle};
continuation and local uniqueness make the common graphs identical
or disjoint, hence ordered.

\vskip0.12in
We next check that the family includes every zero in the required
logarithmic-power cylinder.
A zero there is $O((1+\log N)^{\max\{k,1\}})$ from $z$.
First apply \eqref{eq:r11-two-disk-angle} at this distance and
continue its graph to the normal axis through $z$: the height changes
by $O(N^{-c}(1+\log N)^{C_k})$. Its intercept is therefore within
$(C_0+1)\log N$, where the refined angle estimate applies. Enlarge
the fixed depth constant to include it. Infinitely many distinct graphs
in a closed smaller cylinder would accumulate at a regular zero,
contradicting its uniform implicit chart.

\vskip0.12in
On lateral radius $L=(1+\log N)^k$, the common-plane error is at
most $C(\kappa_NL^2+\theta_NL)$, giving
\eqref{eq:r11-additive-error}. Apply the preceding projection lemma
on a smaller polynomial disk: its slope is $O(\theta_N)$,
its vertical competitor has length $O(\log N)$, and both required
smallness conditions tend to zero. The horizontal foot displacement
is $O(\theta_N\log N)$ and the distance error is
$O(\theta_N^2\log N)$. A point outside that disk cannot be closer than the vertical competitor.
Every closer zero therefore belongs to the graph family already constructed.

\vskip0.12in
For separation, the relative Harnack estimate $|\nabla\log D|\leq C_n$
gives $D\leq CN^{-\omega/2}$ in $B_{4r_0}$ about any participating
zero, with $r_0=c_0\log N$ sufficiently small. Thus
$|D^2U|\leq CN^{-\omega/4}$ and $\nabla U$ stays close to its
initial unit normal. Strict monotonicity in that direction makes all
zeros in a smaller cylinder one graph. Two graphs closer than a small
multiple of $r_0$ would enter that cylinder and coincide by continuation.
The assertion is immediate in the planar case. Finally,
$I+sA=I+o(1)$ gives the metric and Jacobian estimates at
$|s|=O(\log N)$.

\vskip0.12in

A logarithmically close graph has an interior shortest foot whose
chord lies in the plateau of $\Gamma$, supplying objective
$G_1\ell=O(\log N)$. Every better competitor also has logarithmic
length since $\Gamma\geq1$, so is in the complete family. Chords
within the base graph have normal displacement $O(\kappa_Nl^2)$
and objective at least $c\kappa_N^{-1}\gg\log N$, including
diagonal limits. Each other graph's shortest chord attains its lower
bound $G_1$ times length; a chord to a nonadjacent graph crosses an
intervening graph at shorter length. This proves
\eqref{eq:r11-actual-gap-realization}. Uniqueness of the zero pieces and
their perpendicular feet gives agreement on overlaps.
\end{proof}

\vskip0.12in

\begin{lemma}
\label{lem:r11-radial-interpolation-geometry}
\label{lem:r11-moving-band-geometry}
Under the clean-window hypotheses, for each fixed distance tolerance
and all sufficiently large $N$, at any point in the smaller cylinder
at most two local graphs have distances within that tolerance of the minimum.
If there are two, they are adjacent and lie on opposite sides of that point. Fix a base
$z_i$ and a normal side. Its only possible partner is the first
adjacent graph $j$, with shortest gap $\ell_{ij}\leq C_0\log N$.
Along $x=z_i+s\nu_i$ before reaching it, the specified distances obey
\begin{equation}\label{eq:r11-band-distance-calculus}
 s_i=s,\qquad s_i+s_j=\ell_{ij}+o(1),\qquad
 \frac{d}{ds_i}\frac{s_i-s_j}{2}=1+o(1).
\end{equation}
Each band defined by a fixed bound on $|s_i-s_j|$ is traversed at most
once per side; the adjacent graph and $\ell_{ij}$ stay fixed, while
the receiving foot may move. There,
\begin{equation}\label{eq:r11-moving-foot-reach}
 |z_j(x)-z_i|=\ell_{ij}+o(1),\qquad
 d(z_j(x))\leq G_1(\ell_{ij}+o(1)).
\end{equation}
Normal parametrization, volume and base-gradient pullback have
factors $1+o(1)$.
\end{lemma}
\begin{proof}
Ordered normal intercepts are separated by $c_0\log N$ and their
distances differ from true distances by $o(1)$. Two on the same
side cannot compete, and an intervening intercept excludes a
nonadjacent pair. This also excludes three competitors.
The projection lemma gives \eqref{eq:r11-band-distance-calculus}
with error $O(\theta_N^2\log N)=o(1)$ and positive band derivative.
Before crossing the first graph it excludes farther partners;
after crossing it is closer than the base by
$\ell_{ij}+o(1)\to\infty$, excluding a second band.
The same lemma gives the first part of
\eqref{eq:r11-moving-foot-reach}; the chord is in both angular
plateaux, so the fixed base is an actual competitor at the moving
foot. Metric bounds follow from the atlas. The distance is
differentiated with the adjacent graph fixed.
\end{proof}

\vskip0.12in

\begin{corollary}
\label{cor:r11-rough-mask-clean-window}
Let $\rho(x)=\sqrt{R^2+|x-x_c|^2}$ and, for fixed $\alpha,\beta>0$,
$Z_R=\{z\in\Sigma:D(z)\geq\rho(z)^{-\alpha}\}$.
For sufficiently large $R$, every zero with
$\operatorname{dist}(z,Z_R)>\rho(z)^\beta/3$ has the preceding
atlas, projection and band conclusions on fixed logarithmic-depth
and logarithmic-power lateral windows. Errors are uniform for
$\rho\geq R$, independently of the observation center and outer cutoff.
\end{corollary}
\begin{proof}
Put $N=\rho(z)$ and choose $0<\omega<\min\{\alpha,\beta,1\}$.
The ball $B_{2N^\omega}(z)$ avoids $Z_R$ and has
$\rho=N+O(N^\omega)$, hence $D\leq CN^{-\alpha}\leq N^{-\omega}$
at its zeros. Apply the clean-window lemma with an interior margin.
\end{proof}

\vskip0.12in
The reach identities above are identities for $d$. They apply to
$\widehat d$ only on the branch where the physical cap is inactive.
The graph charts may meet the fixed unstable core. Accordingly, subsequent
stability tests must include an ambient cutoff that vanishes near the core,
and the terms from the cap branch must be estimated separately.

\vskip0.12in
\subsection{Profiles near adjacent interfaces}
\label{sec:r11-local-profile-primitives}
We use the variables of Section~\ref{sec:r10-analytic-module}.
The numerical leading coefficients below are independent of dimension;
auxiliary constants and thresholds may depend on the fixed dimension $m\geq2$.

\begin{lemma}
\label{lem:r11-local-jets}
At a zero with $D(z)<1/4$, let
$\nu(z)=\boldsymbol p(z)/|\boldsymbol p(z)|$. For every fixed $P$,
uniformly on $|s|\leq P$,
\begin{equation}\label{eq:r11-fixed-column}
 |\boldsymbol p(z+s\nu)-\nu|+|U(z+s\nu)-s|
 +\left|\frac{q(z+s\nu)}{G(s)}-1\right|\leq C_{m,P}D(z).
\end{equation}
\end{lemma}
\begin{proof}
For $Y(s)=\boldsymbol p(z+s\nu)-\nu$,
$(\nabla^2U)\boldsymbol p=-\nabla D/2$ gives
$Y'=-\nabla D/2-(\nabla^2U)Y$, $|Y(0)|\leq D(z)$.
Relative jets and Gronwall give $Y=O_{m,P}(D(z))$.
Integrate to control $U-s$, then use
$q=G(U)$ and $|(\log G)'|\leq\sqrt2$.
\end{proof}

\subsubsection{Sharp chamber tails}
\label{sec:r11-sharp-tail}
\begin{proposition}
\label{prop:r11-sharp-tail4}
In a fractional interior of a clean window, let $x$ lie between
adjacent actual graphs at distances $s_i,s_j\leq C_0\log N$.
For every $\varepsilon>0$, uniformly at sufficiently large $N$,
\begin{equation}\label{eq:r11-sharp-tail4}
 q(x)\leq(4+\varepsilon)
       (e^{-\sqrt2s_i}+e^{-\sqrt2s_j}).
\end{equation}
\end{proposition}
\begin{proof}
Make the chamber positive and put $w=1-u$. Choose coordinates at
the lower foot so that $x=(0,s_i)$ and $y$ is horizontal. At the two
feet from $x$, take inward affine unit distances $l_i^0,l_j^0$, so
$l_i^0(x)=s_i$, $l_j^0(x)=s_j$. In a lateral box of size
$L=(\log N)^2$ the atlas bounds the tangent-plane errors by
$\delta_N=o(1)$. Increasing this error slightly makes the polytope
\begin{equation}\label{eq:r11-tail-inner-polytope}
 \mathcal P=\{l_i>0,l_j>0,|y_a|<L\},\qquad
 l_i=l_i^0-\delta_N,\quad l_j=l_j^0-\delta_N,
\end{equation}
lie inside the actual chamber. Opposite principal normals make it
bounded, and its heights remain $O(\log N)$. True distances differ
from these affine distances by $o(1)$ throughout its fixed inner
faces, with feet inside the polynomial buffer. Consequently for fixed $M$,
\begin{equation}\label{eq:r11-tail-principal-face}
 w=1-g(M)+o_N(1)\quad\hbox{on }l_i=M\hbox{ or }l_j=M,
\end{equation}
uniformly on the complete faces, including their intersections.

\vskip0.12in
There are $k=2m+2$ unit face distances $l_b$, including
$L\pm y_a$. Choose $k\zeta<1/8$, then $\eta>0$ small in terms of
$\varepsilon$, then a fixed $M$ so large that
\begin{equation}\label{eq:r11-tail-M-choices}
 C/M^2<\zeta/2,\qquad r_M=(2+\eta)e^{-\sqrt2M}<\zeta,
 \qquad r_M-2r_M^2>1-g(M).
\end{equation}
The last strict inequality follows from
$1-g(M)=2e^{-\sqrt2M}+O(e^{-2\sqrt2M})$.
After fixing these parameters, we increase $N$. On $\mathcal P_M=\{l_b>M\ \forall b\}$ put
\[
 \begin{gathered}
 r_i=(2+\eta)e^{-\sqrt2l_i},\quad r_j=(2+\eta)e^{-\sqrt2l_j},\\
 r_{a,\pm}=\zeta e^{\sqrt2(M-l_{a,\pm})},\quad
 R=\sum_b r_b,\quad W=R-2R^2.
 \end{gathered}
\]
Here $R\leq k\zeta<1/8$. Every point of
$\overline{\mathcal P_M}$ has a phase ball of radius arbitrarily
close to $M$, so Lemma~\ref{lem:r4-phase-ball} gives
$w\leq C/M^2<\zeta/2$. The map $r\mapsto r-2r^2$ is increasing
on $[0,1/8]$. Thus a lateral face has
$W\geq\zeta-2\zeta^2\geq\zeta/2>w$, and a principal face has
$W\geq r_M-2r_M^2>w$ by
\eqref{eq:r11-tail-principal-face}--\eqref{eq:r11-tail-M-choices}.
Since $(-\Delta+2)R=0$,
\[
 (-\Delta+2)W=4R^2+4|\nabla R|^2\geq3W^2,
 \qquad (-\Delta+2)w=3w^2-w^3.
\]
With $w,W\leq1/8$, comparison on this bounded polytope gives
$w\leq W\leq R$; its weak Dirichlet form needs no corner regularity.
At the original point this yields
\[
 w(x)\leq(2+\eta)e^{\sqrt2\delta_N}
 (e^{-\sqrt2s_i}+e^{-\sqrt2s_j})
 +2m\zeta e^{\sqrt2M-\sqrt2(\log N)^2}.
\]
The last term is of lower order than either principal exponential.
If instead $\min(s_i,s_j)\leq M+1$, the fixed-column lemma gives
$q(x)=G(\min(s_i,s_j))+o_N(1)$, and
$G(t)\leq4e^{-\sqrt2t}$ supplies the fixed positive slack.
In the remaining case $x\in\mathcal P_M$ for large $N$; use the
displayed bound and $q\leq2w$. With the already fixed small $\eta$,
increasing $N$ proves \eqref{eq:r11-sharp-tail4}. This proves the phase
tail estimate of \cite[Proposition 8.10]{OpenAI2026} using one scalar
quadratic correction for all boundary faces.
\end{proof}

\vskip0.12in

\subsection{Proof of the interaction estimate}
\label{sec:r13-local-master}
\label{sec:r12-additive-interaction}
\label{sec:r12-central-feedback48}
\label{sec:r12-signed-residual-cone}
We now combine the normalized defect limit with the positive directional
source to estimate the normal response. The positive source argument is
related to \cite[Sections 5 and 8]{OpenAI2026}.
Put $\kappa_0=\sqrt2$, $A_*=8d_0=12\sqrt2=48h_*$ and
$f_\pm=\varphi_\pm$ from Theorem~\ref{thm:positive-even-source}.
The rates $\kappa_0$ and $\kappa=\kappa_0/G_1$ are measured with respect
to the Euclidean gap and the penalized reach, respectively.

\begin{lemma}
\label{lem:r13-positive-incoming}
\label{lem:r12-central-incoming}
Suppose
\begin{equation}\label{eq:r13-master-clean-ball}
 D\leq N^{-\omega}\quad\hbox{on }\Sigma\cap B_{N^\omega}(z_0),
 \qquad z\in\Sigma\cap B_{(1-\theta)N^\omega}(z_0),
\end{equation}
with fixed $\omega>0$, $\theta\in(0,1)$. Select any subset of the
two distinct first-adjacent graphs at shortest gaps
$\ell_\alpha\leq C_0\log N$. Set
$\mu_\alpha=e^{-\kappa_0\ell_\alpha}$ on selected sides and zero
otherwise; rotate to $z=0$, $n_0=\nu(z)=e_{m+1}$.
For each fixed $Q,b>0$, there are pairwise disjoint physical patches
$\Omega_\alpha$ outside $C_b=B_b^m\times(-b,b)$ for sufficiently large $N$.
Their positive source contribution from Lemma~\ref{lem:positive-directional-source} obeys
\begin{equation}\label{eq:r12-central-incoming}
 \begin{gathered}
 J_N=\sum_\alpha\mathcal K*(\mathbf1_{\Omega_\alpha}F_{n_0})
 \geq(a_Q-o_N(1))\sum_\alpha\mu_\alpha e^{\alpha\kappa_0s},
 \\
 a_Q=\frac3{\kappa_0}\int_{-Q}^Qe^{\kappa_0t}G(t)^2dt\to8.
 \end{gathered}
\end{equation}
uniformly on $C_b$. Also $q/G(s)=1+o_N(1)$ there. The incoming
error is relative to each side's own exponential, independently
of the gap ratio.
\end{lemma}
\begin{proof}
The feet are $y_\alpha=\alpha\ell_\alpha n_\alpha$, with
$c\log N\leq\ell_\alpha\leq C_0\log N$ and
$|n_\alpha-n_0|\leq\theta_N$. Alternating simple zeros orient the
gradient normal at each adjacent graph as $-n_\alpha$.
Take its tangent disk $|\tau|\leq\ell_\alpha^{2/3}$ and normal
depths $|t|\leq Q$. The projection lemma makes this normal map
injective; its Jacobian is $1+o(1)$ and
\begin{equation}\label{eq:r12-additive-patch}
 X_\alpha(\tau,t)=\alpha\ell_\alpha n_\alpha+\tau-tn_\alpha+o(1).
\end{equation}
Indeed the position error is
$O(N^{-\omega/2}\ell_\alpha^{4/3})=o(1)$ and
$|\tau\cdot n_0|\leq\theta_N\ell_\alpha^{2/3}=o(1)$.
Thus the patches occupy opposite half-spaces and escape fixed boxes.
The fixed-column lemma gives, relatively on $[-Q,Q]$,
\begin{equation}\label{eq:r12-additive-source6}
 q(X_\alpha)=G(t)(1+o(1)),\qquad
 F_{n_0}(X_\alpha)\geq(6-o(1))G(t)^2.
\end{equation}
For $x\in C_b$, Lemma~\ref{lem:r11-kernel-horizontal-mass} gives
the horizontal kernel mass
$(1-o(1))\mu_\alpha e^{\alpha\kappa_0(s+t)}/(2\kappa_0)$.
The full horizontal mass is invariant under a bounded horizontal
translation. The relative tail outside the disk is
$\operatorname{poly}(\ell_\alpha)e^{-c\ell_\alpha^{1/3}}=o(1)$,
and the kernel logarithmic derivative absorbs the additive geometric
error. Integration in $t$ proves \eqref{eq:r12-central-incoming}
separately for each side. The identity
$\int_{\mathbb R}e^{\kappa_0t}G^2=2\sigma$ gives $a_Q\to8$.
Fixed-ball relative jets give $U=s+o(1)$ and hence $q/G=1+o(1)$.
\end{proof}

\vskip0.12in

\begin{theorem}
\label{thm:r13-local-master}
\label{thm:r12-actual-signed-residual-cone}
\label{cor:r12-compact-Jost-response}
\label{lem:r12-central-resolvent}
Under \eqref{eq:r13-master-clean-ball}, select sides as above.
For each fixed compact $K$ and $\delta>0$, there is
$N_0=N_0(m,\Gamma,\omega,\theta,C_0,K,\delta)$ such that, for $N\geq N_0$,
\begin{equation}\label{eq:r13-master-compact-response}
 w(x',s)\geq8\sum_\alpha\mu_\alpha f_\alpha(s)
                   -\delta\sum_\alpha\mu_\alpha\qquad(x\in K).
\end{equation}
For each fixed $P>0$ and $\eta>0$, there is
$N_1=N_1(m,\Gamma,\omega,\theta,C_0,P,\eta)$ such that, for $N\geq N_1$,
\begin{equation}\label{eq:actual-signed-residual-finite-window}
 |H_{n_0}(z)-A_*(\mu_+-\mu_-)|+A_*(\mu_++\mu_-)
 \leq(C_P+\eta)I_P(z),\qquad
 C_P=\frac{h_*}{\tanh(P/\sqrt2)},
\end{equation}
where $I_P=\sigma^{-1}\int_{-P}^Pq(z+s\nu)^3D(z+s\nu)\,ds$.
Thus for every $\varepsilon>0$ one finite $P_\varepsilon$ works
for every fixed $P\geq P_\varepsilon$, with
\begin{equation}\label{eq:r12-actual-signed-residual-cone}
 |H_{n_0}(z)-A_*(\mu_+-\mu_-)|+A_*(\mu_++\mu_-)
 \leq(h_*+\varepsilon)I_P(z).
\end{equation}
\end{theorem}
\begin{proof}
We use the bounded-below maximum principle: if $V\geq v_0>0$ and
$(-\Delta+V)R\geq0$ on all space, with $R$ bounded below, then
$R\geq0$. Otherwise minimize $R(x)+\eta|x-x_0|^2$ with
$R(x_0)<0$. At its minimum,
$0\leq(-\Delta+V)R\leq2(m+1)\eta+v_0R(x_0)<0$ for small $\eta$.

\vskip0.12in

Take any clean sequence $N_j\to\infty$ in the common coordinates.
If $\epsilon_j=D_j(0)=0$, planarity makes all selected weights zero.
Otherwise Lemma~\ref{lem:normalized-defect-limit} gives compact smooth limits
\[
 W_j=w_j/\epsilon_j\to W\geq0,\quad W(0)=\tfrac12,\quad
 L_0W=G^2|V(x')|^2\geq0,\quad L_0=-\Delta+2-G.
\]
The strong minimum principle gives $W>0$. Since $w_j\geq J_{N_j}$,
one fixed $Q_0>0$ in \eqref{eq:r12-central-incoming} gives
\[
 (a_{Q_0}-o(1))\sum_\alpha\mu_{\alpha,j}/\epsilon_j
 \leq w_j(0)/\epsilon_j
 =[1+\sqrt{1-\epsilon_j}]^{-1}\to\tfrac12.
\]
Hence, along a subsequence, $\mu_{\alpha,j}/\epsilon_j\to r_\alpha\geq0$.
First fix $Q,b$ and pass to $j\to\infty$ in the incoming bound;
then exhaust boxes and let $Q\to\infty$. It follows that
$W\geq8\sum r_\alpha j_\alpha$ on all space, where
$j_\alpha=e^{\alpha\kappa_0s}$. Each fixed incoming inequality holds
for all sufficiently large indices, which permits the limits in this order.
Since $L_0f_\pm=0$ and $f_\pm-j_\pm=1\pm g\in[0,2]$, the residual
$W-8\sum r_\alpha f_\alpha$ is a supersolution bounded below by
$-16\sum r_\alpha$. The maximum principle proves
\begin{equation}\label{eq:r13-limit-mode-domination}
 W\geq8(r_+f_++r_-f_-).
\end{equation}
A failure sequence for \eqref{eq:r13-master-compact-response},
divided by $\epsilon_j$, contradicts this if $\sum r_\alpha>0$;
if that sum vanishes it gives $W(x)\leq0$ at a compact limit point.
If no side is selected, the assertion reduces to $w\geq0$.

\vskip0.12in

For the curvature estimate, keep $P,\eta$ fixed. The same normalized
limit has an even source and
\begin{equation}\label{eq:r12-signed-column-limits}
 H_j(0)/\epsilon_j\to W_s(0),\qquad
 I_{P,j}(0)/\epsilon_j\to
 \mathcal I_P=\frac2\sigma\int_{-P}^PG^2W(0,s)\,ds>0.
\end{equation}
Apply \eqref{eq:positive-mode-subtraction} with $F=W$,
$c_\pm=8r_\pm$ and $M_P(W)=\mathcal I_P/2$. It yields
\begin{equation}\label{eq:r12-limit-diamond}
 |W_s(0)-A_*(r_+-r_-)|+A_*(r_++r_-)
 \leq C_P\mathcal I_P.
\end{equation}
The positive error allowance $\eta\mathcal I_P$ contradicts a
normalized failure of \eqref{eq:actual-signed-residual-finite-window}.
This also covers the case in which the limiting weights vanish. Choose $P_\varepsilon$ with
$C_P\leq h_*+\varepsilon/2$ for $P\geq P_\varepsilon$ and take
$\eta=\varepsilon/2$ to obtain the final assertion.
\end{proof}

\vskip0.12in

\begin{corollary}
\label{cor:r13-master-consequences}
\label{thm:r12-additive-mass48}
\label{thm:r12-additive-interaction}
\label{thm:r11-sharp-mass32}
\label{cor:r12-unsigned-two-sided-mean}
For every fixed $P>0$, $\varepsilon>0$, at sufficiently large clean scale,
\begin{align}
 I_P&\geq[48\tanh(P/\sqrt2)-\varepsilon](\mu_++\mu_-),
       \label{eq:r12-additive-mass48}\\
 |H|+24\sqrt2\min\{\mu_+,\mu_-\}&\leq(C_P+\varepsilon)I_P.
       \label{eq:r12-unsigned-two-sided-mean}
\end{align}
For every $\varepsilon>0$, one finite $P_\varepsilon$ gives the
same inequalities with $48-\varepsilon$ and $h_*+\varepsilon$
for every fixed $P\geq P_\varepsilon$. Selecting a single side gives
the weaker coefficient $32$ for sufficiently large fixed $P$.
\end{corollary}
\begin{proof}
Drop the absolute value in \eqref{eq:actual-signed-residual-finite-window}
and choose its slack so that
$A_*/(C_P+\eta)\geq48\tanh(P/\sqrt2)-\varepsilon$.
The triangle inequality gives the second estimate because
$\mu_++\mu_--|\mu_+-\mu_-|=2\min\{\mu_+,\mu_-\}$.
Use \eqref{eq:r12-actual-signed-residual-cone} in the same way
for the large-window assertions.
\end{proof}

\begin{proposition}
\label{prop:periodic-sharpness48}
Without stability, in clean classes with $C_0>\omega/\sqrt2$ the
optimal limiting mass coefficient at fixed $P>0$ is
$48\tanh(P/\sqrt2)$. Thus $48$ is optimal over all finite
window choices, but cannot replace $48-\varepsilon$ at a fixed
finite window. This optimality statement concerns the unrestricted
clean class and its local mass coefficient; the exterior-stable class
and global recovery coefficients fall outside its scope.
\end{proposition}
\begin{proof}
Let $u_A$ be the one-dimensional periodic solution of amplitude
$A\uparrow1$, with $u_A(0)=0$, $u_A'(0)>0$, extended horizontally.
Put $\beta=1-A^2$, $q_A=1-u_A^2$. The conserved ODE energy gives
\[
 (u_A')^2=\tfrac12(q_A^2-\beta^2),\quad
 D_A=\beta^2/q_A^2,\quad
 I_{P,A}=\frac{\beta^2}{\sigma}\int_{-P}^Pq_A(s)\,ds.
\]
The nearest zero on each side is at distance
\[
 \ell_A=2\sqrt2\int_0^A\frac{du}{\sqrt{(1-u^2)^2-\beta^2}}
 =\frac{2\sqrt2}{\sqrt{1+\beta}}J(\lambda),\qquad
 \lambda^2=\frac{2\beta}{1+\beta},
\]
where $u=A\sin\theta$, $x=\cos\theta$ give
\[
 J(\lambda)=\int_0^1
 \frac{dx}{\sqrt{1-x^2}\sqrt{\lambda^2+(1-\lambda^2)x^2}}
 =\log(4/\lambda)+o(1).
\]
To verify the asymptotic formula, subtract $\operatorname{arsinh}(1/\lambda)$.
Replacing the second denominator by $\sqrt{\lambda^2+x^2}$ gives an
error $O(\lambda^2|\log\lambda|)$. Dominated convergence then leaves
$\int_0^1[(1-x^2)^{-1/2}-1]dx/x=\log2$.
Thus $\ell_A=\sqrt2\log(8/\beta)+o(1)$ and
$\mu_A=e^{-\sqrt2\ell_A}\sim\beta^2/64$. Since $u_A\to g$
on fixed intervals,
\[
 \frac{I_{P,A}}{2\mu_A}\longrightarrow
 \frac{32}{\sigma}\int_{-P}^PG(s)\,ds
 =48\tanh(P/\sqrt2).
\]
All zero sheets have defect $\beta^2$ and curvature zero.
For $N=\beta^{-2/\omega}$ they satisfy the clean hypothesis and
$\ell_A\leq C_0\log N$ for any fixed $C_0>\omega/\sqrt2$.
The displayed limit proves every claimed obstruction; the preceding
corollary supplies the matching lower bounds.
\end{proof}

\vskip0.12in

\begin{remark}
\label{rem:r12-additive-network-interface}
Each selected first-adjacent edge contributes one weight, with the same base,
$I_P$ and normal interval. An edge named twice is counted once.
Only buffered sides with uncapped gaps contribute an exponential.
The theorem applies across the physical core; later stability tests
include a cutoff in their support. Its nonnegative density
$(h_*+\varepsilon)I_P-24\sqrt2\min\{\mu_+,\mu_-\}$ can be
transported by nonnegative comparison; projection, Jacobian,
moving-mask and higher-moment errors still require separate estimates.
\end{remark}

\vskip0.12in

\part{Stability and the energy estimate}
\label{part:s12-stability}
In this part we use exterior stability to obtain large-scale regularity
and a weighted curvature and defect estimate. We then construct one
amplitude for the reach and ambient stability inequalities. Comparing
the resulting estimates proves the energy-growth criterion used in
Part~\ref{part:s12-applications}.

\vskip0.12in

\section{\bf Large-scale geometry of exterior-stable interfaces}
\label{sec:s12-geometry}
We first obtain a bounded density on a scale slightly smaller than the
reach. Exterior stability and local sheeting then give an isolated graph
on this scale. After estimating the derivatives of its mean curvature,
we derive weak differential inequalities for the reach. Stability is used
only in balls centered at the base point and disjoint from the physical core. No integrated moment
estimate or natural growth bound is assumed.

\vskip0.12in

\subsection{Density and regularity below the reach scale}
\label{sec:r11-refined-reach}\label{sec:r9-refined-reach}\label{sec:r7-weak-capped-reach}
Throughout, $n=m+1$, $3\leq m\leq8$, and the actual reach and fixed
penalty are those of \eqref{eq:r4-actual-reach}. Write
\[
 M_r(z)=r^{-m}E(u;B_r(z)),\qquad K(r)=1+\sup_xM_r(x).
\]
Modica monotonicity and bounded energy density give $1\leq K(r)\leq1+Cr$
for $r\geq1$. The stress-tensor argument follows \cite[Section~7.2]{OpenAI2026}.
In the recurrence below, we keep the density at the larger radius.

\vskip0.12in

\subsubsection{A density bound below the reach scale}
\label{sec:near-reach-density-window}

\begin{theorem}
\label{thm:actual-near-reach-density-window}
There are $C,K_0,d_*<\infty$, depending only on $m$ and the fixed
penalty $\Gamma$, such that every bounded entire Allen--Cahn
solution and every regular zero $z$ with $d=d(z)\geq d_*$ finite obey
\begin{equation}\label{eq:actual-near-reach-density-window}
 h=d\exp\{-C\sqrt{\log(2+M_d(z))}\},\qquad M_h(z)\leq K_0.
\end{equation}
More precisely, there are fixed $B>1$, $c>0$ and an integer $j_*$
such that
\begin{equation}\label{eq:density-window-gaussian-relaxation}
 M_{dB^{-j}}(z)\leq C\{1+M_d(z)e^{-cj^2}\},
 \qquad j_*\leq j\leq\tfrac13\log_Bd.
\end{equation}
The density conclusion also holds with the smaller radius defined
using $K(d)+1$ in place of $2+M_d(z)$.
In particular $h=d^{1-o(1)}$ and $h\to\infty$, uniformly as
$d\to\infty$.  The constant $C$ can be increased arbitrarily without
invalidating the density conclusion.  If $d(z)=\infty$, the solution
is planar and all its normalized densities have a fixed universal bound.
\end{theorem}
\begin{proof}
Translate $z$ to zero and rotate $\nu(z)$ to $e_n$. By the
chord bound, without any regularity assumption on the other zeros,
\begin{equation}\label{eq:density-window-slab}
 \Sigma\cap B_{3r}\subset\{|x_n|\leq C_\Gamma r^2/d\}.
\end{equation}
Put $q=1-u^2$, $\boldsymbol p=\nabla U$ and
$\mathbb S=e_uI-\nabla u\otimes\nabla u$. Then $\operatorname{div}\mathbb S=0$ and
\[
 \mathbb S_{nn}=\tfrac14q^2(1+|\boldsymbol p|^2-2p_n^2)
 \geq\tfrac14q^2(1-p_n^2),\qquad
 |\mathbb Se_n|^2\leq Cq^2\mathbb S_{nn}.
\]
Testing with $x_ne_n\zeta^2$, for a cutoff from $B_r$ to $B_{2r}$,
and applying Young gives
\begin{equation}\label{eq:density-window-stress}
 \int_{B_r}q^2(1-p_n^2)\leq C\int_{B_{2r}}q^2(x_n/r)^2.
\end{equation}
Outside $|x_n|\leq C_\Gamma r^2/d+\sqrt r$ in $B_{2r}$, the slab
puts zeros in $B_{3r}$ at vertical distance at least $\sqrt r$;
other zeros are at distance at least $r$. Phase decay therefore gives
\begin{equation}\label{eq:density-window-height}
 \int_{B_{2r}}q^2(x_n/r)^2
 \leq C\{(r/d)^2+r^{-1}\}E(B_{2r})+Cr^{m+1}e^{-c\sqrt r}.
\end{equation}
In the integrated monotonicity formula on $[R/B,R]$, the discrepancy
is at most $q^2(1-p_n^2)/4$, and
$|\partial_ru|^2\leq C_Bq^2\{1-p_n^2+(x_n/R)^2\}$.
Thus, for one fixed sufficiently large $B\geq16$, the preceding two
bounds on a fixed larger ball imply
\begin{equation}\label{eq:density-window-delayed-step}
 0\leq M_R-M_{R/B}\leq
 C\{(R/d)^2+R^{-1}\}M_{BR}
 \quad(R_{\rm b}\leq R\leq c_\Gamma d).
\end{equation}
We have absorbed the remainder $CR e^{-c\sqrt R}$ into
$CR^{-1}M_{BR}$. Indeed, the universal Lipschitz bound at a zero gives
a positive lower bound for the potential energy on a fixed ball. By
monotonicity, $M_r\geq m_0>0$ for $r\geq1$.

\vskip0.12in

Set $R_j=dB^{-j}$, $k_j=M_{R_j}$ and
$b_j=C(B^{-2j}+B^j/d)$. Choose fixed $a,R_{\rm f}$ large, and let
$N$ be the largest index with $R_N\geq R_{\rm f}$. Then
\[
 k_j-k_{j+1}\leq b_jk_{j-1}\quad(a\leq j<N),\qquad b_j\leq1/8,
\]
where $BR_j\leq d$ and $R_j\leq c_\Gamma d$ at every used index.
Choose $a$ and $R_{\rm f}$ so that the two geometric series are
sufficiently small. Then
\begin{equation}\label{eq:density-window-endpoint}
 0\leq k_j\leq M_d,\qquad k_N\leq CBR_{\rm f},\qquad
 \sum_{i=a-1}^N b_i\leq1/4.
\end{equation}
Since $R_N\in[R_{\rm f},BR_{\rm f})$, one has
$N>\frac13\log_Bd$ for all sufficiently large $d$.
Telescoping the delayed recurrence and using monotonicity gives
\[
 k_j\leq k_N+b_jk_{j-1}
             +k_j\sum_{i=j+1}^{N-1}b_i
 \quad(a\leq j<N).
\]
Hence $k_j\leq2k_N+2b_jk_{j-1}$. Iterating this first-order inequality,
with $2b_j\leq1/4$, yields
\[
 k_j\leq\tfrac83 k_N+M_dq_j,\qquad q_j=\prod_{i=a}^j2b_i.
\]
For $a\leq i\leq j\leq\frac13\log_Bd$, $b_i\leq2CB^{-2i}$;
increasing the fixed $a$ gives
\begin{equation}\label{eq:density-window-middle-b}
 q_j\leq\prod_{i=a}^jB^{-i}\leq Ce^{-cj^2}.
\end{equation}
This proves \eqref{eq:density-window-gaussian-relaxation}.
Take $J=\lceil A\sqrt{\log(2+M_d)}\rceil$ with fixed $A$ large;
then $J\geq a$, $M_de^{-cJ^2}\leq1$, and, because $M_d\leq Cd$,
$J\leq\frac13\log_Bd<N$ for large $d$. Choosing $C$ so that
$h\leq dB^{-J}$ proves $M_h\leq K_0$ by monotonicity.
The same monotonicity permits increasing $C$ or replacing $2+M_d$
by $K(d)+1$, while $h\geq d e^{-C\sqrt{\log(2+Cd)}}=d^{1-o(1)}$.

\vskip0.12in

If $d=\infty$, the chord bound puts all zeros in $\{x_n=0\}$.
Arbitrarily large zero-free balls tangent at $0$ give $D(0)=0$
by the tangent phase-ball lemma. Relative-defect rigidity makes $U$
affine with unit gradient; direct integration gives the planar density bound.
\end{proof}

\vskip0.12in

\subsubsection{A truncation of the reach near the fixed ball}
\label{sec:physical-core-capped-reach}

Fix $0<\epsilon\leq1/8$, and put
\[
 \rho_o(z)=\sqrt{1+|z|^2},\qquad
 r_\epsilon(z)=\epsilon\rho_o(z),\qquad
 \widehat d(z)=\min\{d(z),r_\epsilon(z)\}.
\]
Here $\rho_o$ is centered at the fixed origin; $\widehat d$
is a function on the entire regular zero set.

\vskip0.12in

\begin{lemma}
\label{lem:physical-cap-admissibility}
Fix $A\geq1$ and $\eta>0$, and take
$0<\epsilon\leq1/(8A)$. There is $R_*=R_*(R_0,\eta)$ such that,
for $R\geq R_*$, every base point $z$ with $\rho_o(z)\geq\eta R$ and
$d(z)<r_\epsilon(z)$ satisfies
\[
 B_{Ad(z)}(z)\cap\overline B_{R_0}=\varnothing.
\]
Its density window therefore lies in a stable ball. The cap branch
is treated by the explicit calculus below.

\vskip0.12in

For an actual active plateau pair based at such a point,
\begin{equation}\label{eq:physical-cap-receiving-order}
 \widehat d(y)\leq d(y)\leq d(z)=\widehat d(z).
\end{equation}
The signed positive-foot area formula applies with separate spatial
weights at the base point $z$ and the receiving point $y$.
\end{lemma}
\begin{proof}
For large $|z|$, $Ad(z)<2A\epsilon|z|\leq|z|/4$; the guard makes
$|z|>2R_0$, proving admissibility. The reverse normal chord gives
$d(y)\leq d(z)$ on an actual plateau pair, and capping only decreases $d(y)$.
\end{proof}

\vskip0.12in

\subsubsection{Sheeting in stable balls}
\label{sec:r11-nearreach-flat-bridge}
\label{sec:guarded-near-reach-exterior}
\label{subsec:graph-defect-bootstrap}

\begin{proposition}
\label{prop:r11-nearreach-flat-bridge}
\label{cor:r11-density-flat-bridge}
\label{prop:actual-guarded-near-reach-isolation}
\label{cor:guarded-capped-near-reach-window}
\label{lem:actual-isolated-graph-defect-bootstrap}
Fix $K_0<\infty$ and $0<\epsilon_0\leq1/10$.
There exist $\theta_*,h_*,c_*,c>0$ and $C\geq1$, depending only on
$m,K_0,\epsilon_0$ and the fixed penalty $\Gamma$, with the following
property. For each integer $k\geq0$, the constant $C_k$ may additionally
depend on $k$. Let $z$ be a regular zero of a bounded
entire solution, with finite actual reach $d=d(z)$, and suppose
\begin{equation}\label{eq:r11-flat-bridge-hyp}
 Q_u\geq0\text{ on }C_c^1(B_{h_0}(z)),\qquad
 M_{h_0}(z)\leq K_0,\qquad h_0\geq h_*,\qquad
 h_0/d\leq\theta_*.
\end{equation}
In tangent coordinates at $z$, put $h=c_*h_0$.
The whole zero set in $\mathcal C_{4h}=B_{4h}^m\times(-4h,4h)$
is one regular graph $x_n=f(y)$, with
\begin{equation}\label{eq:r11-flat-bridge-graph}
 f(0)=0,\quad Df(0)=0,\quad |Df|\leq\epsilon_0,\qquad
 h\|D^2f\|_\infty+h^{3/2}[D^2f]_{C^{1/2}}\leq C.
\end{equation}
On a fixed smaller concentric portion,
\begin{equation}\label{eq:r11-flat-bridge-zero}
 |\nabla u|\geq c>0,
\end{equation}
\begin{equation}\label{eq:actual-graph-defect-bootstrap}
 d\geq ch,\qquad D\leq Ch^{-2},\qquad
 \|H\|_{C^k_{\rm physical}}\leq C_kh^{-2}
 \quad(k\text{ fixed}),
\end{equation}
\begin{equation}\label{eq:actual-graph-curvature-holder-bootstrap}
 |A|+h^{1/2}[A]_{C^{1/2}}\leq C/h.
\end{equation}
\end{proposition}
\begin{proof}
We argue by contradiction. If the graph conclusion did not hold uniformly,
we could choose $h_{0,j}\to\infty$, $h_{0,j}/d_j\to0$ and counterexamples
for the candidate graph constants $C=j$.
After rotating coordinates so that $\nu(z_j)=e_n$, set $v_j(x)=u_j(z_j+h_{0,j}x)$ and
$\varepsilon_j=h_{0,j}^{-1}$. These solutions are stable on $B_1$,
$|v_j|\leq1$, $\varepsilon_j^2\Delta v_j=v_j^3-v_j$, and their full
$\varepsilon_j$-energies are bounded by $K_0$. The chord estimate gives
\begin{equation}\label{eq:r11-flat-bridge-slab}
 \{v_j=0\}\cap B_1\subset
 \{|x_n|\leq C_\Gamma h_{0,j}/d_j\}.
\end{equation}
Hutchinson--Tonegawa compactness and equipartition
\cite[Theorem~1]{HutchinsonTonegawa2000}, with zero Lagrange multiplier,
give a stationary integral $m$-varifold. Phase decay away from the slab
puts its support on the plane. The energy lower bound on fixed-radius balls centered at $z_j$ and
monotonicity make the limit nonzero. Tangential stationarity makes
its integer multiplicity constant on the plane disk, so for some $Q\geq1$,
\begin{equation}\label{eq:r11-flat-bridge-planar-measure}
 \varepsilon_j|Dv_j|^2\,dx\rightharpoonup
 Q\sigma_0\mathcal H^m\lfloor\{x_n=0\},\qquad
 \sigma_0=\int(g')^2=2\sqrt2/3.
\end{equation}
Equipartition gives this normalization because
$\int_{\mathbb R}\tfrac12(g')^2=\sigma_0/2$ and the potential
energy has the same limit. Thus the hypotheses of local stable planar sheeting
\cite[Theorem~1.1]{LiuSheeting2026} hold in every dimension considered.
It gives $Q$ ordered zero graphs on
$B_{1/2}^m\times(-1/3,1/3)$ with uniform interior $C^{2,1/2}$ bounds;
the slab makes their slopes tend to zero. If $Q>1$, another graph meets
the normal line through $z_j$ at an actual competitor of distance at most
$2h_{0,j}/3$, giving $d_j\leq2G_1h_{0,j}/3$, a contradiction.
Hence $Q=1$, and rescaling proves the graph and curvature bounds,
with $c_*=1/32$ and fixed inner margins.

\vskip0.12in

On the inner graph, zero-free tangent phase balls and Hopf give regular
zeros. The $C^2$ bound yields
$|\nu(w)\cdot(v-w)|\leq C|v-w|^2/h$ for nearby points.
Since $\Gamma(s)\geq1/s$, their reach objectives are at least $ch$;
other zeros are at distance at least $ch$ by the margin. Thus $d(w)\geq ch$.
The infinite-reach case is planar. Otherwise
Proposition~\ref{prop:r4-actual-contact-regularity} first gives
$D\leq C/h$, so $|\nabla u|^2=(1-D)/2\geq c$ and all inner zeros
are good for threshold $1/4$ when $h$ is large. Split the comparison
envelope at distance $c_1h$: local curvature is $O(h^{-1})$ and modified
curvature is globally bounded, giving
\[
 J_\lambda\leq Ch^{-2}+Ce^{-c_1\lambda h},\qquad
 E=e^{-\sqrt2d/G_1}\leq e^{-ch}.
\]
Applying the upper comparison in Theorem~\ref{thm:r5-actual-two-sided},
which was proved independently of sheeting, improves this bound to
$D\leq Ch^{-2}$. Relative defect jets then give
$\|D\|_{C^k}\leq C_kh^{-2}$ on fixed balls, and
$H=\partial_\nu D/[2(1-D)]$, together with uniform implicit-graph
bounds, gives bounds for the derivatives of $H$ through each fixed order in unit graph charts. At each step we work on a smaller region with a fixed fractional margin.

\vskip0.12in

For the exterior-stable application, take $K_0$ from
Theorem~\ref{thm:actual-near-reach-density-window} and fix
$\eta>0$, $0<\epsilon\leq1/8$. Choose $C_{\rm den}$ at least as large
as its density-window coefficient and so that
$e^{-C_{\rm den}\sqrt{\log2}}\leq\theta_*$. Set
\[
 h_0=d\exp[-C_{\rm den}\sqrt{\log(K(d)+1)}].
\]
At a base point $z$ with $\rho_o(z)\geq\eta R$ and
$d<\epsilon\rho_o(z)$, Lemma~\ref{lem:physical-cap-admissibility}
places $B_d(z)$ outside the core for $R\geq R_*(R_0,\eta)$.
The density-window theorem gives $M_{h_0}(z)\leq K_0$, while
$K(d)\leq1+C_nd$ gives $h_0=d^{1-o(1)}\to\infty$ as $d\to\infty$.
Thus \eqref{eq:r11-flat-bridge-hyp} holds for $d\geq d_*$ after
increasing $d_*$ in terms of the fixed local data. Since $h=c_*h_0$,
the graph radius also satisfies $h=d^{1-o(1)}$.
\end{proof}

\vskip0.12in

\subsubsection{Mean curvature of an isolated interface}
\label{subsec:actual-parity-jets}

For the quartic potential, the Darboux factorization reduces the normal
parity argument in \cite[Section~7.4]{OpenAI2026} to an equation with a
strictly positive potential. We use it to improve the mean-curvature
derivative estimates.

\vskip0.12in

\begin{theorem}
\label{thm:actual-isolated-parity-jets}
Let $u$ be a bounded entire Allen--Cahn solution whose whole zero
set in $\mathcal C_{4h}$ is one regular graph with small slope.
Assume on fixed inner portions the curvature H\"older bound
\eqref{eq:actual-graph-curvature-holder-bootstrap} and the $D,H$
bounds in \eqref{eq:actual-graph-defect-bootstrap}, with constants
independent of $h$.  In the guarded case,
Proposition~\ref{prop:r11-nearreach-flat-bridge} supplies these local
hypotheses.  For every fixed
$0<\gamma<1/2$, there are $h_*,c,C>0$, depending only on $m,\gamma$,
the fixed geometric bounds and $\Gamma$, such that for $h\geq h_*$,
on a fixed smaller concentric graph portion,
\begin{equation}\label{eq:actual-isolated-parity-jets}
 |\nabla_\Sigma H|+|\nabla_\Sigma^2H|
       \leq Ch^{-3-\gamma},\qquad
 |\nabla_\Sigma A|\leq Ch^{-2}.
\end{equation}
For every fixed $\varepsilon_1>0$, the stronger estimate
$|H|\leq C_{\varepsilon_1}h^{-3+\varepsilon_1}$ holds for
$h\geq h_{\varepsilon_1}$; these two constants may additionally
depend on $\varepsilon_1$.
\end{theorem}

\vskip0.12in

\begin{proof}
Appendix~\ref{app:isolated-parity} uses the Darboux factorization of
Lemma~\ref{lem:s7-normal-darboux} to remove the translational mode.
The resulting equation has a positive potential and controls the profile error.
The trace identity \eqref{eq:s7-darboux-mean-trace} estimates the required derivatives of $H$
from the even part of the error and its normal derivative. The inner portions are chosen with fixed fractional margins.
\end{proof}

\vskip0.12in

\subsection{Differential estimates at contact points}
We next compute the second variation for an off-diagonal contact.
Its special cases give the plateau and unmodified comparisons. For a
diagonal contact we choose the extension of a top eigenvector to optimize
the gradient term.

\vskip0.12in

\subsubsection{The second variation at distinct contact points}
\label{sec:r5-nonplateau}

Work at an Alexandrov point of $d$ with a regular active receiver,
automatic at large reach by Proposition~\ref{prop:r4-actual-contact-regularity}.

\begin{proposition}
\label{prop:r5-transition-exact}
At an Alexandrov point of $d$, let $y$ be a regular active
off-diagonal receiver. Choose $N=\pm\nu(z)$ so that
$N\cdot(y-z)>0$, and put
\[
 e=(y-z)/\ell,\quad s=N\cdot e>0,\quad b=e-sN,\quad
 c_1=\Gamma-s\Gamma',\quad Z=c_1e+\Gamma'N.
\]
The receiving inward normal is $N_y=-Z/|Z|$, and
\begin{equation}\label{eq:r5-transition-gradient}
 \nabla_\Sigma d=(-c_1\operatorname{Id}+\ell\Gamma'A_N)b.
\end{equation}
Set $h_Z=|Z|$, $\alpha_Z=Z\cdot N$, $t_b=|b|$ and
\[
 K_Z=\Gamma/h_Z-s,\qquad \delta_Z=1-\alpha_Z/h_Z.
\]
For $t_b>0$ let $v=b/t_b$.  Then
\begin{align}\label{eq:s7-offdiagonal-uniform}
 \Delta_\Sigma d\leq{}&h_ZH_{N_y}(y)+\alpha_ZH_N(z)
  +\ell\Gamma'b\cdot\nabla_\Sigma H_N-\ell\Gamma'sa^2\notag\\
 &+\frac{c_1K_Z^2}{\ell}-2\Gamma'\delta_Z A_N(v,v)
  +\ell\Gamma''t_b^2\left|A_Nv-\frac{K_Z}{\ell}v\right|^2.
\end{align}
At $t_b=0$ this is interpreted by continuity; the fixed penalty
is constant there, so the last three terms vanish.

\vskip0.12in

For the fixed penalty constructed in Lemma~\ref{lem:r7-small-penalty},
write
\[
 \mathfrak q=-s\Gamma'/\Gamma\in[0,1],\qquad
 r_\Gamma=\sqrt{s^2+\mathfrak q^2t_b^2}.
\]
The angular factors have the exact expressions
\begin{equation}\label{eq:s7-angular-factors}
 K_Z=\frac{s(1-\mathfrak q^2)t_b^2}{r_\Gamma(1+r_\Gamma)},
 \qquad
 \delta_Z=\frac{(1+\mathfrak q)(r_\Gamma+\mathfrak q)t_b^2}
                    {r_\Gamma(1+r_\Gamma)},
\end{equation}
and in particular
\begin{equation}\label{eq:s7-angular-bounds}
 0\leq K_Z\leq1-s,\qquad 1-s\leq\delta_Z\leq2t_b^2.
\end{equation}
The same angular parameter gives the uniform small-curvature bound
\begin{equation}\label{eq:s7-uniform-small-ad}
 ad\leq1\quad\Longrightarrow\quad |\nabla_\Sigma d|\geq\Gamma t_b.
\end{equation}
\end{proposition}
\begin{proof}
For $F(w,N)=|w|\Gamma(N\cdot w/|w|)$, one has
$F_w=Z$, $F_N=\Gamma'w$ and
\[
 F_{ww}=\frac{c_1}{\ell}P_{e^\perp}
       +\frac{\Gamma''}{\ell}(N-se)\otimes(N-se).
\]
Receiving stationarity makes $Z$ normal; the actual zero-free region
fixes $N_y=-Z/|Z|$.  First variation at $z$ gives
\eqref{eq:r5-transition-gradient}.  For orthonormal frames
$v_i$ at $z$ and $v_i'$ at $y$, let $w_i=v_i'-v_i$ and move both
endpoints geodesically.  The mixed and $N$-second derivatives combine
into a square, giving the exact upper-touching trace
\begin{align}\label{eq:r5-transition-exact}
 \Delta_\Sigma d\leq{}&|Z|H_{N_y}(y)
  +H_N(z)Z\cdot N+\ell\Gamma'b\cdot\nabla_\Sigma H_N
  -\ell\Gamma'sa^2\notag\\
 &+\sum_i\left\{\frac{c_1}{\ell}|P_{e^\perp}w_i|^2
       +2\Gamma'w_i\cdot A_Nv_i
       +\ell\Gamma''\left[b\cdot A_Nv_i+
                   \frac{(N-se)\cdot w_i}{\ell}\right]^2\right\}.
\end{align}
The first two trace terms come from the accelerations of the endpoints.
The remaining two follow from
$\Delta_\Sigma N=\nabla_\Sigma H_N-a^2N$.

\vskip0.12in

Take $v_1=v$ and transport the frame by the minimal plane rotation
$R:N\mapsto Z/h_Z$, fixing $\{N,v\}^\perp$.  This is a receiving
frame because its normal can be $-Z/h_Z$.  Only $w_1$ is nonzero:
\[
 w_1=-\delta_Zv-\frac{c_1t_b}{h_Z}N,\quad
 P_{e^\perp}w_1=K_Z(sv-t_bN),\quad
 (N-se)\cdot w_1=-t_bK_Z.
\]
Substitution gives \eqref{eq:s7-offdiagonal-uniform}.
The fixed construction in Lemma~\ref{lem:r7-small-penalty} gives
$0\leq\mathfrak q\leq1$.
Since $h_Z=\Gamma r_\Gamma/s$ and
$\alpha_Z/h_Z=(s^2-\mathfrak q t_b^2)/r_\Gamma$,
rationalization gives \eqref{eq:s7-angular-factors}.
The first bound in \eqref{eq:s7-angular-bounds} follows from
$s\leq r_\Gamma\leq1$; for the second put $u=-\Gamma'/\Gamma$ and use that
$(s-u t_b^2)/(1+u^2t_b^2)^{1/2}$ decreases for $0\leq u\leq1/s$,
with endpoint values $s$ and $2s^2-1$.
Finally, the first variation is
$-\Gamma\{(1+\mathfrak q)I+\mathfrak q dA_N/(\Gamma s)\}b$.
Since $\Gamma s\geq1$, the matrix has smallest eigenvalue at least
$1+\mathfrak q-\mathfrak q ad/(\Gamma s)\geq1$ when $ad\leq1$.
This proves \eqref{eq:s7-uniform-small-ad}.
\end{proof}

\vskip0.12in

\subsubsection{The exact plateau comparison}

\begin{proposition}
\label{prop:r4-plateau-reach-inequality}
At almost every base point $z$ whose reach exceeds the fixed threshold of
Proposition~\ref{prop:r4-actual-contact-regularity} and admits an active plateau contact,
let $N=\pm\nu(z)$ be chosen with $s=N\cdot e\geq1-\delta_2$,
where $e=(y-z)/\ell$, and put $b=e-sN$.  Then
\begin{equation}\label{eq:r4-plateau-gradient}
 \nabla_\Sigma d=-G_1 b,
\end{equation}
and the absolutely continuous part of the surface Laplacian satisfies
\begin{align}\label{eq:r4-plateau-exact}
 \Delta_\Sigma d
 &\leq G_1\{H_{-e}(y)+sH_N(z)\}
       +\frac{1-s}{1+s}\frac{|\nabla_\Sigma d|^2}{d}
 \notag\\
 &\leq G_1\{H_{-e}(y)_++|H(z)|\}
       +\frac{\delta_2}{2-\delta_2}
           \frac{|\nabla_\Sigma d|^2}{d}.
\end{align}
The singular part of $\Delta_\Sigma d$ has the favorable nonpositive
sign.  Thus on any open set where a plateau contact is active almost
everywhere, \eqref{eq:r4-plateau-exact} is also a weak measure inequality.

\end{proposition}

\begin{proof}
In Proposition~\ref{prop:r5-transition-exact}, put
$\Gamma=G_1$ and $\Gamma'=\Gamma''=0$.  Then
$N_y=-e$, $h_Z=G_1$, $\alpha_Z=G_1s$ and $K_Z=1-s$.
The first variation is \eqref{eq:r4-plateau-gradient}, and
\eqref{eq:s7-offdiagonal-uniform} gives
\[
 \Delta d\leq G_1\{H_{-e}(y)+sH_N(z)\}
                 +G_1(1-s)^2/\ell.
\]
Since $d=G_1\ell$ and $|\nabla d|^2=G_1^2(1-s^2)$,
the last term is $(1-s)|\nabla d|^2/((1+s)d)$.
Use $s\geq1-\delta_2$ for the second inequality.
Theorem~\ref{lem:r4-reach-semiconcavity} supplies the nonpositive
singular part and hence the weak-measure conclusion.
\end{proof}

\vskip0.12in

\subsubsection{Positive trace at actual receiving feet}
\label{sec:r5-positive-trace}
Let $\tau>0$ be the decay exponent supplied by
Theorem~\ref{thm:r5-positive-trace}. This exponent is independent of the
exponent $0<\gamma<1/2$ in the isolated-interface estimate. We apply the
tangent-ball estimate first at the receiving point $y$ and then at the base
point $z$ of an actual contact.

\begin{corollary}
\label{cor:r5-incoming-trace}
For every sufficiently large actual off-diagonal active reach pair
$(z,y)$, let $N_y$ be the inward normal of the zero-free region
$\mathcal O_z$ at its receiving foot.  Then
\begin{equation}\label{eq:r5-incoming-trace}
 (H_{N_y}(y))_+\leq C_\Gamma d(z)^{-1-\tau}.
\end{equation}
For plateau contacts $N_y=(z-y)/|z-y|$.  Thus the signed transfer
formula and all its multiplicity-one conclusions from
Theorem~\ref{thm:r4-positive-foot-transfer} remain valid with the
actual positive trace subject to \eqref{eq:r5-incoming-trace}.
\end{corollary}
\begin{proof}
The rolling ball from Proposition~\ref{prop:r4-actual-contact-regularity}
has radius $c_\Gamma d(z)$ and inward normal $N_y$ at $y$.
Apply Theorem~\ref{thm:r5-positive-trace} to this tangent ball, with the
same selection, orientation and Jacobian.
\end{proof}

\begin{corollary}
\label{cor:r5-source-trace}
At every sufficiently large finite actual reach,
\begin{equation}\label{eq:r5-source-trace}
 |H(z)|\leq C_\Gamma d(z)^{-1-\tau}.
\end{equation}
In particular the actual plateau inequality improves to
\[
 \Delta_\Sigma d\leq C_\Gamma d^{-1-\tau}
       +\frac{\delta_2}{2-\delta_2}
          \frac{|\nabla_\Sigma d|^2}{d}
\]
at almost every base point admitting a plateau contact, with the same
favorable singular part as in
Proposition~\ref{prop:r4-plateau-reach-inequality}.
\end{corollary}
\begin{proof}
The chord bound at $z$ gives zero-free tangent phase balls of radius
$c_\Gamma d(z)$ with opposite inward normals. Applying
Theorem~\ref{thm:r5-positive-trace} on both sides bounds $|H|$;
substitute this and \eqref{eq:r5-incoming-trace} in the plateau formula.
\end{proof}

\vskip0.12in

\subsubsection{Differential calculus for the physical cap}

\begin{lemma}
\label{lem:physical-cap-calculus}
Assume $D\not\equiv0$, so the actual reach is finite and locally
semiconcave on $\Sigma_{\rm reg}$.  Then $\widehat d$ is locally
semiconcave, with nonpositive singular part of its surface Laplacian.
On the cap branch $\{d\geq r_\epsilon\}$ one may use the smooth upper
branch $r_\epsilon$, including almost every tie point.  Its exact
calculus is
\begin{equation}\label{eq:physical-cap-calculus}
 \nabla_\Sigma r_\epsilon=\epsilon\frac{z^T}{\rho_o},\qquad
 \Delta_\Sigma r_\epsilon
 =\epsilon\left\{\frac m{\rho_o}
       -\frac{|z^T|^2}{\rho_o^3}
       -H\frac{z\cdot\nu}{\rho_o}\right\}.
\end{equation}
In particular $|\nabla r_\epsilon|\leq\epsilon$ and
$|\Delta r_\epsilon|\leq m\epsilon/\rho_o+\epsilon|H|$.
For sufficiently large cap reach the mean-curvature estimate at $z$ gives
\begin{equation}\label{eq:physical-cap-mean-error}
 |\Delta r_\epsilon|
 \leq\frac{m\epsilon}{\rho_o}
       +C_\Gamma\epsilon^{-\tau}\rho_o^{-1-\tau}
 \leq C_{\Gamma,\epsilon}/\rho_o.
\end{equation}
For fixed $\epsilon$, the first right side is at most
$(m+1)\epsilon/\rho_o$ once $\rho_o$ exceeds a further fixed threshold.

\end{lemma}
\begin{proof}
Subtracting one common local quadratic makes $d$ and $r_\epsilon$
concave, hence also their minimum; its singular Laplacian is nonpositive.
At an Alexandrov tie the smooth cap is upper touching. Its calculus
follows from $\Delta_\Sigma z=-H\nu$ and the Hessian of $\rho_o$.
On the cap branch \eqref{eq:r5-source-trace} and $d\geq\epsilon\rho_o$
give \eqref{eq:physical-cap-mean-error}; the last bound requires only
$C_\Gamma\epsilon^{-1-\tau}\rho_o^{-\tau}\leq1$.
\end{proof}

\vskip0.12in

\subsubsection{The unmodified specialization}

\begin{proposition}
\label{prop:r5-unmodified-exact}
Suppose $D\not\equiv0$.  At an Alexandrov twice-differentiability point
$z$ of the finite reach $d$, suppose an active off-diagonal contact $y$
has $s=N\cdot e\leq1-\delta_1$, where $N=\pm\nu(z)$ is chosen so
$s>0$, $e=(y-z)/\ell$, and $\ell=|y-z|$.  Assume $y$ is regular; this
is automatic for sufficiently large $d$ by
Proposition~\ref{prop:r4-actual-contact-regularity}.  Set
\[
 r=d/2,\qquad b=e-sN,\qquad M=\operatorname{Id}+rA_N,
 \qquad\Lambda=\max_i|\lambda_i(A_N)|.
\]
The receiving inward normal is $N_y=N-2se$.  Then
\begin{equation}\label{eq:r5-unmodified-gradient}
 \nabla_\Sigma r=-Mb/s,
\end{equation}
and the absolutely continuous Laplacian obeys
\begin{align}\label{eq:r5-unmodified-laplacian}
 \Delta_\Sigma d\leq{}&a^2d+
 \left(2-\frac{2}{1+r\Lambda}\right)
              \frac{|\nabla_\Sigma d|^2}{d}
 \notag\\
 &+\frac{H_{N_y}(y)-(1-2s^2)H_N(z)
                   -\ell b\cdot\nabla_\Sigma H_N(z)}{s^2}.
\end{align}
In particular the displayed gradient coefficient is between zero and
one, and is at most $2r\Lambda\leq ad$.  If $ad$ is sufficiently
small, then $|\nabla_\Sigma d|\geq c_{\delta_1}>0$.
\end{proposition}

\begin{proof}
With $\Gamma=1/s$, Proposition~\ref{prop:r5-transition-exact} gives
\[
 h_Z=s^{-2},\quad\alpha_Z=(2s^2-1)s^{-2},\quad
 K_Z=0,\quad\delta_Z=2(1-s^2),\quad N_y=N-2se.
\]
Its minimal rotation restricted to $T_z\Sigma$ is exactly reflection
across $e^\perp$.  The first variation is
\eqref{eq:r5-unmodified-gradient}. Here $r=d/2$ is an algebraic parameter;
the zero-free rolling ball is the smaller one supplied by the
contact-regularity theorem.
The chord bound gives $r\Lambda\leq1$, so
$0\leq M\leq(1+r\Lambda)\operatorname{Id}$.
The nontrace terms of \eqref{eq:s7-offdiagonal-uniform}, apart from
$a^2d$, equal
\[
 \frac4{s^2}\langle A_Nb,b\rangle+
 \frac{4r}{s^2}|A_Nb|^2
 =\frac4{rs^2}\{|Mb|^2-\langle Mb,b\rangle\}
 \leq\left(2-\frac2{1+r\Lambda}\right)
                      \frac{|\nabla d|^2}{d}.
\]
Here $\langle Mb,b\rangle\geq|Mb|^2/(1+r\Lambda)$ and
$\nabla d=-2Mb/s$.  The trace terms are precisely those in
\eqref{eq:r5-unmodified-laplacian}.  Finally,
\eqref{eq:s7-uniform-small-ad} and
$|b|^2\geq1-(1-\delta_1)^2$ prove the small-$ad$ assertion.
\end{proof}

\vskip0.12in

\subsubsection{The diagonal contact estimate}
\label{sec:r5-diagonal}

\begin{proposition}
\label{prop:r5-diagonal-exact}
At an Alexandrov point of the actual reach with an active diagonal contact,
choose the local orientation $N$ so that $S=-A_N$ has a top eigenvalue
$\lambda=2/d>0$.  Write $h=\operatorname{tr}S=-H_N$, and let $v$ be any
unit vector in its top eigenspace.  Then
\begin{equation}\label{eq:r5-diagonal-exact}
 \Delta_\Sigma d\leq a^2d+\frac{|\nabla_\Sigma d|^2}{d}
       -2h-\frac{d^2}{2}\nabla_\Sigma^2h(v,v).
\end{equation}
The statement includes a repeated top eigenvalue.  Its precise meaning is
an upper-touching branch inequality: whenever $d$ has an Alexandrov
second-order expansion, its trace is bounded above as displayed.  The favorable singular measure at switches follows from
Theorem~\ref{lem:r4-reach-semiconcavity}. On the guarded isolated-graph
portions, the mean-curvature terms are controlled by
Theorem~\ref{thm:actual-isolated-parity-jets}.
\end{proposition}
\begin{proof}
Every eigenvalue $\lambda_j$ of $S$ belongs to $[-\lambda,\lambda]$,
by the diagonal chord bound.  Let $V$ be the top eigenspace.  For a principal
geodesic of signed length $L$ in a direction $w\in V$ one has
\[
 \frac{2N\cdot(y(L)-z)}{|y(L)-z|^2}
 =\lambda+\tfrac13(\nabla_wS)(w,w)L+O(L^2).
\]
The reciprocal chord bound gives an upper bound $\lambda$ for both signs
of $L$.  Thus $(\nabla_wS)(w,w)=0$.  Since $\nabla S$ is totally symmetric
by Codazzi, polarization shows that it vanishes on $V^3$.

\vskip0.12in

Choose an orthonormal principal frame $(v_j)_{j=1}^m$ with $v=v_1$
and extend $v_1$ as a local unit vector field whose first
derivatives at the point satisfy
\[
 \langle\nabla_iv_1,v_j\rangle
   =\frac{(\nabla_iS)_{1j}}{\lambda-\lambda_j}
       \quad(\lambda_j<\lambda),
 \qquad
 \langle\nabla_iv_1,v_j\rangle=0\quad(v_j\in V).
\]
For $f=S(v_1,v_1)$, the function $2/f$ is a local upper competitor for
$d$, with equality at the point.  Its derivatives obey
\begin{align*}
 \nabla_if&=(\nabla_iS)_{11},\\
 \Delta f&=(\Delta S)_{11}
       +2\sum_{i,\lambda_j<\lambda}
          \frac{|(\nabla_iS)_{1j}|^2}{\lambda-\lambda_j}.
\end{align*}
At a top eigenvector, the unit-length constraint accounts for the entire
contribution from the second derivatives of the extension.  By Codazzi,
$\nabla_jf=(\nabla_1S)_{1j}$; this vanishes for $v_j\in V$.
Since $\lambda-\lambda_j\leq2\lambda$, the full term $2\sum$ in the
last display is at least $|\nabla f|^2/\lambda$.
The exact nonminimal Simons identity is
\[
 \Delta S=\nabla^2h+hS^2-a^2S.
\]
Consequently
\[
 \Delta f\geq\nabla^2h(v,v)+h\lambda^2-a^2\lambda
                       +|\nabla f|^2/\lambda.
\]
Differentiate $2/f$ and use upper contact to obtain
\eqref{eq:r5-diagonal-exact}, including the exact signed mean term $-2h$.
\end{proof}

\subsection{A weak inequality for the truncated reach}
The mean-curvature derivative estimates make the remainders in the
nonplateau comparisons decay faster than the reciprocal reach. We control
all of them by one lower bound for the curvature and gradient terms.
The physical cap is treated separately as a smooth upper branch.

\subsubsection{The large-reach branches}

Fix $0<\omega<1/4$ and the penalty from Section~\ref{sec:r9-reach-geometry},
choosing its angular accuracies once to satisfy the estimates below.
All $C_\Gamma,c_\Gamma$ depend on these fixed choices, including $\delta_2$.

\begin{proposition}
\label{prop:r7-actual-weak-reach}
There are $d_*,\theta,c_*>0$ and constants $C_\Gamma,c_\Gamma>0$,
depending only on $m,\omega$ and the fixed penalty with its angular
accuracies. For each guard $\eta>0$, there is $R_*=R_*(R_0,\eta)$
such that the following holds for $R\geq R_*$. Consider a base point $z$
with $\rho_o(z)>\eta R$, $d(z)<r_\epsilon(z)$ and $d(z)>d_*$.
At almost every such base point one may select an active contact type.
For a nonplateau off-diagonal contact or a diagonal contact,
\begin{equation}\label{eq:r7-nonplateau-weak}
 \Delta_\Sigma d\leq
    (1+\omega+C_\Gamma d^{-\theta})a^2d
           +B\frac{|\nabla_\Sigma d|^2}{d},\qquad
 B\leq1+C_\Gamma d^{-\theta}.
\end{equation}
If $ad<c_*$, the same estimate may be used with
\begin{equation}\label{eq:r7-small-ad-drift}
 B\leq\omega+C_\Gamma ad+C_\Gamma d^{-\theta},
 \qquad |\nabla_\Sigma d|\geq c_\Gamma>0.
\end{equation}
An active plateau pair obeys
\begin{equation}\label{eq:r7-plateau-weak}
 \Delta_\Sigma d\leq
 G_1\{|H(z)|+H_{N_y}(y)_+\}
            +\omega\frac{|\nabla_\Sigma d|^2}{d}.
\end{equation}
One also has the locally uniform gradient bound
\begin{equation}\label{eq:r7-reach-gradient}
 |\nabla_\Sigma d|\leq
 C_\Gamma\exp\{C_\Gamma\sqrt{\log(K(d)+1)}\}.
\end{equation}
The inequalities concern the absolutely continuous Laplacian at
Alexandrov points.  The singular part has nonpositive sign.
\end{proposition}

\begin{proof}
Fix $\gamma=1/4$ in Theorem~\ref{thm:actual-isolated-parity-jets}.
The guarded graph and parity estimates then give
$h=cd\exp[-C\sqrt{\log(K(d)+1)}]$. Thus $h=d^{1-o(1)}$, with
\begin{equation}\label{eq:r7-source-jets}
 |A|\leq C/h,\quad |\nabla_\Sigma A|\leq C/h^2,\quad
 |H|\leq C/h^2,\quad
 |\nabla_\Sigma H|+|\nabla_\Sigma^2H|\leq Ch^{-3-\gamma}.
\end{equation}
The coordinate Hessian obeys the same bound since its connection error
is $O(h^{-1})\nabla H$. Independently, the receiving rolling ball gives
\begin{equation}\label{eq:r7-receiving-trace}
 H_{N_y}(y)\leq C_\Gamma d^{-1-\tau}.
\end{equation}
Choose $0<\sigma<\min\{\gamma/8,\tau/8,1/8\}$ and then a fixed
$\xi<\min\{\sigma/4,\gamma/[8(3+\gamma)],1/32\}$.
Increase $d_*$ until $h\geq d^{1-\xi}$.

For an unmodified contact, write the last quotient in
\eqref{eq:r5-unmodified-laplacian} as $\mathcal T$.
If $s\geq d^{-\sigma}$, its upper bound is
\[
 \mathcal T\leq C_\Gamma
 (d^{-1-\tau+2\sigma}+d^{2\sigma}h^{-2}
                         +d^{1+\sigma}h^{-3-\gamma}).
\]
If $s<d^{-\sigma}$, then $\ell=ds=o(h)$; both feet lie on the
graph centered at $z$, the receiving normal continues $N$, and the horizontal
displacement is exactly $\ell b$. Taylor expansion before division gives
\[
 H_{N_y}(y)-H_N(z)-\ell b\cdot\nabla H_N(z)
       =O(\ell^2h^{-3-\gamma}),\qquad
 \mathcal T\leq Ch^{-2}+Cd^2h^{-3-\gamma}.
\]
Thus $\mathcal T\leq C_\Gamma d^{-1-\theta}$ for fixed $\theta>0$.
The exact algebraic gradient coefficient is
$B_0=2-2/(1+r\Lambda)\leq\min\{1,ad\}$.
At a diagonal contact the remainder in \eqref{eq:r5-diagonal-exact}
has the same bound, while $ad\geq2$ and the gradient coefficient is one.

\vskip0.12in

In the angular transition, put $t_b=|b|$. Equations
\eqref{eq:s7-offdiagonal-uniform}--\eqref{eq:s7-angular-bounds},
$\ell\asymp d$, $0\leq\mathfrak q\leq1$, and
$\Gamma''t_b^2\leq2\varepsilon_1$ bound the algebraic terms by
\[
 \mathfrak q a^2d+C\varepsilon_1a^2d
          +Ca t_b^2+C(1+\varepsilon_1)t_b^4/d.
\]
Young bounds the mixed term by an arbitrarily small multiple of $a^2d$
plus a fixed multiple of $t_b^4/d$.
For $ad<c_*\leq1$, \eqref{eq:s7-uniform-small-ad} gives
$t_b^4/d\leq C\delta_1|\nabla d|^2/d$; for $ad\geq c_*$ it is
at most $C\delta_1^2a^2d$. Fixing the accuracies small enough leaves
\[
 (1+\omega/2)a^2d+(\omega/2)|\nabla d|^2/d.
\]
The trace terms are bounded by
$C_\Gamma(d^{-1-\tau}+h^{-2}+dh^{-3-\gamma})
\leq C_\Gamma d^{-1-\theta}$, decreasing $\theta$ if necessary.

We have therefore bounded the error on every nonplateau branch by
$C_\Gamma d^{-1-\theta}$.
For $ad<c_*$, the unmodified or transition contact has
$t_b\geq\sqrt{1-(1-\delta_2)^2}>0$, so
\eqref{eq:s7-uniform-small-ad} gives $|\nabla d|\geq c_\Gamma$;
a diagonal contact is impossible. Otherwise $a^2d\geq c_*^2/d$.
Consequently
\[
 a^2d+|\nabla d|^2/d\geq c_\Gamma/d,
\]
and this lower bound absorbs the error. It proves
\eqref{eq:r7-nonplateau-weak} and \eqref{eq:r7-small-ad-drift}.
On the plateau, \eqref{eq:r4-plateau-exact} proves
\eqref{eq:r7-plateau-weak} upon fixing
$\delta_2/(2-\delta_2)\leq\omega$.

\vskip0.12in

For the gradient bound, unmodified contacts with $\ell\geq ch$ obey
$|\nabla d|\leq Cd/h$ by \eqref{eq:r5-unmodified-gradient}.
When $\ell<ch$, normal Taylor expansion gives
$-2sb=\ell A_Nb+O(\ell^2/h^2)$, hence
$|\nabla d|\leq Cd^2/h^2$. Diagonal inverse-curvature differentiation
gives the latter bound from $|\nabla A|\leq Ch^{-2}$.
Transition first variation is $O_\Gamma(1+da)=O_\Gamma(1)$ by the chord
bound $da\leq2\sqrt m$; the plateau gradient is at most $G_1$.
These give \eqref{eq:r7-reach-gradient}. Local semiconcavity, matching
penalty jets at angular endpoints and the nonpositive singular Laplacian
allow all active upper branches, including diagonal limits, without a
positive switch measure.
\end{proof}

\vskip0.12in

\subsubsection{The weak capped-reach inequality}

\begin{theorem}
\label{thm:r7-all-branch-capped-reach}
Keep the parameters and large-reach threshold above, and assume
$D\not\equiv0$.  On the open regular region
\[
 \mathcal U_{R}=\{z\in\Sigma_{\rm reg}:
               \rho_o(z)>\eta R,\quad\widehat d(z)>d_*\},
\]
choose $R$ sufficiently large for physical-core admissibility.
Partition its Alexandrov points into the cap set
$\mathcal C=\{d\geq r_\epsilon\}$ and the actual set
$\{d<r_\epsilon\}$.  On the latter select a plateau contact
where available, calling this measurable subset $\mathcal P$;
call the remaining points $\mathcal N$.  Then, as a weak measure
inequality on $\mathcal U_R$,
\begin{align}\label{eq:r7-all-branch-capped-reach}
 \Delta_\Sigma\widehat d\leq{}&
 \mathbf1_{\mathcal N}\left\{
 (1+\omega+C_\Gamma\widehat d^{-\theta})a^2\widehat d
       +B\frac{|\nabla\widehat d|^2}{\widehat d}\right\}
 \notag\\
 &+\mathbf1_{\mathcal P}\left\{
 G_1(|H(z)|+H_{N_y}(y)_+)
                 +\omega\frac{|\nabla\widehat d|^2}{\widehat d}
                       \right\}
 +\mathbf1_{\mathcal C}\,\mathcal F_{\rm cap},
\end{align}
where $B$ obeys \eqref{eq:r7-nonplateau-weak} and, for small
$a\widehat d$ on $\mathcal N$, \eqref{eq:r7-small-ad-drift}.
The explicit nonnegative cap term may be taken as
\begin{equation}\label{eq:r7-explicit-cap-expense}
 \mathcal F_{\rm cap}
 =\frac{m\epsilon}{\rho_o}+\epsilon|H|
 \leq\frac{m\epsilon}{\rho_o}
       +C_\Gamma\epsilon^{-\tau}\rho_o^{-1-\tau}.
\end{equation}
For fixed $\epsilon$, after a further fixed physical-radius
threshold, $\mathcal F_{\rm cap}\leq(m+1)\epsilon/\rho_o
=(m+1)\epsilon^2/\widehat d$ on $\mathcal C$.
There is no positive singular measure at actual/contact or
actual/cap switches.  Moreover
\begin{equation}\label{eq:r7-capped-gradient}
 |\nabla_\Sigma\widehat d|\leq
 C_\Gamma\exp\{C_\Gamma\sqrt{\log(K(\widehat d)+1)}\}
 \quad\text{almost everywhere on }\mathcal U_R.
\end{equation}
For positive contributions from plateau contacts, the signed multiplicity-one area
formula holds with $\widehat d(y)\leq\widehat d(z)$. The cap branch
contributes only the explicit cap term.
\end{theorem}

\begin{proof}
On $\mathcal N\cup\mathcal P$, apply Proposition~\ref{prop:r7-actual-weak-reach}.
On $\mathcal C$, including Alexandrov ties, use the cap as smooth upper
branch and Lemma~\ref{lem:physical-cap-calculus}; its trace and mean bound
give \eqref{eq:r7-explicit-cap-expense}. The cap gradient is at most
$\epsilon$, proving \eqref{eq:r7-capped-gradient} on this branch as well.

\vskip0.12in

The contact parameters admit a measurable selection on countable regular
patches: on compact patches their active set is compact after adding
diagonal limits, and the objectives and plateau condition are Borel.
Using completed surface measure, select a plateau contact whenever
available and any remaining active contact otherwise. Receiving traces are locally bounded by
\eqref{eq:r7-receiving-trace}; the other coefficients are locally bounded
on $\mathcal U_R$. Since the singular part of the Laplacian of $\widehat d$ is
nonpositive, the almost-everywhere branch estimates imply the stated
weak measure inequality. Finally \eqref{eq:physical-cap-receiving-order}
preserves the actual signed positive-foot formula and receiving order;
the cap creates no contact.
\end{proof}

\subsubsection{Errors for fixed exponential multipliers}

\begin{lemma}
\label{lem:physical-cap-weighted-error}
Fix a good threshold $\delta>0$, $\eta>0$, $\epsilon>0$, $\kappa>0$,
a fixed $p\geq2$, and an integer $N\geq0$. Let
\[
 \begin{gathered}
 \chi_R(z)=\exp[-\sqrt{R^2+|z-x_c|^2}/R],\\
 \mathcal C_R=\{z\in\Sigma_g:d(z)\geq r_\epsilon(z),\quad
                                      \rho_o(z)\geq\eta R\}.
 \end{gathered}
\]
Let $P$ be a fixed nonnegative locally Lipschitz multiplier with
$P(t)+|P'(t)|\leq C_{\rm mult}e^{-\kappa t}$ for almost every sufficiently
large $t$ (using its continuous representative for $P$).
At exceptional levels choose a bounded derivative representative; its
choice is immaterial in the terms multiplied by $|\nabla r_\epsilon|^2$.
Define the explicit cap error term
\[
 \mathfrak B_P=P(r_\epsilon)
        \left\{|\Delta r_\epsilon|
                  +\frac{|\nabla r_\epsilon|^2}{r_\epsilon}
                  +\frac{|\nabla r_\epsilon|}{R}\right\}
                    +|P'(r_\epsilon)||\nabla r_\epsilon|^2.
\]
Uniformly in the translating observation center, for sufficiently
large $R$,
\begin{equation}\label{eq:physical-cap-weighted-error}
 \int_{\mathcal C_R}\chi_R^2(1+\rho_o)^N\mathfrak B_P
 \leq C e^{-cR}K(R)R^m=o\bigl(K(R)R^{m-2p}\bigr).
\end{equation}
Here $C,c$ and the lower threshold for $R$ may depend on $m,\Gamma$,
the displayed fixed parameters and the multiplier $P$, and are independent
of $x_c$. The constants $C,c$ are independent of $R$.
The same conclusion holds for compact tests bounded by
$\chi_R^2$; a fixed guarded spatial cutoff with derivative bounded
by $C_\eta/R$ only changes $C$.  The estimate concerns exactly the displayed fixed-multiplier costs.
\end{lemma}
\begin{proof}
Good-surface area comparison and ball covering give, uniformly in $x_c$,
\begin{equation}\label{eq:physical-cap-weighted-area}
 \int_{\Sigma_g}\chi_R^2\leq CK(R)R^m.
\end{equation}
Indeed covering gives $K(t)\leq C(t/R)K(R)$ for $t\geq R$;
the shell $|z-x_c|\asymp2^jR$ contributes at most
$CK(R)R^m2^{(m+1)j}e^{-c2^j}$, including fixed unit enlargements.
Cap calculus bounds $(1+\rho_o)^N\mathfrak B_P$ by a fixed polynomial
in $1+\rho_o$ times $e^{-\kappa\epsilon\rho_o}$.
Absorb the polynomial with half the exponential and use
$\rho_o\geq\eta R$ to obtain $Ce^{-cR}$; then
\eqref{eq:physical-cap-weighted-area} proves the estimate.
Its normalized ratio is $O(R^{2p}e^{-cR})\to0$.
The guarded-cutoff derivative contributes only the already included
$CP(r_\epsilon)|\nabla r_\epsilon|\chi_R^2/R$ term.
\end{proof}

\vskip0.12in

The preceding estimate is for a fixed multiplier. If an amplitude depends
on the observation scale and can remain constant on an interval, its
exponential bound must be proved uniformly before applying this estimate.

\begin{corollary}
\label{cor:physical-cap-interaction}
Put $\kappa=\sqrt2/G_1$,
$E=e^{-\kappa d}$ and $\widehat E=e^{-\kappa\widehat d}$.
Then $E\leq\widehat E$, with equality on the actual branch.
On $\mathcal C_R$,
\[
 \int_{\mathcal C_R}\chi_R^2
    \{\widehat E^2+D(\widehat E-E)\}
       \leq Ce^{-cR}K(R)R^m.
\]
Thus the replacement of actual interaction by capped interaction adds
the displayed exponentially small error on this guarded region;
the original mixed interaction and bad-net count are estimated separately.
\end{corollary}
\begin{proof}
Monotonicity of the exponential proves the order and equality.
On the cap branch $\widehat E=e^{-\kappa\epsilon\rho_o}$ and
$0\leq D\leq1$.  Apply the weighted area estimate and the same
exponential absorption as in Lemma~\ref{lem:physical-cap-weighted-error}.
\end{proof}

\vskip0.12in

\subsection{Summary of the local estimates}
\label{sec:r13-nine-dimensional-local}

We conclude this section by recording the local estimates used below for
$3\leq m\leq8$.

\begin{proposition}
\label{prop:r13-nine-dimensional-local-package}
For $3\leq m\leq8$, a nonconstant exterior-stable bounded entire
Allen--Cahn solution has the relative defect jets, uniform good charts and
a uniform positive energy lower bound on fixed balls centered at zeros;
the normal mass, lifted quadratic form and
sharp mean estimates; the actual two-sided pointwise comparison;
the guarded weak capped-reach and selected positive-foot formulas;
and the clean adjacent atlas used below.
The penalty can be fixed with $G_1\leq1.001$, the flat radial
positive-foot constant with $J_*=1.01$, and the profile receiver
inflation with $1.05$. The additive mass and remaining-mass theorem
uses one common sufficiently large finite normal window $P$.
\end{proposition}
\begin{proof}
We verify that the preceding estimates give all the stated conclusions
with compatible choices of constants. Relative jets, charts, the
positive energy lower bound at zeros, decay and the density envelope use fixed-dimensional
elliptic bounds with exponent $12>n$. The positive inverse, normal mass
and horizontal estimates give the lifted-form and averaged-mean bounds.
Its normal Jacobian has degree at most eight; higher terms are
$O(H^2+a^3)=o_\delta(D)$ on the fixed cutoffs.
The tangent-ball and tube comparisons preserve the exact exponential
rate: integration over a tangent disk cancels the polynomial factor of
the horizontal kernel. The paired inverse has the same normal spectral
gap, and its Sobolev exponent satisfies $12>9$.
For the guarded graph construction, integral compactness on the limiting
plane and stable sheeting give finitely many sheets. A chord along the normal at the base point
excludes more than one sheet.
The parity theorem and exact contact formulas give the weak capped reach
and signed positive-foot transfer. The relative defect derivative estimates, proved without stability or
energy-growth assumptions, give the polynomial clean atlas, and the additive mass theorem uses one common finite $P$.

\vskip0.12in

It remains to choose the receiving-point constants uniformly over the
stated dimension range.
For $c_3\leq.005$, $\delta_2\leq10^{-4}$ and cutoff/observer inflation
$(1.001)^2$, their worst case is $m=8$:
\[
 \frac{(1.005)^8(1.001)^2}{1-10^{-4}}<1.043<1.05,\qquad
 \frac{(1.001)^8}{1-10^{-4}}<1.009<1.01.
\]
The second bound uses radial flatness $a_1/G_1\leq.001$.
The finite dimension range permits a common lower reach for the
receiver constants. Choose the common normal windows, then the good
threshold, and use these choices in the moment argument. The higher-power Kato condition is
verified when it is used
in Section~\ref{sec:r12-higher-moment-attack}.
\end{proof}

\vskip0.12in

\section{\bf The \texorpdfstring{$p=9/4$}{p=9/4} curvature and defect estimate}
\label{sec:s12-moment}

\vskip0.12in

\subsection{The localization and the local measure}
\label{sec:r9-fourth-order-core}\label{sec:s1-physical-localization}
In this section we prove a weighted moment estimate with exponent
$p=9/4$. This estimate will be used in the proof of the energy growth
bound. Throughout this section, $3\leq m\leq8$, $n=m+1$, and $u$
is an entire solution stable outside $B_{R_0}(0)$. We use the local
estimates in Proposition~\ref{prop:r13-nine-dimensional-local-package}
and do not assume an integrated moment or energy-growth bound.
The normal lifting and partition argument follows \cite{OpenAI2026}.
Here we include the localization needed for stability outside a compact
set. See also \cite{SternbergZumbrun1998} for the reduction of diffuse
stability to interface geometry.
For the small good threshold $\delta$ chosen below, write
\[
 \begin{gathered}
 q=1-u^2,\quad U=\sqrt2\operatorname{arctanh}u,\quad D=1-|\nabla U|^2,
 \\
 \Sigma=\{u=0\},\quad\Sigma_g=\Sigma\cap\{D<\delta\},\quad X=\Sigma\setminus\Sigma_g.
 \end{gathered}
\]
The constant wells require no further argument. For a planar transition,
$D=a=0$, and the cap contributes only the exponential remainder based
at the origin, estimated below. In the remaining case $D>0$.
Fix $1\leq G_1\leq1.001$, $\kappa=\sqrt2/G_1$, the physical cap
$\epsilon$, and a maximal unit-separated net $\mathcal X$ of $X$.
Put $s_X=\operatorname{dist}(\cdot,X)$ and use the single local measure
\[
 d\mu_{\rm loc}=dS\lfloor\Sigma_g+e_u\,dx+
 \sum_{b\in\mathcal X}\delta_b,\qquad
 e_u=\tfrac12|\nabla u|^2+\tfrac14(1-u^2)^2.
\]
Good charts, uniform derivative bounds and separation give uniformly
bounded unit-ball masses, hence polynomial ball growth. All surface
integrals below are over $\Sigma_g$; ambient-volume integrals are
indicated explicitly. Set
\[
 \begin{gathered}
 t=\min(d,\epsilon\rho_o),\quad \rho_o=\sqrt{1+|z|^2},\quad
 E=e^{-\kappa d},\quad\widehat E=e^{-\kappa t},\quad \sigma=4\sqrt2/3,\\
 I=I_P=\sigma^{-1}\int_{-P}^Pq(z+s\nu)^3D(z+s\nu)\,ds,
 \qquad S_0=(a^2+I)/2.
 \end{gathered}
\]
Choose a smooth cutoff $\zeta_T$, flat at both ends, zero on $B_T(0)$,
one outside $B_{2T}(0)$, with
$|D^j\zeta_T|\leq C_jT^{-j}\mathbf1_{A_{T,2T}}$ for $1\leq j\leq4$.
For $2\leq p\leq9/4$, $S\geq16T$, define
\[
 \begin{gathered}
 \chi_S(z)=e^{-\sqrt{S^2+|z-x_c|^2}/S},\quad
 \omega_p=\zeta_T^{2p}\chi_S^2,\quad N_p=\sum_{b\in\mathcal X}\omega_p(b),\\
 h_{c_h,T}=e^{-c_hT-c_h\operatorname{dist}(\cdot,B_{4T}(0))},\qquad
 V_p=\mathcal V_{S,T,c_h}[\mu_{\rm loc}]\\
 :=S^{-2p}\int\chi_S^2d\mu_{\rm loc}
  +T^{-2p}\int_{B_{4T}}\chi_S^2d\mu_{\rm loc}
  +\int h_{c_h,T}\chi_S^2d\mu_{\rm loc}.
 \end{gathered}
\]
Prescribe the eventual profile rate $\lambda_f>0$ before fixing
$c_h>0$. For a fixed small safety factor $c_{\rm halo}>0$, choose
$c_h<c_{\rm halo}\min\{a_{\rm geo},2\epsilon\lambda_f\}$, where
$a_{\rm geo}>0$ is the minimum of the fixed normal-decay rates used
below. The factor $c_{\rm halo}$ is chosen small enough to allow the
finitely many displacement estimates used in the proof. Decreasing it
enlarges only the exponential remainder and leaves the leading coefficients
unchanged. The profiles constructed in Lemma~\ref{lem:r12-nine-quarter-profile}
have the common rate $\lambda_f=5\kappa/8$.
Write $\mathcal A_T=A_{T/2,3T}(0)$.

\begin{theorem}
\label{prop:r12-conditional-natural-p}
\label{thm:r13-natural-nine-quarter-budget}\label{thm:r13-actual-nine-quarter-budget}
Let $3\leq m\leq8$ and let $u:\mathbb R^{m+1}\to[-1,1]$ solve $\Delta u=u^3-u$
and be stable outside $B_{R_0}(0)$. For the good set, bad net, cap
and interaction rate defined above, fix the local choices below, including
the final reach profile and halo rate. There are $C_p,T_0<\infty$
such that, for every $T\geq T_0$,
$S\geq16T$ and observation center $x_c$,
\begin{equation}\label{eq:r12-conditional-natural-p}
 \int\omega_p(a^{2p}+D^p+\widehat E^p)+N_p\leq C_pV_p,
 \qquad p=9/4.
\end{equation}
The constants may depend on $m,u$ and the fixed local, profile and halo
choices. The threshold $T_0$ may additionally depend on $R_0$ and the
fixed exterior guard; these choices are independent of the radial
guard chosen below. Both constants are independent of $T,S,x_c$.
The right side consists of the two scale-weighted integrals of
$\mu_{\rm loc}$ and the exponential halo integral defining $V_p$.
\end{theorem}

\begin{remark}
\label{rem:s5-nine-quarter-choice-order}
Fix the local parameters in the following order: the penalty,
normal windows, the tolerances in the curvature-tube estimates and the
RSS construction,
the $.91$, $.01$ profile tolerances, and a preliminary $t_{\min}$. Next choose the
small averaging-disk radius $\mu_{\rm env}$, the curvature tolerance
$b$, and its finite window $s_0$ for the two absorptions in
Lemma~\ref{lem:r12-higher-moment-coupling}. Then decrease the good
threshold $\delta$ and fix $L_0,t_*$, the positive bad-net coefficient in the RSS estimate and the
small-$at$ threshold $c_*$. This fixes $C_A,c_p>0$ independently
of the reach profile. Prescribe $e_4=v=\nu=c_p/(8C_A)$ and the tolerances for
the compact-interval flux and the large-reach curvature remainders.
Then choose the single shape parameter $\delta_f$ and construct
$t_0,t_e$ and $f$ as in Lemma~\ref{lem:r12-nine-quarter-profile}. Finally increase $T$, take
$S\geq16T$, and exhaust finite supports.

\vskip0.12in

The profile's upper comparison $H_f$, derivative and primitive
constants $C_f$, Young remainders and the lower bound for $T$ may
depend on $f$. The leading RSS coefficients, the bad-net coefficient, and $C_A,c_p$
do not. In particular $t_{\min}$, the good threshold and the averaging
windows are never increased after $C_A$ has been fixed.
\end{remark}
The proof has three steps. We first estimate the curvature in terms
of the defect. We then apply stability to an RSS test and estimate its
reach-dependent derivative term by a suitable profile. Combining these
inequalities proves the theorem in
Section~\ref{sec:r12-higher-moment-closure}. We begin with the
localization estimates used in all three steps.

\vskip0.12in

\subsubsection{Weighted localization estimates}
Every local measure $\nu$ used below is a finite sum of good area,
diffuse energy and bad-net counting measure. Thus
\begin{equation}\label{eq:r9-measure-ledger}
 \int\chi_S^2d\nu\leq C\int\chi_S^2d\mu_{\rm loc},\qquad
 \int_{\mathcal A_T}\chi_S^2d\nu
 \leq C\int_{B_{4T}}\chi_S^2d\mu_{\rm loc}.
\end{equation}
For each fixed list of displacement decay rates choose $c$ below
their reserved fractions and then $0<c_h<c$. Put
$H_T(z)=e^{-c(T+\operatorname{dist}(z,\mathcal A_T))}$.
Both rates may be decreased for a new fixed kernel; this only enlarges
$V_p$. The complete argument uses one finite list of rates, including
the values of $a,b$ selected for the estimates below. For
each such list, all estimates in this subsubsection are used for
$T\geq T_{\rm loc}$, where $T_{\rm loc}<\infty$ may depend on
that list, the reserved rates $c,c_h$ and the fixed cutoffs, uniformly
in $S\geq16T$ and $x_c$. A new fixed kernel may enlarge this
threshold. We shall repeatedly use the fact that the following three
densities have integral at most $CV_p$ against $\chi_S^2d\nu$:
\begin{align}\label{eq:r9-halo}
 &T^{-2p}e^{-b\operatorname{dist}(z,A_{T,2T})},\qquad H_T(z),\notag\\
 &(1+|z|)^Ne^{-b|z|}\mathbf1_{|z|\geq T-C_0},\qquad b>0.
\end{align}
Indeed, inside $B_{4T}$ they are $O(T^{-2p})$. Outside this ball,
half of the fixed exponential decay dominates
$h_{c_h,T}$, and the remaining half absorbs the polynomial. In the first density the
outside distance is at least $2T+\operatorname{dist}(z,B_{4T})$.
In particular the remainder envelope
\begin{equation}\label{eq:physical-cutoff-ideal}
 R_{p,T}=T^{-2p}\mathbf1_{\mathcal A_T}+H_T,\qquad
 \int\chi_S^2R_{p,T}\,d\nu\leq CV_p
\end{equation}
bounds all remainders arising from the physical cutoff below, with the
same local measure for every observation center $x_c$.

\begin{lemma}
\label{lem:physical-transport-calculus}\label{lem:r13-real-power-transport}
After fixing the displacement rates as above, uniformly for
$2\leq p\leq9/4$, every fixed $v>0$, and ambient $x,z$,
\begin{align}\label{eq:r13-full-power-transport}
 \zeta_T(x)^{2p}&\leq(1+v)\zeta_T(z)^{2p}
                +C_v\min(1,(|x-z|/T)^{2p}),\notag\\
 T^{-1}\zeta_T(x)^{2p-1}\mathbf1_{A_{T,2T}}(x)
 &\leq v\zeta_T(z)^{2p}+C_vT^{-2p}(1+|x-z|)^{2p},\\
 &\int(1+|s|)^Ne^{-a|s|}\zeta_T(z+s\nu)^{2p-2}
       |\nabla\zeta_T(z+s\nu)|^2ds\notag\\
 &\quad\leq CT^{-2}\zeta_T(z)^{2p-2}\mathbf1_{\mathcal A_T}(z)
       +CR_{p,T}(z).\notag
\end{align}
Positive remainders in the first two lines are needed only when the
cutoff changes or one endpoint is in its transition. For $j=2,4$,
\begin{equation}\label{eq:physical-young-power}
 T^{-j}\zeta^{2p-j}E^{p-j/2}
 \leq v\zeta^{2p}E^p+C_{p,v}T^{-2p},\qquad 0\leq\zeta,E\leq1,
\end{equation}
where $p=2,j=4$ means $T^{-4}$.

\vskip0.12in
For $d=|x-y|$, the ratio $\chi_S(x)^2/\chi_S(y)^2$ is at most $e^{2d/S}$. The
one-sided and symmetric weight comparisons are
\begin{equation}\label{eq:s8-one-sided-weight-pair}
 \omega_p(x)\leq e^{2d/S}\chi_S(y)^2
 \{(1+v)\zeta_T(y)^{2p}
   +C_v\min(1,(d/T)^{2p})\mathbf1_{B_{2T}}(y)\},
\end{equation}
\begin{equation}\label{eq:physical-weight-pair}
 \omega_p(x)\leq e^{2d/S}\chi_S(y)^2
 \{(1+v)\zeta_T(y)^{2p}
  +C_v\min(1,(d/T)^{2p})\mathbf1_{\{\zeta_T(x)\ne\zeta_T(y)\}}\}.
\end{equation}
Let $\nu_2$ be one of the local measures above and let $\nu_1$ have uniformly
bounded unit-ball masses; ambient volume is permitted for $\nu_1$.
If $k(x,y)\leq C(1+d)^Ne^{-ad}$ and $0\leq F\leq M$, then
\begin{equation}\label{eq:physical-same-weight-kernel}
 \int\omega_p(x)\int k(x,y)F(y)d\nu_2(y)d\nu_1(x)
 \leq C_k\int\omega_pF\,d\nu_2+C_{k,v}MV_p.
\end{equation}
The principal constant $C_k$ depends on the kernel and the unit-ball
mass bound for $\nu_1$, independently of any constant, such as $H_f$,
in an upper comparison for the reach profile.
Finally let $y=y(z)$ be measurable, $v(z)\geq|z-y(z)|$, and
$0\leq A(v)\leq C_f(1+v)^Ne^{-av}$. No inverse-map hypothesis is
needed for the cutoff commutator estimate integrated in $z$
\begin{equation}\label{eq:physical-foot-remainder}
 \int\chi_S(z)^2A(v(z))\min(1,(|z-y(z)|/T)^{2p})
 \mathbf1_{\{\zeta_T(z)\ne\zeta_T(y(z))\}}d\nu(z)\leq C_fV_p.
\end{equation}
For the observer transfer, $v\leq\sqrt S$ has inflation
$e^{2/\sqrt S}$, while the complementary contribution is bounded by
$C_fS^{-2p}\int\chi_S^2d\nu$. The principal foot transfer still
requires the signed area formula with the corresponding multiplicity.
\end{lemma}
\begin{proof}
The first two lines of
\eqref{eq:r13-full-power-transport} and the two weight comparisons follow from
Lipschitz continuity and Young's inequality. In the one-sided formula a positive remainder requires an increasing
cutoff, hence $\zeta_T(y)<1$ and $y\in B_{2T}$. For the normal
line use $|\nabla\zeta_T|\leq C/T$ and
$\zeta_T(z+s\nu)^{2p-2}\leq C[\zeta_T(z)^{2p-2}+(|s|/T)^{2p-2}]$.
Equation~\eqref{eq:physical-young-power} uses the conjugates
$p/(p-j/2)$ and $2p/j$, with its stated endpoint interpreted directly.

\vskip0.12in
We estimate the remainders as follows. If a displacement $d\leq T/4$
changes the cutoff, or has a transition endpoint, its base lies in
$\mathcal A_T$; the error is $T^{-2p}$ times a fixed polynomial
in $d$. If $d>T/4$, the segment meets the transition annulus, so
$d\geq c[T+\operatorname{dist}(z,\mathcal A_T)]$. Reserving half
an exponential for this distance bounds the error by $CR_{p,T}(z)$.
The observer factor consumes at most half a fixed decay rate for
large $T$. This proves the normal estimate and the cutoff
commutator integrated in $z$, using $v\geq d$. On $v>\sqrt S$ the fixed tail
is $O_f(S^{-2p})$, giving the asserted bound for the integral over $z$.

\vskip0.12in

For the spatial kernel use the stronger one-sided comparison and
Tonelli. The polynomial-exponential moments of $\nu_1$ are uniformly
bounded, since its ball masses grow at most like $(1+r)^n$. Thus
its remainder is exactly
$C_{k,v}MT^{-2p}\int_{B_{2T}}\chi_S^2d\nu_2\leq C_{k,v}MV_p$.
For a transfer in the opposite direction, from the base point $z$ to
its receiving foot $y$, the one-sided core would be at $y$. There one
uses the symmetric cutoff commutator estimate integrated in $z$ instead;
no receiving-core indicator is pulled
back without an inverse-map justification.
\end{proof}

\begin{lemma}
\label{lem:s5-rss-growing-collars}
Put
\[
 N_{\rm core}=\sum_{v\in\mathcal X\cap B_{2T}}\chi_S(v)^2,
 \qquad R_{\rm loc}=N_p+T^{-2p}N_{\rm core}\leq N_p+V_p.
\]
For $\nu_{\rm out}=dS$ or $dx$, every $b\geq0$, and $N=2n+8$,
\begin{equation}\label{eq:s5-rss-growing-collar}
 \int_{\{s_X\leq b\}}\omega_p\,d\nu_{\rm out}
 \leq C(1+b)^Ne^{C(1+b)/S}R_{\rm loc}.
\end{equation}
For each fixed $a>0,j\geq0$, there is $T_a<\infty$ such that,
uniformly for $T\geq T_a$, $1\leq L\leq T$ and $M\geq1$,
\begin{align}\label{eq:s5-rss-two-tails}
 &\int_{\{s_X(x)\leq3L,\ \operatorname{dist}(x,\Sigma)>M\}}
                 q(x)^2\omega_p(x)dx\notag\\
 &\quad+\int_{\Sigma_g}\int_{\substack{|s|>M\\s_X(z+s\nu)\leq3L}}
 (1+|s|)^je^{-a|s|}\mathbf1_{|s|a(z)<\alpha}J(z,s)
                 \omega_p(z+s\nu)ds\,dS_z\notag\\
 &\hspace{20mm}\leq C(1+L+M)^{N'}e^{-a'M}R_{\rm loc}.
\end{align}
Here $J=\det(I+sA)>0$ on the curvature tube. For fixed $a,j$, the
constants $C,N',a'>0$ and the threshold $T_a$ depend on $a,j$
and the fixed local geometry, uniformly in $L,M,S,T,x_c$ and the
reach profile on the stated range.
The same bound holds for the tail with collar restriction $s_X(z)\leq3L$ at the base point
\begin{equation}\label{eq:s5-rss-source-collar-tail}
 \begin{aligned}
 &\int_{\{s_X(z)\leq3L\}}\int_{|s|>M}
 (1+|s|)^je^{-a|s|}\mathbf1_{|s|a(z)<\alpha}J(z,s)
       \omega_p(z+s\nu)ds\,dS_z
 \\
 &\quad\leq C(1+L+M)^{N'}e^{-a'M}R_{\rm loc}.
 \end{aligned}
\end{equation}
The same bounds hold after multiplication by masks taking values in
$[0,1]$.
\end{lemma}
\begin{proof}
Cover the collar by $B_{b+2}(v)$, $v\in\mathcal X$, and apply
\eqref{eq:s8-one-sided-weight-pair}. Its ball mass is at most
$C(1+b)^n$, and the remainder is bounded by the net sum inside the physical core.
This proves the stronger exponent $n+2p$ in
\eqref{eq:s5-rss-growing-collar}; it also gives
\[
 T^{-2p}\int_{B_{2T}\cap\{s_X\leq b\}}\chi_S^2d\nu_{\rm out}
 \leq C(1+b)^ne^{C(1+b)/S}R_{\rm loc}.
\]
For this last bound, the net points in the physical core contribute $T^{-2p}N_{\rm core}$;
all other net points have $\zeta_T=1$ and $T^{-2p}\leq1$.
Phase decay proves the first tail. For the normal tail,
$s_X(z)\leq3L+|s|$ and $J\leq(1+\alpha)^m$. Transporting the
weight to $z$ uses the collar bound for its principal part and the
last core bound for its remainder. The resulting integral is
\[
 CR_{\rm loc}\int_M^\infty
 (1+L+s)^{n+2p+j}e^{-as+C(1+L+s)/S}ds.
\]
Since $L\leq T\leq S/16$, it has the stated polynomial-exponential
bound for large $T$. For the version restricted at the base point use $b=3L$
instead of $3L+|s|$ in the same integral over the parametrizing tube.
\end{proof}
In particular, for each fixed $b$,
\begin{equation}\label{eq:r9-collar}
 \int_{\{s_X\leq b\}}\omega_p(dS+dx)\leq C_b(N_p+V_p).
\end{equation}
The corresponding uncutoff cover gives
\[
 \int_{\{s_X\leq b\}}\chi_S^2(dS+dx)
 \leq C(1+b)^ne^{C(1+b)/S}\sum_{v\in\mathcal X}\chi_S(v)^2.
\]

\begin{lemma}
\label{lem:physical-compact-cutoff}
Fix a finite averaging radius $s_0\geq1$. Choose $m_{\rm bad}=0$ on
$s_X\leq3s_0$, $m_{\rm bad}=1$ on $s_X\geq4s_0$ and
$|\nabla m_{\rm bad}|\leq C/s_0$. Let $\xi_B$ be an outer cutoff,
one on $B_B(0)$ and zero off $B_{2B}(0)$, with gradient $O(B^{-1})$.
For $\eta_B=\zeta_Tm_{\rm bad}\xi_B\chi_S^{1/p}$ put
$\eta_B^\sharp(z)=\sup_{y\in\Sigma\cap B_{2s_0}(z)}\eta_B(y)$.
For large $T$, these cutoffs and their enlargements stay on the good
set and outside the unstable core. With $d_c=|z|-R_0$,
\begin{align}\label{eq:physical-compact-cutoff}
 \int\{|\nabla\eta_B|^{2p}+e^{-pd_c/8}\eta_B^{2p}\}
   &\leq C_{s_0,p}(N_p+V_p)+e_B,\notag\\
 \int D^p(\eta_B^\sharp)^{2p}
   &\leq 2\int D^p\omega_p+C_{s_0,p}V_p,
 \qquad e_B\longrightarrow0.
\end{align}
The constant $C_{s_0,p}$ is independent of $B,S,T,x_c$; the error
$e_B$ tends to zero for each fixed $S,T,x_c$.
\end{lemma}
\begin{proof}
Differentiating the four factors gives, respectively, $T^{-2p}$ times
annular area, a fixed bad-collar integral, $B^{-2p}$ times the outer
tail weighted by $\chi_S^2$, and $S^{-2p}$ times the $\chi_S^2$-weighted area. Equations
\eqref{eq:r9-measure-ledger} and \eqref{eq:r9-collar} bound the first,
second and fourth terms, after multiplication by their displayed powers
of $T^{-1}$ or $S^{-1}$; the uniform local measure bound and exponential
$\chi_S$ make the third $e_B\to0$. The core factor has an exponential tail centered at the
origin, which is bounded by \eqref{eq:r9-halo}. A contributing enlargement
point is at distance at least $s_0$ from $X$ and at least $T-2s_0$
from the origin. For $|y-z|\leq2s_0$, the full-power comparison
\eqref{eq:physical-weight-pair} bounds its maximal weight by
$(1+v)e^{4s_0/S}\omega_p(z)$ plus
$C_vs_0^{2p}T^{-2p}\chi_S(z)^2\mathbf1_{\mathcal A_T}(z)$.
Fix $v=1/4$ first and increase $T$ after $s_0$ so that
$(1+v)e^{4s_0/S}\leq2$ uniformly for $S\geq16T$.
This remains valid after taking the supremum; $D\leq1$ and the
core measure integral bound the second term, whose constant may depend
on $s_0$. Thus the principal enlargement coefficient is the fixed
number $2$, independently of the averaging window. All finite-$B$
integrals are legitimate compact exterior tests; bounded $a,D$ and
the observation-weight tail also justify the subsequent exhaustion by dominated convergence.
\end{proof}

\vskip0.12in

\subsubsection{Normal-ray estimates and the partition of unity}
Fix $0<\tau\leq10^{-3}$. First choose finite normal windows $P,Q$ and a
curvature-tube parameter $\alpha$, then decrease $\delta$, so that
$I_P\geq.999a^2$, $I_P\geq c_ID$, and the lifted sheet gain is as
close as prescribed to $\sigma S_0$. The local normal-form calculation
is \eqref{eq:r3-lift-exact}; it is not repeated in the integrated proof.
Normalized compactness also gives any prescribed fixed lower reach
$t_{\min}$ on $\Sigma_g$. These choices are compatible with the
profile-independent pointwise and averaging constants.
For partition cancellation we use the following accurate-ray estimate.
For every $\xi>0$ there is $b_\xi>0$, depending also on the fixed
curvature-tube parameter $\alpha$, such that
\begin{equation}\label{eq:r9-accurate}
 \begin{gathered}
 B_z(s):=D(z)e^{(2\sqrt2+\xi)|s|}\leq b_\xi\\
 \Longrightarrow\quad
 |q(z+s\nu)/G(s)-1|+
       |\nabla_x\log(G(s)/q(x))|\leq C_\xi B_z(s).
 \end{gathered}
\end{equation}
Here $G=1-\tanh^2(s/\sqrt2)$; the ray lies inside
$|s|a<\alpha/2$, and a physical point has at most one
accurate preimage, including signed depths.
This is Lemma~\ref{lem:r4-accurate-ray}, with exactly the displayed
$B_z(s)$ and the same curvature-tube parameter. Its stronger profile
bound includes $e^{-\xi|s|/2}$; its uniqueness statement already
covers both signs of the normal depth, so no global tube assumption
is introduced here.

\vskip0.12in

Fix $L_0$ as below.  On the physical space choose
$\theta=0$ for $s_X\leq2L_0$, $\theta=\pi/2$ for
$s_X\geq3L_0$, $|\nabla\theta|\leq C/L_0$,
and write $\alpha_0=\cos\theta$, $b=\sin\theta$.
On the base surface choose $\mu=0$ for $s_X\leq L_0/2$,
$\mu=1$ for $s_X\geq L_0$, $|\nabla_\Sigma\mu|\leq C/L_0$.
Let $\beta(z,s)$ be an even curvature cutoff equal to one
on $|s|a\leq\alpha/2$ and supported in $|s|a<\alpha$.
Relative Harnack and tangent phase balls bound $d$ on the
$10L_0$ collar: the former gives
$D\geq c\delta e^{-C(10L_0+1)}$ and the latter gives
$D\leq C_\Gamma/d$ for large reach.  The receiving-defect
inequality likewise bounds the reaches at base points whose selected feet
enter that collar.  Fix $t_*$ above both finite bounds.

\vskip0.12in

\subsection{Curvature estimates from stability}
\label{sec:r12-higher-moment-attack}
We first derive a higher-power curvature estimate by combining
the Simons identity with stability. The nonminimal Codazzi estimate in
Lemma~\ref{lem:nonminimal-kato-simons} gives, for every fixed $e>0$,
\begin{equation}\label{eq:r12-optimized-kato}
 |\nabla A|^2\geq\kappa_K|\nabla a|^2-C_{K,m,e}|\nabla H|^2,
 \quad\kappa_K=\frac{m+2}{m(1+e)},\quad
 C_{K,m,e}=\frac3{m+2}\{\kappa_K(1+e^{-1})-1\}.
\end{equation}
Thus any $1<\kappa_K<1+2/m$ is available; $\kappa_K$ is distinct
from the interaction rate $\kappa$. The endpoint is sharp for
trace-free Codazzi jets: take $A=\operatorname{diag}(m-1,-1,\ldots,-1)$
and symmetric $T$ with $T_{111}=(m-1)b$, $T_{1ii}=-b$ for $i>1$,
and their permutations. Then $\operatorname{tr}_{jk}T_{ijk}=0$ and
$|T|^2/|\nabla a|^2=(m+2)/m$.

For $p\geq2$, let $\zeta$ be a compactly supported Lipschitz cutoff
on the regular good set, with support in $\{|z|-R_0>8Q\}$ for the fixed window $Q$ above, and put
$r=p-1$. Write
\[
 \begin{gathered}
 F_p=\int a^{2p}\zeta^2,\quad G_p=\int a^{2p-4}|\nabla a|^2\zeta^2,
 \quad B_p=\int a^{2p-2}|\nabla\zeta|^2,\\
 C_p=\int a^{2p-3}\zeta\nabla a\cdot\nabla\zeta,\quad
 H_p=\int a^{2p-4}|\nabla H|^2\zeta^2,\quad
 L_p=\int |H|a^{2p-1}\zeta^2.
 \end{gathered}
\]
The Simons identity \cite{Simons1968}, in the curvature absorption
method of \cite{SchoenSimonYau1975}, is
\begin{align}\label{eq:r12-power-simons-exact}
 F_p={}&\int a^{2p-4}|\nabla A|^2\zeta^2+(2p-4)G_p+2C_p\notag\\
 &+\int a^{2p-4}\zeta^2A_{ij}H_{ij}
   +\int a^{2p-4}\zeta^2H\operatorname{tr}A^3.
\end{align}
Integration by parts turns its mean Hessian term into
\[
 -H_p-(2p-4)\int a^{2p-5}\zeta^2A_{ij}a_iH_j
       -2\int a^{2p-4}\zeta A_{ij}\zeta_iH_j.
\]
Together with \eqref{eq:r12-optimized-kato} this yields, for fixed $v>0$,
\begin{equation}\label{eq:r12-power-simons-bound}
 F_p\geq(\kappa_K+2p-4-v)G_p+2C_p-C_{m,p,e,v}(H_p+B_p)-L_p.
\end{equation}
The stability inequality on the interface, with any prescribed $\beta<1$ sufficiently
close to one, applied to $a^{p-1}\zeta$ gives
\begin{equation}\label{eq:r12-power-stability}
 \beta F_p\leq r^2G_p+2rC_p+B_p+E_p,\qquad
 E_p=C\int e^{-d_c/8}a^{2p-2}\zeta^2,
 \quad d_c=|z|-R_0.
\end{equation}
These calculations use regularized powers at $a=0$. Comparing the
last two inequalities leads to the condition
\begin{equation}\label{eq:r12-moment-gate}
 \beta(\kappa_K+2p-4)>(p-1)^2.
\end{equation}
Indeed choose $v$ with
$\Delta_p=\beta(\kappa_K+2p-4-v)-r^2>0$ and use
$|C_p|\leq\sqrt{G_pB_p}$:
\[
 \Delta_pG_p\leq2|r-\beta|\sqrt{G_pB_p}
      +C(B_p+H_p+L_p+E_p)
 \leq\tfrac12\Delta_pG_p+C'(B_p+H_p+L_p+E_p).
\]
Stability then also bounds $F_p$, proving
\begin{equation}\label{eq:r12-power-curvature-pre}
 F_p+G_p\leq C(B_p+H_p+L_p+E_p).
\end{equation}
This condition is sufficient and necessary for the preceding absorption of
$G_p$ with a positive coefficient. It is a condition on this argument,
not on all solutions of the equation. For fixed $\beta,\kappa_K$
its upper root is
\begin{equation}\label{eq:r12-fixed-beta-range}
 p<1+\beta+\sqrt{\beta(\kappa_K+\beta-2)}.
\end{equation}
Letting $\beta\uparrow1$ gives $(p-2)^2<\kappa_K-1$;
$\kappa_K=1+1/m$ permits $p<2+1/\sqrt m$, while the optimized range is
\begin{equation}\label{eq:r12-ssy-range}
 2\leq p<2+\sqrt{2/m}.
\end{equation}
The upper endpoint of the range used here is not attained by this strict absorption argument.

\vskip0.12in

Set $\zeta=\eta^p$. Young's inequality bounds $B_p,H_p,L_p,E_p$,
respectively, by a small multiple of $F_p$ plus
\[
 C\int|\nabla\eta|^{2p},\quad
 C\int|\nabla H|^p\eta^{2p},\quad
 C\int|H|^{2p}\eta^{2p},\quad
 C\int e^{-pd_c/8}\eta^{2p}.
\]
For $H_p$ the conjugates are $p/(p-2),p/2$, with $p=2$ direct.
For each prescribed finite $s_0$, Proposition~\ref{prop:r3-small-mean}
and Fubini on the good charts give
\[
 |\nabla H(z)|^p\leq\frac{C_{m,p}}{s_0^{m+p}}
          \int_{\Sigma\cap B_{2s_0}(z)}D^p,\qquad
 \int|\nabla H|^p\eta^{2p}\leq\frac{C_{m,p}}{s_0^p}
          \int D^p(\eta^\sharp)^{2p}.
\]
Also $|H|^{2p}\leq C\delta^pD^p$. Choosing $s_0$ before $\delta$
and absorbing the small $F_p$ terms proves, for every $\varepsilon>0$,
\begin{equation}\label{eq:r12-curvature-only-higher-moment}
 \int a^{2p}\eta^{2p}\leq C_{\varepsilon,p}\int|\nabla\eta|^{2p}
 +\varepsilon\int D^p(\eta^\sharp)^{2p}
 +C_{\varepsilon,p}\int e^{-pd_c/8}\eta^{2p}.
\end{equation}
The defect term on the right is evaluated on the enlarged cutoff,
with the same exterior guard.

\vskip0.12in

\begin{corollary}
\label{cor:r13-nine-quarter-curvature}
For $3\leq m\leq8$, estimate
\eqref{eq:r12-curvature-only-higher-moment} holds at $p=9/4$ with
arbitrarily small prescribed defect tolerance.
\end{corollary}
\begin{proof}
Choose $\beta=.99$, $\kappa_K=1+1/m$. The strict margin in \eqref{eq:r12-moment-gate} is
$\{396-31m\}/(400m)\geq37/800>0$. Proposition~\ref{prop:r13-nine-dimensional-local-package}
supplies the local estimates required for these choices.
\end{proof}

\begin{corollary}
\label{cor:s3-same-weight-curvature}
For $p=9/4$ and every $b>0$, there are a finite averaging window,
a good threshold and $C_b<\infty$, all independent of the reach
profile, such that, for large $T$ and every $S\geq16T$,
\begin{equation}\label{eq:r12-p-small-curvature}
 A_p:=\int a^{2p}\omega_p
 \leq b\int(D^p+\widehat E^p)\omega_p+C_b(N_p+V_p).
\end{equation}
\end{corollary}
\begin{proof}
Use tolerance $b/2$ in \eqref{eq:r12-curvature-only-higher-moment}
and the compact $\eta_B$ of Lemma~\ref{lem:physical-compact-cutoff}.
Its enlargement coefficient is at most $2$ after increasing $T$
after $s_0$. Hence the defect term is at most
$b\int D^p\omega_p+C_bV_p$. The gradient and core terms are bounded by
$C_b(N_p+V_p)+e_B$. Let $B\to\infty$ using the same lemma. The remaining fixed bad
collar is bounded by boundedness of $a$ and \eqref{eq:r9-collar}.
\end{proof}

\vskip0.12in

\subsection{A weighted stability inequality}
\label{sec:r13-rss-power}
In this subsection we reduce the stability inequality to the zero
set. To this end we use the root-sum-squares construction of
\cite[Section~9.2]{OpenAI2026}, with a single ambient scalar weight. Throughout,
$3\leq m\leq8$, $2\leq p\leq9/4$, and $S\geq16T$; the cutoffs,
cap, net and localization quantity $V_p$ remain those fixed above.

\vskip0.12in

\begin{theorem}
\label{thm:r13-actual-p-rss}
Fix $0<\tau\leq10^{-3}$. There are $c_0,L_0,t_*>0$ and
$C<\infty$, determined by $\tau$ and the fixed local choices,
uniformly for $2\leq p\leq9/4$, with the following property.
Let $f\in C^1$ be nonincreasing, $0<f\leq1$, equal to one on
$[0,t_*]$, with $|f'|\leq K_ff$ for some $K_f<\infty$ and eventual tail
$f(t)=c_fe^{-\lambda_ft}$ for some $\lambda_f>0$. Choose the halo
rate $c_h<2\epsilon\lambda_f$ as above. There are
$T_0,C_f<\infty$ such that, for every $T\geq T_0$,
\begin{align}\label{eq:r13-p-rss-interface}
 (1-\tau)\int_{\{s_X\geq4L_0\}}S_0f(t)^2\omega_p+c_0N_p
 \leq{}&(1+\tau)\int f'(t)^2|\nabla t|^2\omega_p
               +CS^{-2}\int f(t)^2\omega_p\notag\\
 &+CT^{-2}\int_{\mathcal A_T}\zeta_T^{2p-2}\chi_S^2f(t)^2+C_fV_p.
\end{align}
The constants $C_f,T_0$ may depend on the chosen profile, its bound
$K_f$, the halo rate and the fixed solution and local data. The threshold
$T_0$ is enlarged as needed for $R_0$ and the fixed exterior guard.
The constants $c_0,L_0,t_*$ and the leading coefficients are independent
of the profile and halo rate.
\end{theorem}
\begin{proof}
We construct a compact ambient test and then estimate the errors
from its cutoffs. Use the partition $\alpha_0=\cos\theta$,
$b=\sin\theta$, the base mask $\mu$, and curvature cutoff $\beta$
fixed above; put
$h_p=\zeta_T^p\chi_S$. Insert physical and normal outer cutoffs
$\rho_B,\gamma_B$, equal to one up to $B$, zero above $2B$, with
first derivatives $O(B^{-1})$. A base cutoff $\kappa_B$, one on
$B_{4B}$ and zero off $B_{8B}$, is inactive because the other
supports imply $|z|\leq|z+s\nu|+|s|\leq4B$. Take the RSS of
\[
 qh_p\rho_B\alpha_0,\qquad
 G\beta\gamma_B\kappa_Bh_p(z+s\nu)\rho_B(z+s\nu)
                b(z+s\nu)\mu(z)f(t(z)).
\]
The compact parametrizing support lies in the locally invertible
curvature tube, so every nonzero fiber is finite; cutoff zero
extensions give a compact Lipschitz exterior test.

\vskip0.12in

For components $\xi_j=ha_j$, the exact variance is
\begin{equation}\label{eq:s5-rss-variance-common}
 \sum_j|\nabla\xi_j|^2-|\nabla(\sum_j\xi_j^2)^{1/2}|^2
       =h^2\min_w\sum_j|\nabla a_j-a_jw|^2.
\end{equation}
Use $h=h_p\rho_B$, the ambient identity
$Q_u(qv)=\int q^2|\nabla v|^2-\int q^3Dv^2$, and the lifted
identity \eqref{eq:r3-lift-exact} with its full scalar
$h_p\rho_Bb\mu f\beta\gamma_B$. The latter gives exactly
$-J_sGG'$ after integration by parts, with inverse tangential metric
$(I+sA)^{-2}$; no derivative of the common scalar enters the variance.
All outer-cutoff errors are $e_B\to0$ at fixed $S,T,x_c,f$:
polynomial area, bounded local jets, the coarse gradient
$|\nabla t|\leq C(1+t)^2$, and the exponential tail give joint domination
\[
 \chi_S(z+s\nu)^2G(s)^2\leq C_{S,x_c}e^{-c_S(|z|+|s|)}.
\]
This also justifies moving $\rho_B$ between base, physical point and
net center in all the collar bounds, with error $e_B$.

\vskip0.12in

Disjoint fixed balls about the bad net have $q^3D\geq c_\delta>0$
and $\alpha_0=1$. Reverse fixed-displacement weight transport gives
the lower bound $c_bN_B-CV_p-e_B$, where $c_b>0$ is fixed before $L_0,f$.
On $\Sigma_{\rm sw}=\{3L_0/2\leq s_X\leq4L_0\}$ use a fixed short
jointly injective tube. Ambient and normal gains are complementary
because $f=\mu=1$ there. Replacing $b(z+s\nu)$ by $b(z)$ introduces an error $O(L_0^{-1})D$:
$|b(z+s\nu)-b(z)|\leq C\min(1,|s|/L_0)$ and
$|J_s|\leq CD(1+|s|)^{m-1}$, using $H=O(D)$ and $a^2\leq CD$.
The remaining part of the normal gain gives the positive term
$c_s\int_{\Sigma_{\rm sw}}D\omega_p$ with $c_s>0$ fixed small
relative to the ambient gain and $\tau c_I$. At the level
$s_X=4L_0$ it is included in that remaining fraction, so the same gain is
not used twice, even for positive-area level sets.

Write
\[
 \begin{gathered}
 \mathcal S_B=\int_{s_X\geq4L_0}\rho_B^2S_0f^2\omega_p,
 \quad \mathcal F_B=\int\rho_B^2f'^2|\nabla t|^2\omega_p,
 \quad N_B=\sum_{v\in\mathcal X}\rho_B(v)^2\omega_p(v),\\
 \mathcal R_B=\int_{\Sigma_{\rm sw}}\rho_B^2D\omega_p,\qquad
 \mathcal C_B=S^{-2}\int\rho_B^2f^2\omega_p
 +T^{-2}\int_{\mathcal A_T}\rho_B^2\zeta_T^{2p-2}\chi_S^2f^2.
 \end{gathered}
\]
Normalize by $\sigma$. Choose the Young cross-term, normal-mass and inverse-metric errors
to have coefficients at most $\tau/16$ each,
so their total coefficient is at most $1+\tau/4$.
The exact identities and variance then give
\begin{equation}\label{eq:s5-rss-finite-master}
 (1-\tau/4)\mathcal S_B+c_bN_B+c_s\mathcal R_B
 \leq(1+\tau/4)\mathcal F_B+C\mathcal C_B
         +\sum_{j=1}^7\mathcal E_{j,B}+C_fV_p+e_B.
\end{equation}
The first six nonnegative error terms come from variation of the
normal weight in the gain terms, the curvature cutoff, the partition
of unity, the base cutoff, the profile--partition cross terms, and
the ambient scalar derivatives. The seventh term records transport
of the spatial weight in the profile-gradient contribution.
All fixed Young and normal-comparison factors are included in its
definition. We estimate these terms below.

\vskip0.12in

First, the two weight comparisons give
\[
 \begin{aligned}
 |h_p(z+s\nu)^2-h_p(z)^2|
 &\leq[v+C|s|e^{2|s|/S}/S]h_p(z)^2\\
 &\quad+C_ve^{2|s|/S}\chi_S(z)^2\min(1,(|s|/T)^{2p})
            \mathbf1_{\{\zeta_T(z+s\nu)\ne\zeta_T(z)\}}.
 \end{aligned}
\]
The last term is bounded by $C_vV_p$ by the localization estimates;
the first coefficient is arbitrarily small. On the curvature boundary,
$|\nabla\beta|\leq C_\alpha(1+|s|)$ and $|s|a\geq\alpha/2$,
so its square and fixed-Young crosses integrate to
$Ce^{-c_\alpha/a}=o_\delta(D)$. Use $I\geq c_ID$ for the
terms on the clean and switch regions. The tail restricted at the base point bounds
the remaining contribution after $L_0$ is chosen. This gives
\begin{equation}\label{eq:s5-rss-local-budget}
 \mathcal E_{1,B}+\mathcal E_{2,B}
 \leq\tfrac\tau4\mathcal S_B+\tfrac{c_s}4\mathcal R_B
              +\tfrac{c_b}8N_B+C_fV_p+e_B.
\end{equation}
All small fractions precede $f$.

\vskip0.12in

Next consider the partition of unity. Put $M_0=C_1\log L_0$. At an accurate
switch point with $|s|\leq M_0$ there is a unique accurate inverse
branch and $\beta=\mu=f=1$. Change its integral to physical measure,
cancelling $J$, before using the two-component variance. With
$A=q^2,B=G^2$, its pure partition remainder is exactly
\[
 Ab^2+B\alpha_0^2-\frac{AB}{A\alpha_0^2+Bb^2}
 =\frac{\alpha_0^2b^2(A-B)^2}{A\alpha_0^2+Bb^2},
\]
and its variance is
\[
 \frac{AB\alpha_0^2b^2}{A\alpha_0^2+Bb^2}
 \left|\nabla\log(q/G)-\frac{\nabla\theta}{\alpha_0b}\right|^2,
\]
interpreted continuously at vanishing factors. The ratio-gradient
square is favorable; its cross and the common-scalar cross are at most
\[
 C_\xi G^2B_z(s)
 \{h_p^2(L_0^{-1}+L_0^{-2})+h_p|\nabla h_p|/L_0\}.
\]
Other inverse branches only increase the variance.
Use $G^2B_z(s)\leq CD(z)e^{\xi|s|}$, observer derivative
$O(S^{-1})h_p^2$, and the second transport inequality for the
physical derivative. The short-ray terms satisfy the bound below.
For an inaccurate short lifted ray,
$G^2\leq C_\xi D(z)e^{\xi M_0}$. At a switch point without
an accurate preimage choose a nearest zero. If its distance
$r\leq M_0$, its regular base lies in $\Sigma_{\rm sw}$ and
$D(z)e^{(2\sqrt2+\xi)r}>b_\xi$, so phase decay gives
$q^2\leq C_\xi D(z)e^{2\xi M_0}$. The normal area formula,
with $J\leq C(1+M_0)^m$, permits overcounting eligible nearest
rays; no single-valued projection is required. The remaining normal
and ambient rays are precisely the two tails of
Lemma~\ref{lem:s5-rss-growing-collars}. Altogether,
\[
 \mathcal E_{3,B}\leq
 \frac{C_\xi(1+M_0)^{N_0}e^{2\xi M_0}}{L_0}\mathcal R_B
 +C_\xi(1+L_0+M_0)^{N_0}e^{-cM_0}N_B+C_vV_p+e_B.
\]
Choose $C_1>(N_0+3)/c$, then $2\xi C_1<1/4$, its accurate-ray
threshold, and finally large $L_0$, to obtain
\begin{equation}\label{eq:s5-rss-partition-budget}
 \mathcal E_{3,B}\leq\tfrac{c_s}4\mathcal R_B+
                    \tfrac{c_b}8N_B+C_vV_p+e_B.
\end{equation}

\vskip0.12in

If $\nabla\mu$ and $b$ are nonzero, then $|s|\geq L_0$;
its base is in the $L_0$ collar, but its endpoint need not be.
The tail restricted at the base point therefore bounds its squared derivatives and
cross terms. In a profile--$b$ cross, $f'\ne0$ forces
$s_X(z)>10L_0$, whereas $s_X(x)<3L_0$, so $|s|>7L_0$.
Young's inequality gives $(\tau/4)\mathcal F_B$ plus another square,
which is bounded by the same tail estimate. Increasing $L_0$ in the preceding geometric choice gives
\begin{equation}\label{eq:s5-rss-mask-profile-budget}
 \mathcal E_{4,B}+\mathcal E_{5,B}
 \leq\tfrac\tau4\mathcal F_B+\tfrac{c_b}8N_B+C_fV_p+e_B.
\end{equation}
The profile--$\mu$ cross vanishes; curvature crosses were already
included. We estimate the remaining profile-gradient weight transport
separately. Put
\[
 F_f(z)=f'(t(z))^2|\nabla_\Sigma t(z)|^2.
\]
The coarse reach-gradient bound and the eventual exponential tail give
\[
 0\leq F_f(z)\leq C(1+t(z))^4f'(t(z))^2\leq C_f.
\]
On the curvature tube, $J$ and $(I+sA)^{-2}$ are uniformly bounded.
After extracting $F_f(z)$, the normal kernel of this contribution is
nonnegative and bounded by $C(1+|s|)^Ne^{-a_0|s|}$ for a fixed
$a_0>0$, with constants independent of the profile. The weight comparison displayed above and
Lemma~\ref{lem:physical-transport-calculus} therefore give
\begin{equation}\label{eq:s5-rss-profile-transport}
 \mathcal E_{7,B}
 \leq C(v+S^{-1})\mathcal F_B+C_{f,v}V_p+e_B.
\end{equation}
The relative observer-weight error is integrated using the first
moment of the normal kernel. The additive physical-cutoff error gives
$T^{-2p}$ on $\mathcal A_T$ and an exponential halo outside it;
using $F_f\leq C_f$, these terms are bounded by $C_{f,v}V_p$.
Moving the finite outer cutoff contributes only $e_B$ by the joint
exponential domination established above. Choose $v>0$ sufficiently
small and then increase $T$ after the profile is fixed, so that
$S\geq16T$ implies
\[
 \mathcal E_{7,B}\leq\tfrac\tau8\mathcal F_B+C_fV_p+e_B.
\]
This choice leaves the previously fixed local geometric parameters
and leading coefficients unchanged.

\vskip0.12in

Finally, the scalar derivatives on the lifted pieces are bounded by
$C\mathcal C_B+C_fV_p+e_B$ by the normal derivative estimate.
On the ambient piece they are supported in the $3L_0$ bad collar.
The observation-weight term is
$C_{L_0}S^{-2}N_B+CV_p+e_B$; the physical-cutoff term is
\[
 C_{L_0,p}T^{-2}\sum_{v\in\mathcal X\cap\mathcal A_T'}
                   \zeta_T(v)^{2p-2}\chi_S(v)^2+CV_p+e_B,
\]
where $\mathcal A_T'\subset B_{4T}$ for large $T$.
By \eqref{eq:physical-young-power} with $j=2,E=1$, the latter is
bounded by $c_bN_B/16$ plus a localization remainder $CV_p$.
Increasing $T$ bounds the observation-weight term by another $c_bN_B/16$. Thus
\begin{equation}\label{eq:s5-rss-ambient-budget}
 \mathcal E_{6,B}\leq\tfrac{c_b}8N_B+C_fV_p+e_B.
\end{equation}
This accounts for derivatives of every scalar factor and cutoff.

\vskip0.12in

Combining the preceding estimates gives
\[
 \sum_{j=1}^7\mathcal E_{j,B}
 \leq\tfrac\tau4\mathcal S_B+\tfrac{3\tau}8\mathcal F_B
       +\tfrac{c_s}2\mathcal R_B+\tfrac{c_b}2N_B+C_fV_p+e_B.
\]
Inserting this bound in \eqref{eq:s5-rss-finite-master} gives
\[
 (1-\tau/2)\mathcal S_B+\tfrac{c_b}2N_B+\tfrac{c_s}2\mathcal R_B
 \leq(1+5\tau/8)\mathcal F_B+C\mathcal C_B+C_fV_p+e_B.
\]
Discard the positive switch term and weaken the clean and profile
coefficients to $1-\tau$ and $1+\tau$.
Joint exponential domination exhausts the displayed
quantities, and Fatou's lemma passes the variance lower bound to the limit. The positive bad-net
coefficient is $c_0=c_b/2$. All choices of small leading fractions,
$L_0$ and $t_*$ precede $f$; only $T$ increases afterwards.
\end{proof}

\vskip0.12in

\subsection{A profile for the reach function}
\label{sec:r12-moment-profile-pay}

In this subsection we estimate the reach-gradient term in
Theorem~\ref{thm:r13-actual-p-rss}. We fix $3\leq m\leq8$ and
$p=9/4$, and use the weak reach and selected positive-foot estimates
from Proposition~\ref{prop:r13-nine-dimensional-local-package}.
The multiplier and signed-foot argument is adapted from
\cite[Sections~9.3--9.4]{OpenAI2026}. The receiving point is evaluated
at its own capped reach, and the physical-cutoff remainders are
estimated with power $2p$.

\vskip0.12in

In this subsection, write $\gamma>0$ for the decay exponent $\theta$
in Proposition~\ref{prop:r7-actual-weak-reach}. Put $p=9/4$,
$r=p-1=5/4$ and $\widehat E=e^{-\kappa t}$,
where $t=\min(d,\epsilon\rho_o)$ and
$\kappa=\sqrt2/G_1$, $1\leq G_1\leq1.001$.
We use the following weak branch bounds:
\[
 \Delta t\leq\mathcal A+B|\nabla t|^2/t
\]
with favorable singular part. On the nonplateau branch,
$\mathcal A\leq(1+\omega+Ct^{-\gamma})a^2t$ and
$B\leq1+Ct^{-\gamma}$; when $at$ is sufficiently small,
$B\leq\omega+Cat+Ct^{-\gamma}$ and
$|\nabla t|\geq c_\Gamma>0$.  On a plateau,
$\mathcal A=G_1(|H|+H_{N_y,+})$, $B\leq\omega$.
On the cap we have
$\Delta t\leq m\epsilon/\rho_o+\epsilon|H|$.
On clean zeros assume
\[
 I\geq.999a^2,\qquad I\geq c_ID,\qquad |H|\leq.355I.
\]
The selected positive feet have multiplicity one, bounded inverse
Jacobian $J_\Gamma$, receiving order $t(y)\leq t(z)$,
and $H_{N_y,+}\leq C_\Gamma/t(z)$.
For $a(y)t(z)$ sufficiently small, the inverse Jacobian with its
fixed small weight inflation is at most $1.05$.
All thresholds and these constants are fixed before the profile.

\vskip0.12in
For $S\geq16T$ use
\[
 \omega_p=\zeta_T^{2p}\chi_S^2,\qquad
 V_p=\mathcal V_{S,T,c_h}[\mu_{\rm loc}].
\]
Unit good charts and the fixed-ball positive energy lower bound at zeros suffice for
every local area and net estimate below. Transition terms are bounded by
$T^{-2p}\int_{B_{4T}}\chi_S^2d\mu_{\rm loc}$; tails based at the origin
are bounded by $\int h_{c_h,T}\chi_S^2d\mu_{\rm loc}$. Both estimates
follow directly from the local measure $\mu_{\rm loc}$.

\begin{lemma}
\label{lem:r12-nine-quarter-profile}
There is $t_{\min}$, depending only on the fixed local geometry and
the $.91$, $.01$ table tolerances, which can be fixed before the final
good threshold and coupled-absorption constant. For every subsequently
prescribed $s_*,e_3>0$, there is a nonincreasing $C^1$ profile
$0<f\leq1$, equal to one on the prescribed compact collar-reach
interval, with the following properties. In the definition of $B_*$
below, $C$ is one fixed coefficient dominating the weak nonplateau
branch bounds. Set $k=1.78$ and
\[
 B_*(t)=1+Ct^{-\gamma},\qquad
 \pi(t)=\exp\int_{t_{\min}}^tB_*(v)\,dv/v,\qquad
 P(t)=\pi(t)^{-1}\int_t^\infty\pi(v)f'(v)^2\,dv.
\]
There is a fixed $t_0$ for which $P$ is decreasing on
$[t_{\min},\infty)$ and, on this interval,
\begin{align}\label{eq:r12-nine-quarter-package}
 -P'&=f'^2+B_*P/t,\nonumber\\
 tP&\leq.91f^2\quad(t<t_0),&
 P&\leq(5\kappa/16+.001)f^2\quad(t\geq t_0),\nonumber\\
 t^{7/2}f^2&\leq e_3\quad(t\geq t_0),&
 \sup_{[t_{\min},t_*]}P&\leq s_*.
\end{align}
Moreover
\begin{equation}\label{eq:r12-nine-quarter-tail}
 \widehat E^{5/4}\leq f^2\leq H_f\widehat E^{5/4},\qquad
 \Psi(t):=\int_t^\infty P\leq C_ff^2,\qquad
 \sup t^{7/2}P(t)<\infty.
\end{equation}
The eventual exponential rate is $5\kappa/8$, and $P$ has two-sided
exponential bounds on that branch. The constants $H_f,C_f$, the threshold $t_0$, and the starting point
of the exponential tail may depend on the chosen profile; $B_*$ and
$t_{\min}$ do not.
\end{lemma}
\begin{proof}
We construct the profile in three steps, and then smooth it without
changing the stated bounds. Use the fixed coefficient $C$ in $B_*$
and put $\lambda_f=5\kappa/8$ and fix $t_{\min}$ so that
\[
 C_\pi:=\exp\{C/(\gamma t_{\min}^{\gamma})\}\leq1.001,
 \qquad k/t_{\min}<\lambda_f.
\]
The fixed lower threshold was chosen so that
$\omega+Ct_{\min}^{-\gamma}\leq.01$, as required for the $t<t_0$ rows of the table below.
Then $\pi(v)/\pi(t)\leq C_\pi v/t$ for every $v\geq t$.
The construction has one shape parameter and three branches:
a plateau, a soft power, and an exponential tail.

\vskip0.12in

Fix $t_s>\max(t_*,t_{\min})$.  For $0<\delta_f<1/4$ put
\[
 F_{\delta_f}(t)=
 \begin{cases}
 1,&t\leq t_s,\\
 [1-\delta_f+\delta_f(t/t_s)^{\delta_f}]^{-k/\delta_f},&t>t_s.
 \end{cases}
\]
Its logarithmic slope is zero on the plateau and, for $t>t_s$,
\[
 \alpha_f=-tF_{\delta_f}'/F_{\delta_f}
 =\frac{kx}{1-\delta_f+x},\qquad
 x=\delta_f(t/t_s)^{\delta_f},\qquad
 \frac{d\alpha_f}{d\log t}
       =\delta_f\alpha_f(1-\alpha_f/k).
\]
Define
\[
 Q_{\delta_f}(t)=\frac{k(\delta_f+2\alpha_f(t))}
                        {2(\delta_f+2k)}\quad(t>t_s),\qquad
 Q_{\delta_f}(t)=E_{\delta_f}:=
 \frac{k\delta_f(1+2k)}{2(\delta_f+2k)}\quad(t\leq t_s).
\]
Direct differentiation gives the Riccati identity
$dQ_{\delta_f}/d\log t=2\alpha_fQ_{\delta_f}-\alpha_f^2$
on the soft-power branch.  Consequently
\begin{equation}\label{eq:s7-soft-profile-flux}
 \int_t^\infty vF_{\delta_f}'(v)^2\,dv
       =F_{\delta_f}(t)^2Q_{\delta_f}(t),\qquad
 0<Q_{\delta_f}\leq k/2,\qquad E_{\delta_f}\leq1.14\delta_f.
\end{equation}
Indeed the derivative of $F_{\delta_f}^2Q_{\delta_f}$ is
$-tF_{\delta_f}'^2$, and this product tends to zero at infinity.
The equality extends across the plateau because both sides are
constant there.  Choose $\delta_f$ so that
$C_\pi E_{\delta_f}/t_{\min}\leq s_*/4$.

Since $\alpha_f\to k>7/4$, choose $t_0\geq\max(t_s,4000)$ with
\[
 \alpha_f(t)\geq1.76,\qquad
 t^{7/2}F_{\delta_f}(t)^2\leq e_3/2\quad(t\geq t_0).
\]
The latter product tends to zero and decreases once
$\alpha_f>7/4$.  No further change of $\delta_f$ or $t_{\min}$
is needed.

\vskip0.12in

At $t_e>t_0$ splice to
$f(t)=F_{\delta_f}(t_e)e^{-\lambda_f(t-t_e)}$;
before $t_e$ use $f=F_{\delta_f}$.  Initially work with this
piecewise smooth profile.  On its exponential branch,
\[
 Q_e:=P(t_e)/f(t_e)^2
 \leq\frac{\lambda_f^2}{2\lambda_f-B_*(t_e)/t_e}.
\]
Choose $t_e$ sufficiently large that
$Q_e\leq\lambda_f/2+10^{-4}$ and $t_e>7/(4\lambda_f)$.
The same upper bound holds for $P(t)/f(t)^2$ for all $t\geq t_e$,
since $B_*(t)/t$ is decreasing.  Set
$\Theta_e=t_eQ_eF_{\delta_f}(t_e)^2$.  As $t_e\to\infty$,
$\Theta_e\to0$, because $F_{\delta_f}(t)^2=O(t^{-2k})$.
Enlarge $t_e$ so that
\[
 \frac{C_\pi\Theta_e}{F_{\delta_f}(t_0)^2}\leq.005,
 \qquad \frac{C_\pi\Theta_e}{t_{\min}}\leq s_*/4.
\]
The kernel comparison and \eqref{eq:s7-soft-profile-flux} now give
one bound for every $t_{\min}\leq t\leq t_e$:
\begin{equation}\label{eq:s7-soft-profile-terminal}
 \frac{tP(t)}{f(t)^2}
 \leq C_\pi\left\{Q_{\delta_f}(t)
                       +\frac{\Theta_e}{F_{\delta_f}(t)^2}\right\}.
\end{equation}
For $t<t_0$ it gives
$tP/f^2\leq1.001(.89)+.005=.89589<.91$.
At $t=t_{\min}$ it gives $P(t_{\min})\leq s_*/2$.
For $t_0\leq t\leq t_e$, use $\alpha_f\geq1.76$ to obtain
\[
 \frac{\Theta_e}{tF_{\delta_f}(t)^2}
 =\frac{t_e}{t}Q_e\frac{F_{\delta_f}(t_e)^2}{F_{\delta_f}(t)^2}
 \leq Q_e(t_e/t)^{1-3.52}\leq Q_e.
\]
It follows that
\[
 P/f^2\leq C_\pi\{k/(2t_0)+\lambda_f/2+10^{-4}\}
 <\lambda_f/2+.000766<5\kappa/16+.001,
\]
uniformly for $\kappa\in[\sqrt2/1.001,\sqrt2]$.
The exponential bound already treats $t\geq t_e$.

\vskip0.12in

The bound on $t^{7/2}f^2$ is preserved by this change because
$t^{7/2}f^2$ decreases on the exponential branch by
$t_e>7/(4\lambda_f)$.  Also $-f'/f\leq\lambda_f$, since
$\alpha_f\leq k$ and $t\geq t_{\min}$.  Thus
$f^2e^{2\lambda_f t}$ is nondecreasing before the tail and constant
on it, proving the two exponential comparisons in
\eqref{eq:r12-nine-quarter-tail}.

\vskip0.12in

Smooth the logarithmic derivative in small neighborhoods of $t_s,t_e$,
keeping it in $[0,\lambda_f]$, leaving the plateau through
$\max(t_*,t_{\min})$,
and making it exactly $\lambda_f$ on the final tail.  The smoothed
profiles converge uniformly on intervening compact intervals, and
$f_j'\to f'$ in $L^2([t_{\min},\infty),\pi(t)\,dt)$.
Outside the shrinking neighborhoods the soft-power and exponential
branches each keep their shape multiplied by a constant tending to
one; their original weighted derivative energy is finite.
Cauchy--Schwarz therefore gives uniform convergence of
$\int_t^\infty\pi f_j'^2$ in $t$, and uniform convergence of
$P_j/f_j^2$ on compact intervals.  On the final exponential
branch its explicit bound is independent of the multiplier.
The strict $.91$, $.001$ margins and the half-size compact-flux
and curvature bounds preserve all required inequalities under a
sufficiently small smoothing.  The logarithmic derivative bound
preserves the exponential floor.
Finally $-P'=f'^2+B_*P/t$ makes $P$ decreasing.  On the final tail
$P$ is comparable to $f^2$, so its primitive and the remaining
finite moment follow there by integration and on the finite
complement by continuity. The construction therefore leaves the good
threshold, averaging windows, and previously fixed $C_A$ unchanged.
\end{proof}

\vskip0.12in

\begin{lemma}
\label{lem:r12-receiver-own-reach}
With $t_{\min}$ and the branch tolerances fixed for the preceding
profile construction, the sum of the outgoing and clean incoming
densities, divided by the
single common factor $f(t)^2\omega_p$, has the following upper
bounds, up to terms with arbitrarily small $a^{2p}$ coefficient and
terms bounded by the net sum in a fixed bad collar:
\begin{equation}\label{eq:r12-nine-quarter-density-table}
 \begin{array}{c|c}
 \text{type of evaluation point}&\text{upper bound}\\\hline
 \text{nonplateau, own }t<t_0&.92a^2+.01I\\
 \text{plateau, own }t<t_0&.02I\\
 \text{nonplateau, own }t\geq t_0&.17I\\
 \text{plateau, own }t\geq t_0&.33I
 \end{array}
\end{equation}
At a base point on the nonplateau branch with $t\geq t_0$, the outgoing term is absorbed by
the improved drift, or is bounded by a small multiple of $a^{2p}$.
\end{lemma}
\begin{proof}
For a clean incoming point $y$ write $t_y=t(y)$ for its own
capped reach and $t_z=t(z)$ for the capped reach at the selected base point.
The reverse chord and cap order give $t_y\leq t_z$.  Since $P$
is decreasing, we have the stronger comparison
\begin{equation}\label{eq:r12-P-monotone-foot}
 P(t_z)\leq P(t_y).
\end{equation}
If $t_y<t_0$, use the full bounded inverse Jacobian $J_\Gamma$
and $P(t_y)\leq.91f(t_y)^2/t_y$.  Choosing $t_{\min}$ first
large enough makes the incoming density at most $.01If(t_y)^2$.
This includes base points whose capped reach $t_z$ lies in the
final exponential branch.  Replacing \eqref{eq:r12-P-monotone-foot}
only by a uniform exponential bound would give a coefficient of
at least $5/32$ at this receiver with $t_y<t_0$, which is too large for the
following absorption.

\vskip0.12in

If $t_y\geq t_0$, then $t_z\geq t_0$.  Where $a(y)t_z$ is small,
the inflated inverse Jacobian is at most $1.05$, and
\[
 1.05G_1(.355)(5\kappa/16+.001)<.17.
\]
Monotonicity of $f$ gives the incoming bound for $t_y\geq t_0$.
Where $a(y)t_z>c_3$, use the tangent-ball trace bound
$H_{N_y,+}\leq C_\Gamma/t_z$ and bounded Jacobian.  Its contribution is
\[
 CP(t_z)/t_z\leq Ce_3t_z^{-2p}\leq C_{c_3}e_3a(y)^{2p}.
\]
Base points whose receiving points enter the fixed bad collar have
bounded $t_z\leq t_*$ by the receiving-defect input; their compact
$P$ flux gives the stated bad-net bound without using bad-receiver
surface regularity.

\vskip0.12in

Using the accuracy fixed for the $t<t_0$ rows,
$\omega+Ct_{\min}^{-\gamma}\leq.01$, put
$L_A=1+\omega+Ct_{\min}^{-\gamma}\leq1.01$.
At a base point on the nonplateau branch with $t<t_0$, \eqref{eq:r12-nine-quarter-package} gives
$\mathcal AP\leq.91L_Aa^2f^2\leq.9191a^2f^2<.92a^2f^2$.
Choose $c_*>0$ below the improved-drift threshold so that
\[
 \omega+Cc_*\leq1/4,\qquad
 L_Ac_*^2\leq c_\Gamma^2/100.
\]
This choice uses fixed local data and precedes the profile.
At a base point on the nonplateau branch with $t\geq t_0$ and $at\leq c_*$, the improved bounds
and the present $B_*=1+Ct^{-\gamma}$ give
\[
 B_*-B\geq1-\omega-Cat\geq3/4,\qquad
 \mathcal AP\leq L_A(at)^2P/t
       \leq .01P|\nabla t|^2/t.
\]
Thus the nonnegative drift term absorbs $\mathcal AP$ with the same $B_*$
on both ranges $t<t_0$ and $t\geq t_0$.
Otherwise $at\geq c_*$ and
\[
 \mathcal AP\leq Ca^2tf^2
   \leq Ce_3a^2t^{-2p+2}\leq C_{c_*}e_3a^{2p}.
\]
By the initial choice of $t_{\min}$, the outgoing plateau bound is
at most $.01If^2$ for $t<t_0$, and at most $.16If^2$ for
$t\geq t_0$, since
$G_1(.355)(5\kappa/16+.001)<.16$.
Adding the outgoing and incoming estimates proves the table.
\end{proof}

For $\tau\leq.001$ and $S_0=(a^2+I)/2$, subtracting
$(1+\tau)$ times each row's upper bound from $(1-\tau)S_0$
leaves a strictly positive lower bound.
The nonplateau row with $t<t_0$ is the limiting one:
\[
 .48949I-.42142a^2
 \geq(.48949-.42142/.999)I>.06I.
\]
The other rows leave more than $.47I$, $.32I$ and $.16I$,
respectively. Thus these positive margins are independent of the tolerances
$s_*$ and $e_3$ used to construct the profile.

\begin{proposition}
\label{prop:r12-nine-quarter-profile-pay}
Put $Z_p=\int(D^p+\widehat E^p)\omega_p$. There is $c_p>0$,
determined by the fixed local choices and the RSS coefficients, such
that for every $e_4,v>0$ one can choose the preceding profile so that
\begin{equation}\label{eq:r12-nine-quarter-pay}
 c_p(M_p+N_p)\leq e_4\int a^{2p}\omega_p+vZ_p
                  +C_fS^{-2}\int f^2\omega_p+C_{f,v}V_p,
 \qquad M_p=\int_{\{s_X\geq4L_0\}}Df^2\omega_p,
\end{equation}
The constants $C_f,C_{f,v}$ and the lower threshold for $T$ may
depend on the chosen profile. All cap and signed-foot physical errors
have scale $V_p$.
\end{proposition}
\begin{proof}
We apply the weak reach inequality with the profile from the previous
lemma and then use the signed-foot area formula. Choose $m_0=0$ on $s_X\leq5L_0$, $m_0=1$ on
$s_X\geq6L_0$, with $|\nabla m_0|\leq C/L_0$.
The constant part of $f$ covers the reaches at both base and receiving
points in the $10L_0$ bad collar; for large $T$ its cap also exceeds those finite bounds.
Thus $f'(t)=0$ wherever $m_0\ne1$.
Test the weak reach bound with $P(t)m_0\omega_p$; the singular measure
contributes with the favorable sign. The ODE gives exactly
\[
 \int f'^2|\nabla t|^2\omega_p
 +\int m_0\omega_p(B_*-B)P|\nabla t|^2/t
 \leq\int m_0\omega_p\mathcal AP
             +\int P\nabla t\cdot\nabla(m_0\omega_p).
\]
For the last term use $\nabla\Psi=-P\nabla t$ and
\[
 \int m_0P\nabla t\cdot\nabla\omega_p
 =\int\Psi\{m_0\Delta_\Sigma\omega_p+
                            \nabla m_0\cdot\nabla\omega_p\}.
\]
The direct collar derivative $P\nabla t\cdot\nabla m_0$ is bounded by
$C_{\rm col}s_*N_p+C_fV_p$. For its mixed primitive term,
\begin{equation}\label{eq:s3-profile-collar-mixed}
 \left|\int\Psi\nabla m_0\cdot\nabla\omega_p\right|
 \leq C_fS^{-1}N_p+\varepsilon N_p+C_{f,\varepsilon}V_p.
\end{equation}
Indeed $\Psi\leq C_f$ on this fixed collar; the observer derivative
uses its collar cover, while the physical derivative gives an annular bad-net sum.
Pointwise Young gives
\[
 C_fT^{-1}\zeta_T(b)^{2p-1}
 \leq\varepsilon\zeta_T(b)^{2p}+C_{f,\varepsilon}T^{-2p};
\]
after multiplying by $\chi_S(b)^2$ and summing, its remainder is
$C_{f,\varepsilon}T^{-2p}\int_{B_{4T}}\chi_S^2d\mu_{\rm loc}$.
The fixed collar cover and the transport estimates for the cutoff
power $\zeta_T^{2p}$ yield the same local bound. Fix $\varepsilon$ below the positive bad-net coefficient before the
profile; after constructing $f$, increase $T$ so that
$C_f/S\leq\varepsilon$ for every $S\geq16T$.
The surface Laplacian of $\omega_p$ is bounded by
\[
 C(S^{-2}+S^{-1}D)\omega_p
 +C\chi_S^2\mathbf1_{\mathcal A_T}
       (T^{-2}\zeta_T^{2p-2}+T^{-1}D\zeta_T^{2p-1}).
\]
The needed physical Young inequalities are
\begin{align*}
 T^{-2}\zeta_T^{2p-2}f^2
   &\leq v\zeta_T^{2p}\widehat E^p+C_{f,v}T^{-2p},\\
 T^{-1}m_0\zeta_T^{2p-1}Df^2
   &\leq v m_0\zeta_T^{2p}Df^2+C_{f,v}T^{-2p}.
\end{align*}
The first uses $f^2\leq H_f\widehat E^{p-1}$. For the second, use
$Df^2\leq1$ with the same factor $m_0$ on both sides.  Integrating these remainders on the good area in the transition region
or bad net gives
$C_{f,v}T^{-2p}\int_{B_{4T}}\chi_S^2d\mu_{\rm loc}\leq C_{f,v}V_p$.
Fixed collar enlargements obey the same $2p$-power Young comparison
and remain in $B_{4T}$ for large $T$.  Spatial mean terms are
$O_f(S^{-1})M_p$ and are absorbed after fixing $f$.

\vskip0.12in

For clean positive feet use the signed multiplicity-one area
formula, Theorem~\ref{thm:r4-positive-foot-transfer}. At a base point $z$ with an uncapped plateau contact,
$v=t(z)=d(z)$ and $|z-y|=v/G_1\leq v$; the
incoming coefficient is bounded by a fixed polynomial times $P(v)$.
Equation~\eqref{eq:physical-foot-remainder} therefore bounds the full
physical commutator by $C_fV_p$. Its observer split at $\sqrt S$
has arbitrarily small inflation on the short part. The long part is
bounded pointwise at its base point by $C_fS^{-2p}$ and hence contributes at most
$C_fS^{-2p}\int\chi_S^2dS\leq C_fV_p$.
Estimate the cutoff remainder at the base points using the local core and
exponential halo. Apply Lemma~\ref{lem:r12-receiver-own-reach} to the
principal transferred part, with its inverse Jacobian, own-reach order
and bad-receiver compact-flux term. Receiving-core terms are transferred
only where an inverse-map bound is available.

\vskip0.12in

On the smooth cap $t=\epsilon\rho_o$, $P$ and $f'^2$ are
fixed exponentially small functions of physical radius.
The cap Laplacian and all mean factors have fixed polynomial
growth. Their densities with respect to $dS_z$ are bounded by
$C_f(1+|z|)^Ne^{-c_f|z|}\mathbf1_{|z|\geq T}$ on good area.
Split at $B_{4T}$: the inside part is bounded by
$C_fT^{-2p}\int_{B_{4T}}\chi_S^2dS$, while the outside part is
bounded by $C_f\int h_{c_h,T}\chi_S^2dS$. Thus the total is at
most $C_fV_p$, with $c_h$ fixed below the common eventual decay
rate before the profile shape parameter is selected.
The receiving order gives the same $C_fV_p$ bound for
receiver cap branches, rather than a fixed $C_u$ term.

\vskip0.12in

Combine the table's strict pointwise margins with
\eqref{eq:r13-p-rss-interface}.  Use $I\geq c_ID$ for $M_p$,
choose $s_*$ and the collar Young fraction so that their terms are
absorbed by the positive bad-net term. Choose $e_3$ so that the coefficients of
$a^{2p}$ in the preceding curvature remainders, at base points with
$t_z\geq t_0$ and receiving points with $t_y\geq t_0$, have sum less than
$e_4$. Choose the Young fractions for the
remaining $M_p,N_p,Z_p$ terms, then increase $T$ after the
profile is fixed.  This proves \eqref{eq:r12-nine-quarter-pay}.
All weak calculations begin with finite supports; fixed exponential
tails and the coarse reach gradient justify their exhaustion.
\end{proof}

\vskip0.12in

\subsection{Proof of the moment estimate}
\label{sec:r12-higher-moment-closure}

We now combine the estimates of the preceding subsections to prove
Theorem~\ref{thm:r13-natural-nine-quarter-budget}. Throughout this
subsection, $p=9/4$ and $3\leq m\leq8$.
Proposition~\ref{prop:r13-nine-dimensional-local-package} gives the local comparison $D\leq C(J_\lambda+E)$ and $E\leq C(J_\lambda+D)$,
where $E=e^{-\kappa d}$ and
$J_\lambda(z)=\sup_w e^{-\lambda|z-w|}\bar a(w)^2$, with
$\bar a=a$ on the good set and $\bar a=1$ on the bad set.
Put $\widehat E=e^{-\kappa\min(d,\epsilon\rho_o)}$ and
\[
 \omega_p=\zeta_T^{2p}\chi_S^2,\quad
 A_p=\int a^{2p}\omega_p,\quad
 Z_p=\int(D^p+\widehat E^p)\omega_p,\quad
 N_p=\sum_{b\in\mathcal X}\omega_p(b),
\]
\[
 V_p=\mathcal V_{S,T,c_h}[\mu_{\rm loc}],\qquad S\geq16T.
\]
The required same-weight curvature estimate is Corollary~\ref{cor:s3-same-weight-curvature};
the Simons--Kato condition has a strict margin at least $37/800$.
The weighted stability inequality is
Theorem~\ref{thm:r13-actual-p-rss}. In the following argument the
good threshold and curvature windows are fixed before the reach profile.

\begin{lemma}
\label{lem:r12-higher-moment-coupling}
The local comparison implies
\begin{equation}\label{eq:r12-p-pointwise}
 D^p+E^p\leq C_pJ_\lambda^p+C_pDE^{p-1}.
\end{equation}
For the fixed good threshold and local curvature-absorption windows,
there is $C_A\geq1$, depending only on the fixed local choices and
these windows, such that every subsequently chosen profile with
$\widehat E^{p-1}\leq f^2\leq H_f\widehat E^{p-1}$ satisfies
\begin{equation}\label{eq:r12-p-coupled}
 A_p+Z_p\leq C_A(M_p+N_p+V_p),\qquad
 M_p=\int_{\{s_X\geq4L_0\}}Df^2\omega_p.
\end{equation}
In particular, the upper comparison constant $H_f$ does not enter $C_A$.
\end{lemma}
\begin{proof}
Apply Corollary~\ref{cor:s2-pointwise-moments} with $r=p$ to obtain
\eqref{eq:r12-p-pointwise}.

\vskip0.12in

For every fixed small disk radius $\mu>0$, the relative defect
derivative estimates and uniform good charts give
\[
 a(w)^{2p}\leq C\mu^{-m}\int_{\Sigma_g\cap B_1(w)}a^{2p}
                    +C\mu^{2p}\int_{\Sigma_g\cap B_1(w)}D^p
\]
away from the fixed bad collar.  Indeed average
$a(w)\leq a(v)+C\mu\sqrt{D(w)}$ on the $\mu$-disk, then
replace $D(w)^p$ by its relative-comparable fixed-disk average.
The second coefficient does not acquire $\mu^{-m}$.
Near the bad set use the same separated net.  Taking the weighted
supremum and then bounding the supremum of nonnegative disk integrals
by their sum proves
\[
 J_\lambda(z)^p\leq C_\lambda\int e^{-p\lambda|z-v|}
       \{\mu^{-m}a(v)^{2p}+\mu^{2p}D(v)^p\}\,dS_v
       +C_\lambda\sum_{b\in\mathcal X}e^{-p\lambda|z-b|}.
\]
Apply the identical-weight kernel bound
\eqref{eq:physical-same-weight-kernel} to the two terms, with
$F=\mu^{-m}a^{2p}+\mu^{2p}D^p$ on the good surface and $F=1$
on the bad net. These functions are bounded before any moment
estimate. It gives
\begin{equation}\label{eq:r12-p-envelope}
 \int\omega_p J_\lambda^p
 \leq C_\lambda\mu^{-m}A_p+C_\lambda\mu^{2p}Z_p
                  +C_\lambda N_p+C_{\lambda,\mu,p}V_p.
\end{equation}
The kernel estimate is applied to surface area on the good set and the bad-net counting measure;
its annular remainder is $C_{\lambda,\mu,p}T^{-2p}$ times their
weighted core masses, and its observer ratio has the same bounded
displacement integral as in the RSS test. Thus its remainder has the same $V_p$, and its principal coefficient
is fixed before the profile is chosen.

\vskip0.12in

For the curvature term use Corollary~\ref{cor:s3-same-weight-curvature},
whose compact-test and maximal-enlargement bounds are established
before the profile is chosen.

\vskip0.12in

First choose $\mu$ so the coefficient of $Z_p$ in
\eqref{eq:r12-p-envelope} after \eqref{eq:r12-p-pointwise} is below
$1/4$.  Its coefficient of $A_p$ is then fixed; choose $b$ in
\eqref{eq:r12-p-small-curvature} so its resulting $Z_p$ coefficient
is below $1/4$.  Fix its finite averaging windows and then the good
threshold.  Absorbing both quarters proves
\[
 A_p+Z_p\leq C\left\{\int D\widehat E^{p-1}\omega_p+N_p+V_p\right\}.
\]
Replacing $E$ by $\widehat E$ introduces only the good-area
integral
$C\int_{|z|\geq T}e^{-c|z|}\chi_S^2dS$, because their difference
occurs at $d\geq\epsilon\rho_o$. Its core/halo split is bounded
by $CV_p$ directly from the local measure $\mu_{\rm loc}$.
On the fixed bad collar, the integrand is bounded and the collar
cover gives $C(N_p+V_p)$.  On its complement use
$\widehat E^{p-1}\leq f^2$.  This proves
\eqref{eq:r12-p-coupled}, with the principal coupling constant fixed by
the local estimates and the two absorptions above.
\end{proof}

\vskip0.12in

We complete the proof of Theorem~\ref{thm:r13-natural-nine-quarter-budget}.
\begin{proof}
Put $X=A_p+Z_p$, $Y=M_p+N_p$ and $V=V_p$.
Enlarge the profile-independent constant $C_A$ to be at least one.
The coupling \eqref{eq:r12-p-coupled} is
\[
 X\leq C_A(Y+V).
\]
The positive constant $c_p>0$ in
Proposition~\ref{prop:r12-nine-quarter-profile-pay} is also fixed
before the profile.  Prescribe
$e_4=v=\nu=c_p/(8C_A)$ and construct the profile with these
fractions. With this profile, Young's inequality
with exponents $p/(p-1),p$ gives
\[
 C_fS^{-2}\int f^2\omega_p
 \leq C_fH_fS^{-2}\int\widehat E^{p-1}\omega_p
 \leq\nu Z_p+C_{f,\nu}S^{-2p}\int\chi_S^2
 \leq\nu Z_p+C_{f,\nu}V.
\]
The profile estimate therefore gives
\[
 c_pY\leq\frac{c_p}{8C_A}A_p+
             \frac{c_p}{4C_A}Z_p+C_fV,
 \qquad Y\leq X/(4C_A)+C_f'V.
\]
For the single nonnegative functional
$\mathfrak E_f=X+2C_AY$, the two estimates imply
\[
 \mathfrak E_f
 \leq C_AY+C_AV+2C_A\{X/(4C_A)+C_f'V\}
 =\tfrac12\mathfrak E_f+C_A(1+2C_f')V.
\]
Thus $\mathfrak E_f\leq C_f''V$.  Since it controls
$A_p+Z_p+N_p$, this proves the claimed moment estimate.
The fractions precede $f$; the Young remainder constants and the
lower threshold for $T$ may depend on it.  All calculations use
the common weight, cap and bad net, with the already established
finite-support estimates and exhaustion.
\end{proof}

\begin{corollary}
\label{cor:s8-energy-only-moment}
The right side of Theorem~\ref{thm:r13-natural-nine-quarter-budget}
is bounded by
\begin{equation}\label{eq:s8-energy-only-moment}
 \begin{gathered}
 C\left\{S^{-2p}\int\chi_S^2e_u\,dx
 +T^{-2p}\int_{B_{5T}}\chi_S^2e_u\,dx
 +\int h_{c_h,T}^{\sharp}\chi_S^2e_u\,dx\right\},\\
 h_{c_h,T}^{\sharp}=e^{-c_hT-c_h\operatorname{dist}(\cdot,B_{5T})}.
 \end{gathered}
\end{equation}
If $c_hT\geq2p\log S$, the last term can be omitted.
In particular it can be omitted for all sufficiently large $R$
when $T=\eta R$ with fixed $\eta>0$ and $S$ is at most any fixed
power of $R$.
\end{corollary}
\begin{proof}
For a positive weight $w$ with logarithmic Lipschitz constant at
most $L$, the fixed-ball positive energy lower bound at zeros and
the uniform local area bound on the good set give
\[
 \int w\,dS\lfloor\Sigma_g+
       \sum_{b\in\mathcal X}w(b)\leq C_L\int we_u\,dx.
\]
Indeed choose a maximal unit-separated cover of the good zero set.
Its unit balls cover the surface and have uniformly bounded good
area. Concentric balls of one fixed radius less than $1/4$ are
disjoint and each has uniformly positive diffuse energy.
The weight is comparable to its center value on both balls, proving
the assertion. The bad net has the same direct argument.
For a restriction to $B_{4T}$ the small energy balls lie in $B_{5T}$
for large $T$. Apply this comparison to $w=\chi_S^2$ and to
$w=h_{c_h,T}\chi_S^2$, whose logarithmic Lipschitz constants are
at most $2/S$ and $c_h+2/S$. This proves
\eqref{eq:s8-energy-only-moment}; the enlargement of the halo ball
only increases its weight. Finally
$h_{c_h,T}^{\sharp}\leq e^{-c_hT}\leq S^{-2p}$ under the
stated scale restriction, so its entire integral is bounded by the
first term. This proves the corollary.
\end{proof}

\begin{corollary}
\label{cor:s8-density-envelope-moment}
With $K(r)=1+\sup_xr^{-m}E(u;B_r(x))$, the moment estimate takes the following density-envelope form:
\begin{equation}\label{eq:s8-density-envelope-moment}
 \int\omega_p(a^{2p}+D^p+\widehat E^p)+N_p
 \leq C_p\{K(S)S^{m-2p}
       +K(8T)T^{m-2p}\chi_S(0)^2\}.
\end{equation}
\end{corollary}
\begin{proof}
The fixed-ball energy comparison and exponential shell summation
give
\[
 \int\chi_S^2d\mu_{\rm loc}\leq CK(S)S^m,\qquad
 \int_{B_{4T}}\chi_S^2d\mu_{\rm loc}
       \leq CK(8T)T^m\chi_S(0)^2.
\]
For the halo, use the same centered shell decomposition. Inside
$B_{4T}$ the observer ratio is bounded. Outside, the factor
$e^{2|z|/S}$ consumes at most half the fixed exponential rate once
$T$ is large, while covering by $8T$ balls bounds all shell masses.
Consequently
\[
 \int h_{c_h,T}\chi_S^2d\mu_{\rm loc}
 \leq C_{c_h}K(8T)T^me^{-c_hT}\chi_S(0)^2
 \leq CK(8T)T^{m-2p}\chi_S(0)^2.
\]
Insert these evaluations into the local moment theorem.
\end{proof}

\vskip0.12in

\section{\bf Penalized reach and ambient stability}
\label{sec:s12-radial}

\vskip0.12in

\subsection{The radial calculus}
\label{sec:r9-radial-core}\label{sec:r10-optimized-radial-core}

We construct a radial test function and compare two inequalities for
its restriction to the zero set. The resulting comparison is the main
step in the energy growth proof. Throughout this section and
Section~\ref{sec:s12-growth-criterion}, we use the zero set and local estimates of
Proposition~\ref{prop:r13-nine-dimensional-local-package}, with
$3\leq m\leq8$, $n=m+1$, $p=9/4$, energy measure $\mu=e_u\,dx$,
$a=|A|$, $H=\operatorname{tr}A$, $\Delta_\Sigma x=-H\nu$,
$a_m=(m-2)/2$, and $I=I_P$. Write
$K(r)=1+\sup_x r^{-m}\mu(B_r(x))$ for $r\geq1$.
Fix the penalty with $G_1\leq1.001$, the flat
positive-foot allowance $J_*\leq1.01$, and $\kappa=\sqrt2/G_1$.
The good set $\Sigma_g=\{D<\delta\}\cap\Sigma$ and separated bad net
$\mathcal X$ are those used in the moment theorem. In particular,
the unit good charts and fixed-ball positive energy lower bounds at zeros are uniform,
$a^2+|H|\leq CD$, phase energy decays exponentially away from
$\Sigma$, and $D\to0$ forces actual reach $d\to\infty$.

\vskip0.12in
Write
\begin{equation}\label{eq:r9-capped-reach}
 t=\min(d,\epsilon r_o),\qquad r_o=\sqrt{1+|z|^2},\qquad
 \widehat E=e^{-\kappa t},\qquad0<\epsilon\leq1/8.
\end{equation}
The capped reach is locally Lipschitz and semiconcave, with
nonpositive singular Laplacian. On its uncapped branch use the
nonplateau, small-$ad$ and plateau inequalities of
Proposition~\ref{prop:r7-actual-weak-reach}; on its cap branch use
$|\nabla t|\leq\epsilon$ and
$\Delta_\Sigma(\epsilon r_o)\leq m\epsilon/r_o+\epsilon|H|$.
We parametrize the selected positive feet with their prescribed
orientation and multiplicity one. Their inverse Jacobians are bounded, and
\begin{equation}\label{eq:r9-primitive-feet}
 |z-y|\leq d(z),\quad d(y)\leq d(z),\quad
 D(y)+a(y)+H_{N_y,+}\leq C_\Gamma/d(z).
\end{equation}
If $a(y)d(z)\leq a_1$, their inverse Jacobian is at most $J_*$.
The chord geometry used below is
\begin{equation}\label{eq:r9-chord-geometry}
 |\nu(z)\cdot(y-z)|\leq C_\Gamma|y-z|^2/d(z),\qquad
 |\lambda_i(z)|\leq C_\Gamma/d(z).
\end{equation}
These local inputs supply the estimates used in the moment and
radial arguments.

\vskip0.12in
Fix $\gamma,c,q_*>0$, $0<a_*<c$ and $0<\vartheta<1/2$ satisfying
the strict conditions in \eqref{eq:r13-joint-p-parameters}, with
$c_0=c(a_*/c)^\vartheta>\gamma(m-2)/4$. Fix
$\alpha,\beta_0>0$ with
$b_p=p\alpha+(m+1)\beta_0<\min\{2p-3,m-2\}$.
The local parameters, normal windows and radial parameters are fixed
throughout these two sections. Unless another dependence is indicated,
comparison constants may depend on these choices, but not on
$R,B,x_c$. The guard $\eta$ is fixed before $R$; independence of
small $\eta$ is stated in the estimates where it is needed.
We follow the moving-weight calculations in
\cite[Sections~10.1--10.4]{OpenAI2026}, keeping the outer cutoff in
place while absorbing the coefficients in each comparison.

\vskip0.12in

\subsubsection{The radial amplitude}
Suppose that $K(R)\to\infty$.  Monotonicity, covering and bounded
density give
\begin{equation}\label{eq:r9-envelope-facts}
 K(R)\leq C(1+R),\qquad K(s)\leq C(s/R)K(R)\ (s\geq R),
 \quad L:=R/\sqrt{K(R)}\geq c\sqrt R.
\end{equation}
Put $T=\eta R$, where $\eta>0$ will be fixed sufficiently small,
and define
\[
 \begin{gathered}
 \rho=\sqrt{R^2+|z-x_c|^2},\quad\tau=\log\rho,\quad
 h=\rho^{-a_m},\quad\varphi=h^2=\rho^{-(m-2)},\\
 T_\rho=1-|\nabla_\Sigma\rho|^2,\qquad
 \mathscr A=K(R)L^{1-\vartheta}.
 \end{gathered}
\]
Fix a smooth nondecreasing $f_*:\mathbb R\to(0,1]$, with
$0\leq(\log f_*)'\leq q_*$, $f_*=1$ on $[1,\infty)$ and
$f_*(s)=C_*e^{q_*s}$ on $(-\infty,-1]$.  Let
\begin{equation}\label{eq:r9-moving-family}
 \begin{gathered}
 b_L(t)=(t^{1-\vartheta}-L^{1-\vartheta})_+^{1/(1-\vartheta)},
 \quad s=c\tau-b_L(t),\quad W=t^{-\vartheta}f_*(s),\\
 F(\tau,t)=\int_t^\infty u^{-\vartheta}f_*(c\tau-b_L(u))\,du,
 \qquad\lambda=(\log f_*)'(s).
 \end{gathered}
\end{equation}
Every evaluation uses $t=\widehat d$.  Let
$Z_R=\{z\in\Sigma:D(z)\geq\rho(z)^{-\alpha}\}$ be the
\emph{global} rough set.  A fixed-scale smoothing of a distance
cutoff gives a smooth ambient $m_H$ such that
\begin{equation}\label{eq:r9-rough-mask}
 \begin{gathered}
 m_H=0\ \text{if }\operatorname{dist}(x,Z_R)<\rho^{\beta_0},
 \quad m_H=1\ \text{if }\operatorname{dist}(x,Z_R)>4\rho^{\beta_0},\\
 |\nabla m_H|+|D^2m_H|\leq C,
 \quad\operatorname{supp}(1-m_H)\subset
       \mathcal H_R:=\{\operatorname{dist}(x,Z_R)\leq5\rho^{\beta_0}\}.
 \end{gathered}
\end{equation}
Use $m_H=1$ if $Z_R$ is empty.  One obtains the displayed mask by
first using cutoff levels $3/2,7/2$ for the Lipschitz distance
divided by $\rho^{\beta_0}$ and then convolving at a fixed small
scale.  The gaps preserve the zero and one regions, and yield the
stronger bounds $|\nabla m_H|+|D^2m_H|\leq C\rho^{-\beta_0}$.
All singular zeros lie in its open zero collar.

\vskip0.12in

Let $\zeta_T=0$ on $B_T(0)$, $\zeta_T=1$ outside
$B_{2T}(0)$, with $|D^j\zeta_T|\leq C_jT^{-j}$.
For $B\geq4R$, take $m_B=1$ on $B_B(x_c)$ and zero outside
$B_{2B}(x_c)$, with $|D^jm_B|\leq C_jB^{-j}$, and use exactly
\begin{equation}\label{eq:r9-common-amplitude}
 \mathfrak{m}=\zeta_T^pm_Hm_B,\qquad\psi=\mathfrak{m}h\sqrt W,\qquad
 \mathscr L_B=\int I\psi^2.
\end{equation}
Define the common Dirichlet quantity by
$\mathscr D_B=\int|\nabla_\Sigma\psi|^2$.

\vskip0.12in

The symbol $\mathfrak{m}$ denotes the mask; $m$ continues to denote
the fixed interface dimension.
On its support $D<\rho^{-\alpha}\to0$ and the actual small-defect
reach criterion gives $t\geq t_{\min}(R)\to\infty$; the cap
itself is at least $\epsilon T$.  All supported integrals are
finite at finite $B$, and are supported on regular zeros.

\begin{lemma}
\label{lem:r9-radial-calculus}
With
\[
 J_1=\int_{b_L(t)}^\infty u^{-\vartheta-1}f_*'(c\tau-u)\,du,
\]
one has
\begin{align}\label{eq:r9-weight-chain}
 \Delta_\Sigma F\geq{}&F_\tau\Delta_\Sigma\tau
 -W(\Delta_\Sigma t)^{\rm upper}
 -c^2\vartheta J_1|\nabla\tau|^2\\
 &+\frac{\vartheta W}{t}|\nabla t|^2
                  +\frac{W\lambda}{b_L'}|\nabla s|^2.\nonumber
\end{align}
Here $J_1\leq CF_\tau/\tau$, and the last square is zero where
$b_L'=0$.  Uniformly for $\tau\geq\log R$ and supported $t$,
\begin{align}\label{eq:r9-weight-ratios}
 &F/W\geq(q_*+\vartheta/t)^{-1},\quad
 (F+F_\tau)/W\leq C(L+\tau),\quad
 F+F_\tau\leq C(L+\tau)^{1-\vartheta},\nonumber\\
 &F_\tau/W=(1+o_R(1))c
       [t/(c\tau\vee b_L(t))]^\vartheta,\quad
 \sup_{t>0}t^{2p-1}W\leq C(L+\tau)^{2p-1-\vartheta}.
\end{align}
The exact geometric identities are
\begin{align}\label{eq:r9-radial-identities}
 \Delta_\Sigma\tau&=(m-2+2T_\rho-HX\cdot\nu)/\rho^2,\nonumber\\
 -\Delta_\Sigma\varphi/\varphi
   &=(m-2)(mT_\rho-HX\cdot\nu)/\rho^2,\\
 -\Delta_\Sigma h/h
   &=[a_m^2+a_m(a_m+2)T_\rho-a_mHX\cdot\nu]/\rho^2,
 \qquad X=z-x_c.\nonumber
\end{align}
\end{lemma}
\begin{proof}
For $t>L$, $b_L'=(b_L/t)^\vartheta$ and the substitution
$u=b_L(v)$ gives $v^{-\vartheta}dv=u^{-\vartheta}du$.
Consequently
\[
 F_t=-W,\quad F_{tt}=W(\vartheta/t+\lambda b_L'),\quad
 F_{\tau t}=-cW\lambda,\quad
 F_{\tau\tau}=c^2b_L^{-\vartheta}f_*'(c\tau-b_L)
                                      -c^2\vartheta J_1.
\]
Since $b_L^{-\vartheta}f_*'=W\lambda/b_L'$, the quadratic
form is the completed square in \eqref{eq:r9-weight-chain}.
Near $t=L$ the argument $c\tau-b_L(t)$ exceeds one, so all
terms involving $\lambda/b_L'$ vanish on a neighborhood.  The
nonpositive singular part of $\Delta t$, multiplied by $-W$,
is nonnegative.  Local semiconcave smoothing therefore proves
the distributional form without differentiating a branch selection.

\vskip0.12in

The measure $f_*'(c\tau-u)du$ has total mass
$f_*(c\tau-b_L(t))$, is supported above $c\tau-1$, and after
normalization is exponentially concentrated at
$c\tau\vee b_L(t)$ with bounded mean displacement.  Compare
$u^{-\vartheta}$ and $u^{-\vartheta-1}$ there to obtain the
uniform $F_\tau$ asymptotic and the $J_1$ bound.  The logarithmic
decrease of $W$ in $t$ is at most $q_*+\vartheta/t$, proving the
lower ratio.  Split its remaining integrals at $C(L+\tau)$ and
use the exponential tail to obtain the upper bounds and the
$(2p-1)$st moment.  Finally trace the radial ambient Hessians on an
$m$-plane and use $\Delta_\Sigma x=-H\nu$.
\end{proof}

\vskip0.12in

For later comparison the exact ground-state identity, valid for the
compact amplitude above, is
\begin{equation}\label{eq:s1-ground-state}
 \int|\nabla_\Sigma\psi|^2
 =\int\varphi|\nabla_\Sigma(\mathfrak m\sqrt W)|^2
   +\int\varphi\mathfrak m^2W(-\Delta_\Sigma h/h).
\end{equation}
The masks vanish near singular zeros, so regular surface integration
by parts proves the identity by expansion. The cutoff power $\zeta_T^p$ is $C^2$.

\vskip0.12in

Define the radial form
\begin{align}\label{eq:s7-raw-radial-form}
 \mathscr P_B=\int\varphi\mathfrak m^2\bigg[
 &\frac{(m-2)F_\tau+\{m(m-2)F+2F_\tau\}T_\rho}{\rho^2}
 +\frac{3\vartheta W}{4t}|\nabla_\Sigma t|^2\notag\\
 &+\frac{W\lambda}{b_L'}|\nabla_\Sigma s|^2
       -c^2\vartheta J_1|\nabla_\Sigma\tau|^2\bigg].
\end{align}
The term with $\lambda/b_L'$ is zero where $b_L'=0$.
This form is finite at finite support; its sign is determined by
the comparison below.  Put $M=m-2$ and define
\begin{equation}\label{eq:s3-radial-charges}
 \mathscr V_B=\int\frac{\varphi\mathfrak m^2F T_\rho}{\rho^2},
 \qquad
 \mathscr C_B=\int\frac{\varphi\mathfrak m^2W}{\rho^2}
                         \left(1+\frac{2T_\rho}{M}\right).
\end{equation}

\begin{lemma}\label{lem:s7-raw-ground-state}
Suppose $m>2$, $\gamma,\vartheta>0$, $\gamma q_*/4<1$,
$t_{\min}(R)\to\infty$. Let $h_0>0$ and let $Q\geq0$ be locally integrable on the supported
regular bases, with $|H|\leq h_0Q$, and put $\mathscr M_B=\int Q\psi^2$.  Suppose also
$c_0=c(a_*/c)^\vartheta>\gamma M/4$.  Set
\[
 d_H=\frac M2(c_0-\gamma M/4)>0,\quad
 c_*=\frac{mM}{2}(1-\gamma q_*/4)>0,\quad
 C_\gamma=\frac{2\gamma}{1-\gamma q_*/4}.
\]
For each fixed small $\eta$, all sufficiently large $R$ and finite
$B$, one has
\begin{align}\label{eq:s7-raw-ground-lower}
 \mathscr P_B-\gamma\mathscr D_B\geq{}&
 c_*\mathscr V_B+d_H\mathscr C_B
 -\frac{\gamma a_mh_0}{R}\mathscr M_B
 -C_\gamma\int\varphi W|\nabla_\Sigma\mathfrak m|^2\notag\\
 &-C_H\int_{t<a_*\tau}
                     \frac{\varphi\mathfrak m^2W}{\rho^2},
\end{align}
where $C_H=(d_H+\gamma M^2/4)(1+2/M)$.
All displayed constants are independent of small $\eta$;
the threshold for $R$ may depend on it.  The pointwise mask bound \eqref{eq:s7-mask-schur} below follows
under the other hypotheses alone; the condition on $c_0$ enters only
when estimating the radial coefficient.
\end{lemma}
\begin{proof}
Write $E_\gamma=F_\tau-\gamma MW/4$ and
\[
 Z=\frac{\lambda}{b_L'}|\nabla_\Sigma s|^2+
                   \frac{3\vartheta}{4t}|\nabla_\Sigma t|^2.
\]
The ground-state identity \eqref{eq:s1-ground-state} gives exactly
\begin{align}\label{eq:s7-raw-ground-identity}
 \mathscr P_B-\gamma\mathscr D_B={}&
 \int\frac{\varphi\mathfrak m^2}{\rho^2}
 \big[ME_\gamma+\{2E_\gamma+mM(F-\gamma W/4)\}T_\rho
                +\gamma a_mH(X\cdot\nu)W\big]\notag\\
 &-c^2\vartheta\int\varphi\mathfrak m^2
                  J_1|\nabla_\Sigma\tau|^2
 +\int\varphi\{\mathfrak m^2WZ-
                   \gamma|\nabla_\Sigma(\mathfrak m\sqrt W)|^2\}.
\end{align}
Thus $E_\gamma$ controls both the constant and trace coefficients.

\vskip0.12in

Weighted Cauchy--Schwarz yields
\[
 \gamma|\nabla_\Sigma\log\sqrt W|^2
 \leq\left(\frac{\gamma\lambda b_L'}4+
                   \frac{\gamma\vartheta}{3t}\right)Z.
\]
For large $R$ the factor on the right is at most
$\mu_0=(1+\gamma q_*/4)/2<1$, because $\lambda b_L'\leq q_*$.
Apply $|g+v|^2\leq |g|^2/(1-\mu_0)+|v|^2/\mu_0$ to
$g=\nabla_\Sigma\mathfrak m$ and
$v=\mathfrak m\nabla_\Sigma\log\sqrt W$ to obtain
\begin{equation}\label{eq:s7-mask-schur}
 \mathfrak m^2WZ-
 \gamma|\nabla_\Sigma(\mathfrak m\sqrt W)|^2
 \geq-C_\gamma W|\nabla_\Sigma\mathfrak m|^2.
\end{equation}
This proves the pointwise bound for arbitrary $c_0$.

The bounds $J_1\leq CF_\tau/\tau$ and
$|\nabla_\Sigma\tau|^2\leq\rho^{-2}$ replace the constant
radial coefficient by
$(M-O(1/\log R))F_\tau-\gamma M^2W/4$.
On $t\geq a_*\tau$, use
$F_\tau/W\geq c_0-o_R(1)$ to bound it below by $d_HW$,
and $2E_\gamma\geq2d_HW/M$ for large $R$.
On the complementary set, the remaining $F_\tau$ terms are
nonnegative; omitting them introduces an error at most
$(d_H+\gamma M^2/4)W(1+2T_\rho/M)$ relative to this lower
bound. This is bounded by the last term of
\eqref{eq:s7-raw-ground-lower} since $0\leq T_\rho\leq1$.
The global bound $W\leq(q_*+o_R(1))F$ supplies
$c_*\mathscr V_B$.  Finally the mean term has absolute value at
most $\gamma a_mh_0\mathscr M_B/R$, since
$|X\cdot\nu|/\rho^2\leq1/R$ and
$\psi^2=\varphi\mathfrak m^2W$.  These comparisons and
\eqref{eq:s7-mask-schur} prove the result at finite support.
\end{proof}

\vskip0.12in

\subsection{Localization and the physical test}
\label{sec:r11-interpolated-stability}

We now construct the ambient test that yields the radial stability
inequality. The same amplitude \eqref{eq:r9-common-amplitude} is used
in every estimate. In addition to the preceding local estimates, the construction uses the clean
adjacent atlas and moving-band geometry of
Lemmas~\ref{lem:r11-clean-window-geometry} and
\ref{lem:r11-moving-band-geometry}. The sharp estimates of
Corollary~\ref{cor:r13-master-consequences} and
Proposition~\ref{prop:r11-sharp-tail4} give, for logarithmic
first-neighbor gaps and their balanced bands,
\begin{equation}\label{eq:r11-switch-sharp-inputs}
 d=G_1\ell_i,\qquad I_P\geq47.999e^{-\sqrt2\ell_i},\qquad
 q(x)\leq4.001(e^{-\sqrt2s_i}+e^{-\sqrt2s_j}),\quad q=1-u^2.
\end{equation}
Here $I_P=\sigma^{-1}\int_{-P}^Pq^3D\,ds$ and
$\sigma=\int_{\mathbb R}G^2=4\sqrt2/3$.
Use the common normal window $P$ fixed before the moment argument,
chosen large enough for these estimates and the remaining-mass theorem.
Next fix a large depth constant $M_{\rm dep}$, then
$C_0>4M_{\rm dep}$. Let $\mu_\pm=e^{-\sqrt2\ell_\pm}$
for the first gap on each side when $\ell_\pm\leq C_0\log\rho$
in the buffered clean polynomial window, and zero on an absent or
unselected side. Define
\begin{equation}\label{eq:r12-common-bonus}
 b_R=\min(\mu_+,\mu_-),\qquad
 \mathscr N_B=\int b_R\psi^2.
\end{equation}
Adjacent graphs agree on overlaps, so the selection is measurable.
We use $b_R$ as a measurable weight in the remaining-mass estimate, which gives
\begin{equation}\label{eq:r12-global-positive-improved-density}
 |H|\leq B_R:=.355I_P-24\sqrt2b_R,\qquad B_R\geq0.
\end{equation}
To combine the switching estimate with the reach inequality, define the positive density
\begin{equation}\label{eq:s8-main-positive-mass}
 k_{\rm sw}=25/.998,\quad J_R=I_P-k_{\rm sw}b_R,\quad
 \mathscr J_B=\int J_R\psi^2=\mathscr L_B-k_{\rm sw}\mathscr N_B.
\end{equation}
Since $b_R\leq.355I_P/(24\sqrt2)$ and
$.355k_{\rm sw}<24\sqrt2$,
\begin{equation}\label{eq:s8-main-positive-density}
 .737I_P\leq J_R\leq I_P,\qquad |H|\leq B_R\leq.355J_R.
\end{equation}
These bounds hold also at clean receiving feet: a ball of radius
$\rho^{\omega_{\rm cl}}$, with
$0<\omega_{\rm cl}<\min(\alpha,\beta_0)/2$, avoids $Z_R$ and
has $D\leq C\rho^{-\alpha}<\rho^{-\omega_{\rm cl}}$.
Thus the same selected window applies there. This density will be
used directly in the selected positive-foot transfer; the two terms in its definition are kept separate when estimating
the ambient test.

\vskip0.12in

\subsubsection{Cap and physical-cutoff estimates}

We first prove the cap and physical-cutoff estimates independently of
the compact interpolated test. Throughout, fixed normal depth gives $0\leq I_P\leq2P/\sigma$.

\begin{lemma}
\label{lem:r11-switch-cap-payment}
On $r_o\geq T/2$ let $\mathcal C=\{d\geq\epsilon r_o/4\}$.
For each fixed polynomial power $N$,
\begin{equation}\label{eq:r11-cap-master}
 \int_{\mathcal C\cap\Sigma_g}\varphi W(1+r_o+L+\tau)^N
       \leq C_{N,\eta}e^{-a_\eta R}.
\end{equation}
Consequently the cap part of $\mathscr L_B$, and the cap part of
$\int|\nabla\psi|^2$, have this bound, uniformly in $B,x_c$.
Moreover every switching band with at least one foot in $\mathcal C$
has an exponentially small switching contribution, including the term
with the other foot's amplitude. The mass estimates below use the
boundedness of $I_P$ itself.
\end{lemma}
\begin{proof}
On $\mathcal C$, $\epsilon r_o/4\leq t\leq\epsilon r_o$ and
$t/L\geq c\epsilon\eta\sqrt{K(R)}\to\infty$. Thus
$b_L(t)\geq t/2$ for all sufficiently large $R$.
If $c\tau\leq t/4$, the exponential tail of $f_*$ gives
$W\leq Ce^{-q_*t/4}$; all polynomial factors are absorbed by this
exponential, since $L\leq C_\eta r_o$ and $\tau\leq C t$.
Otherwise $\tau\geq t/(4c)\geq c_\epsilon r_o$, and
$\varphi=e^{-(m-2)\tau}$ absorbs every fixed polynomial in
$r_o,L,\tau$, leaving $Ce^{-a r_o}$. Uniform unit charts give
$\mathcal H^m(\Sigma_g\cap B_s(0))\leq C(1+s)^{m+1}$;
physical annulus summation proves \eqref{eq:r11-cap-master}.
This bound is uniform in arbitrary translating observers.

\vskip0.12in

On a noncap branch with $t\leq\epsilon r_o$ the coarse reach-gradient
bound gives $|\nabla t|\leq C(1+r_o)^2$; on the smooth cap branch
$|\nabla t|\leq\epsilon$. Since
$|\nabla\log W|\leq C(|\nabla t|+1)$ on the region in question,
and all mask derivatives are bounded, $|\nabla\psi|^2$ is bounded by
$C\varphi W$ times a fixed polynomial. This proves the gradient claim;
the mass claim uses only the boundedness of $I_P$.

\vskip0.12in

For the last claim fix a band between two feet $z_i,z_j$, and write
$\ell=\ell_{ij}$ for the gap measured from a fixed base on its normal
ray. If $z_j\in\mathcal C$, the chord competitor in
\eqref{eq:r11-moving-foot-reach} gives
$\ell+o(1)\geq d(z_j)/G_1\geq c_\epsilon r_o(z_j)$.
Also $r_o(z_i)\leq r_o(z_j)+|z_i-z_j|\leq r_o(z_j)+\ell+o(1)$.
Thus $\ell\geq c_\epsilon\max(r_o(z_i),r_o(z_j))$.
In the fixed-width band the sharp phase bound, or its coarse phase
consequence, gives $q^2\leq Ce^{-\sqrt2\ell}$. Bound the amplitude
difference by $2\psi(z_i)^2+2\psi(z_j)^2$, and parametrize each term by
its own normal rays. Each nonzero base has $|z|\geq T$, at most two
bands, bounded band width, and normal Jacobian $1+o(1)$. Its contribution is
therefore at most $C\varphi(z)W(z)e^{-a r_o(z)}$.
Physical annulus summation proves the claimed uniform exponential
bound. This estimates both the uncapped and the capped foot.
\end{proof}

\begin{lemma}
\label{lem:r11-switch-physical-difference}
In the clean, noncap bands, replacing two physical factors
$\zeta_T(z_i)^p,\zeta_T(z_j)^p$ by a common assigned factor introduces an error
at most $C\eta^{m-2}K(R)$. The estimate is uniform in $B,x_c$.
\end{lemma}
\begin{proof}
The squared difference is at most
$CT^{-2}|z_i-z_j|^2\leq CT^{-2}(\ell_{ij}+o(1))^2$.
By the band phase bound and normal parametrization, the error
assigned to one fixed base is bounded by
\[
 CT^{-2}\varphi(z)\ell_{ij}^2e^{-\sqrt2\ell_{ij}}
 \mathbf1_{\{\zeta_T(z_i)\ne\zeta_T(z_j)\}}.
\]
Here $W\leq1$ on noncap clean bases for large $R$, and at most two
bands contribute per base. If $|z_i|>4T$ and the cutoff values differ,
then $|z_j|<2T$ and $\ell_{ij}+o(1)\geq|z_i|/2$.
Consequently the integral is at most
\begin{equation}\label{eq:r11-switch-physical-located}
 CT^{-2}R^{-(m-2)}\left[
  \mathcal H^m(\Sigma_g\cap B_{4T}(0))+
  \int_{|z|>4T}e^{-a|z|}\,dS_z\right].
\end{equation}
Uniform good charts and the energy lower bound in fixed-radius balls centered at zeros give
area at most $CK(8T)T^m$ in the first term. Centered monotonicity
below $8T$ and ambient covering above it give the physical envelope
$CK(8T)(1+s)^m[1+(1+s)/(8T)]$. It bounds the second term by
$CK(8T)T^me^{-a'T}$. Since $8T\leq R$, the result is at most
$CR^{-(m-2)}T^{m-2}K(R)=C\eta^{m-2}K(R)$.
The cutoff difference is therefore controlled by the physical-core term
in \eqref{eq:r11-switch-physical-located} and its exponentially decaying tail.
\end{proof}

\vskip0.12in

\subsubsection{Radial localization from the moment estimate}

\begin{lemma}
\label{lem:s6-radial-localization}
Use the fixed local data and moving family of
Section~\ref{sec:r9-radial-core}, the moment estimate
\eqref{eq:s8-density-envelope-moment}, and the cap and physical-cutoff
estimates above.  Put $p=9/4$, $T=\eta R$,
$L=R/\sqrt{K(R)}$, and $\mathscr A=K(R)L^{1-\vartheta}$ along a
putative sequence with $K(R)\to\infty$.  In the radial formulas
use the profile slope $q_*$.  Suppose
\[
 \begin{gathered}
 q_*<p\kappa,\qquad q_*(2/\kappa-c)>1,\\
 p\kappa a_*<\min\{2p-1,m\},\qquad
 b_p=p\alpha+(m+1)\beta_0<\min\{2p-3,m-2\}.
 \end{gathered}
\]
with $0<a_*<c$ and the other fixed moving-weight parameters as in
Section~\ref{sec:r9-radial-core}.
Then the estimates for the rough set, physical cutoff, cap branch,
large-$at$ terms, small-reach deficits, selected positive feet and outer cutoff
hold for the same amplitude and masks.  In particular they give
\eqref{eq:r13-p-rough-count}, \eqref{eq:r13-p-large-at-payment},
\eqref{eq:r13-p-outer-mass}, and \eqref{eq:r13-p-physical-foot}.
After reserving the mean-curvature terms proportional to the interaction
mass for absorption at finite support, the remaining localization errors
are bounded by
\begin{equation}\label{eq:s6-radial-localization-ledger}
 C\eta^{m-2}\mathscr A
 +C\eta^{m-3}\mathscr A/R
 +C\eta^{m-2p}\mathscr A/K(R)^{p-1}
 +o_{R;\eta}(\mathscr A)+e_{R,\eta,x_c}(B).
\end{equation}
The leading $C$ is independent of small fixed $\eta$;
the little-oh is uniform in $x_c,B$; and $e(B)\to0$ at fixed
$R,\eta,x_c$. All interaction masses in these estimates carry the
finite outer cutoff.
\end{lemma}
\begin{proof}
\emph{Shell estimates and rough masks.}
For $s_j=2^jR$, apply \eqref{eq:s8-density-envelope-moment} with auxiliary
physical radius $T/32$ and observer scale $8s_j$.  On the support of the physical cutoff this gives the unweighted shell moment and bad-net bound
\[
 V_{p,j}=C_p\{K(s_j)s_j^{m-2p}
       +K(R)T^{m-2p}\chi_{8s_j}(0)^2\}.
\]
A separated rough center with $D\geq\rho^{-\alpha}$ either
belongs to a bounded-multiplicity enlargement of the bad net or
has a fixed good disk contributing $cs_j^{-p\alpha}$ to
$\int D^p$.  Its count is at most $Cs_j^{p\alpha}V_{p,j}$.
Enlargement to radius $Cs_j^{\beta_0}$ therefore has area and
ambient volume at most
\begin{equation}\label{eq:r13-p-rough-count}
 Cs_j^{b_p}V_{p,j}
 +Cs_j^{(m+1)\beta_0}
                 \mathbf1_{s_j\geq cT^{1/\beta_0}},
 \qquad b_p=p\alpha+(m+1)\beta_0.
\end{equation}
Choose $\alpha,\beta_0>0$ so that
$b_p<\min\{2p-3,m-2\}$.  Multiplying this count by
$s_j^{-(m-2)}(L+\log s_j)^{1-\vartheta}(1+\log s_j)^N$
and summing gives $o_{R;\eta}(\mathscr A)$: the main shell
ratio is $2^{3-2p+b_p}<1$, the physical ratio is
$2^{-(m-2)+b_p}<1$, and the initial normalized factor is
$O_\eta(R^{2-2p+b_p}(1+\log R)^N)$.  The origin-core halo
starts at $T^{1/\beta_0}$ and has the negative exponent
$-(m-2)/\beta_0+m+1$.  This bounds the rough-mask derivatives, their mixed terms, and
the rough receiving collars.  The coarse reach
gradient is subpolynomial in $t$. Together with $W\leq t^{-\vartheta}$
and the fixed logarithmic slope, this gives
\[
 W|\nabla\log W|^2+W|\nabla t|^2/t+W|\nabla t|\leq C,
 \qquad |\nabla F|\leq C\{1+(L+\tau)^{1-\vartheta}/\rho\}.
\]
Also $|\Delta_\Sigma m_H|\leq C$ on good zeros. Hence first,
second and mixed mask derivatives are bounded by the displayed
weighted count. Each moment estimate is applied outside its own
physical cutoff.

\emph{Physical derivatives and cap.}
On $A_{T,2T}(0)$ apply the moment input with observer $0$,
observer radius $8T$ and auxiliary physical core $T/32$. The
observer weight is bounded below there, so it gives $\int D^p\leq CK(R)T^{m-2p}$ and area
$CK(R)T^m$.  H\"older implies
$\int |H|\leq C\int D\leq CK(R)T^{m-2}$.
The exact surface Laplacian of $\zeta_T^{2p}$ is bounded by
$C_p(T^{-2}+|H|/T)$ on this annulus.  With
$\varphi=\rho^{-(m-2)}$ and the uniform suprema
\[
 \sup\varphi(F+F_\tau)\leq CR^{-(m-2)}L^{1-\vartheta},
 \qquad
 \sup\rho^{-1}\varphi(F+F_\tau)
              \leq CR^{-(m-1)}L^{1-\vartheta},
\]
the total physical-Hessian and mask-gradient error is
$C_p\eta^{m-2}\mathscr A+C_p\eta^{m-3}\mathscr A/R$.
The constants are independent of small $\eta$.
For $t_c=\epsilon(1+|z|^2)^{1/2}$, either
$c\tau\leq t_c/4$, when $W,F,F_\tau$ are exponentially small
in physical radius, or $\tau\geq t_c/(4c)$, when
$\varphi$ has that decay.  Summation using the unconditional
area $O(r^{m+1})$ bounds the cap contributions to the radial, mass,
gradient and mixed terms uniformly in the translating observer. The switching bands use
the two-foot cap estimate of
Lemma~\ref{lem:r11-switch-cap-payment}.

\vskip0.12in

\emph{Large-$at$ curvature and small-reach deficits.}
At a remaining base point with $at\geq a_0$,
\[
 a^2t+|\nabla t|^2/t
 \leq C_{a_0}a^{2p}t^{2p-1}
               e^{C\sqrt{\log(K(t)+1)}}.
\]
The same bound applies to a positive foot with large $a(y)t$, using
$H_{N_y,+}\leq C/t\leq C_{a_1}a(y)^{2p}t^{2p-1}$.
The profile has
$\sup_{t>0}t^{2p-1}W\leq C(L+\tau)^{2p-1-\vartheta}$,
by \eqref{eq:r9-weight-ratios}.  Their total normalized shell contribution is at most
\begin{equation}\label{eq:r13-p-large-at-payment}
 C_\eta\frac{e^{C\sqrt{\log(K(R)+1)}}}{K(R)^{p-1}}
 \sum_{j\geq0}2^{(3-2p)j}e^{C\sqrt{j+1}}(1+j)^C=o_R(1).
\end{equation}
Its physical contribution has the additional fixed factor
$\eta^{m-2p}$ and the faster geometric ratio
$2^{-(m-2)j}$.  On $t<a_*\log\rho$ one has
$\widehat E^p\geq\rho^{-p\kappa a_*}$, so the constant and
trace deficits in the radial comparison are bounded by
\[
 C\sum_j s_j^{-m+p\kappa a_*}V_{p,j}
       =o_{R;\eta}(\mathscr A).
\]
Here the main shell ratio is $2^{1-2p+p\kappa a_*}<1$;
the fixed physical term has ratio $2^{-m+p\kappa a_*}<1$.
Both strict bounds are explicitly assumed.  Their initial normalized factor is
$C_\eta R^{-2p+p\kappa a_*}/L^{1-\vartheta}$.
All finitely many polynomial reach factors on
$t>M_0(L+\tau)$ are bounded by increasing the fixed $M_0$:
the exponential profile tail absorbs them and any prescribed
power of $\rho$, before unconditional area is summed.

\vskip0.12in

\emph{Finite outer exhaustion.}
Choose fixed $\delta>0$ so that
$r=q_*(1-\delta)/\kappa<p$ and
$\varepsilon'=q_*[2(1-\delta)/\kappa-c]-1>0$.
At fixed $R,T,L$,
$W\leq C_Rs^{c q_*}\widehat E^r$ on a distant shell.
For this distant shell use the moment input with the auxiliary
physical radius $s/16$ and observer radius $8s$.  Once $s$ exceeds
a fixed multiple of $|x_c|+T$, its physical cutoff is identically
one on the shell.  Both terms of that moment input are then bounded
by $CK(8s)s^{m-2p}\leq Cs^{m-2p+1}$.  This avoids treating the remainder from the fixed inner core as
decaying when $m-2p+1<0$.
Thus $\int\widehat E^p+\int D^p\leq C_Rs^{m-2p+1}$,
and the unconditional area bound is $O(s^{m+1})$.
H\"older's inequality then gives
\begin{equation}\label{eq:r13-p-outer-mass}
 \int_{\rho\asymp s}W\leq C_Rs^{m-\varepsilon'}.
\end{equation}
Pure outer terms are $C_Rs^{-m}(1+\log s)^N\int W=o(1)$.
Since $F\leq C_R(1+\log s)W$ and $W\leq C_R$,
the outer mean term is at most
\[
 C_R(1+\log s)s^{-(m-1)}
 (\int D^p)^{1/p}(\int W^{p/(p-1)})^{(p-1)/p}
 \leq C_R(1+\log s)
 s^{-(p-1)(1+\varepsilon')/p}=o(1).
\]
A mismatch between the outer cutoffs at the endpoints of a selected
positive-foot pair adds only a fixed power of $L+\log s$, hence has the
same limit.  An outer amplitude difference, bounded by $C\log B/B$ on
$\rho\asymp B$, consequently contributes at most
\begin{equation}\label{eq:r11-switch-outer-mass}
 C_{R,\eta,x_c}(1+\log B)^N B^{-\varepsilon'}\longrightarrow0.
\end{equation}
These bounds allow outer exhaustion after the finite-support
absorptions; $\int I_P\psi^2$ carries the cutoff throughout those steps.

\vskip0.12in

\emph{Positive-foot weight comparison and cutoff errors.}
A selected clean positive foot has $t(y)\leq t(z)$ and lies at
$o_R(\rho)$ distance from its base point $z$ after removing
$t>M_0(L+\tau)$. Therefore monotonicity of $W$ in $t$ and
$|\partial_\tau\log W|\leq c q_*$ give
\[
 \varphi(z)W(\tau(z),t(z))
 \leq(1+o_R(1))\varphi(y)W(\tau(y),t(y)).
\]
Rough receiving feet and feet with large $a(y)t$ have already been
estimated at the receiving points. Transfer the physical mask by
\[
 \zeta_T(z)^{2p}\leq(1+v)\zeta_T(y)^{2p}
       +C_{p,v}\min\{1,(|z-y|/T)^{2p}\}.
\]
If $\zeta_T(z)\ne\zeta_T(y)$, then $z\in B_{3T}(0)$ and $y\in B_{4T}(0)$, because
$|z-y|\leq t(z)\leq\epsilon r_o(z)$. The signed positive-foot area
formula, $H_{N_y,+}\leq C/t(z)$, and
\eqref{eq:r9-weight-ratios} give the located remainder
\begin{align}\label{eq:r13-p-physical-foot}
 CT^{-2p}K(R)T^mR^{-(m-2)}L^{2p-1-\vartheta}
   =C\eta^{m-2p}\mathscr A/K(R)^{p-1}.
\end{align}
The bound is uniform over observers since a negative power of
$\rho/R$ absorbs each fixed power of $L+\log\rho$. Outer mask
mismatches were estimated above. Apply this trace estimate to the cutoff remainder. In the principal
radial comparison, leave the negative neighbor term in its signed form.

\end{proof}

For fixed $\eta$, the second and third terms of
\eqref{eq:s6-radial-localization-ledger} are
$o_{R;\eta}(\mathscr A)$ along the putative unbounded-density
sequence. The leading constant is independent of small $\eta$.
In the proof below we first choose $\eta$ from the area lower bound,
then let $B\to\infty$ at fixed $R,\eta,x_c$, and finally let
$R\to\infty$.

\vskip0.12in

\subsubsection{Construction of the ambient test}

\begin{proposition}
\label{prop:r13-physical-construction}
For sufficiently large $R$ and subsequent finite $B$, one admissible
compact exterior test gives
\begin{equation}\label{eq:r13-physical-construction-output}
 .998\mathscr L_B\leq1.005\int_\Sigma|\nabla_\Sigma\psi|^2
       +25\mathscr N_B+C\eta^{m-2}K(R)
       +o_{R;\eta}(\mathscr A)+e_{R,\eta,x_c}(B).
\end{equation}
Here $e(B)\to0$ at fixed $R,\eta,x_c$, and the little-oh is uniform
in the observer and finite support. The normalization is by the
normal mass $\sigma=4\sqrt2/3$.
\end{proposition}

\begin{proof}

Use $M=M_{\rm dep}$, chosen sufficiently large below. In the region
\begin{equation}\label{eq:r11-ambient-clean-domain}
 \operatorname{dist}(x,\Sigma)<M\log\rho(x)+1,
 \qquad\operatorname{dist}(x,Z_R)>\tfrac12\rho(x)^{\beta_0},
\end{equation}
use buffered graph families containing every candidate within
$2M\log\rho+10$. The clean-window corollary gives them, because
every participating foot has distance at least
$\rho(z)^{\beta_0}/3$ from $Z_R$ for large $R$.
Let $s_i,s_j$ be the two smallest distances from $x$ to the local graphs,
and suppose $|s_i-s_j|\leq6$. The corresponding graphs are adjacent and
lie on opposite sides of $x$; no third graph has distance within six
of the minimum. Define
\[
 \mathcal U(t)=\frac{\int_t^3\operatorname{sech}^2(\sqrt2u)\,du}
                   {\int_{-3}^3\operatorname{sech}^2(\sqrt2u)\,du},
 \qquad -3\leq t\leq3,
\]
extended constantly by one and zero at the two ends. Set
\[
 \begin{aligned}
 \widetilde\eta(x)&=\psi(z_i(x))&&\text{off the bands},\\
 \widetilde\eta(x)&=\mathcal U(t)\psi(z_i(x))+
       [1-\mathcal U(t)]\psi(z_j(x))&&\text{on a band},\\
 t&=(s_i-s_j)/2.
 \end{aligned}
\] The symmetry $\mathcal U(-t)=1-\mathcal U(t)$ makes
this independent of labeling. Endpoint values match the nearest-sheet
formula, and the projections agree on overlaps. Near a zero only
its smooth projection is used. Thus this is one locally Lipschitz
function; there is no differentiation of a global layer enumeration.

\vskip0.12in

Take a Lipschitz depth cutoff $\zeta_{\rm dep}$ equal to one for
$\operatorname{dist}(x,\Sigma)\leq M\log\rho$ and zero above
$M\log\rho+1$, with bounded derivative. Take an additional physical
cutoff $\chi_T$, zero on $B_{T/4}(0)$, one outside $B_{T/2}(0)$,
with $|\nabla\chi_T|\leq C/T$. Define
\[
 \eta=\chi_T\zeta_{\rm dep}\widetilde\eta
\]
in \eqref{eq:r11-ambient-clean-domain} and extend it by zero.
At the auxiliary rough boundary, participating feet are at distance
less than $\rho(z)^{\beta_0}$ from $Z_R$, so all their $m_H$
factors vanish on a full neighborhood. Thus extension by zero is
locally Lipschitz. The factor $1/2$ in
\eqref{eq:r11-ambient-clean-domain} ensures this neighborhood.

\vskip0.12in

If $\eta(x)\ne0$, a participating base has
$|z-x_c|\leq2B$ and $|x-z|\leq C(1+\log\rho(x))$.
Therefore $|x-x_c|\leq2B+C[1+\log(R+|x-x_c|)]$, which bounds
its support for fixed $R,B,x_c$. A finite cover by projection charts
gives a finite Lipschitz constant. Also $T/4>R_0$ for large $R$.
Thus $q\eta$ is an admissible compactly supported test function for
exterior stability, and the exact
identity $(-\Delta+3u^2-1)q=-q^2D$ yields
\begin{equation}\label{eq:r11-exact-ambient-q-test}
 \int q^3D\eta^2\leq\int q^2|\nabla\eta|^2.
\end{equation}

\vskip0.12in

We first estimate the two auxiliary cutoff errors. On the depth-cutoff support,
the unconditional phase bound gives $q^2\leq C\rho^{-aM}$, while
$\widetilde\eta^2\leq C\rho^{-(m-2)}$. Ambient shell volume is
$O(\rho^{m+1})$, so the resulting error is at most
$C\sum_j s_j^{3-aM}=o_R(\mathscr A)$ after fixing $M$ large.
On $\operatorname{supp}\nabla\chi_T$, any nonzero participating
amplitude has a foot outside $B_T(0)$, at distance at least $T/2$
from the ambient point. In a band the nearest distance differs from
that distance by at most six. Hence
$\operatorname{dist}(x,\Sigma)\geq T/2-6$ wherever this cutoff term is nonzero.
It is bounded by
$CR^{-(m-2)}T^{m-1}e^{-aT}$, uniformly over observers.
All fixed Young losses for these cutoffs can be prescribed arbitrarily
small before $R$ increases; their reciprocal constants multiply only
the negligible terms just estimated.

\vskip0.12in

\subsubsection{Mass, gradients, and the normal-band error}

For all supported bases the fixed tubes $|s|\leq P$ are jointly
injective, since their actual reaches tend uniformly to infinity.
They have Jacobian $1+o_R(1)$. There $\chi_T=\zeta_{\rm dep}=1$,
the base is nearest, and $\eta(z+s\nu)=\psi(z)$ exactly.
Integrating over these tubes and using positivity elsewhere gives
\begin{equation}\label{eq:r11-ambient-full-mass}
 \int q^3D\eta^2\geq(1-o_R(1))\sigma\mathscr L_B.
\end{equation}
The fixed-depth integration includes capped bases with their full $I_P$.

\vskip0.12in

In a nearest-sheet region the projected gradient is
$(\operatorname{Id}+sA)^{-1}\nabla_\Sigma\psi$. Metric and volume
factors are $1+o_R(1)$ through the logarithmic window. Along each
nearest ray, $\operatorname{dist}(x,\Sigma)=|s|$, and
\[
 \int_{\rm nearest}q^2\,ds\leq\sigma+o_R(1).
\]
For this uniform limit first choose a large fixed normal depth to
make the phase tail small, then use profile convergence on that fixed
interval as $R$ increases.
In a band the tangential part is a convex combination of the two
projected gradients. Its square is bounded by the corresponding
combination of squares; the total $q^2$ mass of a fixed-width band
tends uniformly to zero since sheet separation is at least
$c\log\rho$. Each base has at most two bands. After a fixed Young
inequality separating its normal switch derivative, these tangential
band terms are $o_R(1)\int|\nabla\psi|^2$ and are absorbed into
the eventual coefficient $1.005$. They are not called absolute
$o_R(\mathscr A)$ errors.

\vskip0.12in

The remaining term is
\begin{equation}\label{eq:r11-switch-integral}
 S_{\rm sw}=\int_{\rm bands}q^2\mathcal U'(t)^2
                  |\psi(z_i)-\psi(z_j)|^2\,dx,
\end{equation}
using $|\nabla t|\leq1$. For bands touching the region where the rough mask is nonconstant,
the rough-set estimate applies: every foot with nonzero amplitude
$\psi$ lies in a fixed logarithmic enlargement of the rough collar,
and normal parametrization
has bounded width and at most two traversals per base. Their contribution
is at most $C\int_{\rm rough}\varphi=o_R(\mathscr A)$.
The switching contribution of bands touching $\mathcal C$ is exponentially small by
Lemma~\ref{lem:r11-switch-cap-payment}.

\vskip0.12in

On the remaining bands both $m_H$ equal one and both reaches are
uncapped. Put
$V_i=h(z_i)\sqrt{W(\tau(z_i),t(z_i))}$ and
$a_i=\zeta_T(z_i)^pm_B(z_i)$, so $\psi(z_i)=a_iV_i$.
Assign each physical point to a foot $i$ for which
$V_i\geq V_j$, breaking ties measurably; this partitions an integral,
not the function being differentiated. Write $r=V_j/V_i\leq1$. Then
\begin{equation}\label{eq:r11-switch-cutoff-difference}
 |a_iV_i-a_jV_j|^2\leq(1+\delta_s)a_i^2V_i^2(1-r)^2
       +C_{\delta_s}V_i^2\{ |\zeta_T(z_i)^p-\zeta_T(z_j)^p|^2
                             +|m_B(z_i)-m_B(z_j)|^2\}.
\end{equation}
The physical term is bounded by
Lemma~\ref{lem:r11-switch-physical-difference}, and the outer term
by \eqref{eq:r11-switch-outer-mass}. A base with $a_i=0$ contributes zero to the principal term.

\vskip0.12in

Fix an assigned base $z_i$ and one side. Its partner is the first
adjacent graph and is fixed during this normal traversal. Put
$\ell_i=d(z_i)/G_1$, $u=\ell_{ij}-\ell_i\geq0$. The moving-foot
geometry gives $d(z_j)\leq G_1(\ell_{ij}+o_R(1))$.
For uncapped reaches,
$|\partial_d\log\sqrt W|\leq q_*/2+\vartheta/(2d)$ and
$|\partial_\tau\log\sqrt W|\leq cq_*/2$.
Their reaches diverge uniformly, their observer radii have ratio
$1+o_R(1)$, and their logarithms differ by $O(\log\rho/\rho)$.
Monotonicity in $d$ and integration of these inequalities therefore give
\begin{equation}\label{eq:r11-switch-moving-ratio}
 r\geq(1-\epsilon_R)
       \exp[-(q_*G_1/2+\epsilon_R)u],\qquad\epsilon_R\to0,
\end{equation}
uniformly over the bands. If the receiving reach is smaller, its weight
only increases and the same conclusion holds. The additive chord
error changes the ratio by a factor tending to one.

\vskip0.12in

The sharp tail gives
\[
 q^2\leq64(1+.001/4)^2(1+o_R(1))
        e^{-\sqrt2\ell_{ij}}\cosh^2(\sqrt2t).
\]
Both the normal volume Jacobian and $ds_i/dt$ tend to one. The exact
finite-width calculation is
\begin{equation}\label{eq:r11-switch-normal-integral}
 \mathcal U'(t)=-\frac{\operatorname{sech}^2(\sqrt2t)}
                         {\sqrt2\tanh(3\sqrt2)},\qquad
 64\int_{-3}^3\cosh^2(\sqrt2t)\mathcal U'(t)^2\,dt
                  =\frac{32\sqrt2}{\tanh(3\sqrt2)}.
\end{equation}
Choose the finitely many fixed normal Young and geometry tolerances
so that their combined comparison factor is at most $(1.001)^6$.
The normal term, divided by the normal mass $\sigma$, is
therefore at most
\[
 \frac{24}{\tanh(3\sqrt2)}
       (1+.001/4)^2(1.001)^6
       e^{-\sqrt2\ell_{ij}}a_i^2V_i^2(1-r)^2
 \ <24.2 e^{-\sqrt2\ell_{ij}}\psi(z_i)^2(1-r)^2.
\]
All reciprocal Young constants multiply only the rough, cap, physical,
depth or outer errors already estimated, or a tangential band coefficient
that tends to zero.  Normal Jacobians and the exponential of the
additive distance error are included in these fixed factors by
increasing $R$.

\vskip0.12in

It remains to estimate this band term before concluding the test inequality.
There are at most two contributing sides at a fixed base. On a nearest
assigned side $u=0$ in \eqref{eq:r11-switch-moving-ratio}, hence
$r=1-o_R(1)$ for the fixed finite slope. The mass bound for a single selected side,
$I_P\geq47.999e^{-\sqrt2\ell_{ij}}$, bounds all such terms by
$o_R(1)\mathscr L_B$. Choose $R$ so large that their total is at most
$.001\mathscr L_B$.

\vskip0.12in

There is at most one nonnearest side. Its first-neighbor gap is
logarithmic because the band lies in the fixed depth window; the
opposite nearest gap is no larger. The choice $C_0>4M_{\rm dep}$
therefore includes both in the definition of $b_R$, and
\[
 e^{-\sqrt2\ell_{ij}}=\min(\mu_+,\mu_-)=b_R(z_i).
\]
Since $0\leq1-r\leq1$, its normalized contribution is at most
$24.2b_R(z_i)\psi(z_i)^2$. One normal traversal on that side and
measurable assignment to the larger $V_i$ give a total at most
$25\mathscr N_B$. We integrate over the assigned base throughout,
using its normal Jacobian once; the measurable assignment only
partitions this integral.

\vskip0.12in

The fixed tubes contribute at least $.999\sigma\mathscr L_B$,
while the tangential gradient contribution is at most
$1.005\sigma\int|\nabla\psi|^2$. The exact ambient inequality
\eqref{eq:r11-exact-ambient-q-test}, the normal estimate above and
the cap, rough-set, physical-cutoff, depth and outer estimates give
\eqref{eq:r13-physical-construction-output}, after subtracting the
$.001$ nearest-side allowance. Every absorption is at finite support.
\end{proof}

\vskip0.12in

\subsection{Comparison of the radial and stability inequalities}
\label{sec:r13-physical-master}

\begin{theorem}
\label{thm:s1-source-switch}\label{thm:r13-common-physical-test}
Assume the local estimates and the moment bound with exponent $9/4$, and use
the same $\psi,J_R$ from \eqref{eq:r9-common-amplitude} and
\eqref{eq:s8-main-positive-mass}. Assume the strict localization
conditions of Lemma~\ref{lem:s6-radial-localization}. At finite
outer support,
\begin{align}\label{eq:s8-main-positive-physical}
 .998\mathscr J_B&\leq1.005\mathscr D_B+\mathcal E_B^{\rm test},
 \notag\\
 \mathscr P_B&\leq.718\mathscr J_B+\mathcal E_B^{\rm rad},
\end{align}
where
\begin{align}\label{eq:r13-two-ledgers}
 \mathcal E_B^{\rm test}
 &=C\eta^{m-2}K(R)+o_{R;\eta}(\mathscr A)+e_B,\notag\\
 \mathcal E_B^{\rm rad}
 &=C\eta^{m-2}\mathscr A+C\eta^{m-3}\mathscr A/R
  +C\eta^{m-2p}\mathscr A/K(R)^{p-1}
  +o_{R;\eta}(\mathscr A)+e_B.
\end{align}
The little-oh is uniform in $x_c,B$, $e_B\to0$ at fixed
$R,\eta,x_c$, and the leading physical constant is independent
of small $\eta$. If the radial ground-state conditions hold and
$\delta_J=\gamma(.998)/1.005-.718>0$, then
\begin{equation}\label{eq:s8-main-positive-comparison}
 c_*\mathscr V_B+d_H\mathscr C_B+
 (\delta_J-\gamma a_m(.355)/R)\mathscr J_B
 \leq\mathcal E_B^{\rm rad}
       +\frac\gamma{1.005}\mathcal E_B^{\rm test}+\mathcal G_B,
\end{equation}
where $\mathcal G_B$ is the mask and small-reach error in
\eqref{eq:s7-raw-ground-lower}.
\end{theorem}
\begin{proof}
Proposition~\ref{prop:r13-physical-construction} gives
$.998\mathscr L_B\leq1.005\mathscr D_B+25\mathscr N_B+
\mathcal E_B^{\rm test}$. Subtract $25\mathscr N_B$ to obtain
the first inequality with exactly the density $J_R$.

\vskip0.12in

For the radial inequality, control the cap terms, terms with $t>M_0(L+\tau)$,
large-$at$ terms and rough-receiver terms using
Lemma~\ref{lem:s6-radial-localization}. Small-$at$
nonplateau terms and plateau drift contribute at most
$\vartheta|\nabla t|^2/(4t)$: choose their fixed flatness
threshold and contact accuracy sufficiently small, then increase
$R$. The local plateau term is at most
$.355G_1\mathscr J_B$ by \eqref{eq:s8-main-positive-density}.

\vskip0.12in

For the incoming term, apply the signed positive-foot area formula
to the subset where $H_{N_y,+}>0$. Rough receivers have already been estimated;
receivers with $a(y)t(z)>a_1$ are estimated using
$H_{N_y,+}\leq C/t(z)\leq C_{a_1}a(y)^{2p}t(z)^{2p-1}$.
On the remaining pairs the inverse Jacobian is at most $J_*$.
The reach order transfers the unmasked weight to its receiving
value with factor $1+o_R(1)$. Split physical and outer cutoff
mismatches before applying the mean bound, and bound their remainders using $H_{N_y,+}\leq C/t$, as in
\eqref{eq:r13-p-physical-foot} and \eqref{eq:r13-p-outer-mass}.
The principal incoming term is therefore at most
$.355G_1J_*(1+v)^2(1+o_R(1))\mathscr J_B$.
The subsets of base points and receiving points may both be extended to the whole
surface because $J_R\geq0$; no density is differentiated.
When combining the trace terms below, we enlarge the coefficient of
the nonnegative density $B_R$ before expanding its two terms.
With $v=.0005$ and $R$ large,
\[
 1+1.01(1.0005)^2<2.012,\qquad
 .355(1.001)(2.012)<.715.
\]
Before bounding $B_R$ by $.355J_R$, the principal local and
incoming trace sum $\mathcal T_B$ satisfies
\begin{equation}\label{eq:s9-principal-source-trace}
 \mathcal T_B\leq G_1(2.012)\int\psi^2B_R.
\end{equation}
The full negative neighbor term in this estimate will also be used
in the signed comparison below.
Here the upper branch is the absolutely continuous bound from the
weak capped-reach estimate. Multiplication by $-W$ makes its nonpositive singular Laplacian
contribute nonnegatively. The cap terms enter $\mathcal E_B^{\rm rad}$. Consequently
\begin{equation}\label{eq:s9-direct-positive-source}
 \int\psi^2(\Delta_\Sigma t)^{\rm upper}
 \leq\frac\vartheta4\int\psi^2|\nabla t|^2/t
       +.716\mathscr J_B
       +C\eta^{m-2p}\mathscr A/K(R)^{p-1}
       +o_{R;\eta}(\mathscr A)+e_B.
\end{equation}

\vskip0.12in

Test \eqref{eq:r9-weight-chain} against $\varphi\mathfrak m^2$
and integrate twice by parts at finite support. The identities
\eqref{eq:r9-radial-identities} produce exactly the constant and
$T_\rho$ coefficients of $\mathscr P_B$; substitution of
\eqref{eq:s9-direct-positive-source} leaves its
$3\vartheta/4$ gradient coefficient together with the negative
$J_1$ term. Physical, rough and outer derivatives have the
bounds of Lemma~\ref{lem:s6-radial-localization}.
The remaining mean term is bounded by
\[
 C\sup_{\rho\geq R}\frac{L+\log\rho}{\rho}\,\mathscr J_B
      =o_R(1)\mathscr J_B,
\]
using $|H|\leq.355J_R$ and
$(F+F_\tau)/W\leq C(L+\tau)$.
Absorb this between $.716$ and $.718$ to obtain the second
inequality in \eqref{eq:s8-main-positive-physical}.
Finally apply Lemma~\ref{lem:s7-raw-ground-state} with
$Q=J_R$, $h_0=.355$, and combine the two inequalities after
multiplying the stability inequality by $\gamma/1.005$.
This proves \eqref{eq:s8-main-positive-comparison} by comparing the
same positive density in the two finite-support inequalities.
\end{proof}

\vskip0.12in

\section{\bf The energy growth criterion}
\label{sec:s12-growth-criterion}

In this section we prove the energy growth bound from the radial
comparison established above. The main point is to obtain a lower
bound for the area on which the radial test remains nonzero. This proves
the criterion used in Section~\ref{sec:r13-actual-growth}.

\subsection{The radial comparison and a lower area bound}
\label{sec:r13-moment-radial-master}

\begin{theorem}
\label{thm:r13-moment-radial-master}
\label{thm:r13-joint-radial-criterion}
\label{prop:r13-radial-ledger-closure}
Let $3\leq m\leq8$, $p=9/4$, and let $u$ be a bounded
entire Allen--Cahn solution stable outside a fixed ball. Use the
local estimates, selected finite-window remaining mass and
natural moment of
Proposition~\ref{prop:r13-nine-dimensional-local-package} and
Corollary~\ref{cor:s8-density-envelope-moment}, with the fixed
penalty, cap and net of Section~\ref{sec:r9-radial-core}.
Suppose $\gamma,c,q_*,a_*>0$ satisfy
\begin{align}\label{eq:r13-joint-p-parameters}
 \gamma\frac{.998}{1.005}&>.718,&&\notag\\
 (m-2)c&>\gamma a_m^2,&\gamma q_*/4&<1,\notag\\
 q_*&<p\kappa,&q_*(2/\kappa-c)&>1,\\
 0<a_*&<c,&p\kappa a_*&<\min\{2p-1,m\}.\notag
\end{align}
Choose $\vartheta>0$ small enough that the constant condition
remains strict with $c_0=c(a_*/c)^\vartheta$ in place of $c$, and fix the contact accuracy $0<\omega\leq\vartheta/16$.
Then $K$ is bounded. Thus the corresponding energy has uniform
natural growth.
\end{theorem}

\begin{proof}
Suppose $K(R)\to\infty$ and use
\eqref{eq:r9-common-amplitude}, with $T=\eta R$,
$L=R/\sqrt{K(R)}$ and $\mathscr A=K(R)L^{1-\vartheta}$.
Lemma~\ref{lem:s6-radial-localization} bounds the mask and
small-reach errors in
\eqref{eq:s8-main-positive-comparison}. Absorbing its $R^{-1}$
mass term at finite support and dropping $d_H\mathscr C_B\geq0$
gives
\begin{equation}\label{eq:s1-final-radial-upper}
 c_*\mathscr V_B+(\delta_J/2)\mathscr J_B
 \leq C\eta^{m-2}\mathscr A+o_{R;\eta}(\mathscr A)
                  +e_{R,\eta,x_c}(B),
\end{equation}
where $e(B)\to0$ at fixed $R,\eta,x_c$.

\vskip0.12in

\emph{A lower bound for the good area.}
We next derive an opposite bound from the energy of a ball on which
the density is large.  Choose
$x_c$ so that $\mu(B_R(x_c))\geq a_0K(R)R^m$, for a fixed
$a_0>0$, along the unbounded-density sequence.  First remove
$B_{4T}(0)$, whose energy is at most
$C\eta^mK(R)R^m$ by centered monotonicity.  Choose a fixed
large $P_0$ so that energy at distance more than $P_0$ from
$\Sigma$ is less than a fixed small fraction of this selected
energy.  This choice is uniform over all centers: a separated net
of zeros has at most
$CK(R)R^m(1+j/R)^{m+1}$ elements in $B_{R+j+C}(x_c)$,
and phase decay bounds that far energy by
\[
 CK(R)R^m\sum_{j\geq P_0-C}(1+j)^C e^{-cj}.
\]
Remove also the $8\rho^{\beta_0}+2P_0$ rough collar outside
$B_{4T}(0)$. A rough center $y$ whose collar meets $B_R(x_c)$ obeys
$R\leq\rho(y)\leq\sqrt2R+C\rho(y)^{\beta_0}\leq2\sqrt2R$
for large $R$, since $\rho$ is $1$-Lipschitz and $\beta_0<1$.
Thus its observer radius $\rho(y)$ is comparable to $R$, uniformly in $x_c$, and
$R^{\beta_0}\ll T$ makes every needed center exterior to the
physical core. The preceding counting gives volume
$C_\eta K(R)R^{m-2p+b_p}=o(K(R)R^m)$; bounded $e_u$ gives
the same energy loss.  With $\eta$ below a fixed threshold and
$R$ large, a diffuse set of energy at least
$(a_0/2)K(R)R^m$ remains.

Every remaining diffuse point has a nearest zero at distance at
most $P_0$, lying in $B_{2R}(x_c)\setminus B_{3T}(0)$, with
$m_H=1$ and $D\to0$. Indeed, if $x$ is a retained diffuse point
and $z$ is a nearest zero, then $|x-z|\leq P_0$ and
$|\rho(x)-\rho(z)|\leq P_0$. The removed collar of width
$8\rho^{\beta_0}+2P_0$ therefore places $z$ beyond the mask's
$4\rho(z)^{\beta_0}$ collar for large $R$. It is regular, and distance minimization
on its graph makes the point a normal image $z+s\nu(z)$,
$|s|\leq P_0$.  The normal area formula in the direction which
counts \emph{all} preimages yields
\[
 (a_0/2)K(R)R^m
 \leq\int_{\mathcal G_R^0}\int_{-P_0}^{P_0}
        e_u(z+s\nu(z))|\det(I+sA)|\,ds\,dS_z
 \leq C_{m,P_0}\mathcal H^m(\mathcal G_R^0),
\]
where
$\mathcal G_R^0=\Sigma_g\cap B_{2R}(x_c)\cap\{|z|>3T,\ m_H=1\}$
is a measurable good base set covering all selected nearest zeros,
with unrestricted reach. Since $a^2\leq CD\to0$, the Jacobian is
uniformly bounded. Counting all preimages gives the required upper
bound even where distinct normal images overlap.

For removing large \emph{actual} reach, injectivity is needed and
follows directly from \eqref{eq:r9-chord-geometry}.  If
$x=z+s\nu(z)$, $d(z)>\ell$, and $|s|<c_\Gamma\ell$, then
for every different zero $y$,
\[
 |x-y|^2-|x-z|^2
 \geq[1-2C_\Gamma|s|/d(z)]|y-z|^2>0.
\]
Thus $z$ is the unique nearest zero, simultaneously over all
bases and both normal signs.  The curvature bound also gives a
positive Jacobian lower bound.  Disjoint images of bases in
$B_{2R}(x_c)$, with depths of length $c_\Gamma\ell$, fit in
$B_{3R}(x_c)$ and imply
\[
 \mathcal H^m\{z\in\Sigma_g\cap B_{2R}(x_c):d(z)>\ell\}
                                    \leq C_\Gamma R^{m+1}/\ell.
\]
Taking $\ell=L/2$ removes only
$C\sqrt{K(R)}R^m=o(K(R)R^m)$.  Hence the set
\begin{equation}\label{eq:r9-retained-set}
 \mathcal G_R=\{z\in\Sigma_{\rm reg}\cap B_{2R}(x_c):
             |z|>3T,\ m_H=1,\ d(z)\leq L/2\}
\end{equation}
has area at least $c_{\rm ret}K(R)R^m$, where $c_{\rm ret}>0$ is
independent of sufficiently small fixed $\eta$.

On $\mathcal G_R$, $\mathfrak{m}=1$ for $B\geq4R$,
$\widehat d\leq L/2$, and
\[
 F\geq\int_{L/2}^Lu^{-\vartheta}\,du
       =c_\vartheta L^{1-\vartheta},\qquad
 T_\rho\geq R^2/\rho^2\geq1/5,\quad
 \varphi T_\rho/\rho^2\geq cR^{-m}.
\]
Therefore
\begin{equation}\label{eq:r9-final-radial-lower}
 \int\varphi \mathfrak{m}^2F T_\rho/\rho^2\geq c_{\rm ret}'\mathscr A,
 \qquad c_{\rm ret}'>0,
\end{equation}
uniformly for sufficiently small fixed $\eta$.
Choose $\eta$ so $C\eta^{m-2}<c_*c'_{\rm ret}/2$, exhaust $B$ at
fixed $R,\eta,x_c$, and let $R$ increase.  The contradiction in
\eqref{eq:s1-final-radial-upper} bounds the density envelope.

\end{proof}

\begin{corollary}
\label{cor:r13-nine-dimensional-fixed-window}
\label{cor:r13-uniform-growth-certificate}
For every $3\leq m\leq8$, the criterion of
Theorem~\ref{thm:r13-moment-radial-master} admits
\[
 p=\frac94,\quad\gamma=\frac{363}{500},\quad c=\frac{547}{500},
 \quad q_*=\frac{157}{50},\quad a_*=\frac12,
 \quad\vartheta=\frac1{1000},\quad
 \alpha=\beta_0=\frac1{100}.
\]
One may prescribe the contact accuracy $\omega\leq1/32000$.

\end{corollary}
\begin{proof}
Put $M=m-2$ and $c_0=c(a_*/c)^\vartheta$.
Since $\log(547/250)<4/5$, one has
\[
 c_0>\frac{547}{500}\left(1-\frac1{1250}\right)
       =\frac{683203}{625000}.
\]
The first four positive terms of the series for $e^{4/5}$ already
give the stated logarithmic bound.
For $1\leq M\leq6$, the concave quadratic
$M(c_0-363M/2000)/2$ takes its smallest endpoint value at $M=6$;
therefore
\[
 d_H=\frac M2(c_0-\gamma M/4)
       >\frac{7734}{625000}>\frac3{250}.
\]
The coefficient margin $\delta_J$ is $493/167500$ and
$\gamma q_*/4=56991/100000<1$, so the explicit
$C_\gamma$ in Lemma~\ref{lem:s7-raw-ground-state} is finite.
The positive-foot Jacobian has already entered the coefficient $.718$ in the radial
inequality and does not multiply this gradient term.
Uniformly for $\kappa\in[\sqrt2/1.001,\sqrt2]$,
\[
 p\kappa-q_*>.0387,\qquad
 q_*\{2(1-10^{-4})/\kappa-c\}-1>.00498.
\]
Moreover $p\kappa a_*\leq9\sqrt2/8<8/5<3
\leq\min\{2p-1,m\}$, and
\[
 p\alpha+(m+1)\beta_0\leq\frac9{80}<1
                 \leq\min\{2p-3,m-2\}.
\]
Finally $\omega\leq\vartheta/16$. Every required margin is strict.
\end{proof}

\vskip0.12in

\appendix

\section{\bf The isolated-layer mean-curvature estimate}
\label{app:isolated-parity}

\begin{lemma}
\label{lem:s7-normal-darboux}\label{lem:s2-normal-schur-trace}
Let $L_t=-\partial_t^2+3g^2-1$ and suppose
\[
 (-\Delta_y+L_t)\phi=Hg'+E,\qquad \phi(y,0)=0.
\]
For each requested horizontal multi-index $\alpha$, assume
$\partial_y^\alpha\phi$, $\partial_t\partial_y^\alpha\phi$ and
$\partial_y^\alpha E$ are bounded on the whole space. Set
$E_e(y,t)=[E(y,t)+E(y,-t)]/2$. Then
\begin{align}
 \|\partial_y^\alpha\phi\|_\infty+
 \|\partial_t\partial_y^\alpha\phi\|_\infty
 &\leq C\|\partial_y^\alpha E\|_\infty,
 \label{eq:s7-darboux-profile-bound}\\
 \|\partial_y^\alpha H\|_\infty
 &\leq C(\|\partial_y^\alpha E_e\|_\infty+
          \|\partial_t\partial_y^\alpha E_e\|_\infty),
 \label{eq:s7-darboux-mean-bound}
\end{align}
where the second bound additionally assumes its right side finite.
In these linear estimates, $C$ may depend on $m$, the requested
horizontal derivative order and the fixed profile $g$.
\end{lemma}
\begin{proof}
By \eqref{eq:s7-darboux-factorization}, $Z=\mathcal D\phi$ solves
$L_0Z=\mathcal DE$, where $L_0=-\Delta+2-G$. For
$R_2=(-\Delta+2)^{-1}$ the bounded distributional identity is
\[
 Z=\partial_tR_2E+R_2(\sqrt2gE)+R_2(GZ),\qquad
 \|R_2\|_{\infty\to\infty}=\tfrac12,
 \quad\|\partial_tR_2\|_{\infty\to\infty}\leq1/\sqrt2.
\]
The last estimate follows by integrating
$e^{-2s}(\pi s)^{-1/2}$ in the heat representation. Absorption gives
$\|Z\|_\infty\leq2\sqrt2\|E\|_\infty$; the reconstruction formula
\eqref{eq:s7-pinned-recovery}, with $\phi(y,0)=0$, and $\phi_t=Z-\sqrt2g\phi$ give
the first bound directly in terms of $E$.

\vskip0.12in

The maximum principle for $2-G\geq1$, followed by unit-scale
$W^{2,p_0}$ estimates with $p_0>m+1$, gives
\begin{equation}\label{eq:s7-positive-partner-bound}
 \|U\|_\infty+\|\nabla U\|_\infty\leq C\|L_0U\|_\infty
 \quad(U,L_0U\text{ bounded}).
\end{equation}
Since $\mathcal D$ reverses parity, $L_0Z_o=\mathcal DE_e$.
At $t=0$, pinning gives $\Delta_y\phi=0$, $Z_t=\phi_{tt}$, hence
\begin{equation}\label{eq:s7-darboux-mean-trace}
 H=-\sqrt2\{E_e(y,0)+\partial_tZ_o(y,0)\}.
\end{equation}
This and \eqref{eq:s7-positive-partner-bound} prove the second bound.
Horizontal differentiation commutes with every step.
\end{proof}

\vskip0.12in
We next explain how to apply these estimates on nested horizontal boxes
of radius $R_h$.
After the smooth normal cutoff and extension, take $\chi=1$ on an
intermediate box, supported strictly inside the larger box where the
initial profile bounds hold.
For $\partial_y^\alpha\phi$, multiplication by this cutoff
changes its error source to $F+C$, where
$F=\chi\partial_y^\alpha E$ and
$C=-2\nabla\chi\cdot\nabla\partial_y^\alpha\phi
-(\Delta\chi)\partial_y^\alpha\phi$.
The transformed solution generated by $F$ is defined by
\[
 Z_F=(I-R_2G)^{-1}
       \{\partial_tR_2F+R_2(\sqrt2gF)\}.
\]
Since $\|R_2G\|_{\infty\to\infty}\leq1/2$, this gives
$\|Z_F\|_\infty\leq C\|F\|_\infty$; the resolvent derivative acts
on its kernel, so $\partial_tF$ does not enter this bound. Since $\chi$ depends only on $y$, $\mathcal D$
annihilates $\chi(\partial_y^\alpha H)g'$. The remaining transformed
solution satisfies
$L_0Z_C=\mathcal DC$, with source supported at distance $cR_h$
from the inner box. Kato's inequality gives
$|L_0^{-1}\mathcal DC|\leq(-\Delta+1)^{-1}|\mathcal DC|$.
Its exponential tails, unit interior gradient estimates and the reconstruction
formula \eqref{eq:s7-pinned-recovery} make this contribution $O(h^Me^{-cR_h})$ when the
commutator derivatives are $O(h^M)$. At horizontal order four,
$\mathcal DC$ uses at most six total profile derivatives.
The mean estimate uses the same argument with the stated bound on
the normal derivative of the even part of the source. The fourth-order profile estimate combines the norm of
the fourth horizontal source derivative with the exponentially small
commutator term estimated above; it requires no smallness assumption on
$\partial_t\partial_y^4E$.

\begin{proof}
We prove Theorem~\ref{thm:actual-isolated-parity-jets}. We work on inner portions of the graph. At each application of an
interior estimate, shrink the portion by a fractional margin chosen
independently of $h$. The constants may depend on the geometric bounds,
derivative orders and these margins, independently of $h$. The hypotheses give
\begin{equation}\label{eq:parity-starting-data}
 |A|+h^{1/2}[A]_{C^{1/2}}\leq C/h,\qquad
 \|H\|_{C^k}\leq C_kh^{-2}\quad(k\text{ fixed}).
\end{equation}
The defect bound gives $|\nabla u|\geq c$ on the graph; ambient
interior estimates and the implicit function theorem give uniform
fixed-order unit-chart derivatives. Rotate at an inner center and
use its tangent graph $X(y)$ and Fermi coordinates $X(y)+t\nu(y)$.
Nested boxes of radius $R_h=h^\epsilon\ll h$ are injective and
isolated; the curvature bound and small slope give distance to the
whole zero set comparable to $|t|$ on inner boxes.

\vskip0.12in

\emph{A quadratic graph comparison.}
Write the prescribed-mean-curvature equation as
$a^{ij}(Df)f_{ij}=-H$, with
$a^{ij}(p)=\delta_{ij}/\sqrt{1+|p|^2}-p_ip_j/(1+|p|^2)^{3/2}$.
For the quadratic Taylor polynomial $P$ at a unit-chart center, put
$B=D^2P$ and $v=f-P$. Then $\|v\|_\infty\leq Ch^{-3/2}$ and
\[
 a^{ij}(Df)v_{ij}+b^lv_l=-H-a^{ij}(DP)B_{ij},\qquad
 b^l=B_{ij}\int_0^1a^{ij}_{p_l}(DP+\theta Dv)\,d\theta.
\]
Uniform $C^{6,1/2}$ graph bounds give uniformly elliptic
$C^{5,1/2}$ coefficients. The right side vanishes at the center
and has $C^{5,1/2}$ norm $O(h^{-2})$. Indeed, on the fixed unit chart,
the coarse derivatives of $H$ through order six and the uniform graph
bounds give
\[
 \|H-H(0)\|_{C^{5,1/2}}
 +\|a^{ij}(DP)B_{ij}-a^{ij}(DP(0))B_{ij}\|_{C^{5,1/2}}
 \leq Ch^{-2};
\]
every horizontal derivative of positive order in the second term contains two factors
$B=O(h^{-1})$. At the center, $a^{ij}(DP(0))B_{ij}=-H(0)$, so the constants cancel.
One Schauder estimate gives
$\|v\|_{C^{7,1/2}}\leq Ch^{-3/2}$, and therefore
\begin{equation}\label{eq:parity-curvature-interpolation}
 \|\partial_y^kA\|_{C^{0,1/2}}\leq Ch^{-3/2},\qquad1\leq k\leq5.
\end{equation}
Metric and frame derivatives add products of two small graph factors.

\vskip0.12in

\emph{Initial estimate for the profile error.}
Choose the upper positive phase and put
$\phi(y,t)=u(X(y)+t\nu(y))-g(t)$. For every fixed $M>0$, the fixed-window defect estimates give
$|\phi|\leq C_Mh^{-2}$ on $|t|\leq M$, uniformly in the inner
horizontal centers. On $M\leq|t|\leq R_h$, the isolated graph and
the fixed fractional margins give
$\operatorname{dist}(X(y)+t\nu(y),\Sigma)\geq c|t|$.
The phase-decay estimate therefore gives $|\phi|\leq Ce^{-cM}$
on this part. First let $h\to\infty$ with $M$ fixed and then let
$M\to\infty$; this proves $\phi=o(1)$ uniformly on the growing
boxes. At the distant normal faces, phase decay and the fixed-order
Fermi-coordinate bounds give $O(h^{M_0}e^{-cR_h})$ for the requested
derivatives, with a fixed finite $M_0$. The exact equation is
\begin{align}\label{eq:parity-exact-fermi}
 (-\Delta_y+L_t)\phi={}&K_tg'+K_t\phi_t+
 (g_t^{ij}-\delta^{ij})\phi_{ij}+b_t^i\phi_i-3g\phi^2-\phi^3,
 \\
 K_t={}&\operatorname{tr}\{A(I+tA)^{-1}\}.\notag
\end{align}
Here subscripts $t$ on metric coefficients denote the parallel metric.
Its matrix expansion and the tangent slope $O(h^{-1+\epsilon})$ give
\begin{equation}\label{eq:parity-coefficient-ledger}
 \begin{gathered}
 K_t=H-t|A|^2+t^2\operatorname{tr}A^3+O(h^{-4+C\epsilon}),
 \qquad K_t=O(h^{-2+C\epsilon}),\\
 g_t^{ij}-\delta^{ij},\ b_t^i=O(h^{-1+C\epsilon})
 \end{gathered}
\end{equation}
through the derivative orders used below. Fix the derivative orders,
cutoff shapes, inner margins and auxiliary comparison parameters independently
of $\epsilon$. For some $\epsilon_0>0$, every exponent loss in the
finite-order estimates below is at most $C_{\rm exp}\epsilon$ for
$0<\epsilon\leq\epsilon_0$, with $C_{\rm exp}$ independent of $\epsilon$.
Multiplicative constants and the lower threshold for $h$ may depend on
the chosen $\epsilon$.

On either normal half-box use $v_0$ from
Lemma~\ref{lem:r5-sharp-barrier}, with $L_tv_0=1$ and
$0<c\leq v_0\leq C$, reflected separately on the lower half-line.
The comparison is carried out on each half-line separately. The
operator in the exact equation, with
the metric, drift and small nonlinear potential moved left, sends
$v_0$ to at least $1/2$. For small fixed $c_0>0$ the face functions
$v_0e^{-c_0\operatorname{dist}(\cdot,F)}$ remain supersolutions:
their additional terms are $O(c_0+c_0^2)$.
Their sum plus $Ch^{-2+C\epsilon}v_0$ dominates both signs of
$\phi$, which vanishes at $t=0$. Division by $v_0$ gives the
comparison principle. Inner face errors are exponentially small;
semilinear unit-scale Schauder estimates for the smooth equation,
including across $t=0$, give
\begin{equation}\label{eq:parity-profile-initial}
 \|\phi\|_{C^{6,1/2}}\leq Ch^{-2+C\epsilon}.
\end{equation}
The source and principal coefficients need four derivatives, while
$b_t$ needs curvature through order five, all supplied above.

\vskip0.12in

\emph{Two source estimates.}
For an even cutoff $\chi_h(t)$ whose derivatives lie at $|t|\asymp R_h$,
extend $\Phi=\chi_h\phi$ by zero. Its exact error source is
$E=\chi_hE_{\rm raw}+(\chi_h-1)Hg'-2\chi_h'\phi_t-\chi_h''\phi$,
where $E_{\rm raw}$ is the preceding right side minus $Hg'$.
All added terms are superpolynomially small. Restore the leading odd
source on the whole normal line and extend the coefficients outside
the cutoff interval smoothly by their finite matrix expansions;
all differences enter $\mathcal R$. Writing $\phi$ for $\Phi$,
\begin{equation}\label{eq:s3-parity-source-decomposition}
 \begin{aligned}
 E&=\mathcal O+\mathcal C+\mathcal B+\mathcal N+\mathcal R,\\
 \mathcal O&=-t|A|^2g',&\mathcal C&=(K_t-H+t|A|^2)g',\\
 \mathcal B&=(g_t^{ij}-\delta^{ij})\phi_{ij}+b_t^i\phi_i,
 &\mathcal N&=K_t\phi_t-3g\phi^2-\phi^3.
 \end{aligned}
\end{equation}
The remainder and its requested derivatives are $O(h^Me^{-cR_h})$.
For $1\leq|\alpha|\leq4$, a derivative of $|A|^2$ has factors
$h^{-3/2}$ and $h^{-1}$; $\mathcal C$ starts cubically in $A$,
$\mathcal B$ is $O(h^{-3+C\epsilon})$, and $\mathcal N$ is
$O(h^{-4+C\epsilon})$. Hence
\begin{equation}\label{eq:parity-first-error}
 \sup|\partial_y^\alpha E|\leq Ch^{-5/2+C\epsilon},
 \qquad1\leq|\alpha|\leq4.
\end{equation}
The highest derivative of the profile error required in this estimate is $\partial_y^6\phi$.
The localized first estimate of the lemma gives
\begin{equation}\label{eq:parity-positive-profile-gain}
 \sup(|\partial_y^\alpha\phi|+|\partial_t\partial_y^\alpha\phi|)
 \leq Ch^{-5/2+C\epsilon},\qquad1\leq|\alpha|\leq4.
\end{equation}
The source involving $H$ is eliminated by $\mathcal D(Hg')=0$
before the horizontal derivative estimates are applied.

\vskip0.12in

The leading term $\mathcal O$ is odd. Taking the even part $E_e$ of $E$,
\begin{equation}\label{eq:s7-parity-even-zero}
 \sup(|E_e|+|\partial_tE_e|)\leq Ch^{-3+C\epsilon}.
\end{equation}
A horizontal derivative of order one or two of the cubic curvature term in $\mathcal C$ has size $O(h^{-7/2})$;
every profile factor in $\mathcal B$, through horizontal order two
and normal order one, has horizontal derivative order between one and four
and now gains $h^{-1/2}$. The metric identity
$g_t=g_0(I+tA)^2$ produces a curvature factor under normal
differentiation; $\mathcal N,\partial_t\mathcal N$ remain
$O(h^{-4+C\epsilon})$, using at most $\partial_y^2\phi_{tt}$.
Taking the normal derivative after the even part therefore gives
\begin{equation}\label{eq:parity-second-error}
 \sup(|\partial_y^\alpha E_e|+|\partial_t\partial_y^\alpha E_e|)
 \leq Ch^{-7/2+C\epsilon},\qquad1\leq|\alpha|\leq2.
\end{equation}
Applying \eqref{eq:s7-darboux-mean-bound} for $1\leq|\alpha|\leq2$ gives the same bound for $H_y,H_{yy}$;
the zeroth-order case gives $|H|\leq Ch^{-3+C\epsilon}$.
For prescribed $0<\gamma<1/2$, choose $0<\epsilon\leq\epsilon_0$
after the fixed choices above so that
$C_{\rm exp}\epsilon<1/2-\gamma$. Bounded connection coefficients
give the asserted covariant derivative estimates. For a prescribed
$\varepsilon_1>0$, the same argument with the additional restriction
$C_{\rm exp}\epsilon<\varepsilon_1$ gives the stated estimate for $H$,
with its own constants and lower threshold for $h$.

\vskip0.12in

Finally rescale the original graph by $h^{-1}$. Its $C^{2,1/2}$ norm
is uniform, and $hH$ has bounded $C^{1,1}$ norm because its first
two derivatives scale by $h^2,h^3$. Differentiating its prescribed
mean-curvature equation once and applying Schauder gives a uniform
third graph derivative, hence $|\nabla_\Sigma A|\leq Ch^{-2}$.
All centers use the same fractional margins.
\end{proof}

\vskip0.12in

\section{\bf Scalar feasibility for the radial comparison}
\label{app:s6-scalar-feasibility}

\begin{theorem}\label{thm:s2-exact-feasibility}
For $m\geq3$, $p\geq2$, $\kappa>0$ and the coefficients of
\eqref{eq:s2-recovery-input}, put
$g_0=\kappa dA/s$ and $g_1=\kappa dB/b$ when $b>0$, otherwise
$g_1=\infty$. The following scalar conditions
\begin{gather}
 \gamma s/d>A,\quad\gamma b/d\leq B,\quad
 c>\gamma(m-2)/4,\notag\\
 \gamma q_*/4<1,\quad q_*<p\kappa,\quad q_*(2/\kappa-c)>1
 \label{eq:s2-scalar-conditions}
\end{gather}
are feasible exactly when
\begin{equation}\label{eq:s2-exact-gate}
 g_0<\min\{g_1,8/(m-1),(8-4/p)/(m-2)\}.
\end{equation}
All auxiliary strict localization margins can then be chosen if
contact accuracy is arbitrarily prescribable.
\end{theorem}
\begin{proof}
Set $M=m-2$, $g=\gamma\kappa$, $C=c\kappa$, $U=\kappa/q_*$.
The last four conditions become
\[
 C>Mg/4,\qquad U>\max(g/4,1/p),\qquad C+U<2.
\]
They hold for some $C,U$ exactly when
$g<\min(8/(M+1),(8-4/p)/M)$; the first two conditions in
\eqref{eq:s2-scalar-conditions} require
$g>g_0$ and $g\leq g_1$. This proves necessity.
Conversely choose $g$ strictly between $g_0$ and the right side
of \eqref{eq:s2-exact-gate},
put $h_g=\max(g/4,1/p)$, $r_g=2-Mg/4-h_g>0$, and take
\[
 C=Mg/4+r_g/3,\qquad U=h_g+r_g/3.
\]
The four normalized radial, gradient, moment and outer margins
are at least $r_g/3$.
Recover $\gamma,c,q_*$, choose $a_*>0$ sufficiently small that
$a_*<c$ and $p\kappa a_*<\min(2p-1,m)$, then choose
$\vartheta>0$ small so $c(a_*/c)^\vartheta>\gamma M/4$.
Finally take small $\alpha,\beta_0>0$ with
$p\alpha+(m+1)\beta_0<\min(2p-3,M)$ and prescribe
$\omega\leq\vartheta/16$. These are parameter choices only;
the moment and physical tests remain separate analytic inputs.
\end{proof}

\vskip0.12in

For the actual coefficients $A=.718$, $s=.998$, $d=1.005$,
$B=80/\kappa$, $b=25$, one has $g_1=402/125$ and
$g_0\leq(72159/99800)\sqrt2<28/27$, the most restrictive moment
threshold for $m\leq8$, $p=9/4$. The explicit uniform parameter
choice is verified in Corollary~\ref{cor:r13-uniform-growth-certificate}.

\section{\bf The dimension restriction in the radial comparison}
\label{sec:r13-final-gap}
In this section we examine the dimension range of the radial argument.
The scalar criterion in Appendix~\ref{app:s6-scalar-feasibility}
concerns the chosen family of test functions. Its failure does not give
an obstruction for solutions of the equation.

\begin{corollary}
\label{prop:r13-ten-dimensional-obstruction}\label{cor:s2-combined-obstruction}
For fixed coefficients in \eqref{eq:s2-recovery-input}, scalar feasibility for some
$2\leq p<p_{\max}(m)=2+\sqrt{2/m}$ is equivalent to
\[
 g_0<\min\{g_1,8/(m-1),(8-4/p_{\max}(m))/(m-2)\}.
\]
In ambient dimension ten, the moment endpoint is
\begin{equation}\label{eq:s2-combined-threshold}
 (8-4/(2+\sqrt2/3))/7=(100+6\sqrt2)/119<1;
\end{equation}
at $p=9/4$ it is $8/9$. Even arbitrary moments do not give feasibility
for $m\geq9$ with the ideal coefficients $s=d=1$, $A=1/\kappa$,
$g_1=\infty$.
\end{corollary}
\begin{proof}
Take the increasing union of the open scalar intervals in
\eqref{eq:s2-exact-gate}. For the ideal coefficients $g_0=1$,
whereas the all-moment condition $g_0<8/(m-1)$ fails for $m\geq9$.
\end{proof}
Only the $p=9/4$ full moment is proved here; taking this scalar union
does not prove the weighted stability or profile estimates at other powers.

\vskip0.12in

\subsection{A conditional estimate for the signed deficit}
We give a conditional form of the radial argument that keeps the
signed constant coefficient. At the same finite support, define
\begin{align}\label{eq:s1-radial-positive}
 \mathscr Q_B=\int\varphi\mathfrak m^2\bigg[
 &\frac{A_0W+\{m(m-2)F+B_0W\}T_\rho}{\rho^2}\notag\\
 &+\frac{3\vartheta W}{4t}|\nabla t|^2+
       \frac{W\lambda}{b_L'}|\nabla s|^2\bigg],
\end{align}
with the last quotient zero when $b_L'=0$.

\begin{lemma}
\label{lem:s3-signed-ground-state}\label{lem:s2-trace-reserve}
Put $M=m-2>0$, $a_m=M/2$, $B_0=2A_0/M$ and
$\Delta_0=A_0-\gamma M^2/4$, of either sign. If
$\gamma,\vartheta>0$, $\gamma q_*/4<1$, $t_{\min}(R)\to\infty$
and $|H|\leq h_0I_P$, then at finite support, for large $R$,
\begin{equation}\label{eq:s3-signed-ground-state}
 \mathscr Q_B-\gamma\mathscr D_B\geq
 c_*\mathscr V_B+\Delta_0\mathscr C_B
 -\gamma a_mh_0\mathscr L_B/R
 -C_\gamma\int\varphi W|\nabla\mathfrak m|^2,
\end{equation}
where $c_*=mM(1-\gamma q_*/4)/2$ and
$C_\gamma=2\gamma/(1-\gamma q_*/4)$ are independent of small $\eta$.
\end{lemma}
\begin{proof}
The constant coefficient is $\Delta_0W$, and
\begin{equation}\label{eq:s2-trace-identity}
 mMF+\{B_0-\gamma a_m(a_m+2)\}W
 =mM(F-\gamma W/4)+(2\Delta_0/M)W.
\end{equation}
Thus the signed term is exactly $\Delta_0\mathscr C_B$.
Use $W\leq(q_*+o_R(1))F$, the mask inequality
\eqref{eq:s7-mask-schur}, and the mean bound from the exact
ground-state identity. The signed coefficient remains explicit
throughout the calculation.
\end{proof}

\begin{lemma}\label{lem:s7-frozen-recovery}
The established localization and radial estimates imply, whenever
$A_0<Mc_0$ and $B_0=2A_0/M$,
\[
 \mathscr Q_B\leq.718\mathscr L_B-(80/\kappa)\mathscr N_B
                  +\mathcal E_B^{\rm rad}.
\]
\end{lemma}
\begin{proof}
Section~\ref{sec:r13-physical-master} bounds the principal trace
by $G_1(2.012)\int\psi^2B_R$, where
$B_R=.355I_P-24\sqrt2b_R$. Substituting $24\sqrt2G_1=48/\kappa$ gives
\[
 \mathcal T_B\leq.715\mathscr L_B
                    -(96.576/\kappa)\mathscr N_B.
\]
The remaining radial mean term is bounded by
$\varepsilon_R\mathscr L_B$, where $\varepsilon_R\to0$, using
$|H|\leq.355I_P$. For $\varepsilon_R<.003$, the form $\mathscr P_B$ therefore
satisfies
\[
 \mathscr P_B\leq.718\mathscr L_B-(80/\kappa)\mathscr N_B
                    +\mathcal E_B^{\rm rad},
\]
with strict margins in both coefficients.
To estimate $\mathscr Q_B$, subtract \eqref{eq:s7-raw-radial-form} from
\eqref{eq:s1-radial-positive}. On $t\geq a_*\tau$, the strict
coefficient inequalities and $J_1|\nabla\tau|^2\leq CF_\tau/(\tau\rho^2)$
absorb the difference. On the complement its absolute value is at most
$C\int_{t<a_*\tau}\varphi\mathfrak m^2W/\rho^2=o_{R;\eta}(\mathscr A)$,
already bounded by the localization estimates. This proves the lemma.
\end{proof}

\vskip0.12in

\begin{lemma}
\label{thm:s2-finite-recovery}\label{cor:s3-deficit-absorption}
Let $A,B,b\geq0$, $s,d>0$, and at each finite support suppose
\begin{equation}\label{eq:s2-recovery-input}
 \mathscr P_B\leq A\mathscr L_B-B\mathscr N_B+\mathcal E_r,
 \qquad s\mathscr L_B\leq d\mathscr D_B+b\mathscr N_B+\mathcal E_t,
\end{equation}
where $\mathscr L_B,\mathscr N_B,\mathscr D_B\geq0$.
If $\delta_*=\gamma s/d-A>0$ and $\chi_*=B-\gamma b/d\geq0$, then
\begin{equation}\label{eq:s2-recovery-output}
 \mathscr P_B+\delta_*\mathscr L_B+\chi_*\mathscr N_B
 \leq\gamma\mathscr D_B+\mathcal E_r+(\gamma/d)\mathcal E_t.
\end{equation}
Suppose additionally
\begin{equation}\label{eq:s2-structural-closure-input}
 \mathscr P_B-\gamma\mathscr D_B\geq c_*\mathscr V_B
 -\mathscr H_B-\varepsilon_R\mathscr L_B-\mathcal E_g,\qquad
 \mathscr H_B\leq r_L\mathscr L_B+r_N\mathscr N_B+\mathcal E_H,
\end{equation}
along a sequence $R\to\infty$ with $\mathscr A_R>0$,
fixed $0\leq r_L<\delta_*$, $0\leq r_N\leq\chi_*$ and
$\varepsilon_R\to0$. If $\mathscr V_B\geq c_{\rm ret}\mathscr A_R$
and the total error is at most
$C\eta^{m-2}\mathscr A_R+o_{R;\eta}(\mathscr A_R)+e_{R,\eta}(B)$,
where $e_{R,\eta}(B)\to0$ for each fixed $R,\eta$, $m>2$, and
$c_*,c_{\rm ret}>0$ and $C$ are uniform for small fixed $\eta$,
then no such sequence exists.
\end{lemma}
\begin{proof}
Multiply stability by $\gamma/d$ and substitute the upper bound for $\mathscr P_B$.
Combining the resulting inequality with the lower bound gives
\[
 c_*\mathscr V_B+(\delta_*-r_L-\varepsilon_R)\mathscr L_B
 +(\chi_*-r_N)\mathscr N_B
 \leq\mathcal E_r+(\gamma/d)\mathcal E_t+\mathcal E_g+\mathcal E_H.
\]
Absorb $\varepsilon_R$ at finite support, choose $\eta$ small,
exhaust $B$, and then increase $R$. Division by $\mathscr A_R>0$
contradicts the lower area bound. The case $\mathscr H_B=0$
is the ordinary recovery principle.
\end{proof}

\vskip0.12in

For extending the radial argument beyond the dimension range of
Theorem~A, an additional estimate remains unproved. With
$\lambda_{\rm H}=(-\Delta_0)_+$, it is
\begin{equation}\label{eq:s2-missing-coercivity}
 \lambda_{\rm H}\mathscr C_B\leq
 (\delta_*/4)\mathscr L_B+(\chi_*/2)\mathscr N_B
                    +o_{R;\eta}(\mathscr A).
\end{equation}
If this estimate and the other higher-dimensional inputs hold, the
preceding lemma gives the same contradiction with the area lower bound.
The estimate remains an additional hypothesis.

\begin{proposition}\label{prop:s2-no-local-coercivity}
Fix $P<\infty$ and selection depth $C_0\log R$. Clean-window
solutions without a stability hypothesis can satisfy
$b_{R_j}=0$ and $R_j^2I_P\to0$. Thus these local hypotheses imply
no uniform bound $R^{-2}\leq C(I_P+b_R)$.
\end{proposition}
\begin{proof}
Take the periodic solutions of Proposition~\ref{prop:periodic-sharpness48}
with adjacent gaps $\ell_j\to\infty$, and set $R_j=e^{\ell_j/C}$
for $C>\max(C_0,\sqrt2)$. Both neighbors are omitted and
$I_P=O(e^{-\sqrt2\ell_j})$, so $R_j^2I_P\to0$.
Their sheets have zero curvature and defect $O(R_j^{-\sqrt2 C})$;
any $0<\omega<\sqrt2 C$ verifies the clean condition on the entire
polynomial window. Hence new global information is needed for
\eqref{eq:s2-missing-coercivity}.
\end{proof}

\vskip0.63in


\section*{{\bf Declarations}}

\noindent\textbf{Funding.} {W. Zou is supported by the National Key R\&D Program of China (Grant No.~2023YFA1010001) and the National Natural Science Foundation of China (Grant No.~12171265).}

\smallskip

\vskip0.12in
\noindent\textbf{Availability of data and materials.} No datasets were generated or analysed during the current study.

\vskip0.12in
\smallskip
\noindent\textbf{Conflict of interest.} The authors declare no conflict of interest.

\vskip0.12in

\smallskip
\noindent\textbf{Use of artificial intelligence.}
The authors acknowledge the use of AI tools. The authors assume full
responsibility for the mathematical validity, accuracy, and integrity
of the proofs presented in the manuscript.

\end{document}